%% file: thesis_graduation.tex
\documentclass[leqno]{article}

\usepackage{amsmath}
\usepackage{amssymb}
\usepackage{amsfonts}
\usepackage{amsthm}
\usepackage{textcomp}
\usepackage{float}
\usepackage{afterpage}
\usepackage{type1cm}

\usepackage{color}
\usepackage{wrapfig}
\usepackage{enumerate}
\usepackage{graphicx}

\def\qed{\vbox{\hrule\hbox to 6pt{\vrule height 6pt\hfill\vrule}\hrule}\vskip6mm} 

\newtheorem{thm}{Theorem}[section]

\newtheorem{cor}[thm]{Corollary}
\newtheorem{defi}{Definition}[section]

\numberwithin{equation}{section}

\begin{document}
\frenchspacing

\title{\textbf{Formulation of  Convergence and Continuity
in Variation of Sets
in a Different Way from Sequences of Sets
and Correspondence}}
\date{}
\author{By Takefumi Fujimoto\thanks{
%Department of International Studies,
%School of Frontier Sciences, The University of Tokyo
%\,\,
E-mail address: the-gift-of-31t@hotmail.co.jp; 2876681407@edu.k.u-tokyo.ac.jp}}
\maketitle

\vspace*{-25pt}

\begin{abstract}
This paper treats the variation of sets.
We attempt to formulate convergence and continuity of set-valued functions
in a different way from the theories on sequences of sets and correspondence.
In the final section, we also attempt to define differentiation of set-valued functions
in a Euclidean space by bijection between two sets.
\end{abstract}
\begin{center}
\textit{Key Words and Phrases}
\end{center}
\vspace{-6pt}
set-valued functions, sequences of sets, correspondence,
convergence of set-valued functions at infinity,
convergence of set-valued functions at a point,
continuity of set-valued functions at a point,
differentiation of set-valued functions in a Euclidean space

\tableofcontents

\section{Introduction}
\label{intro}

%{\fontsize{6pt}{6pt} \selectfont aaa}

This paper studies a theory on the variation of sets.
In the paper, we deal with a set-valued function
that are defined on the real line. The paper focuses on the following three
concepts: \textit{convergence of set-valued functions at infinity, convergence of
set-valued functions at a point, and continuity of set-valued functions
at a point}.
We define these concepts in a different way from 
the theories on sequences of sets and correspondence,
and investigate properties of them.

We should mention the convergence of sequences of sets
as previous works on the variation of sets.
Let $\{ A_n \}_{n=1}^\infty$ be a sequence of sets
$A_1, A_2, \dotsc, A_n, \dotsc$ and let $A$ be a set.
Then, we say that $\{ A_n \}_{n=1}^\infty$ \textit{converges to} $A$
if {\tiny $\displaystyle \bigcap_{n=1}^{\infty} \bigcup_{k=n}^{\infty}$} $A_k$ =
$A$ = {\tiny
$\displaystyle \bigcup_{n=1}^{\infty} \bigcap_{k=n}^{\infty}$} $A_k$.
The set {\tiny $\displaystyle \bigcap_{n=1}^{\infty} \bigcup_{k=n}^{\infty}$} $A_k$ is called the \textit{superior limit}, and
the set {\tiny
$\displaystyle \bigcup_{n=1}^{\infty} \bigcap_{k=n}^{\infty}$} \nolinebreak$A_k$
is called the \textit{inferior limit}. Here, we point two facts.
First, the convergence of sequences of sets is characterized by
the equality of the two limits. Second, sequences of sets can be interpreted
as the dependence of a set on natural numbers.

This paper differs from the study of sequences of sets in three ways.
First, the paper considers a set that changes depending on a real variable,
a set-valued function.
We use notations $A(t)$, $B(t)$, $C(t)$, and so forth to denote set-valued functions
that depend on the real variable $t$.
Second, convergence of set-valued functions at infinity, which we shall deal with in the paper,
is defined by a statement.
Note that the notation $A \triangle B$
represents the symmetric difference of the sets $A$ and $B$, and that
$A \triangle B = \{ A \cup B \} \cap \{ A \cap B\}^c$. Then, we define
the convergence of set-valued functions at infinity as follows:

\begin{defi} \label{def1.1}
Let $A(t)$ be a set-valued function that depends on the variable $t$, and let
$A$ be a set. We say that $A(t)$ converges to $A$ as $t \to \infty$ if
for all $t_i$ and for all $x \in A(t_i) \triangle A$,
there exists $t_j$ such that for all $t \ge t_j$
we have  $x \not\in A(t) \triangle A$.
\end{defi}
Third, convergence and continuity of set-valued functions at a point is also
defined in a similar way. We define the convergence and continuity
of set-valued functions at a point as follows:

\begin{defi}
Let $A(t)$ be a set-valued function that depends on the variable $t$, and let
$A$ be a set. We say that $A(t)$ converges to $A$ as $t \to t_0$ if
for all $t_i \,\, (t_i \neq t_0)$ and for all $x \in A(t_i) \triangle A$,
there exists $\delta >0$ such that for all $t \in \{ s \bigm| 0< |s - t_0| < \delta \}$
we have  $x \not\in A(t) \triangle A$.
\end{defi}

\begin{defi}\label{cont1.3}
Let $A(t)$ be a set-valued function that depends on the variable $t$, and let
$A$ be a set. We say that $A(t)$ is continuous at $t_0$ if
for all $t_i$ and for all $x \in A(t_i) \triangle A(t_0)$,
there exists $\delta >0$ such that for all $t \in \{ s \bigm| |s - t_0| < \delta \}$
we have  $x \not\in A(t) \triangle A(t_0)$.
\end{defi}
The continuity in Definition~\ref{cont1.3} is not equivalent to
that of correspondence.
We could say that Definition~\ref{cont1.3} requires more strictness
in the sense of continuity than correspondence.
We shall clarify this in Subsection~\ref{sec4.exa}.

Concerning Definition~\ref{def1.1}, we should add that the previous work
of sequences of sets has already investigated the
relation between the symmetric difference and the convergence of sequences of sets
(See for examples Niizeki~\cite{Ni2} and~\cite{Niizeki}).
If a sequence of sets $\{ A_n \}_{n=1}^\infty$ converges to a set $A$
in the sense
{\tiny $\displaystyle \bigcap_{n=1}^{\infty} \bigcup_{k=n}^{\infty}$} $A_k$ =
$A$ = {\tiny
$\displaystyle \bigcup_{n=1}^{\infty} \bigcap_{k=n}^{\infty}$} $A_k$,
we denote it by $\displaystyle \lim_{n \to \infty} A_n = A$ here.
In the previous work, it is revealed that  $\displaystyle \lim_{n \to \infty} A_n = A$
is equivalent to $\displaystyle \lim_{n \to \infty} (A_n \triangle A) = \emptyset$.
However, note that both $\displaystyle \lim_{n \to \infty} A_n = A$ and
$\displaystyle \lim_{n \to \infty} (A_n \triangle A) = \emptyset$
mean the equality of the superior limit and the inferior limit:
$\displaystyle \lim_{n \to \infty} A_n = A$ means both of the two limits of
the sequence $\{ A_n \}_{n=1}^\infty$ coincide with the set A;
$\displaystyle \lim_{n \to \infty} (A_n \triangle A) = \emptyset$
means both of the two limits of
the sequence $\{ A_n \triangle A \}_{n=1}^\infty$ coincide with the empty set $\emptyset$.
Hence, the paper differs from the previous work in that
it does not define convergence of ``changing sets" at infinity by
the equality of the two limits.

Here, we adapt Definition~\ref{def1.1} to sequences of sets, that is,
\begin{defi} \label{def1.4}
Let $\{ A_n \}_{n=1}^{\infty}$ be a sequence of sets
and let $A$ be a set. We say that $\{ A_n \}_{n=1}^{\infty}$ converges to $A$ if
for all $n_i \in \mathbb{N}$ and for all $x \in A_{n_i} \triangle A$,
there exists $n_j \in \mathbb{N}$ such that for all $n \ge n_j$
we have  $x \not\in A_n \triangle A$.
%and we write $\displaystyle \lim_{n \to \infty}A_n = A$.
\end{defi}
In reality, we can prove that Definition~\ref{def1.4} is equivalent
to the equality of the two limits
of the sequence $\{ A_n \}_{n=1}^\infty$. 
This fact suggests that
Definition~\ref{def1.1} represents one of the intuitive ideas on convergence
of ``changing sets" at infinity.
Here, we would like to stress that
Definition~\ref{def1.1} is formulated by a statement, not by an equality.
Thus, the paper proves theorems on
convergence of sequences of sets
in a different way from the theory based on the equality of the two limits.
Therefore, the equivalence also suggests that
the paper provides a different way of proof in the theory
of sequences of sets.
\footnote{
We explain the intuitive idea in Subsection~\ref{sec2.def}.
For the proof of the equivalence, see Subsection~\ref{adapt}.}

The plan of this paper is as follows.
We deal with the convergence of set-valued functions at infinity in Section~\ref{sec02},
the convergence of set-valued functions at a point in Section~\ref{sec03},
and the continuity of set-valued functions at a point in Section~\ref{sec04}.
Each of the sections from Section~\ref{sec02} to Section~\ref{sec04} puts
a definition on the respective concepts, gives examples, and then proves theorems.
As examples of the theorems,
we enumerates important theorems derived from Definition~\ref{def1.1}
in the initial part of Section~\ref{sec02}.
In Section~\ref{sec05}, we investigate other possibilities of development in the
study of set-valued functions. Section~\ref{sec05} first introduces a method to describe
the behavior of set-valued functions by one-to-one correspondence. By this method,
we then define the differentiation of set-valued functions in a Euclidean space. In the
rest of Section~\ref{sec05}, we attempt an extension to multivariable 
set-valued functions and an adaption to sequences of sets.
In the subsection on the adaption, we prove
Definition~\ref{def1.4} is equivalent to
the equality of the two limits of the sequence $\{ A_n \}_{n=1}^\infty$.
Finally, Appendix adds other proofs of two theorems in Section~\ref{sec02}.

Before closing Introduction, we make settings of the paper.

\paragraph{Settings}
Throughout the paper, we suppose that set-valued functions depend on a real variable $t$.
We denote them by $A(t)$, $B(t)$, $C(t)$, and so forth.
To avoid problems with complicated domains, assume that set-valued functions are
defined on $\mathbb{R}$ unless expressly stated otherwise.

We use a lower-case letter $x$ to denote an element belonging to a set or a
set-valued function. We suppose that we are working in a fixed universe $X$
that is not empty.

In sets $A$ and $B$, $A \subset B$ allows the possibility $A=B$.
If $A$ is a proper subset of $B$, we write $A \subsetneq B$.

Finally, we let $A \triangle B$ be the symmetric difference of the sets $A$ and
$B$. The symmetric difference $A \triangle B$ is defined by
\begin{equation*}
A \triangle B = \{ A \cup B \} \cap \{ A \cap B\}^c.
\end{equation*}
Intuitively, the symmetric difference $A \triangle B$ is the non-common part
of $A$ and $B$ inside the union $A \cup B$. In this paper, we use the following
three properties of symmetric difference:
\begin{enumerate}
\item $A \triangle B = A^c \triangle B^c$,
\item $A \triangle B = \{ A \cap B^c \} \cup \{ A^c \cap B \}$,
\item $\{ A \cup B \} \triangle \{C \cup D \} \subset
\{A \triangle C \} \cup \{ B \triangle D \}$,
\end{enumerate}
where $A$, $B$, $C$, $D$ are sets.
\footnote{
Part $(3)$ is generalized to
\begin{equation*}
\left( \bigcup_{\lambda \in \Lambda} A_{\lambda} \right)
\triangle \left( \bigcup_{\lambda \in \Lambda} B_{\lambda} \right)
\subset \bigcup_{\lambda \subset \Lambda} \left( A_{\lambda} \triangle B_{\lambda}
\right),
\end{equation*} %adapt size of big union
where the index $\lambda$ is an element of the index set $\Lambda$.
}

\section{Convergence of Set-Valued Functions at Infinity}
\label{sec02}
%\subsection{Definition}\label{sec2.def}
%\subsection{Examples}\label{sec2.exa}
%\subsection{Theorems and Proofs}\label{sec2.thm}

%Definition1.1 \ref{def2.conver}
%Definition2.2 \ref{def2.const}
%Definition2.3 \ref{expa}
%Definition2.4 \ref{shri}
%Definition2.5 \ref{supre}
%Definition2.6 \ref{infim}

%Theorem2.1 \ref{thm2.start}
%Theorem2.7 \ref{thm.uni}
%Theorem2.10 \ref{thm.wakyo}
%Theorem2.12 \ref{thm.exsup}
%Theorem2.13 \ref{thm2.shinf}
%Theorem2.14 \ref{thm2.squ}

Section~\ref{sec02} discusses the convergence of set-valued functions at infinity.
We define it in Subsection~\ref{sec2.def}, give examples in Subsection~\ref{sec2.exa},
and then move to theorems and their proofs in Subsection~\ref{sec2.thm}.
Subsection~\ref{sec2.thm} aims at proving:
\begin{itemize}
\item a set-valued function converges if and only if its complement converges
(Theorem~\ref{thm2.start}),
\item a set-valued function uniquely converges at infinity (Theorem~\ref{thm2.uni}),
\item the union and intersection of convergent set-valued functions
are convergent (Theorem~\ref{thm2.wakyo}),
\item a set-valued function that is ``expanding" on $\mathbb{R}$
converges to its ``supremum" (Theorem~\ref{thm2.exsup});
a set-valued function that is ``shrinking"
on $\mathbb{R}$ converges to its ``infimum" (Theorem~\ref{thm2.shinf}),
\item a set-valued function ``squeezed" between two set-valued functions
converges if the two set-valued functions converge to the same set
(Theorem~\ref{thm2.squ}).
\end{itemize}
To obtain the above conclusions,
we need other theorems.

\subsection{Definition}
\label{sec2.def}

To begin with, we define the convergence of set-valued functions at infinity.

\begin{defi} \label{def2.conver}
Let $A(t)$ be a set-valued function that depends on the variable $t \in \mathbb{R}$,
and let $A$ be a set. We say that $A(t)$ converges to $A$ as $t \to \infty$ if
for all $t_i \in \mathbb{R}$ and for all $x \in A(t_i) \triangle A$,
there exists $t_j \in \mathbb{R}$ such that for all $t \ge t_j$
we have  $x \not\in A(t) \triangle A$,
and we write $\displaystyle \lim_{t \to \infty}A(t) = A$.
Here, we call the set $A$ \textit{limit} of $A(t)$ at infinity.
\end{defi}

We explain the intuitive idea of Definition~\ref{def2.conver}.
Suppose that $A(t)$ converges to limit $A$ as $t \to \infty$.
We would agree that
the convergence of $A(t)$ to $A$ intuitively means $A(t)$ gets ``closer"
to $A$ as $t$ gets larger.
This suggests that the non-common part of $A(t)$ and $A$
inside the union $A(t) \cup A$ gets ``smaller" and ``closer" to
empty set $\emptyset$ as $t$ gets larger.
From this idea, we could say that the symmetric difference
$A(t) \triangle A$, the non-common part, gets ``smaller" and
``closer" to $\emptyset$ as $t$ gets larger.

Figure~\ref{sec2_1_def} pictures this idea. Take an arbitrary point $t_i$
and choose an arbitrary element $x$ from the symmetric difference
$A(t_i) \triangle A$.
At the point $t_i$, the element $x$ belongs to $A(t_i) \triangle A$
(the left-hand portion of Figure~\ref{sec2_1_def}).
For sufficiently large $t_j$, however, the element $x$ does not
belong to $A(t_j) \triangle A$
(the right-hand portion of Figure~\ref{sec2_1_def}),
and also not to $A(t) \triangle A$ for all $t \ge t_j$.

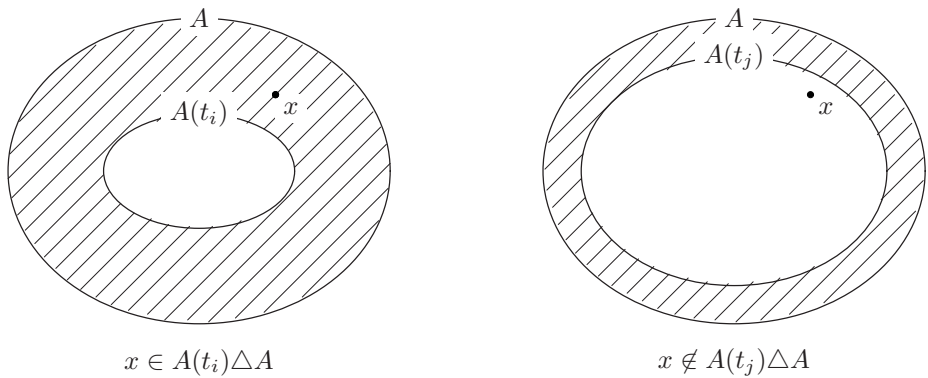
\begin{figure}
\centering
\input{sec2_1_def.tex}
\vspace{0\baselineskip}
\caption{The shaded parts are the symmetric differences of $A(t)$ and
$A$ at $t_i$ (left-hand) and at $t_j$ (right-hand).}
\label{sec2_1_def}
\end{figure}%

We can sum up the above explanation as follows. The convergence of $A(t)$ to $A$
means that the symmetric difference $A(t) \triangle A$
gets ``smaller" and ``closer" to $\emptyset$.
Hence, if we take an arbitrary point $t_i$, all of the elements of $A(t_i) \triangle A$
do not belong to $A(t) \triangle A$ for all sufficiently large $t$.

\subsection{Examples}
\label{sec2.exa}

This subsection gives examples of the convergence of
set-valued functions at infinity.

First, consider the following four set-valued functions. We set them
defined for $t>1$.
\footnote{
Each of the set-valued functions $A(t)$ through $D(t)$ is undefined
for $t \le 1$. Though we define a set-valued function on $\mathbb{R}$
in Definition~\ref{def2.conver}, there's no problem here.
}
\begin{gather*}
A(t) = \Bigl\{ (x , y ) \in \mathbb{R}^2 \bigm|
x^2 + y^2 < \Bigl(1-\frac{1}{t} \Bigr)^2 \Bigr\}. \\
B(t) = \Bigl\{ (x , y ) \in \mathbb{R}^2 \bigm|
x^2 + y^2 \le \Bigl(1-\frac{1}{t}\Bigr)^2 \Bigr\}. \\
C(t) = \Bigl\{ (x , y ) \in \mathbb{R}^2 \bigm|
x^2 + y^2 \le \Bigl(1+\frac{1}{t}\Bigr)^2 \Bigr\}. \\
D(t) =  \Bigl\{ (x , y ) \in \mathbb{R}^2 \bigm|
x^2 + y^2 < \Bigl(1+\frac{1}{t}\Bigr)^2 \Bigr\}.
\end{gather*}
Second, consider the following two sets.
\begin{gather*}
A = \{ (x , y ) \in \mathbb{R}^2 \bigm|
x^2 + y^2 < 1 \}. \\
B = \{ (x , y ) \in \mathbb{R}^2 \bigm|
x^2 + y^2 \le 1 \}.
\end{gather*}
The set-valued functions $A(t)$ and $D(t)$ are open for all $t > 1$;
$B(t)$ and $C(t)$ closed for all $t > 1$.
The set $A$ is open;
$B$ closed.
Here, we show that $A(t)$ and $B(t)$ converge to $A$ as $t \to \infty$
and that $C(t)$ and $D(t)$ converge to $B$ as  $t \to \infty$.

Figure~\ref{sec2_2_exA} illustrates $A(t)$ and $A$.
Take an arbitrary point $t_i >1$ and choose an arbitrary element
$(x_i,y_i)$ from $A(t_i) \triangle A$.
Then, if we take $t_j >1$ such that $t_j > 1/1-\sqrt{\smash[b]{{x_i}^2+{y_i}^2}}$,
we have $(x_i,y_i) \not\in A(t) \triangle A$ for all $t \ge t_j$.
\footnote{
To show the convergence of $A(t)$ to $A$, we need to take
$t_j >1$ such that $1-1/t_j$, the radius of $A(t_j)$,
is larger than $\sqrt{\smash[b]{{x_i}^2+{y_i}^2}}$, the distance from the origin
to the element $(x_i,y_i)$.
Thus, we derive the inequality $t_j > 1/1-\sqrt{\smash[b]{{x_i}^2+{y_i}^2}}$
from the calculation below:
\begin{gather*}
1-\frac{1}{t_j} >\sqrt{\smash[b]{{x_i}^2+{y_i}^2}} \\
\Leftrightarrow t_j > \frac{1}{1-\sqrt{\smash[b]{{x_i}^2+{y_i}^2}}}.
\end{gather*}
}
Since $A(t)$ and $A$ satisfy Definition~\ref{def2.conver},
$A(t)$ converges to $A$ as $t \to \infty$.
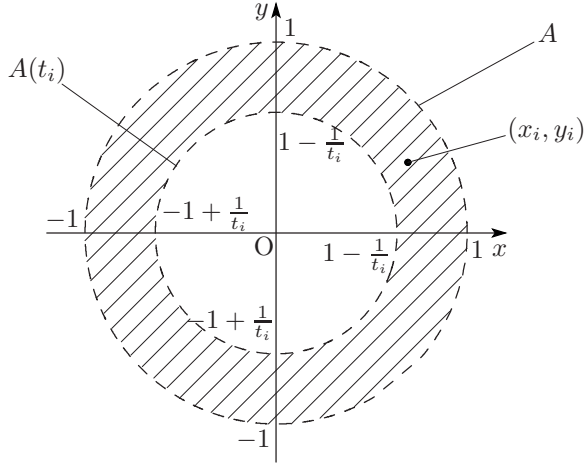
\begin{figure}
\centering
\input{sec2_2_exA.tex}
\caption{The shaded part is $A(t) \triangle A$ at $t_i$,
which does not contain the boundary of A.
The element $(x_i,y_i)$ belongs to $A(t_i) \triangle A$.
}
\label{sec2_2_exA}
\end{figure}%

For the same reason, $B(t)$ converges to $A$ as $t \to \infty$.
It might be counterintuitive that the \textit{closed} set-valued function $B(t)$
converges to the \textit{open} set $A$, not to the \textit{closed} set $B$.

Figure~3
%\ref{sec2_2_exB}
illustrates $B(t)$ and $B$. %%%%%%%%%%%%%%%%%%%%%%%% bug ref
Take an arbitrary point $t_i>1$ and pick an element $(x_i,y_i)$
from $\partial B$, the boundary of B. Here, the element $(x_i,y_i)$
belongs to $B(t_i) \triangle B$ because $\partial B \subset B(t) \triangle B$
for all $t>1$.
Then, we find that $(x_i,y_i) \in B(t) \triangle B$ for all $t \ge t_i$  .
This means that there does not exist $t_j >1$ such that for all $t \ge t_j$,
we have $(x_i,y_i) \not\in B(t) \triangle B$. Since $B(t)$ and $B$ does not
satisfy Definition~\ref{def2.conver}, $B(t)$ does not converge to $B$ as $t \to \infty$.

\begin{figure}
\centering
\input{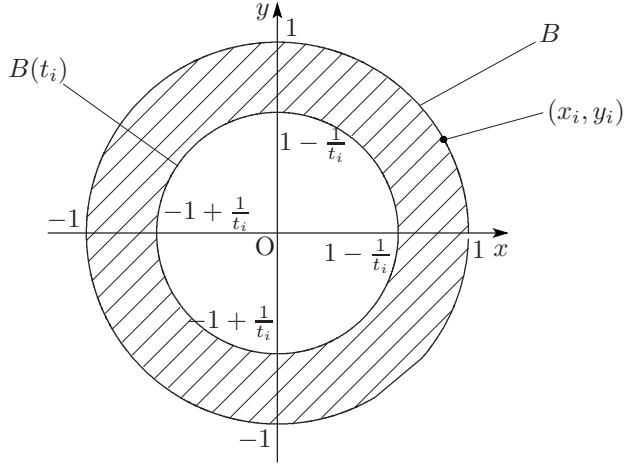}
\label{sec2_2_exB}
\caption{The shaded part is $B(t) \triangle B$ at $t_i$,
which contains the boundary of $B$, $\partial B$.
The element $(x_i,y_i)$ belongs to $\partial B$ and $B(t_i) \triangle B$.
Here, $(x_i,y_i) \in B(t) \triangle B$ for all $t \ge t_i$.
}
\end{figure}%

Note that $A(t_i) \triangle A$ and  $B(t_i) \triangle A$
does not contain the boundary of A, whereas  $A(t_i) \triangle B$ and
$B(t_i) \triangle B$ contain the boundary of B.

Similarly, we can show that $C(t)$ and $D(t)$ converge as $t \to \infty$
to $B$, not to $A$. Here, the \textit{open} set-valued function $D(t)$ converges
to the \textit{closed} set $B$.

Figure~\ref{sec2_2_exC} illustrates $C(t)$ and $B$.
Take an arbitrary point $t_i >1$ and choose an arbitrary element
$(x_i,y_i)$ from $C(t_i) \triangle B$.
Then, if we take $t_j >1$ such that $t_j > 1/\sqrt{\smash[b]{{x_i}^2+{y_i}^2}}-1$,
we have $(x_i,y_i) \not\in C(t) \triangle B$ for all $t \ge t_j$.
\footnote{
To show the convergence of $C(t)$ to $B$, we need to take
$t_j >1$ such that $1+1/t_j$, the radius of $C(t_j)$,
is smaller than $\sqrt{\smash[b]{{x_i}^2+{y_i}^2}}$, the distance from the origin
to the element $(x_i,y_i)$.
Thus, we derive the inequality $t_j > 1/\sqrt{\smash[b]{{x_i}^2+{y_i}^2}}-1$
from the calculation below:
\begin{gather*}
1+\frac{1}{t_j} < \sqrt{\smash[b]{{x_i}^2+{y_i}^2}} \\
\Leftrightarrow t_j > \frac{1}{\sqrt{\smash[b]{{x_i}^2+{y_i}^2}}-1}.
\end{gather*}
}
Since $C(t)$ and $B$ satisfy Definition~\ref{def2.conver},
$C(t)$ converges to $B$ as $t \to \infty$.

\begin{figure}
\centering
\input{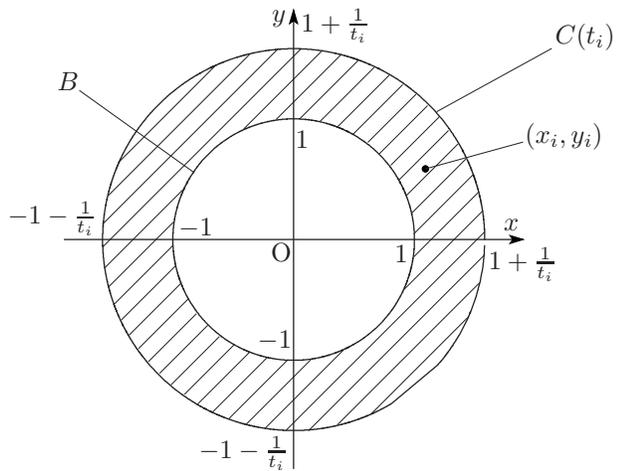}
\caption{The shaded part is $C(t) \triangle B$ at $t_i$,
which does not contain the boundary of B.
The element $(x_i,y_i)$ belongs to $C(t_i) \triangle B$.}
\label{sec2_2_exC}
\end{figure}%

We can show the convergence of $D(t)$ to $B$ in a similar fashion.
However, $D(t)$ does not converge to $A$ as $t \to \infty$.

Figure~\ref{sec2_2_exD} illustrates $D(t)$ and $A$.
Take an arbitrary point $t_i>1$ and pick an element $(x_i,y_i)$
from $\partial A$, the boundary of A. Here, the element $(x_i,y_i)$
belongs to $D(t_i) \triangle A$ because $\partial A \subset D(t) \triangle A$
for all $t>1$.
Then, we find that $(x_i,y_i) \in D(t) \triangle A$ for all $t \ge t_i$  .
This means that there does not exist $t_j >1$ such that for all $t \ge t_j$,
we have $(x_i,y_i) \not\in D(t) \triangle A$. Since $D(t)$ and $A$ does not
satisfy Definition~\ref{def2.conver}, $D(t)$ does not converge to $A$ as $t \to \infty$.
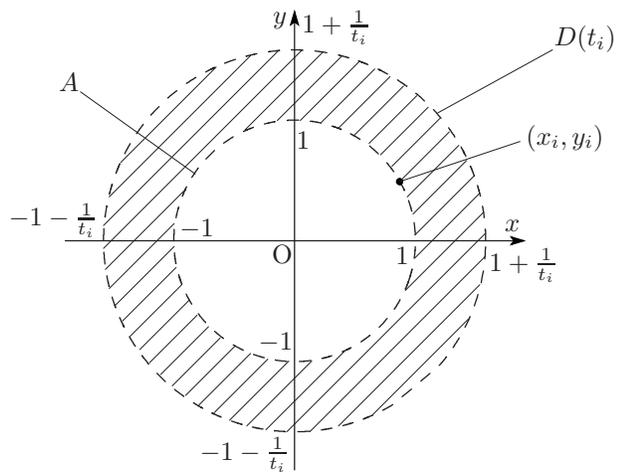
\begin{figure}
\centering
\input{sec2_2_exD.tex}
\caption{The shaded part is $D(t) \triangle A$ at $t_i$,
which contains the boundary of $A$, $\partial A$.
The element $(x_i,y_i)$ belongs to $\partial A$ and $D(t_i) \triangle A$.
Here, $(x_i,y_i) \in D(t) \triangle A$ for all $t \ge t_i$.
}
\label{sec2_2_exD}
\end{figure}%

Note that $C(t_i) \triangle B$ and  $D(t_i) \triangle B$
does not contain the boundary of B, whereas  $C(t_i) \triangle A$ and
$D(t_i) \triangle A$ contain the boundary of A.

We have explained all of the examples.
As a conclusion of this subsection,
we can say that an \textit{open} set-valued function
may converge to a \textit{closed} set as $t \to \infty$ and that a \textit{closed}
set-valued function may converge to an \textit{open} set as $t \to \infty$.
\footnote{
This fact is related to the following properties in a Euclidean space:
the intersection of infinitely-many open sets is not necessarily open;
the union of infinitely-many closed sets is not necessarily closed.
}

%%%%%%%%%%%%%%%%%%%%%%%%%%   subsection %%%%%%%%%%%%%%%%%%%%%%%%%%%%%%%%%%%%

\subsection{Theorems}
\label{sec2.thm}

In this subsection, we state and prove a series of theorems
on Definition~\ref{def2.conver}, the convergence of set-valued functions
at infinity.
For simplicity, set-valued functions are defined on $\mathbb{R}$
throughout this subsection.

Our first theorem follows easily from Definition~\ref{def2.conver}
by using the property of symmetric difference.

\begin{thm} \label{thm2.start}
Let $A(t)$ be a set-valued function and $A(t)^c$ the complement
of $A(t)$ for each $t \in \mathbb{R}$. Let $A$ be a set and $A^c$ the complement
of $A$. Then, $A(t)$ converges to $A$ as $t \to \infty$ if and only if
$A(t)^c$ converges to $A^c$ as $t \to \infty$.
\end{thm}

\begin{proof}
Suppose that $A(t)$ converges to $A$ as $t \to \infty$.
By Definition~\ref{def2.conver},
for all $t_i \in \mathbb{R}$ and for all $x \in A(t_i) \triangle A$,
there exists $t_j \in \mathbb{R}$ such that for all $t \ge t_j$
we have  $x \not\in A(t) \triangle A$.
Since $A \triangle B = A^c \triangle B^c$ with the sets $A$ and $B$,
for all $t_i \in \mathbb{R}$ and for all $ x \in A(t_i)^c \triangle A^c$,
there exists $t_j \in \mathbb{R}$ such that for all $t \ge t_j$
we have  $x \not\in A(t)^c \triangle A^c$.
It follows from Definition~\ref{def2.conver}
that $A(t)^c$ converges to $A^c$ as $t \to \infty$.

We can prove the converse similarly.
\end{proof}

We need the next theorem to derive further results.
\begin{thm} \label{thm2.2}%%%%%%%%%%%%%%%%%%%%%% theorem
Let $A(t)$ be a set-valued function that converges to limit $A$ as $t \to\infty$.
Then, for all $t_i \in \mathbb{R}$ and for all $x \in \{ A(t_i) \triangle A \}^c$,
there exists $t_j \in \mathbb{R}$ such that for all $t \ge t_j$
we have  $x \in \{ A(t) \triangle A \}^c$.
\end{thm}
\begin{proof}
Take an  arbitrary point $t_i \in \mathbb{R}$ and choose an
arbitrary element $x$ from $\{ A(t_i) \triangle A \}^c$.
Then, we can consider two possibilities for the element $x \in
\{ A(t_i) \triangle A \}^c$:
first, for some $t_k \in \mathbb{R}$, we have $x \in A(t_k) \triangle A$;
second, for all $t \in \mathbb{R}$, we have $x \in \{ A(t) \triangle A \}^c$.

To begin with, consider the first case where
$x \in A(t_k) \triangle A$ for some $t_k \in \mathbb{R}$.
The convergence of $A(t)$ to $A$ as $t \to \infty$ means that
for all $t_i \in \mathbb{R}$ and for all $x \in A(t_i) \triangle A$,
there exists $t_j \in \mathbb{R}$ such that
for all $t \ge t_j$ we have $x \in \{ A(t) \triangle A \}^c$.
Thus, it naturally holds that
for  the point $t_k \in \mathbb{R}$ and for the element $x \in A(t_k) \triangle A$,
there exists $t_j \in \mathbb{R}$ such that
for all $t \ge t_j$ we have $x \in \{ A(t) \triangle A \}^c$.

Next consider the second case where
for all $t \in \mathbb{R}$, we have $x \in \{ A(t) \triangle A \}^c$.
Then, this implies that
there exists $t_j \in \mathbb{R}$ such that
for all $t \ge t_j$ we have $x \in \{ A(t) \triangle A \}^c$.

In either case, there exists $t_j \in \mathbb{R}$ such that
the element $x \in \{ A(t_i) \triangle A \}^c$ belongs to
$x \in \{ A(t) \triangle A \}^c$ for all $t \ge t_j$.
Since we take $t_i$ and $x \in \{ A(t_i) \triangle A \}^c$
arbitrarily, it holds that
for all $t_i \in \mathbb{R}$ and for all $x \in \{ A(t_i) \triangle A \}^c$,
there exists $t_j \in \mathbb{R}$ such that for all $t \ge t_j$
we have  $x \in \{ A(t) \triangle A \}^c$, which completes
the proof of Theorem~\ref{thm2.2}.
\end{proof}
\paragraph{Remark}
In general, the converse of Theorem~\ref{thm2.2} is not true.
Let $A(t)$ be a set-valued function such that
$A(t_1) = A(t_2)$ for all $t_1, t_2$ with $t_1 \neq t_2$ and
let $A$ be a set with $A(t) \neq A$ for all $t$.
Figure~\ref{thm2_2rem} illustrates $A(t)$ and $A$
as a Venn diagram.

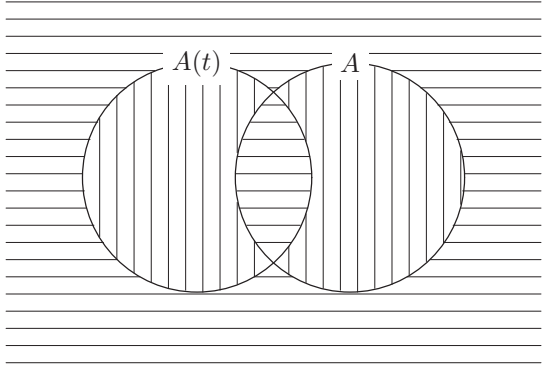
\begin{figure} %%%%%%%%%%%%%%%%%%%%%%%%%%%%%%%% Figure thm2.2rem
\centering
\input{thm2_2rem}
\caption{
The vertically lined part is $A(t) \triangle A$.
The horizontally lined part is $\{ A(t) \triangle A \}^c$.
Both of $A(t) \triangle A$ and $\{ A(t) \triangle A \}^c$
hold stationary for all $t$ because $A(t)$ is the same for all $t$.
}
\label{thm2_2rem}
\end{figure}%

Since $A(t)$ is the same for all $t$, so is $A(t) \triangle A$.
Thus, if we take an arbitrary point $t_i$, all elements of $A(t_i) \triangle A$
belong to $A(t) \triangle A$ for all $t \ge t_i$.
Hence, it does not hold that
for all $t_i \in \mathbb{R}$ and for all $x \in A(t_i) \triangle A$
there exists $t_j \in \mathbb{R}$ such that for all $t \ge t_j$
we have $x \in \{ A(t) \triangle A \}^c$.
Thus, we see that $A(t)$ does not converge to $A$ as $t \to \infty$.
On the other hand, if we take an arbitrary point $t_i$,
all elements of $\{ A(t_i) \triangle A \}^c$ belong to $\{ A(t) \triangle A \}^c$
for all $t \ge t_i$. Hence, it holds that
for all $t_i \in \mathbb{R}$ and for all $x \in \{ A(t_i) \triangle A \}^c$
there exists $t_j \in \mathbb{R}$ such that for all $t \ge t_j$
we have $x \in \{ A(t) \triangle A \}^c$.

\vspace*{1\baselineskip}

The next theorem gives a necessary and sufficient condition for
a set-valued function to converge at infinity.

\begin{thm} \label{thm2.8} %%%%%%%%%%%%%%%%%%%%%%%%%%%%%%%%%%% Theorem
Let $A(t)$ be a set-valued function and $A$ a set. Then,
$A(t)$ converges to $A$ as $t \to \infty$ if and only if
for all $x \in X$, there exists $t_0 \in \mathbb{R}$ such that
for all $t \ge t_0$ we have $x \in \{ A(t) \triangle A \}^c$.
\end{thm}

Before we prove the theorem, it is helpful to clarify the following three facts.
We shall use them in proofs from here on.
\begin{description}
\item[First:] Let $A'$ be a subset of a set $A$ and
let $B'$ be a subset of a set $B$.
If every element of $A$ belongs to $B$, then every element
of $A'$ belongs to $B$.
\footnote{
If `` $x \in A \Rightarrow x \in B$," then
`` $x \in A' \Rightarrow x \in B$."
}
On the other hand,
if every element of $A$ belongs to $B'$, then every element
of $A$ belongs to $B$.
\footnote{
If `` $x \in A \Rightarrow x \in B'$," then
`` $x \in A \Rightarrow x \in B$."
}
\item[Second:] Let $A$, $B$, $C$, and $D$ be sets.
If every element of both $A$ and $B$ belongs to $C$,
then every element of $A \cup B$ belongs to $C$.
\footnote{
If `` $x \in A \Rightarrow  x \in C$ " and `` $x \in B \Rightarrow  x \in C$,"
then `` $x \in A \cup B \Rightarrow x \in C$."
}
On the other hand, if every element of $A$ belongs to both $C$ and $D$,
then every element of $A$ belongs to $C \cap D$.
\footnote{
If `` $x \in A \Rightarrow  x \in C$ " and
`` $x \in A \Rightarrow x \in D$,"
then `` $x \in A \Rightarrow x \in C \cap D$."
}
\item[Third:] Let $A$, $B$, and $C$ be sets. The distributive laws
of set theory state that
\begin{gather*}
(A \cup B) \cap C = (A \cap C) \cup (B \cap C). \\
(A \cap B) \cup C = (A \cup C) \cap (B \cup C).
\end{gather*}
\end{description}

Having explained these properties, we then move to the proof of Theorem~\ref{thm2.8}.

\begin{proof}
Suppose that $A(t)$ converges to $A$ as $t \to \infty$.
By Definition~\ref{def2.conver},
for all $t_{i_1} \in \mathbb{R}$ and for all $x \in A(t_{i_1}) \triangle A$
there exists $t_{j_1} \in \mathbb{R}$ such that for all $t \ge t_{j_1}$
we have  $x \in \{ A(t) \triangle A \}^c$.
On the other hand, by Theorem~\ref{thm2.2},
for all $t_{i_2} \in \mathbb{R}$ and for all $x \in \{ A(t_{i_2}) \triangle A \}^c$
there exists $t_{j_2} \in \mathbb{R}$ such that for all $t \ge t_{j_2}$
we have  $x \in \{ A(t) \triangle A \}^c$.
Take the same point $t_i$ for $t_{i_1}$, $t_{i_2}$ and let
$t_0 = \max \{ t_{j_1}, t_{j_2} \}$. Then, every element of both
$A(t_i) \triangle A$ and  $\{ A(t_i) \triangle A \}^c$ belongs to $\{ A(t) \triangle A \}^c$
for all $t \ge t_0$. Thus,
for all $t_i \in \mathbb{R}$ and
for all $x \in \{ A(t_i) \triangle A \} \cup \{ A(t_i) \triangle A \}^c$
there exists $t_0 \in \mathbb{R}$ such that for all $t \ge t_0$
we have  $x \in \{ A(t) \triangle A \}^c$.
Since $\{ A(t_i) \triangle A \} \cup \{ A(t_i) \triangle A \}^c = X$,
for all $t_i \in \mathbb{R}$ and for all $x \in X$
there exists $t_0 \in \mathbb{R}$ such that for all $t \ge t_0$
we have  $x \in \{ A(t) \triangle A \}^c$.
We can leave out the part ``for all $t_i \in \mathbb{R}$"
because the universe X is independent of the choice of $t_i$.
Thus, it follows that for all $x \in X$
there exists $t_0 \in \mathbb{R}$ such that for all $t \ge t_0$
we have  $x \in \{ A(t) \triangle A \}^c$.

Next suppose that for all $x \in X$
there exists $t_j \in \mathbb{R}$ such that for all $t \ge t_j$
we have  $x \in \{ A(t) \triangle A \}^c$.
%Note that $A(t_i) \triangle A \subset X$ for all $t_i \in \mathbb{R}$.
Since $A(t_i) \triangle A$ is a subset of X for all $t_i \in \mathbb{R}$,
we see that
every element of $A(t_i) \triangle A$ at any $t_i \in \mathbb{R}$
belongs to $\{ A(t) \triangle A\}^c$ for all $t \ge t_j$.
Thus,
for all $t_i \in \mathbb{R}$ and for all $x \in A(t_i) \triangle A$
there exists $t_j \in \mathbb{R}$ such that for all $t \ge t_j$
we have  $x \in \{ A(t) \triangle A \}^c$.
By Definition~\ref{def2.conver}, $A(t)$ converges to $A$ as $t \to \infty$.

Therefore, we have completed the proof of Theorem~\ref{thm2.8}.
\end{proof}

\paragraph{Remark}
Let $A(t)$ be a set-valued function that converges to limit $A$ as $t \to \infty$.
As was discussed in Subsection~\ref{sec2.def},
the convergence of $A(t)$ to $A$ intuitively means that
$A(t) \triangle A$ gets ``smaller" and ``closer" to $\emptyset$ as $t$
gets larger.
Conversely, $\{ A(t) \triangle A\}^c$ gets ``larger" and ``closer"
to the universe $X$ as $t$ gets larger.
So we could say that every element of $X$ belongs to
$\{ A(t) \triangle A\}^c$ for all sufficiently large $t$.
For this reason, we could expect that Definition~\ref{def2.conver}
is equivalent to Theorem~\ref{thm2.8}.

\vspace*{1\baselineskip}
Using the preceding theorem,
we shall prove the next theorem.
\begin{thm} \label{thm2.3}%%%%%%%%%%%%%%%%%%%%%%%%%%%%%% Theorem 2.3
Let $A(t)$ be a set-valued function that converges to limit $A$ as $t \to\infty$.
Then, for all $x \in A$, there exists $t_0 \in \mathbb{R}$
such that for all $t \ge t_0$ we have $x \in A(t)$.
\end{thm}

\begin{proof}
Since $A(t)$ converges to $A$ as $t \to \infty$, by Theorem~\ref{thm2.8},
for all $x \in X$ there exists $t_0 \in \mathbb{R}$ such that
for all $t \ge t_0$ we have $x \in \{ A(t) \triangle A \}^c$.
Since $A$ is a subset of $X$, we see that every element of
$A$ belongs to $\{ A(t) \triangle A \}^c$ for all $t \ge t_0$.
Thus,
for all $x \in A$ there exists $t_0 \in \mathbb{R}$ such that
for all $t \ge t_0$ we have $x \in \{ A(t) \triangle A \}^c$.
Then, this substantially implies that
for all $x \in A$ there exists $t_0 \in \mathbb{R}$ such that
for all $t \ge t_0$ we have $x \in \{ A(t) \triangle A \}^c \cap A$,
because we pick the element $x$ from $A$.
Note that $\bigl[ \{ A(t) \triangle A \}^c \cap A \bigr]
\subset A(t)$, because
\begin{align*}
\{ A(t) \triangle A \}^c \cap A
&= \bigl[ \{ A(t) \cap A \} \cup \{ A(t)^c \cap A^c \} \bigr] \cap A \\
&= \bigl[ \{ A(t) \cap A \} \cap A \bigr] 
\cup \bigl[ \{ A(t)^c \cap A^c \} \cap A \bigr] \\
&= A(t) \cap A \\
&\subset A(t).
\end{align*}
Since  $\{ A(t) \triangle A \}^c \cap A$ is a subset of $A(t)$ for all $t \in \mathbb{R}$,
we see that every element of $A$ belongs to $A(t)$
for all $t \ge t_0$.
Hence, for all $x \in A$ there exists $t_0 \in \mathbb{R}$
such that for all $t \ge t_0$ we have $x \in A(t)$.
We have thus proved the theorem.
\end{proof}

\paragraph{Remark}
We can change the order of the proofs of
Theorem~\ref{thm2.8} and~\ref{thm2.3}.
In Appendix, we first prove Theorem~\ref{thm2.3},
and then prove Theorem~\ref{thm2.8}.
\vspace*{1\baselineskip}

We now obtain a corollary to Theorem~\ref{thm2.3}.
\begin{cor} \label{cor2.4}
Let $A(t)$ be a set-valued function that converges to limit $A$ as $t \to\infty$.
Then, for all $x \in A^c$, there exists $t_0 \in \mathbb{R}$
such that for all $t \ge t_0$ we have $x \in A(t)^c$.
\end{cor}
\begin{proof} 
Since $A(t)$ converges to $A$ as $t \to \infty$,
it follows from Theorem~\ref{thm2.start} that
$A(t)^c$ converges to $A^c$ as $t \to \infty$.
Applying Theorem~\ref{thm2.3}, we conclude that
for all $x \in A^c$, there exists $t_0 \in \mathbb{R}$ such that
for all $t \ge t_0$ we have $x \in A(t)^c$.
\end{proof}

The next theorem is needed to prove the unique
convergence of set-valued functions at infinity.

\begin{thm} \label{thm2.5}%%%%%%%%%%%%%%%%%%%%%%%% Theorem2.5 
Let $A(t)$ be a set-valued function that converges to limit $A$ as $t \to
\infty$, and let $B$ be a set with $B \not\subset A$.
Then, for some $x \in B$ and for all $t \in \mathbb{R}$,
there exists $t_0 \ge t$ such that we have $x \in A(t_0)^c$.
\end{thm}
\begin{proof}
Since $A(t)$ converges to $A$ as $t \to \infty$,
it follows from Corollary~\ref{cor2.4} that
for all $x \in A^c$, there exists $t_1 \in \mathbb{R}$
such that for all $t \ge t_1$ we have $x \in A(t)^c$.
Since $A^c \cap B$ is a subset of $A^c$, we see
that every element of $A^c \cap B$ belongs to $A(t)^c$ for all $t \ge t_1$.
Thus, for all $x \in A^c \cap B$, there exists $t_1 \in \mathbb{R}$
such that for all $t \ge t_1$ we have $x \in A(t)^c$.
Here, notice that there necessarily exists an element $x$ such that
$x \in A^c$ and $x \in B$, because $B \not\subset A$.
Thus, $A^c \cap B$ is not empty.
By this fact, we can say that 
for some $x \in B$ there exists
$t_1 \in \mathbb{R}$ such that for all $t \ge t_1$ we have $x \in A(t)^c$.
This implies that
for some $x \in B$ and for all $t \in \mathbb{R}$, there exists $t_0 \ge t$
such that we have $x \in A(t_0)^c$.
This completes the proof of Theorem~\ref{thm2.5}.
\end{proof}

\paragraph{Remark} 
The statement of Theorem~\ref{thm2.5},
``for some $x \in B$ and for all $t \in \mathbb{R}$, there exists $t_1 \ge t$
such that we have $x \in A(t_1)^c$," is the negation
of the statement of Theorem~\ref{thm2.3},
``for all $x \in B$, there exists $t_2 \in \mathbb{R}$
such that for all $t \ge t_2$ we have $x \in A(t)$."
This means that if $A$ is limit of $A(t)$,
the negation of the statement of
Theorem~\ref{thm2.3} holds
for all sets $B$ with $B \not\subset A$.

\vspace*{1\baselineskip}
We are now in a position to prove the uniqueness of limit of a set-valued function at infinity.
\begin{thm} \label{thm2.uni}
Let $A(t)$ be a set-valued function that converges to limits $A$ and $B$ as $t \to
\infty$. Then, $A = B$.
\end{thm}
\begin{proof}
Suppose for a contradiction that $A \not\subset B$ or
$B \not\subset A$ holds.

In case $A \not\subset B$, we see from Theorem~\ref{thm2.5} that
for some $x \in A$ and for all $t \in \mathbb{R}$, there
exists $t_1 \ge t$ such that we have $x \in A(t_1)^c$.
However, by Theorem~\ref{thm2.3} we find that
for all $x \in A$, there exists $t_2 \in \mathbb{R}$ such that for all $t \ge t_2$
we have $x \in A(t)$, which is a contradiction.

In case $B \not\subset A$, we can derive a contradiction similarly.

By the contradictions above, we conclude that $A \subset B$
and $B \subset A$, that is, $A = B$.
\end{proof}

\paragraph{Remark}
Theorem~\ref{thm2.uni} allows us to speak of \textit{the} limit
of a set-valued function at infinity.

\vspace*{1\baselineskip}
Now we prove the convergence of  the union and intersection
of convergent set-valued functions.
In proving the result, we use the fact that
\linebreak $\{ A \cup B \} \triangle \{C \cup D \} \subset
\{A \triangle C \} \cup \{ B \triangle D \}$
in the sets $A$, $B$, $C$, and $D$.
\begin{thm} \label{thm2.wakyo} %%%%%%%%%%%%%%%%%%%%%%%%% Theorem
Let $A(t)$ and $B(t)$ be set-valued functions that converge to the limits $A$ and $B$,
respectively, as $t \to \infty$. Then, the following hold:
\begin{enumerate}[$(1)$]
\item $\displaystyle \lim_{t \to \infty}A(t) \cup B(t) = A \cup B$.
\item $\displaystyle \lim_{t \to \infty}A(t) \cap B(t) = A \cap B$.
\end{enumerate}
\end{thm}

\begin{proof}
We first prove part $(1)$. Since $\displaystyle \lim_{t \to \infty}A(t) =A$,
it follows from Theorem~\ref{thm2.8} that
for all $x \in X$, there exists $t_1 \in \mathbb{R}$ such that
for all $t \ge t_1$ we have $x \in \{ A(t) \triangle A \}^c$.
Similarly, since
$\displaystyle \lim_{t \to \infty}B(t) = B$, it follows from Theorem~\ref{thm2.8}
that for all $x \in X$, there exists $t_2 \in \mathbb{R}$ such that
for all $t \ge t_2$ we have $x \in \{ B(t) \triangle B \}^c$.
Let $t_0 = \max \{ t_1,t_2\}$. Then, we see that every element of X belongs to both
$\{ A(t) \triangle A\}^c$ and $\{ B(t) \triangle B\}^c$ for all $t \ge t_0$.
Hence,
for all $x \in X$, there exists $t_0 \in \mathbb{R}$ such that
for all $t \ge t_0$ we have
$x \in \{ A(t) \triangle A\}^c \cap \{ B(t) \triangle B\}^c$.
Since $\{ A(t) \triangle A\}^c \cap \{ B(t) \triangle B\}^c
= \bigl[ \{ A(t) \triangle A\} \cup \{ B(t) \triangle B\} \bigr]^c$,
it follows that
for all $x \in X$ there exists $t_0 \in \mathbb{R}$ such that
for all $t \ge t_0$ we have
$x \in \bigl[ \{ A(t) \triangle A\} \cup \{ B(t) \triangle B\} \bigr]^c$.
Note that
$\{ A(t) \cup B(t) \} \triangle \{ A \cup B \}
\subset \{ A(t) \triangle A\} \cup \{ B(t) \triangle B\}$
for all $t \in \mathbb{R}$.
\footnote{
In the sets $A$, $B$, $C$, and $D$,
$\{ A \cup B \} \triangle \{C \cup D \} \subset
\{A \triangle C \} \cup \{ B \triangle D \}$
holds. For the properties of symmetric difference used in this paper,
see Settings in Introduction.
}
Thus, 
$\bigl[ \{ A(t) \cup B(t) \} \triangle \{ A \cup B \} \bigr]^c
\supset \bigl[ \{ A(t) \triangle A\} \cup \{ B(t) \triangle B\} \bigr]^c$
for all $t \in \mathbb{R}$.
Since $\bigl[ \{ A(t) \triangle A\} \cup \{ B(t) \triangle B\} \bigr]^c$
is a subset of $\bigl[ \{ A(t) \cup B(t) \} \triangle \{ A \cup B \} \bigr]^c$
for all $t \in \mathbb{R}$, we see
that every element of X belongs to
$\bigl[ \{ A(t) \cup B(t) \} \triangle \{ A \cup B \} \bigr]^c$
for all $t \ge t_0$.
Hence,
for all $x \in X$, there exists $t_0 \in \mathbb{R}$ such that
for all $t \ge t_0$ we have
$x \in \bigl[ \{ A(t) \cup B(t) \} \triangle \{ A \cup B \} \bigr]^c$.
By Theorem~\ref{thm2.8}, this means
$\displaystyle \lim_{t \to \infty}A(t) \cup B(t) = A \cup B$,
which establishes the proof of part $(1)$.

We next prove part $(2)$.
Since $\displaystyle \lim_{t \to \infty}A(t) =A$ and
$\displaystyle \lim_{t \to \infty}B(t) = B$,
we get
$\displaystyle \lim_{t \to \infty}A(t)^c =A^c$ and
$\displaystyle \lim_{t \to \infty}B(t)^c = B^c$
from Theorem~\ref{thm2.start}.
Using the result of part $(1)$, we obtain
\begin{equation*}
\lim_{t \to \infty}A(t)^c \cup B(t)^c = A^c \cup B^c.
\end{equation*}
Note that $A(t)^c \cup B(t)^c = [ A(t) \cap B(t) ]^c$
and that $A^c \cup B^c = [ A \cap B ]^c$.
Using these facts, we have
\begin{equation*}
\lim_{t \to \infty} [ A(t) \cap B(t) ]^c = [ A \cap B ]^c.
\end{equation*}
By Theorem~\ref{thm2.start}, the preceding equation implies
\begin{equation*}
\lim_{t \to \infty} A(t) \cap B(t) = A \cap B.
\end{equation*}
We have thus proved part $(2)$ of Theorem~\ref{thm2.wakyo}.
\end{proof}

\vspace*{1\baselineskip}
Using induction on $n$, we can extend Theorem~\ref{thm2.wakyo}
to $n$ set-valued functions.

\begin{cor}
Let $n$ be a positive integer.
If a set-valued function $A_i(t)$ converges to the limit
$A_i$ as $t \to \infty$ for each $1 \le i \le n$,
then the following hold:
\begin{enumerate}[$(1)$]
\item $\displaystyle \lim_{t \to \infty}$
{\scriptsize $\displaystyle \bigcup_{1 \le i \le n}$} $A_i(t)$
$=$ {\scriptsize $\displaystyle \bigcup_{1 \le i \le n}$} $A_i$.
\item $\displaystyle \lim_{t \to \infty}$
{\scriptsize $\displaystyle \bigcap_{1 \le i \le n}$} $A_i(t)$
$=$ {\scriptsize $\displaystyle \bigcap_{1 \le i \le n}$} $A_i$.
\end{enumerate}
\end{cor}

\begin{proof}
We first prove part $(1)$.
We use induction on $n$. The claim is trivial for $n=1$.
Assume it is true for $n= k-1$. Take $k-1$ convergent
set-valued functions arbitrarily from $k$ convergent set-valued functions.
We denote the $k-1$ convergent set-valued functions
each by $A_1 (t), A_2 (t), \dots,
A_{k-1} (t)$ and denote the rest by $A_k (t)$.
Then by the induction assumption,
\begin{equation*}
\lim_{t \to \infty}
[ A_1 (t) \cup A_2 (t) \cup \dots
\cup A_{k-1} (t) ] = A_1 \cup A_2 \cup \dots \cup A_{k-1}.
\end{equation*}
Note that $\{ A_1 (t) \cup A_2 (t) \cup \dots
\cup A_{k-1} (t) \} \cup A_k (t)$ is the union of the
two set-valued functions $A_1 (t) \cup A_2 (t) \cup \dots
\cup A_{k-1} (t)$ and $A_k (t)$.
Thus, by Theorem~\ref{thm2.wakyo} we obtain
\begin{align*}
\lim_{t \to \infty}
[ \{ A_1 (t) \cup A_2 (t) \cup \dots
\cup A_{k-1} (t) \} \cup A_k (t) ]
&= \{ A_1 \cup A_2 \cup \dots \cup A_{k-1} \} \cup A_k \\
&= A_1 \cup A_2 \cup \dots \cup A_{k-1} \cup A_k.
\end{align*}
Since we arbitrarily sort the $k$ convergent set-valued functions
into $A_1 (t), A_2 (t), \dots, A_{k-1} (t)$
and $A_k (t)$, it follows that
\begin{equation*}
\lim_{t \to \infty}
[ A_1 (t) \cup A_2 (t) \cup \dots
\cup A_k (t) ] = A_1 \cup A_2 \cup \dots \cup A_k.
\end{equation*}
By the principle of induction, we have completed the proof of part$(1)$.

Part $(2)$ is proved similarly.
\end{proof}

We move to theorems on the convergence
of a set-valued function to ``supremum" and ``infimum."
We need to define some concepts.

\begin{defi} \label{expa}
Let $A(t)$ be a set-valued function and $I$ an interval in $\mathbb{R}$.
We say that $A(t)$ is expanding on $I$ if for all $t_1,t_2 \in I$ with
$t_1 < t_2$ we have $A(t_1) \subset A(t_2)$.
\end{defi}

\begin{defi} \label{shri}
Let $A(t)$ be a set-valued function and $I$ an interval in $\mathbb{R}$.
We say that $A(t)$ is shrinking on $I$ if for all $t_1,t_2 \in I$ with
$t_1 < t_2$ we have $A(t_2) \subset A(t_1)$.
\end{defi}

\begin{defi} \label{supre}
Let $A(t)$ be a set-valued function and $U$ a set.
We say that $U$ is an upper bound of $A(t)$ if for all $t \in \mathbb{R}$
we have $A(t) \subset U$. In particular, $S$, an upper bound of $A(t)$,
is called supremum of $A(t)$ if for all upper bounds $U$ of $A(t)$ we have
$S \subset U$.
\end{defi}

\begin{defi} \label{infim}
Let $A(t)$ be a set-valued function and $L$ a set.
We say that $L$ is a lower bound of $A(t)$ if for all $t \in \mathbb{R}$
we have $L \subset A(t)$. In particular, $N$, a lower bound of $A(t)$,
is called infimum of $A(t)$ if for all lower bounds $L$ of $A(t)$ we have
$L \subset N$.
\end{defi}

If a set-valued function $A(t)$ has an upper bound, then
supremum of $A(t)$ exists. Furthermore, it is unique.

Let $\{ S_{\lambda} \}_{\lambda \in \Lambda}$ be an
indexed family of upper bounds for $A(t)$,
where the index $\lambda$ is an element of the index set $\Lambda$.
Then, the intersection
{\tiny $\displaystyle \bigcap_{\lambda \in \Lambda}$} $S_{\lambda}$
is supremum of $A(t)$ because
{\tiny $\displaystyle \bigcap_{\lambda \in \Lambda}$} $S_{\lambda}$
$\subset S_{\lambda}$ for all $\lambda \in \Lambda$
and $A(t) \subset$
{\tiny $\displaystyle \bigcap_{\lambda \in \Lambda}$} $S_{\lambda}$
for all $t$.
In addition, {\tiny $\displaystyle \bigcap_{\lambda \in \Lambda}$} $S_{\lambda}$
is uniquely determined by the definition of the intersection.
\footnote{
The intersection {\tiny $\displaystyle \bigcap_{\lambda \in \Lambda}$} $S_{\lambda}$
is defined as follows:
\begin{equation*}
{\tiny \bigcap_{\lambda \in \Lambda}} S_{\lambda}
= \{ x \in X \bigm| x \in S_{\lambda} \; \; \text{for all} \; \lambda \in \Lambda \}.
\end{equation*}
}

Similarly, if a set-valued function $A(t)$ has a lower bound, infimum of
$A(t)$ exists and is unique.
Let $\{ N_{\sigma} \}_{\sigma \in \Sigma}$ be an indexed family of lower bounds for $A(t)$,
where the index $\sigma$ is an element of the index set $\Sigma$.
Then, the union \linebreak
{\tiny $\displaystyle \bigcup_{\sigma \in \Sigma}$} $N_{\sigma}$
is infimum of $A(t)$ because $N_{\sigma} \subset$
{\tiny $\displaystyle \bigcup_{\sigma \in \Sigma}$} $N_{\sigma}$
for all $\sigma \in \Sigma$ and
{\tiny $\displaystyle \bigcup_{\sigma \in \Sigma}$} $N_{\sigma}$ $\subset A(t)$
for all $t$.
Additionally, {\tiny $\displaystyle \bigcup_{\sigma \in \Sigma}$} $N_{\sigma}$
is uniquely determined by the definition of the union.
\footnote{
The union {\tiny $\displaystyle \bigcup_{\sigma \in \Sigma}$} $N_{\sigma}$
is defined as follows.
\begin{equation*}
{\tiny \bigcup_{\sigma \in \Sigma}} N_{\sigma}
= \{ x \in X \bigm| x \in N_{\sigma} \; \; \text{for some} \; \sigma \in \Sigma \}.
\end{equation*}
}

Since supremum and infimum of a set-valued function are unique,
we speak of \textit{the} supremum and \textit{the} infimum.

In fact, every set-valued function has an upper bound and a lower bound
because it is contained in the universe $X$, and contains the empty set
$\emptyset$. Thus, there necessarily exist the supremum and the infimum
for every set-valued function.

\vspace*{1\baselineskip}
Having defined the basic concepts, we shall derive
theorems concerning them. The next theorem is on
the convergence to the supremum.

\begin{thm} \label{thm2.exsup}
Let $A(t)$ be a set-valued function that has the supremum $S$.
If $A(t)$ is expanding on $\mathbb{R}$, then $A(t)$ converges to $S$
as $t \to \infty$.
\end{thm}

\begin{proof}
Take an arbitrary point $t_i \in \mathbb{R}$ and choose
an arbitrary element $x$ from $A(t_i) \triangle S$.
Then, we show that for some point $t_j \in \mathbb{R}$
we have $x \in A(t_j)$.
Suppose for a contradiction that for all $t \in \mathbb{R}$
we have $x \in A(t)^c$.
Since $x \in A(t)^c$ and $A(t) \subset S$ for all $t \in \mathbb{R}$,
we get $A(t) \subset S \cap \{ x \}^c$ for all $t \in \mathbb{R}$,
which means that $S \cap \{ x \}^c$ is an upper bound of $A(t)$.
Note that $x$, the element of $A(t_i) \triangle S$,
belongs to $S$, because
\begin{align*}
A(t_i) \triangle S
&= \{ A(t_i) \cup S \} \cap \{ A(t_i) \cap S \}^c \\
&= S \cap A(t_i)^c
\end{align*}
Thus, $S \cap \{ x \}^c \subsetneq S$, which contradicts
the assumption that $S$ is the supremum of $A(t)$.

Therefore, for some point $t_j \in \mathbb{R}$ we have $x \in A(t_j)$.
Since $A(t)$ is expanding on $\mathbb{R}$,
$x \in A(t)$ for all $t \ge t_j$.
Note that $A(t) \triangle S \subset A(t)^c$ because
$A(t) \triangle S = S \cap A(t)^c$.
Considering $x \in A(t)$ for all $t \ge t_j$ and $A(t) \triangle S \subset A(t)^c$
for all $t \in \mathbb{R}$, we see that $x \not\in A(t) \triangle S$
for all $t \ge t_j$.

Therefore, we have proved the theorem.
\end{proof}

\paragraph{Remark}
If a set-valued function $A(t)$ is expanding and has no
upper bound except the universe $X$,
then $A(t)$ converges to $X$ as $t \to \infty$
because X is the supremum of $A(t)$.
We give its example below:
\begin{equation*}
A(t) = \{ (x,y) \in \mathbb{R}^2 \bigm| x^2 + y^2 < t^2 \},
\end{equation*}
where we define $A(t)$ for $t >0$.
The set-valued function $A(t)$ is expanding for all $t >0$ and
$A(t) \subset \mathbb{R}^2$ for all $t >0$.
Also, $A(t)$ has no upper bound except the universe $\mathbb{R}^2$,
because for any set $B \subsetneq \mathbb{R}^2$, there exists $t_0 > 0$
such that $B \subset A(t_0)$, which implies $A(t_0) \not\subset B$.
Thus, $A(t)$ converges to the universe $\mathbb{R}^2$ as $t \to \infty$.

\vspace*{1\baselineskip}
The next theorem is on the convergence to
the infimum.

\begin{thm} \label{thm2.shinf}
Let $A(t)$ be a set-valued function that has the infimum $N$.
If $A(t)$ is shrinking on $\mathbb{R}$, then $A(t)$ converges to $N$
as $t \to \infty$.
\end{thm}

\begin{proof}
Take an arbitrary point $t_i \in \mathbb{R}$ and choose
an arbitrary element $x$ from $A(t_i) \triangle N$.
Then, we show that for some point $t_j \in \mathbb{R}$
we have $x \in A(t_j)^c$.
Suppose for a contradiction that for all $t \in \mathbb{R}$
we have $x \in A(t)$.
Since $x \in A(t)$ and $N \subset A(t)$ for all $t \in \mathbb{R}$,
we get $N \cup \{ x \} \subset A(t)$ for all $t \in \mathbb{R}$,
which means that $N \cup \{ x \}$ is a lower bound of $A(t)$.
Note that $x$, the element of $A(t_i) \triangle N$,
belongs to $N^c$, because
\begin{align*}
A(t_i) \triangle N
&= \{ A(t_i) \cup N \} \cap \{ A(t_i) \cap N \}^c \\
&= A(t_i) \cap N^c
\end{align*}
Thus, $N \subsetneq N \cup \{ x \}^c$, which contradicts
the assumption that $N$ is the infimum of $A(t)$.

Therefore, for some point $t_j \in \mathbb{R}$ we have $x \in A(t_j)^c$.
Since $A(t)$ is shrinking on $\mathbb{R}$,
$x \in A(t)^c$ for all $t \ge t_j$.
Note that $A(t) \triangle N \subset A(t)$ because
$A(t) \triangle N = A(t) \cap N^c$.
Considering $x \in A(t)^c$ for all $t \ge t_j$ and $A(t) \triangle N \subset A(t)$
for all $t \in \mathbb{R}$, we see that $x \not\in A(t) \triangle N$
for all $t \ge t_j$.

Therefore, we have proved the theorem.
\end{proof}

\paragraph{Remark}
Consider again the set-valued functions $A(t)$, $B(t)$, $C(t)$, $D(t)$ and
the sets $A$, $B$ in Subsection~\ref{sec2.exa}.

From Definition~\ref{expa} and~\ref{supre},
we find that $A(t)$ and $B(t)$ are expanding for all $t > 1$,
and that $A$ is their supremum.
On the other hand, we see from Definition~\ref{shri} and~\ref{infim}
that $C(t)$ and $D(t)$ is shrinking for all $t>1$, and that
$B$ is their infimum.

Theorem~\ref{thm2.exsup} corresponds to the fact that
the expanding set-valued functions $A(t)$ and $B(t)$ converge to
their supremum $A$ as $t \to \infty$; Theorem~\ref{thm2.shinf} corresponds to
the fact that the shrinking set-valued functions $C(t)$ and $D(t)$
converge to their infimum B as $t \to \infty$.

\vspace*{1\baselineskip}
We next show the squeeze theorem of set-valued functions.
In proving the theorem, we use the fact that
$A \triangle B = \{ A \cap B^c \} \cup \{ A^c \cap B \}$
in the sets $A$ and $B$.
\begin{thm} \label{thm2.squ}
Let $A(t)$, $B(t)$, and $C(t)$ be set-valued functions such that
$A(t) \subset B(t)$ and $B(t) \subset C(t)$ for all $t \in \mathbb{R}$.
If $A(t)$ and $C(t)$ converge to a set $A$ as $t \to \infty$, then
$B(t)$ converges to A as $t \to \infty$.
\end{thm}

\begin{proof}
To prove the theorem, we first show that
\begin{equation*}
B(t) \triangle A \subset
\{ A(t) \triangle A \} \cup \{ C(t) \triangle A \}.
\end{equation*}
Note $\{ A(t) \cap A^c \} \subset \{ C(t) \cap A^c \}$
and $\{ C(t)^c \cap A \} \subset \{ A(t)^c \cap A \}$
because $A(t) \subset C(t)$ for all $t \in \mathbb{R}$.
Using these inclusions, we get
\begin{align*}
\{ A(t) \triangle A \} \cup \{ C(t) \triangle A \}
&= \bigl[ \{ A(t) \cap A^c \} \cup \{ A(t)^c \cap A \} \bigr] \cup
\bigl[ \{ C(t) \cap A^c \} \cup \{ C(t)^c \cap A \} \bigr] \\
&=  \bigl[ \{ A(t) \cap A^c \} \cup \{ C(t) \cap A^c \} \bigr]
\cup \bigl[ \{ A(t)^c \cap A \} \cup \{ C(t)^c \cap A \} \bigr] \\
&= \{ C(t) \cap A^c \} \cup \{ A(t)^c \cap A \}. 
\end{align*}
Note $\{ B(t) \cap A^c \} \subset \{ C(t) \cap A^c \}$
and $\{ B(t)^c \cap A \} \subset \{ A(t)^c \cap A \}$ for all $t \in \mathbb{R}$
because $B(t) \subset C(t)$ and $A(t) \subset B(t)$ for all $t \in \mathbb{R}$.
By these facts, we are led to
\begin{align*}
\{ C(t) \cap A^c \} \cup \{ A(t)^c \cap A \}
&\supset \{ B(t) \cap A^c \} \cup \{ B(t)^c \cap A \} \\
&= B(t) \triangle A.
\end{align*}
Thus, we obtain
$\{ B(t) \triangle A \}
\subset
\{ A(t) \triangle A \} \cup \{ C(t) \triangle A \}$.

By using the inclusion above, we show that
for all $x \in X$, there exists $t_0 \in \mathbb{R}$
such that for all $t \ge t_0$ we have $x \in \{ B(t) \triangle A \}^c$.
Since $A(t)$ converges to $A$ as $t \to \infty$,
it follows from Theorem~\ref{thm2.8} that
for all $x \in X$, there exists $t_1 \in \mathbb{R}$
such that for all $t \ge t_1$ we have $x \in \{ A(t) \triangle A \}^c$.
Similarly, since $C(t)$ converges to $A$ as $t \to \infty$,
it follows from Theorem~\ref{thm2.8} that
for all $x \in X$, there exists $t_2 \in \mathbb{R}$
such that for all $t \ge t_2$ we have $x \in \{ C(t) \triangle A \}^c$.
Let $t_0 = \max \{ t_1,t_2 \}$. Then, we see that
every element of $X$ belongs to both $\{ A(t) \triangle A \}^c$
and $\{ C(t) \triangle A \}^c$ for all $t \ge t_0$.
Thus,
for all $x \in X$, there exists $t_0 \in \mathbb{R}$
such that for all $t \ge t_0$ we have
$x \in \{ A(t) \triangle A \}^c \cap \{ C(t) \triangle A \}^c$.
Since $\{ B(t) \triangle A \}
\subset
\{ A(t) \triangle A \} \cup \{ C(t) \triangle A \}$,
we get $\{ B(t) \triangle A \}^c
\supset
\bigl[ \{ A(t) \triangle A \} \cup \{ C(t) \triangle A \} \bigr]^c$.
Hence,
$\{ B(t) \triangle A \}^c
\supset
\{ A(t) \triangle A \}^c \cap \{ C(t) \triangle A \}^c$.
Since $\{ A(t) \triangle A \}^c \cap \{ C(t) \triangle A \}^c$
is a subset of $\{ B(t) \triangle A \}^c$, we see that
every element of $X$ belongs to $\{ B(t) \triangle A \}^c$
for all $t \ge t_0$.
Consequently,
for all $x \in X$, there exists $t_0 \in \mathbb{R}$
such that for all $t \ge t_0$ we have
$x \in \{ B(t) \triangle A \}^c$.
By Theorem~\ref{thm2.8},
B(t) converges to $A$ as $t \to \infty$.
\end{proof}

In the rest of the subsection, we collect some miscellaneous theorems.

\begin{thm} \label{cor2.6}
Let $A(t)$ be a set-valued function that converges to the limit $A$ as $t \to
\infty$, and let $B$ be a set such that for all $x \in B$,
there exists $t_0 \in \mathbb{R}$ such that for all $t \ge t_0$
we have $x \in A(t)$.
Then, $B \subset A$.
\end{thm}
\begin{proof}
For a contradiction, suppose that $B \not \subset A$.
Then, it follows from Theorem~\ref{thm2.5} that
for some $x \in B$ and for all $t \in \mathbb{R}$,
there exists $t_1 \ge t$ such that we have $x \in A(t_1)^c$.
But this contradicts the assumption that
for all $x \in B$ there exists
$t_0 \in \mathbb{R}$ such that 
for all $t \ge t_0$ we have $x \in A(t)$.
\end{proof}

\begin{thm}
Let $A(t)$ and $B(t)$ be set-valued functions that converge to the limits $A$ and $B$,
respectively, as $t \to \infty$.
If $A(t) \subset B(t)$ for all $t \in \mathbb{R}$,
then $A \subset B$.
\end{thm}

\begin{proof}
For a contradiction, suppose that $A \not\subset B$.
Since $A(t)$ converges to $A$ as $t \to \infty$,
we find from Theorem~\ref{thm2.3} that
for all $x \in A$ there exists $t_0 \in \mathbb{R}$
such that for all $t \ge t_0$ we have $x \in A(t)$.
Since $A(t)$ is a subset of $B(t)$ for all $t \in \mathbb{R}$,
we see that every element of $A$ belongs to $B(t)$ for all $t \ge t_0$.
Thus,  for all $x \in A$ there exists $t_0 \in \mathbb{R}$
such that for all $t \ge t_0$ we have $x \in B(t)$.
Since $A \not\subset B$, there exists an element $x$ such that
$x \in A$ and $x \in B^c$, that is, $x \in A \cap B^c$.
Hence, $A \cap B^c$ is not empty.
By the fact that $A \cap B^c$ is a subset of $A$, it holds that
for all $x \in  A \cap B^c$
there exists $t_0 \in \mathbb{R}$ such that for all $t \ge t_0$
we have $x \in B(t)$.
Since $B(t)$ converges to $B$ as $t \to \infty$,
$A \cap B^c \subset B$ by Theorem~\ref{cor2.6}.
As was mentioned, $A \cap B^c$ is not empty.
Thus, $A \cap B^c \subset B$ means that there exists an element $x$
such that $x \in B$ and $x \in B^c$, a contradiction.
Therefore, $A \subset B$ holds.
\end{proof}

\begin{thm} \label{thm2.16}
Let $A(t)$ be a set-valued function and $A$ a set.
If there exists $t_0 \in \mathbb{R}$ such that
$A(t) = A$ for all $t \ge t_0$,
then $A(t)$ converges to $A$ as $t \to \infty$.
\end{thm}
\begin{proof}
Since $A(t) \triangle A = \emptyset$ for all $t \ge t_0$,
we have $\{ A(t) \triangle A \}^c = X$ for all $t \ge t_0$.
Thus, if we take an arbitrary element $x$ from the universe $X$,
then $x \in \{ A(t) \triangle A \}^c$ for all $t \ge t_0$.
We conclude, by Theorem~\ref{thm2.8}, that $A(t)$ converges to $A$ as $t \to \infty$.
\end{proof}

\begin{thm} \label{thm2.17}
Let $A(t)$ and $B(t)$ be set-valued functions that converge to the limits $A$ and $B$,
respectively, as $t \to \infty$, and let $C$ be a set.
If there exists $t_0 \in \mathbb{R}$ such that $A(t) \cup B(t) = C$ for all $t \ge t_0$,
then $A \cup B = C$.
\end{thm}
\begin{proof}
Since $\displaystyle \lim_{t \to \infty}A(t) =A$ and
$\displaystyle \lim_{t \to \infty}B(t) = B$, by Theorem~\ref{thm2.wakyo}
we get
\begin{equation} \label{eq2.45}
\lim_{t \to \infty}A(t) \cup B(t) = A \cup B.
\end{equation}
Since $A(t) \cup B(t) = C$ for all $t \ge t_0$,
by Theorem~\ref{thm2.16} we get
\begin{equation} \label{eq2.46}
\lim_{t \to \infty}A(t) \cup B(t) = C.
\end{equation}
Combining equation~\eqref{eq2.45} and~\eqref{eq2.46},
we get $A \cup B = C$.
\end{proof}

\begin{thm} \label{thm2.18}
Let $A(t)$ and $B(t)$ be set-valued functions that converge to the limits $A$ and $B$,
respectively, as $t \to \infty$, and let $C$ be a set.
If there exists $t_0 \in \mathbb{R}$ such that $A(t) \cap B(t) = C$ for all $t \ge t_0$,
then $A \cap B = C$.
\end{thm}
\begin{proof}
The proof of Theorem~\ref{thm2.18} is similar to the proof
of Theorem~\ref{thm2.17}.
\end{proof}

\section{Convergence of Set-Valued Functions at a Point}
\label{sec03}

Section~\ref{sec03} discusses the convergence of set-valued functions at a point.
Similarly to Section~\ref{sec02}, we define it in Subsection~\ref{sec3.def},
give examples in Subsection~\ref{sec3.exa},
and then move to theorems in Subsection~\ref{sec3.thm}.
We formulate the convergence at a point on the basis of the same idea
as the convergence at infinity. Thus, we can derive similar theorems
to those of Subsection~\ref{sec2.thm} and prove them in almost all the same way.

\subsection{Definition}
\label{sec3.def}

We begin with a definition.
\begin{defi} \label{def3.conver}
Let $A(t)$ be a set-valued function that depends on the variable $t \in \mathbb{R}$,
and let $A$ be a set. We say that $A(t)$ converges to $A$ as $t \to t_0$ if
for all $t_i \in \mathbb{R}\,\, (t_i \neq t_0)$ and for all $x \in A(t_i) \triangle A$,
there exists $\delta >0$ such that for all $t \in \{ s
\in \mathbb{R} \bigm| 0< |s - t_0| < \delta \}$
we have  $x \not\in A(t) \triangle A$,
and we write $\displaystyle \lim_{t \to t_0} A(t) = A$.
Here, we call the set $A$ \textit{limit} of $A(t)$ at $t_0$.
\end{defi}

Definition~\ref{def3.conver} is based on the same intuitive idea as
Definition~\ref{def2.conver} of Section~\ref{sec02}.
Suppose that a set-valued function $A(t)$ converges to limit $A$ as $t \to t_0$.
The convergence of $A(t)$ to $A$ at $t_0$ means that
$A(t) \triangle A$, the non-common part of $A(t)$ and $A$,
gets ``smaller" and ``closer" to $\emptyset$ as $t$ gets closer to $t_0$.
Hence, if we take an arbitrary point $t_i$, all of the elements of
$A(t_i) \triangle A$ do not belong to $A(t) \triangle A$ for all $t$
sufficiently close to $t_0$.

We can adapt Definition~\ref{def3.conver} to one-sided convergence
at a point:

\begin{defi} \label{def3.left}
Let $A(t)$ be a set-valued function that depends on the variable $t \in \mathbb{R}$,
and let $A$ be a set. We say that $A(t)$ converges to $A$ as $t \to t_0$
from the left if
for all $t_i < t_0$ and for all $x \in A(t_i) \triangle A$,
there exists $\delta >0$ such that for all $t \in \{ s \in \mathbb{R}
\bigm| 0< t_0 - s < \delta \}$
we have  $x \not\in A(t) \triangle A$,
and we write $\displaystyle \lim_{t \to {t_0 - 0}}$A(t) = A.
Here, we call the set $A$ \textit{left-hand limit} of $A(t)$ at $t_0$.
\end{defi}

\begin{defi} \label{def3.right}
Let $A(t)$ be a set-valued function that depends on the variable $t \in \mathbb{R}$,
and let $A$ be a set. We say that $A(t)$ converges to $A$ as $t \to t_0$
from the right if
for all $t_i > t_0$ and for all $x \in A(t_i) \triangle A$,
there exists $\delta >0$ such that for all $t \in \{ s \in \mathbb{R}
\bigm| 0< s - t_0 < \delta \}$
we have  $x \not\in A(t) \triangle A$,
and we write $\displaystyle \lim_{t \to {t_0 + 0}}$A(t) = A.
Here, we call the set $A$ \textit{right-hand limit} of $A(t)$ at $t_0$.
\end{defi}

We close this subsection with a remark.
\paragraph{Remark}
The set-valued function $A(t)$ can converge
to $A$ as $t \to t_0$ even in the case $A(t_0)$ is not equal to $A$.
Thus, the set $A(t_0)$ is irrelevant to whether $A(t)$ converges
to $A$ as $t \to t_0$. 
Contrary, the set $A(t_0)$ is relevant to whether
$A(t)$ is ``continuous" at $t_0$, as we shall see in Section~\ref{sec04}.

\subsection{Examples}
\label{sec3.exa}

Subsection~\ref{sec3.exa} gives examples of the convergence
of set-valued functions at a point.
One of the implications in this subsection is that
the left-hand limit seldom coincides with the right-hand one in
a Euclidean space.

First, consider the following two set-valued functions. We set them
defined for $t>0$.
\begin{gather*}
A(t) = \{ (x , y ) \in \mathbb{R}^2 \bigm|
x^2 + y^2 < t^2 \}. \\
B(t) = \{ (x , y ) \in \mathbb{R}^2 \bigm|
x^2 + y^2 \le t^2 \}.
\end{gather*}
Second, consider the following two sets.
\begin{gather*}
A = \{ (x , y ) \in \mathbb{R}^2 \bigm|
x^2 + y^2 < 1 \}. \\
B = \{ (x , y ) \in \mathbb{R}^2 \bigm|
x^2 + y^2 \le 1 \}.
\end{gather*}
The set-valued function $A(t)$ is open for all $t > 0$;
$B(t)$ closed for all $t > 0$.
The set $A$ is open;
$B$ closed.

We show that $A(t)$ converges to $A$ as $t \to 1$
from the left.
Figure~\ref{sec3_2_exA} illustrates $A(t)$ and $A$ in case $0< t <1$.
Take an arbitrary point $0 < t_i < 1$
and choose an arbitrary element $(x_i,y_i)$ from $A(t_i) \triangle A$.
Then, if we take $\delta > 0$ such that
$\delta = 1- \sqrt{\smash[b]{{x_i}^2+{y_i}^2}}$,
we have $(x_i,y_i) \not\in A(t) \triangle A$
for all $0 < 1-t < \delta$, that is, for all
$\sqrt{\smash[b]{{x_i}^2+{y_i}^2}} < t < 1$.
\footnote{
To show the convergence of $A(t)$ to $A$,
we need to take $\delta >0$ such that
$t$, the radius of $A(t)$, is larger than
$\sqrt{\smash[b]{{x_i}^2+{y_i}^2}}$,
the distance from the origin to the element $(x_i,y_i)$.
Thus, we obtain $\delta = 1- \sqrt{\smash[b]{{x_i}^2+{y_i}^2}}$
by the calculation below:
\begin{gather*}
\sqrt{\smash[b]{{x_i}^2+{y_i}^2}} < t < 1 \\
\Leftrightarrow 0 < 1-t < 1-\sqrt{\smash[b]{{x_i}^2+{y_i}^2}}.
\end{gather*}
Let $\delta = 1-\sqrt{\smash[b]{{x_i}^2+{y_i}^2}}$.
Then, the inequality above becomes
\begin{equation*}
0 < 1-t < \delta.
\end{equation*}
}
Since $A(t)$ and $A$ satisfy Definition~\ref{def3.left},
$A(t)$ converges to $A$ as $t \to 1$ from the left.

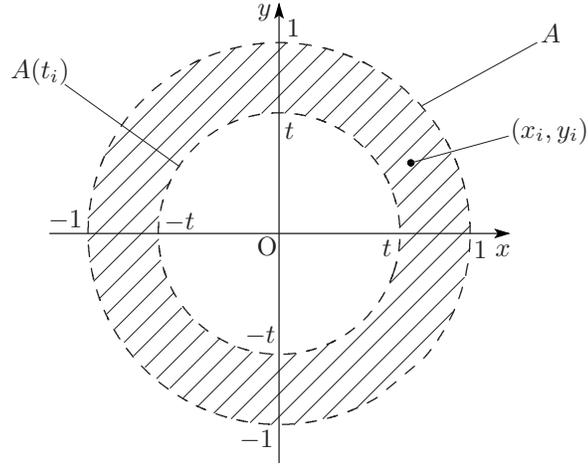
\begin{figure}
\centering
\input{sec3_2_exA.tex}
\caption{
The shaded part is $A(t_i) \triangle A$, which does not
contain the boundary of $A$. The element $(x_i,y_i)$
belongs to $A(t_i) \triangle A$.
}
\label{sec3_2_exA}
\end{figure}%

For the same reason, $B(t)$ converges to $A$ as $t \to \infty$
from the left. However,
$B(t)$ does not converge to $B$ as $t \to \infty$ from the left
because $B(t_i) \triangle B$ contains the boundary $\partial B$
for all $0 < t_i < 1$.
Since $\partial B \subset B(t) \triangle B$ for all $0 < t < 1$,
we find that some element of $B(t_i) \triangle B$ belongs to
$B(t) \triangle B$ for all $0 < t <1$.

We next show that $A(t)$ converges $B$ as $t \to 1$ from the right.
Figure~\ref{sec3_2_exB} illustrates $A(t)$ and $B$ in case $t > 1$.
Take an arbitrary point $t_i > 1$ and choose an
arbitrary element $(x_i,y_i)$ from $A(t_i) \triangle B$.
Then, if we take $\delta > 0$ such that
$\delta = \sqrt{\smash[b]{{x_i}^2+{y_i}^2}} - 1$,
we have $(x_i,y_i) \not\in A(t) \triangle B$
for all $0 < t-1 < \delta$, that is, for all
$1 < t < \sqrt{\smash[b]{{x_i}^2+{y_i}^2}} $.
\footnote{
To show the convergence of $A(t)$ to $B$,
we need to take $\delta >0$ such that
$t$, the radius of $A(t)$, is smaller than
$\sqrt{\smash[b]{{x_i}^2+{y_i}^2}}$,
the distance from the origin to the element $(x_i,y_i)$.
Thus, we obtain $\delta = \sqrt{\smash[b]{{x_i}^2+{y_i}^2}} - 1$
by the calculation below:
\begin{gather*}
1 < t < \sqrt{\smash[b]{{x_i}^2+{y_i}^2}} \\
\Leftrightarrow 0 < t-1 < \sqrt{\smash[b]{{x_i}^2+{y_i}^2}}- 1.
\end{gather*}
Let $\delta = \sqrt{\smash[b]{{x_i}^2+{y_i}^2}} - 1$.
Then, the inequality above becomes
\begin{equation*}
0 < t-1 < \delta.
\end{equation*}
}
Since $A(t)$ and $B$ satisfy Definition~\ref{def3.right},
$A(t)$ converges to $B$ as $t \to 1$ from the right.

\begin{figure}
\centering
\input{sec3_2_exB}
\caption{
The shaded part is $A(t_i) \triangle A$,
which does not contain the boundary of $B$.
The element $(x_i,y_i)$
belongs to $A(t_i) \triangle B$.
}
\label{sec3_2_exB}
\end{figure}
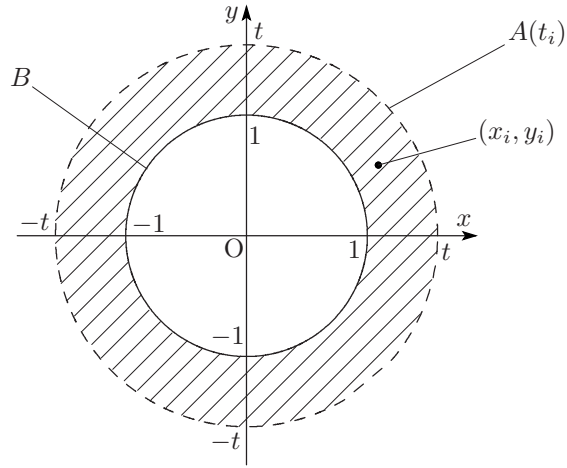%

For the same reason, $B(t)$ converges to $B$ as $t \to 1$
from the right.
Here, $A(t)$ does not converge to $A$ as $t \to 1$ from the right
because $A(t_i) \triangle A$ contains the boundary
$\partial A$ for all $t_i >1$. Since $\partial A \subset
A(t) \triangle A$ for all $t > 1$, we find that some element
of $A(t_i) \triangle A$ belongs to $A(t) \triangle A$
for all $t > 1$.

As a conclusion, we can say that it is difficult to find ``both-handed"
limit at a point in a Euclidean space.
In the examples earlier, the left-hand limit does not coincide with
the right-hand one.
In a Euclidean space, the coincidence occurs at a point in a case
where a set-valued function holds stationary in a neighborhood of the point
(See Theorem~\ref{thm3.const} in Subsection~\ref{sec3.thm}).

Finally, we note there is a set-valued function that do not have even
one-sided limit at a point.
Now consider a point set below:
\begin{equation*}
C(t) = \{ t \},
\end{equation*}
where we define $C(t)$ on $\mathbb{R}$.
Let $C$ be a set such that $C = \{ 1 \}$.
Figure~\ref{sec3_2_exC} illustrates $C(t)$ and $C$.

We first show that $C(t)$ does not converge to $C$ as
$t \to 1$ from the left.
Take an arbitrary point $t_i <1$.
Then $C(t) \triangle C = \{ t_i,1 \}$.
If we choose an element $\{ 1 \}$ from $C(t) \triangle C$,
we have $\{ 1 \} \in C(t) \triangle C$ for all $t >1$.
Since $C(t)$ and $C$ does not satisfy Definition~\ref{def3.left},
$C(t)$ does not converge to $C$ as $t \to 1$ from the left.
Similarly, we can show that $C(t)$ does not converge to $C$
as $t \to 1$ from the right.

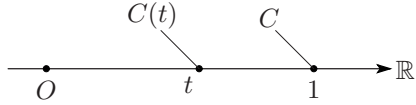
\begin{figure}
\centering
\input{sec3_2_exC.tex}
\caption{In both cases $t \to 1$ from the left and from the right,
$C(t) \triangle C = \{ t,1 \}$. The element $\{ 1 \}$
belongs to $C(t) \triangle C$ for all $t<1$ and for all $t >1$.}
\label{sec3_2_exC}
\end{figure}%

We need to consider a real number and a set separately.
If a real variable $t$ approaches $1$, we can regard the variable $t$ as
convergent to the number $1$ though $t \neq 1$.
But, we cannot regard the set-valued function $\{ t \}$ as convergent to
the set $\{ 1 \}$ because $t \neq 1$.
On account of the dense property of real numbers,
we can consider the variable $t$ the same as the number $1$ if $t$ approaches $1$.
However, we cannot consider the set-valued function $\{ t \}$ the same as the set $\{ 1 \}$
because $\{ t \}$ is completely different from $\{ 1 \}$ as an element.

\subsection{Theorems}
\label{sec3.thm}

In Sebsection~\ref{sec3.thm}, we discuss theorems on Definition~\ref{def3.conver}.
Definition~\ref{def2.conver} and Definition~\ref{def3.conver} are based on
the same intuitive idea and constructed in a similar way.
So, adapting theorems on Definition~\ref{def2.conver} to Definition~\ref{def3.conver},
we can obtain many theorems similar to those of Subsection~\ref{sec2.thm}.
Furthermore, we can prove them in almost all the same way
as Subsection~\ref{sec2.thm}. Thus, in Subsection~\ref{sec3.thm} we only
list theorems on Definition~\ref{def3.conver} and do not prove them.
For simplicity, set-valued functions are defined on $\mathbb{R}$.
\footnote{
In the preceding subsection, we illustrated that it is difficult to find
both-handed limit at a point in a Euclidean space. Thus, the listed theorems might
be unmeaningful themselves. However, we can easily adapt the theorems to
one-handed limit at a point. This means we can practically apply the theorems
by the adaption because it is not difficult to find one-handed limit in a Euclidean space.
Considering convenience, we list the theorems of both-handed limit, not
one-handed limit.
}

\begin{thm} 
Let $A(t)$ be a set-valued function and $A(t)^c$ the complement
of $A(t)$ for each $t \in \mathbb{R}$. Let $A$ be a set and $A^c$ the complement
of $A$. Then, $A(t)$ converges to $A$ as $t \to t_0$ if and only if
$A(t)^c$ converges to $A^c$ as $t \to t_0$.
\end{thm}

\begin{thm}
Let $A(t)$ be a set-valued function that converges to limit $A$ as $t \to t_0$.
Then, for all $t_i \in \mathbb{R} \, \,
(t_i \neq t_0)$ and for all $x \in \{ A(t_i) \triangle A \}^c$,
there exists $\delta >0$ such that for all $t \in \{ s \in
\mathbb{R} \bigm| 0< |s - t_0| < \delta \}$
we have  $x \in \{ A(t) \triangle A \}^c$.
\end{thm}

\begin{thm} \label{thm3.ns} %%%%%%%%%%%%%%%%%%%%
Let $A(t)$ be a set-valued function and $A$ a set. Then,
$A(t)$ converges to $A$ as $t \to t_0$ if and only if
for all $x \in X$, there exists
$\delta >0$ such that for all $t \in \{ s \in
\mathbb{R} \bigm| 0< |s - t_0| < \delta \}$
we have $x \in \{ A(t) \triangle A \}^c$.
\end{thm}

\begin{thm} 
Let $A(t)$ be a set-valued function that converges to limit $A$ as $t \to t_0$.
Then, for all $x \in A$,
there exists $\delta >0$ such that for all $t \in \{ s
\in \mathbb{R} \bigm| 0< |s - t_0| < \delta \}$
we have $x \in A(t)$.
\end{thm}

\begin{cor}
Let $A(t)$ be a set-valued function that converges to $A$ as $t \to t_0$.
Then, for all $x \in A^c$,
there exists $\delta >0$ such that for all $t \in \{ s
\in \mathbb{R} \bigm| 0< |s - t_0| < \delta \}$
we have $x \in A(t)^c$.
\end{cor}

\begin{thm}  
Let $A(t)$ be a set-valued function that converges to limit $A$ as $t \to
t_0$, and let $B$ be a set with $B \not\subset A$.
Then, for some $x \in B$ and for all $\delta >0$,
there exists $t_j \in
\{ s \in \mathbb{R} \bigm| 0< |s - t_0| < \delta \}$ such that we have $x \in A(t_j)^c$.
\end{thm}

\begin{thm}
Let $A(t)$ be a set-valued function that converges to limits $A$ and $B$ as $t \to
t_0$. Then, $A = B$.
\end{thm}

\begin{thm} 
Let $A(t)$ and $B(t)$ be set-valued functions that converge to the limits $A$ and $B$,
respectively, as $t \to t_0$. Then, the following hold:
\begin{enumerate}[$(1)$]
\item $\displaystyle \lim_{t \to t_0}A(t) \cup B(t) = A \cup B$.
\item $\displaystyle \lim_{t \to t_0}A(t) \cap B(t) = A \cap B$.
\end{enumerate}
\end{thm}

\begin{cor}
Let $n$ be a positive integer.
If a set-valued function $A_i(t)$ converges to the limit
$A_i$ as $t \to t_0$ for each $1 \le i \le n$,
then the following hold:
\begin{enumerate}[$(1)$]
\item $\displaystyle \lim_{t \to t_0}$
{\scriptsize $\displaystyle \bigcup_{1 \le i \le n}$} $A_i(t)$
$=$ {\scriptsize $\displaystyle \bigcup_{1 \le i \le n}$} $A_i$.
\item $\displaystyle \lim_{t \to t_0}$
{\scriptsize $\displaystyle \bigcap_{1 \le i \le n}$} $A_i(t)$
$=$ {\scriptsize $\displaystyle \bigcap_{1 \le i \le n}$} $A_i$.
\end{enumerate}
\end{cor}

\begin{thm} 
Let $A(t)$, $B(t)$, and $C(t)$ be set-valued functions such that
$A(t) \subset B(t)$ and $B(t) \subset C(t)$ for all $t \in \mathbb{R}$.
If $A(t)$ and $C(t)$ converge to a set $A$ as $t \to t_0$, then
$B(t)$ converges to A as $t \to t_0$.
\end{thm}

\begin{thm} 
Let $A(t)$ be a set-valued function that converges to the limit $A$ as $t \to
t_0$, and let $B$ be a set such that for all $x \in B$,
there exists $\delta >0$ such that for all $t \in \{ s
\in \mathbb{R} \bigm| 0< |s - t_0| < \delta \}$
we have $x \in A(t)$.
Then, $B \subset A$.
\end{thm}

\begin{thm}
Let $A(t)$ and $B(t)$ be set-valued functions that converge to the limits $A$ and $B$,
respectively, as $t \to t_0$.
If $A(t) \subset B(t)$ for all $t \in \mathbb{R}$,
then $A \subset B$.
\end{thm}

\begin{thm} \label{thm3.const}
Let $A(t)$ be a set-valued function and $A$ a set.
If there exists $\delta > 0$ such that $A(t) = A$ for all $t \in \{ s
\in \mathbb{R} \bigm| 0< |s - t_0| < \delta \}$,
then $A(t)$ converges to $A$ as $t \to t_0$.
\end{thm}

\begin{thm} 
Let $A(t)$ and $B(t)$ be set-valued functions that converge to the limits $A$ and $B$,
respectively, as $t \to t_0$, and let $C$ be a set.
If there exists $\delta > 0$ such that $A(t) \cup B(t) = C$ for all
$t \in \{ s
\in \mathbb{R} \bigm| 0< |s - t_0| < \delta \}$,
then $A \cup B = C$.
\end{thm}

\begin{thm} 
Let $A(t)$ and $B(t)$ be set-valued functions that converge to the limits $A$ and $B$,
respectively, as $t \to t_0$, and let $C$ be a set.
If there exists $\delta > 0$ such that $A(t) \cap B(t) = C$ for all $t \in \{ s
\in \mathbb{R} \bigm| 0< |s - t_0| < \delta \}$,
then $A \cap B = C$.
\end{thm}

\vspace*{1\baselineskip}
We can adapt the above theorems to Definition~\ref{def3.left}
and~\ref{def3.right}.
The same results hold for one-sided convergence.

We conclude this subsection by proving a theorem on the
coincidence of the left-hand limit and the right-hand limit.

\begin{thm} \label{thm3.coincide}
Let $A(t)$ be a set-valued function and $A$ a set.
Then, $\displaystyle \lim_{t \to t_0} A(t) = A$
if and only if
$\displaystyle \lim_{t \to t_0 -0} A(t) = A
= \displaystyle \lim_{t \to t_0 + 0} A(t)$.
\end{thm}
\begin{proof}
Suppose $\displaystyle \lim_{t \to t_0} A(t) = A$.
By Theorem~\ref{thm3.ns},
for all $x \in X$, there exists
$\delta >0$ such that for all $t \in \{ s \in
\mathbb{R} \bigm| 0< |s - t_0| < \delta \}$
we have $x \in \{ A(t) \triangle A \}^c$.
This implies that
for all $x \in X$, there exists
$\delta >0$ such that for all $t \in \{ s \in
\mathbb{R} \bigm| 0< t_0 - s < \delta \}$
we have $x \in \{ A(t) \triangle A \}^c$.
Thus, we get $\displaystyle \lim_{t \to t_0 -0} A(t) = A$
by adapting Theorem~\ref{thm3.ns} to Definition~\ref{def3.left}.
Similarly, it holds that
for all $x \in X$, there exists
$\delta >0$ such that for all $t \in \{ s \in
\mathbb{R} \bigm| 0< s - t_0 < \delta \}$
we have $x \in \{ A(t) \triangle A \}^c$.
Thus, we get $\displaystyle \lim_{t \to t_0 + 0} A(t)$
by adapting Theorem~\ref{thm3.ns} to Definition~\ref{def3.right}.

Now suppose that 
$\displaystyle \lim_{t \to t_0 -0} A(t) = A
= \displaystyle \lim_{t \to t_0 + 0} A(t)$.
Adapting Theorem~\ref{thm3.ns} to Definition~\ref{def3.left},
we see that
for all $x \in X$, there exists
$\delta_1 >0$ such that for all $t \in \{ s \in
\mathbb{R} \bigm| 0< t_0 - s < \delta_1 \}$
we have $x \in \{ A(t) \triangle A \}^c$.
Adapting Theorem~\ref{thm3.ns} to Definition~\ref{def3.right},
on the other hand,
we see that
for all $x \in X$, there exists
$\delta_2 >0$ such that for all $t \in \{ s \in
\mathbb{R} \bigm| 0< t_0 - s < \delta_2 \}$
we have $x \in \{ A(t) \triangle A \}^c$.
Let $\delta = \min \{ \delta_1, \delta_2 \}$.
Then, we see that
for all $x \in X$, there exists
$\delta >0$ such that for all $t \in \{ s \in
\mathbb{R} \bigm| 0< |s - t_0| < \delta \}$
we have $x \in \{ A(t) \triangle A \}^c$.

Therefore, we have proved the theorem.
\end{proof}

\section{Continuity of Set-Valued Functions at a Point}
\label{sec04}

Section~\ref{sec04} discusses the continuity of set-valued functions at a point.
We define it in Subsection~\ref{sec4.def}, give examples in Subsection~\ref{sec4.exa},
and then move to theorems in Subsection~\ref{sec4.thm}.
Similarly to Subsection~\ref{sec3.thm}, we only list theorems and
don't prove them in Subsection~\ref{sec3.thm}.

\subsection{Definition}
\label{sec4.def}

Section~\ref{sec04} starts with the definition of the continuity of set-valued functions
at a point.

\begin{defi} \label{def4.conti}
Let $A(t)$ be a set-valued function that depends on the variable $t \in \mathbb{R}$.
We say that $A(t)$ is continuous at $t_0$ if
for all $t_i \in \mathbb{R}$ and for all $x \in A(t_i) \triangle A(t_0)$,
there exists $\delta >0$ such that for all $t \in \{ s
\in \mathbb{R} \bigm| |s - t_0| < \delta \}$
we have  $x \not\in A(t) \triangle A(t_0)$,
and we write $\displaystyle \lim_{t \to t_0}A(t) = A(t_0)$.
\end{defi}

The continuity at a point stands on the same idea as the convergence at a point.
Definition~\ref{def4.conti} also utilizes the property that $A(t) \triangle A(t_0)$
gets ``smaller" and ``closer" to $\emptyset$ as $t$ gets closer to $t_0$.
Unlike the convergence at a point, however, the continuity at a point is
critically related to the set $A(t_0)$. The continuity at a point
demands that $A(t)$ gets ``closer" to the set $A(t_0)$,
as well as to the limit $\displaystyle \lim_{t \to t_0}A(t)$,
as $t$ gets closer to $t_0$.
In other words, it demands that
$\displaystyle \lim_{t \to t_0}A(t)$ coincides with $A(t_0)$.
This fact justifies the notation $\displaystyle \lim_{t \to t_0}A(t) = A(t_0)$,
which simplifies the definition of the continuity at a point.

We can define one-sided continuity at a point, similarly
to one-sided convergence at a point.

\begin{defi} \label{def4.left}
Let $A(t)$ be a set-valued function that depends on the variable $t \in \mathbb{R}$.
We say that $A(t)$ is continuous on the left at $t_0$ if
for all $t_i \le t_0$ and for all $x \in A(t_i) \triangle A(t_0)$,
there exists $\delta >0$ such that for all $t \in \{ s
\in \mathbb{R} \bigm|0 \le t_0 - s < \delta \}$
we have  $x \not\in A(t) \triangle A(t_0)$,
and we write $\displaystyle \lim_{t \to t_0 - 0}A(t) = A(t_0)$.
\end{defi}

\begin{defi} \label{def4.right}
Let $A(t)$ be a set-valued function that depends on the variable $t \in \mathbb{R}$.
We say that $A(t)$ is continuous on the right at $t_0$ if
for all $t_i \ge t_0$ and for all $x \in A(t_i) \triangle A(t_0)$,
there exists $\delta >0$ such that for all $t \in \{ s
\in \mathbb{R} \bigm|0 \le s - t_0 < \delta \}$
we have  $x \not\in A(t) \triangle A(t_0)$,
and we write $\displaystyle \lim_{t \to t_0 + 0}A(t) = A(t_0)$.
\end{defi}

\subsection{Examples}
\label{sec4.exa}

Subsection~\ref{sec4.exa} gives examples of the continuity of set-valued functions
at a point. Similarly to convergence at a point,
both-sided continuity seldom exists in a Euclidean space.
Consider the following two set-valued functions.
We set them defined for $t > 0$.
\begin{gather*}
A(t) = \{ (x , y ) \in \mathbb{R}^2 \bigm|
x^2 + y^2 < t^2 \}. \\
B(t) = \{ (x , y ) \in \mathbb{R}^2 \bigm|
x^2 + y^2 \le t^2 \}.
\end{gather*}
The set-valued functions above are the same ones as $A(t), B(t)$ in Subsection~\ref{sec3.exa}.
Substituting $t=1$ in $A(t)$ and $B(t)$, we get
\begin{gather*}
A(1) = \{ (x , y ) \in \mathbb{R}^2 \bigm|
x^2 + y^2 < 1 \}. \\
B(1) = \{ (x , y ) \in \mathbb{R}^2 \bigm|
x^2 + y^2 \le 1 \}.
\end{gather*}
Note that $A(1)$ is equal to the left-hand limit
$\displaystyle \lim_{t \to 1-0}A(t)$,
and that $B(1)$ is equal to the right-hand limit
$\displaystyle \lim_{t \to 1+0}B(t)$
(See Subsection~\ref{sec3.exa}).
Thus, in the same way as Subsection~\ref{sec3.exa}
we can show that $A(t)$ is continuous on the left at $t=1$,
and that $B(t)$ is continuous on the right at $t=1$.
It can also be shown that
$A(t)$ is not continuous on the right at $t=1$,
and that $B(t)$ is not continuous on the left at $t=1$.

Both-sided continuity does not exist in the examples here.
It is difficult to find it in a Euclidean space.
This difficulty stems from the fact that the left-hand limit
seldom coincides with the right-hand one in a Euclidean space.
In a Euclidean space, we can find out both-sided continuity at a point in case
a set-valued function holds stationary in a neighborhood of the point
(See Theorem~\ref{thm4.const}
in Subsection~\ref{sec4.thm}).

Finally, we explain that the continuity of Definition~\ref{def4.conti}
is different from the continuity of correspondence.
The set-valued functions $A(t)$ and $B(t)$ in this subsection are continuous
in terms of correspondence because $A(t)$ and $B(t)$
are upper and lower hemi-continuous.
Here, the degree of strictness causes the difference.
Namely,
Definition~\ref{def4.conti} demands more strictness in continuity
than correspondence. Definition~\ref{def4.conti}
does not permit any small sudden change in continuity,
whereas correspondence permits a sudden change if it is not large.

We explain this through the set-valued functions $A(t)$ and $B(t)$
in this subsection.
The set-valued function $A(t)$ gets ``closer'' to the set $A(t_0)$
as $t$ gets closer to a point $t_0$ from the left.
At last $A(t)$ gets equal to $A(t_0)$
when $t$ gets equal to $t_0$.
On the other hand, $B(t)$ gets ``closer'' to the set $A(t_0)$
as $t$ gets closer to $t_0$ from the left. But $B(t)$ ``skips'' $A(t_0)$
and gets equal to $B(t_0)$ when $t$ gets equal to $t_0$.
This ``skip'' of $B(t)$ is a small but sudden change.
Definition~\ref{def4.conti} does not permit this sudden change
even though it is small.
Thus, $B(t)$ is not continuous in Definition~\ref{def4.conti}.

Similarly, $A(t)$ is not continuous in Definition~\ref{def4.conti}.
The set-valued function $A(t)$ gets ``closer'' to the set $B(t_0)$
as $t$ gets closer to $t_0$ from the right. But $A(t)$ ``skips'' $B(t_0)$
and gets equal to $A(t_0)$ when $t$ gets equal to $t_0$.

In correspondence, a sudden change is permitted if it is sufficiently small.
Thus, $A(t)$ and $B(t)$ is continuous in correspondence.

Therefore, we could say that Definition~\ref{def4.conti} demands \textit{strict}
continuity, whereas correspondence demands \textit{loose} continuity.

\subsection{Theorems}
\label{sec4.thm}

In Subsection~\ref{sec4.thm}, we discuss theorems on
Definition~\ref{def4.conti}.
Definition~\ref{def4.conti} is constructed in a similar way
to Definition~\ref{def2.conver} and~\ref{def3.conver}.
Thus, we can derive many theorems similar to those of
Subsection~\ref{sec2.thm} and~\ref{sec3.thm} in almost all the same way.
For this reason, we only list theorems and do not
prove them. For simplicity,
set-valued functions are defined on $\mathbb{R}$.
\footnote{
Similarly to Subsection~\ref{sec3.thm}, we practically apply the listed theorems
by adapting them to one-sided continuity.
For convenience, we list the theorems of both-sided continuity, not
one-sided continuity.
}

\begin{thm} 
Let $A(t)$ be a set-valued function and $A(t)^c$ the complement
of $A(t)$ for each $t \in \mathbb{R}$.
Then, $A(t)$ is continuous at $t_0$ if and only if
$A(t)^c$ is continuous at $t_0$.
\end{thm}

\begin{thm}
Let $A(t)$ be a set-valued function that is continuous at $t_0$.
Then, for all $t_i \in \mathbb{R}$ and for all $x \in \{ A(t_i) \triangle A(t_0) \}^c$,
there exists $\delta >0$ such that for all $t \in \{ s \in
\mathbb{R} \bigm| |s - t_0| < \delta \}$
we have  $x \in \{ A(t) \triangle A(t_0) \}^c$.
\end{thm}

\begin{thm} \label{thm4.ns} %%%%%%%%%%%%%%%%%%%%
Let $A(t)$ be a set-valued function and $A$ a set. Then,
$A(t)$ is continuous at $t_0$ if and only if
for all $x \in X$, there exists
$\delta >0$ such that for all $t \in \{ s \in
\mathbb{R} \bigm| |s - t_0| < \delta \}$
we have $x \in \{ A(t) \triangle A(t_0) \}^c$.
\end{thm}

\begin{thm} 
Let $A(t)$ be a set-valued function that is continuous at $t_0$.
Then, for all $x \in A(t_0)$,
there exists $\delta >0$ such that for all $t \in \{ s
\in \mathbb{R} \bigm| |s - t_0| < \delta \}$
we have $x \in A(t)$.
\end{thm}

\begin{cor}
Let $A(t)$ be a set-valued function that is continuous at $t_0$.
Then, for all $x \in A(t_0)^c$,
there exists $\delta >0$ such that for all $t \in \{ s
\in \mathbb{R} \bigm| |s - t_0| < \delta \}$
we have $x \in A(t)^c$.
\end{cor}

\begin{thm}  
Let $A(t)$ be a set-valued function that is continuous at
$t_0$, and let $B$ be a set with $B \not\subset A(t_0)$.
Then, for some $x \in B$ and for all $\delta >0$,
there exists $t_j \in
\{ s \in \mathbb{R} \bigm| |s - t_0| < \delta \}$ such that we have $x \in A(t_j)^c$.
\end{thm}

\begin{thm} 
Let $A(t)$ and $B(t)$ be set-valued functions that are continuous
at $t_0$. Then, the following hold:
\begin{enumerate}[$(1)$]
\item $A(t) \cup B(t)$ is continuous at $t_0$.
\item $A(t) \cap B(t)$ is continuous at $t_0$.
\end{enumerate}
\end{thm}

\begin{cor}
Let $n$ be a positive integer.
If a set-valued function $A_i(t)$ is continuous
at $t_0$ for each $1 \le i \le n$,
then the following hold:
\begin{enumerate}[$(1)$]
\item
{\scriptsize $\displaystyle \bigcup_{1 \le i \le n}$} $A_i(t)$
is continuous at $t_0$.
\item
{\scriptsize $\displaystyle \bigcap_{1 \le i \le n}$} $A_i(t)$
is continuous at $t_0$.
\end{enumerate}
\end{cor}

\begin{thm} 
Let $A(t)$ be a set-valued function that is continuous at
$t_0$, and let $B$ be a set such that for all $x \in B$,
there exists $\delta >0$ such that for all $t \in \{ s
\in \mathbb{R} \bigm| |s - t_0| < \delta \}$
we have $x \in A(t)$.
Then, $B \subset A(t_0)$.
\end{thm}

\begin{thm} \label{thm4.const}
Let $A(t)$ be a set-valued function.
If there exists $\delta > 0$ such that $A(t_1) = A(t_2)$ for all
$t_1,t_2 \in \{ s
\in \mathbb{R} \bigm| |s - t_0| < \delta \}$
with $t_1 \neq t_2$,
then $A(t)$ is continuous at $t_0$.
\end{thm}

\vspace*{1\baselineskip}
We can adapt the above theorems to Definition~\ref{def4.left}
and~\ref{def4.right}.
The same results hold for one-sided continuity

Similarly to Theorem~\ref{thm3.coincide} in Subsection~\ref{sec3.thm},
we can prove the following theorem.

\begin{thm} \label{thm4.coincide}
Let $A(t)$ be a set-valued function.
Then, $\displaystyle \lim_{t \to t_0} A(t) = A(t_0)$
if and only if
$\displaystyle \lim_{t \to t_0 -0} A(t) = A(t_0)
= \displaystyle \lim_{t \to t_0 + 0} A(t)$.
\end{thm}
\begin{proof}
The proof of Theorem~\ref{thm4.coincide} is similar to
the proof of Theorem~\ref{thm3.coincide}.
\end{proof}

\section{Other Possibilities of Development}
\label{sec05}

Section~\ref{sec05} investigates other possibilities of development
in the study of set-valued functions.
Subsection~\ref{ES} presents a method to describe the behavior of
a set-valued function by bijection between two sets, \textit{Element Specification}.
Using this method, we define the differentiation of set-valued functions
in a Euclidean space
in Subsection~\ref{diffe}. In Subsection~\ref{multi},
we consider an extension to multivariable set-valued functions.
Finally, we adapt the convergence of set-valued functions at infinity
to sequences of sets in Subsection~\ref{adapt}.
This subsection investigates the equivalence to the equality of the
two limits of sequences of sets.

\subsection{Element Specification}
\label{ES}

In this subsection, we introduce one of the methods to describe the
behavior of a set-valued function, \textit{Element Specification}.
\footnote{
We derive the name ``Element Specification'' from
the \textit{specification} of the path of each changing
\textit{element}.
}
This is the method to specify the changing process of a set-valued function
by one-to-one correspondence.

Let an index $\lambda$ be an element of an index set $\Lambda$.
Then, we denote by $a_{\lambda}(t)$ an element of a set-valued function $A(t)$.
The element $a_{\lambda}(t)$ itself changes depending on the variable
$t \in \mathbb{R}$. For simplicity, suppose that
$a_{\lambda_i}(t) \neq a_{\lambda_j}(t)$ for all $\lambda_i, \lambda_j \in
\Lambda$ and for all $t \in \mathbb{R}$.
Element Specification defines the set-valued function $A(t)$ as the collection
of each changing element $a_{\lambda}(t)$, that is,
\begin{equation*}
A(t) = \{ a_{\lambda}(t) \in X \bigm| \lambda \in \Lambda \}.
\footnote{
Here, the equation
$A(t) = \{ a_{\lambda}(t) \in X \bigm| \lambda \in \Lambda \}$
can be also described 
as the union of each point set,
{\tiny $\displaystyle \bigcup_{\lambda \in \Lambda}$} \nolinebreak$\{a_{\lambda}(t)\}$.
}
\end{equation*}

Figure ~\ref{sec5_1_es1} sketches the changing process of each $a_{\lambda}(t)$.
Note that the cardinality of $A(t)$ coincides with that of $\Lambda$
and is constant for all $t \in \mathbb{R}$.
\footnote{
We owe this constancy to the assumption that
$a_{\lambda_i}(t) \neq a_{\lambda_j}(t)$ for all $\lambda_i, \lambda_j \in
\Lambda$ and for all $t \in \mathbb{R}$.
Without the assumption, we must say that
the cardinality of $A(t)$ is less than or equal to that of $\Lambda$.
This is because some element $a_{\lambda_i}(t_0)$ may be the same
as another element $a_{\lambda_j}(t_0)$ at some point $t_0$.
In that case,
the cardinality of $A(t)$ depends on the degree of the ``overlap"
among elements $a_{\lambda}(t)$.
}
Thus, there exists one-to-one correspondence between
$A(t_1)$ and $A(t_2)$ for all $t_1,t_2 \in \mathbb{R}$.
Here, each element $a_{\lambda}(t)$ constructs
a bijection between $A(t_1)$ and $A(t_2)$.

\begin{figure}
\centering
\input{sec5_1_es1.tex}
\caption{We list the changing of each element $a_{\lambda}(t)$
at $t_1,t_2$, and $t_3$.
The arrows indicate how $a_{\lambda}(t)$
changes depending on $t$. The set-valued function $A(t)$ is defined as
the collection of $a_{\lambda}(t)$.
Thus, the change of $a_{\lambda}(t)$ means the change of $A(t)$.}
\label{sec5_1_es1}
\end{figure}
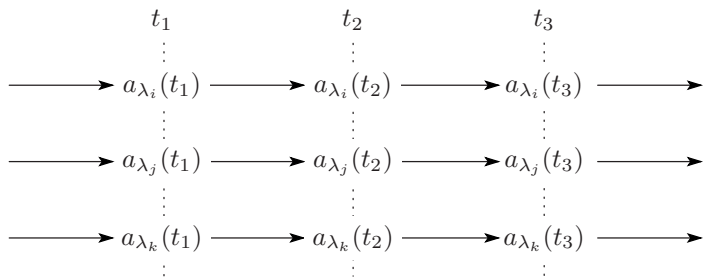%

Element Specification utilizes the property of cardinality that
an infinite set has a one-to-one correspondence
with its infinite proper-subset. This means that a set-valued function described
by Element Specification can change independently of containment
if its cardinality is infinite.

For example, consider the set-valued function
$B(t)= \{ (x,y) \in \mathbb{R}^2 \bigm|  x^2 + y^2 < t^2 \}$, which is defined for $t>0$.
We can describe $B(t)$ with Element Specification as follows:
\begin{align*}
B(t) &= \{ b_{p,q}(t) \in \mathbb{R}^2
\bigm| 0 \le p <1, \, 0 \le q < 2\pi \} \\
&= \{ (tp \cos q , tp \sin q) \in \mathbb{R}^2
\bigm| 0 \le p <1, \, 0 \le q < 2\pi \},
\end{align*}
where we define $B(t)$ for $t>0$.
Let $t_1, t_2$ with $t_1 < t_2$. We find a one-to-one correspondence
between $B(t_1)$ and $B(t_2)$, though $B(t_1) \subset B(t_2)$.
Figure~\ref{sec5_1_es2} pictures two sets $B(1)$ and $B(2)$
and gives an example on how each element $b_{p,q}(t)$ changes when we
specify the values $p$ and $q$.

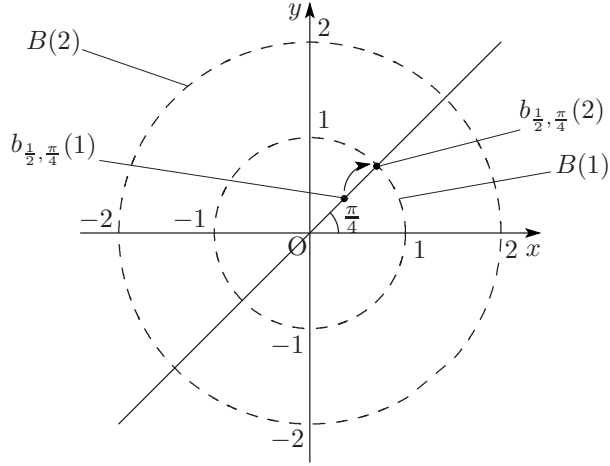
\begin{figure}
\centering
\input{sec5_1_es2.tex}
\caption{
The sets $B(1)$ and $B(2)$ are pictured above.
By substituting the specific value of $p$ and $q$ in
$b_{p,q}(t) = (tp \cos q, tp \sin q)$,
we see how each element $b_{p,q}(t)$ changes
depending on $t > 0$.
The arrow in the figure indicates that the element $b_{\frac{1}{2},\frac{\pi}{4}}(1)
= (1/2\sqrt{2}, 1/2\sqrt{2})$ changes to the element
$b_{\frac{1}{2},\frac{\pi}{4}}(2) = (1/\sqrt{2}, 1/\sqrt{2})$ at $t=2$.
}
\label{sec5_1_es2}
\end{figure}

\subsection{Differentiation of Set-Valued Functions in a Euclidean Space}
\label{diffe}

In this Subsection, we define the differentiation of set-valued functions in a Euclidean space
by using Element Specification. In advance, we note that the
differentiation introduced here is not general for two reasons.
First, we differentiate set-valued functions that are described by Element
Specification. Second, we differentiate set-valued functions in a Euclidean space.

We start by setting a set-valued function used in this subsection.
Let $A(t) \subset \mathbb{R}^n$ be a set-valued function with
Element Specification. Then, we set an index $\lambda$
an element of an index set $\Lambda$,
and define the set-valued function $A(t)$ as
\begin{align*}
A(t) &= \{ a_\lambda(t) \in \mathbb{R}^n \bigm| \lambda \in \Lambda \} \\
&= \{ \, \big(a^1_{\lambda}(t), a^2_{\lambda}(t), \dots,
a^n_{\lambda}(t)\big) \in \mathbb{R}^n \bigm| \lambda \in \Lambda \, \}.
\end{align*}
Suppose that $a^i_{\lambda}(t)$ is differentiable at $t$ for each
$1 \le i \le n$. Then, the differentiation of $A(t)$ at $t$,
denoted $\displaystyle \frac{dA(t)}{dt}$, is defined by
\begin{align*}
\frac{dA(t)}{dt} &= \biggl\{ \frac{d a_\lambda(t)}{dt} \in \mathbb{R}^n
\biggm| \lambda \in \Lambda \biggr\} \\
&= \biggl\{ \left(
\frac{d a^1_\lambda(t)}{dt}, \frac{d a^2_\lambda(t)}{dt},
\dots, \frac{d a^n_\lambda(t)}{dt}
\right) \in \mathbb{R}^n
\biggm| \lambda \in \Lambda \biggr\}.
\end{align*}

We next define the addition of $A(t)$ and $\displaystyle \frac{dA(t)}{dt}$,
denoted by $A(t) + \displaystyle \frac{dA(t)}{dt}$, as follows:
\begin{align*}
A(t) + \displaystyle \frac{dA(t)}{dt}
&= \biggl\{ a_\lambda(t) + \frac{d a_\lambda(t)}{dt} \in \mathbb{R}^n
\biggm| \lambda \in \Lambda \biggr\} \\
&= \biggl\{ \left( a^1_{\lambda}(t) +
\frac{d a^1_\lambda(t)}{dt},
\dots, a^n_{\lambda}(t) + \frac{d a^n_\lambda(t)}{dt}
\right) \in \mathbb{R}^n
\biggm| \lambda \in \Lambda \biggr\}
\end{align*}
Now consider the set-valued function $B(t)$ given by
\begin{align*}
B(t) &= \{ b_{p,q}(t) \in \mathbb{R}^2
\bigm| 0 \le p <1, \, 0 \le q < 2\pi \} \\
&= \{ (tp \cos q , tp \sin q) \in \mathbb{R}^2
\bigm| 0 \le p <1, \, 0 \le q < 2\pi \},
\end{align*}
where we define $B(t)$ for $t>0$.
By definition, we have
\begin{align*}
\frac{d b_{p,q}(t)}{dt}
&= \left( \frac{d}{dt}(tp \cos q), \frac{d}{dt}(tp \sin q) \right) \\
&= (p \cos q, p \sin q).
\end{align*}
Hence, we get
\begin{align*}
\frac{dB(t)}{dt} &= \biggl\{ \frac{d b_{p,q}(t)}{dt}
\in \mathbb{R}^2 \biggm|  0 \le p <1, \, 0 \le q < 2\pi \biggr\} \\
&= \biggl\{ (p \cos q, p \sin q)
\in \mathbb{R}^2 \biggm|  0 \le p <1, \, 0 \le q < 2\pi \biggr\}.
\end{align*}
Using this, we obtain the addition of $B(t)$ and $\frac{d B(t)}{dt}$,
that is,
\begin{align*}
B(t) + \frac{d B(t)}{dt}
&= \biggl\{ 
b_{p,q}(t) + \frac{d b_{p,q}(t)}{dt} \in \mathbb{R}^2
\biggm|  0 \le p <1, \, 0 \le q < 2\pi
\biggr\} \\
&= \biggl\{ (tp \cos q, tp \sin q) +  (p \cos q, p \sin q) \in \mathbb{R}^2
\biggm| 0 \le p <1, \, 0 \le q < 2\pi \biggr\} \\
&= \biggl\{ \big( (t+1)p \cos q, (t+1)p \sin q \big) \in \mathbb{R}^2
\biggm| 0 \le p <1, \, 0 \le q < 2\pi \biggr\} \\
&= \biggl\{ 
b_{p,q}(t+1) \in \mathbb{R}^2
\biggm|  0 \le p <1, \, 0 \le q < 2\pi
\biggr\} \\
&= B(t+1).
\end{align*}
Here, $dB(t)/dt$ means the difference between $B(t)$ and $B(t+1)$.
\footnote{
Naturally, the differentiation of set-valued functions does not necessarily
mean the difference between a set-valued function at $t$ and at $t+1$.
}
We can interpret the differentiation of set-valued functions as the change
in a set-valued function per infinitesimal change in a variable.

\subsection{Extension to Multivariable Set-Valued Functions}
\label{multi}

Subsection~\ref{multi} gives an example of an extension
to multivariable set-valued functions. We consider $n$-variable set-valued functions
in this subsection.

Let $\mathbf{t} = (t^1, \dots, t^n)$ and denote by $A(\mathbf{t})$
a $n$-variable set-valued function.
If we set the distance between points $\mathbf{t_i}$ and $\mathbf{t_j}$,
we can extend Definition~\ref{def3.conver} and~\ref{def4.conti}
to $n$-variable set-valued functions.

Here, we use the Euclidean distance. That is,
\begin{equation*}
\| \mathbf{t}_i - \mathbf{t}_j \|
= \sqrt{\smash[b]{(t^1_i - t^1_j)^2 +
\dotsb + (t^n_i - t^n_j)}},
\end{equation*}
where $\mathbf{t}_i = (t^1_i, \dots , t^n_i )$ and
$\mathbf{t}_j = (t^1_j, \dots , t^n_j )$.
Then, we define the convergence of $n$-variable set-valued functions at a point
as follows:
\begin{defi} \label{def5.conver}
Let $A(t)$ be a set-valued function that depends on the variable $\mathbf{t} \in \mathbb{R}^n$,
and let $A$ be a set. We say that $A(t)$ converges to $A$
as $\mathbf{t} \to \mathbf{t}_0$ if
for all $\mathbf{t}_i \in \mathbb{R}\,\, (\mathbf{t}_i \neq \mathbf{t}_0)$
and for all $x \in A(\mathbf{t}_i) \triangle A$,
there exists $\delta >0$ such that for all $\mathbf{t} \in \{ \mathbf{s}
\in \mathbb{R}^n \bigm| 0< \| \mathbf{s} - \mathbf{t}_0 \| < \delta \}$
we have  $x \not\in A(\mathbf{t}) \triangle A$.
\end{defi}
Similarly, we define the continuity of $n$-variable set-valued functions at a point
as follows:
\begin{defi} \label{def5.conti}
Let $A(t)$ be a set-valued function that depends on the variable $\mathbf{t} \in \mathbb{R}^n$.
We say that $A(t)$ is continuous at $\mathbf{t}_0$ if
for all $\mathbf{t}_i \in \mathbb{R}$
and for all $x \in A(\mathbf{t}_i) \triangle A(\mathbf{t}_0)$,
there exists $\delta >0$ such that for all $\mathbf{t} \in \{ \mathbf{s}
\in \mathbb{R}^n \bigm| \| \mathbf{s} - \mathbf{t}_0 \| < \delta \}$
we have  $x \not\in A(\mathbf{t}) \triangle A(\mathbf{t}_0)$.
\end{defi}
We can derive similar results to those of Subsection~\ref{sec3.thm}
and~\ref{sec4.thm} in almost all the same way.

\subsection{Adaption to Sequences of Sets} 
\label{adapt}

Adapting Definition~\ref{def2.conver} to sequences of sets,
we can define their convergence
away from the equality of the superior limit
and the inferior limit.
We formulate it as follows:
\begin{defi} \label{def5.adapt}
Let $\{ A_n \}_{n=1}^{\infty}$ be a sequence of sets
and let $A$ be a set. We say that $\{ A_n \}_{n=1}^{\infty}$ converges to $A$ if
for all $n_i \in \mathbb{N}$ and for all $x \in A_{n_i} \triangle A$,
there exists $n_j \in \mathbb{N}$ such that for all $n \ge n_j$
we have  $x \not\in A_n \triangle A$.
%and we write $\displaystyle \lim_{n \to \infty}A_n = A$.
\end{defi}

We can prove that Definition~\ref{def5.adapt} is equivalent
to the equality of the superior limit and inferior limit
of the sequence $\{ A_n \}_{n=1}^{\infty}$.
To prove the equivalence, we adapt theorems in Subsection~\ref{sec2.thm}
to sequences of sets and get:
\footnote{
Be careful that the convergence mentioned in the theorems below
is based on Definition~\ref{def5.adapt}, not on the equality.
}

\begin{thm} \label{thm5.3} %%%%%%%%%%%%%%%%%%%%%%%%%%%%%%%%%%% Theorem
Let $\{ A_n \}_{n=1}^{\infty}$ be a sequence of sets and $A$ a set. Then,
$\{ A_n \}_{n=1}^{\infty}$ converges to $A$ if and only if
for all $x \in X$, there exists $n_0 \in \mathbb{N}$ such that
for all $n \ge n_0$ we have $x \in \{ A_n \triangle A \}^c$.
\end{thm}

\begin{thm} \label{thm5.1}%%%%%%%%%%%%%%%%%%%%%%%%%%%%%% Theorem 2.3
Let $\{ A_n \}_{n=1}^{\infty}$ be a sequence of sets that converges to a set $A$.
Then, for all $x \in A$, there exists $n_0 \in \mathbb{N}$
such that for all $n \ge n_0$ we have $x \in A_n$.
\end{thm}

\begin{cor} \label{cor5.2}
Let $\{ A_n \}_{n=1}^{\infty}$ be a sequence of sets that converges to a set $A$.
Then, for all $x \in A^c$, there exists $n_0 \in \mathbb{N}$
such that for all $n \ge n_0$ we have $x \in A_n^c$.
\end{cor}
These theorems can be derived in a similar way to Subsection~\ref{sec2.thm}.
Using theorems above, we first prove the following:

\begin{thm} \label{thm5.4}
Let $\{ A_n \}_{n=1}^{\infty}$ be a sequence of sets and $A$ a set. Then,
$\{ A_n \}_{n=1}^{\infty}$ converges to $A$ if and only if we have that
for all $x \in A$ there exists $n_1 \in \mathbb{N}$
such that for all $n \ge n_1$ we have $x \in A_n$,
and that for all $x \in A^c$ there exists $n_2 \in \mathbb{N}$
such that for all $n \ge n_2$ we have $x \in A_n^c$.
\end{thm}

\begin{proof}
Suppose that $\{ A_n \}_{n=1}^{\infty}$ converges to $A$.
It follows from Theorem~\ref{thm5.1}
that for all $x \in A$ there exists $n_1 \in \mathbb{N}$
such that for all $n \ge n_1$ we have $x \in A_n$.
On the other hand, it follows from Corollary~\ref{cor5.2}
that for all $x \in A^c$ there exists $n_2 \in \mathbb{N}$
such that for all $n \ge n_2$ we have $x \in A_n^c$.

Next suppose that for all $x \in A$ there exists $n_1 \in \mathbb{N}$
such that for all $n \ge n_1$ we have $x \in A_n$,
and that for all $x \in A^c$ there exists $n_2 \in \mathbb{N}$
such that for all $n \ge n_2$ we have $x \in A_n^c$.
Then, on the one hand, it holds that for all $x \in A$ there exists $n_1 \in \mathbb{N}$
such that for all $n \ge n_1$ we have $x \in A_n \cap A$,
because we pick the element $x$ from $A$.
Similarly, on the other hand, it holds that for all $x \in A^c$ there exists $n_2 \in \mathbb{N}$
such that for all $n \ge n_2$ we have $x \in A_n^c$.
Note that $A_n \cap A \subset \{ A_n \triangle A \}^c$
and that $A_n^c \cap A^c \subset \{ A_n \triangle A \}^c$,
because $\{ A_n \triangle A\}^c
=\{ A_n \cap A \} \cup \{ A_n^c \cap A^c \}$.
Since $A_n \cap A$ is a subset of $\{ A_n \triangle A \}^c$,
we see that every element of $A$ belongs to $\{ A_n \triangle A \}^c$
for all $n \ge n_1$.
Thus, for all $x \in A$ there exists $n_1 \in \mathbb{N}$
such that for all $n \ge n_1$ we have $x \in \{ A_n \triangle A \}^c$.
Similarly, since $A_n^c \cap A^c$ is subset of $\{ A_n \triangle A \}^c$,
we see that every element of $A^c$ belongs to $\{ A_n \triangle A \}^c$
for all $n \ge n_2$.
Thus, for all $x \in A^c$ there exists $n_2 \in \mathbb{N}$
such that for all $n \ge n_2$ we have $x \in \{ A_n \triangle A \}^c$.
Let $n_0 = \max \{ n_1, n_2 \}$.
Then, we see that every element of both $A$ and $A^c$ belongs to
$\{ A_n \triangle A \}^c$ for all $n \ge n_0$.
Thus, for all $x \in A \cup A^c$ there exists $n_0 \in \mathbb{N}$
such that for all $n \ge n_0$ we have $x \in \{ A_n \triangle A \}^c$.
Since $A \cup A^c = X$, it follows that for all $x \in X$ there exists $n_0 \in \mathbb{N}$
such that for all $n \ge n_0$ we have $x \in \{ A_n \triangle A \}^c$.
By Theorem~\ref{thm5.3}, we conclude that
$\{ A_n \}_{n=1}^{\infty}$ converges to $A$.

We have thus completed the proof.
\end{proof}

Here, we refer to Theorem 1.3.12 of Niizeki~\cite{Niizeki}.
Let $\{ A_n \}_{n=1}^{\infty}$ be a sequence of sets and $A$ a set.
The theorem states that
{\tiny $\displaystyle \bigcap_{n=1}^{\infty} \bigcup_{k=n}^{\infty}$} $A_k$ =
$A$ = {\tiny
$\displaystyle \bigcup_{n=1}^{\infty} \bigcap_{k=n}^{\infty}$} $A_k$
if and only if
for all $(x , y) \in A \times A^c$, there exists $n_{(x,y)} \in \mathbb{N}$
such that for all $n \ge n_{(x,y)}$ we have $(x , y) \in A_n \times A_n^c$.
Equivalently,
{\tiny $\displaystyle \bigcap_{n=1}^{\infty} \bigcup_{k=n}^{\infty}$} $A_k$ =
$A$ = {\tiny
$\displaystyle \bigcup_{n=1}^{\infty} \bigcap_{k=n}^{\infty}$} $A_k$
if and only if we have that
for all $x \in A$ there exists $n_1 \in \mathbb{N}$
such that for all $n \ge n_1$ we have $x \in A_n$,
and that for all $y \in A^c$ there exists $n_2 \in \mathbb{N}$
such that for all $n \ge n_2$ we have $y \in A_n^c$.

Thus, we find from Theorem~\ref{thm5.4} that
Definition~\ref{def5.adapt} is equivalent
to the equality of the superior limit and the inferior limit of the sequence $\{ A_n \}_{n=1}^{\infty}$.
This fact ensures that both Definition~\ref{def5.adapt} and the equality
of the two limits could serve as a definition of convergence of
sequences of sets.
\footnote{
In addition, Theorem~\ref{thm5.3} could also serve as the definition.
}

\appendix
\section{Appendix}
\label{appa}
Appendix provides other proofs of Theorem~\ref{thm2.8} and~\ref{thm2.3}
in Subsection~\ref{sec2.thm}. We first prove Theorem~\ref{thm2.3} by Theorem~\ref{thm2.start}
and~\ref{thm2.2}, and then prove Theorem~\ref{thm2.8} by Theorem~\ref{thm2.3}.

Note that the order of the proofs of Theorem~\ref{thm2.8} and~\ref{thm2.3} in
Appendix is the opposite order of those in Subsection~\ref{sec2.thm}.
In Appendix, the proof of Theorem~\ref{thm2.3} does not require Theorem~\ref{thm2.8},
whereas the proof of Theorem~\ref{thm2.8} requires Theorem~\ref{thm2.3}.

The next theorem is Theorem~\ref{thm2.3} in Subsection~\ref{sec2.thm}.
In proving it, be careful that we can apply Theorem~\ref{thm2.start} and~\ref{thm2.2}.
\begin{thm} \label{A.1} %%%%%%%%%%%%%%%%%%%%%%%%%%%%%%%%%%%%%%%%%%%%%%%%%%%%%
Let $A(t)$ be a set-valued function that converges to limit $A$ as $t \to\infty$.
Then, for all $x \in A$, there exists $t_0 \in \mathbb{R}$
such that for all $t \ge t_0$ we have $x \in A(t)$.
\end{thm}
\begin{proof}
Since $A(t)$ converges to $A$ as $t \to \infty$,
for all $t_{i_1} \in \mathbb{R}$ and for all $x \in A(t_{i_1}) \triangle A$,
there exists $t_{j_1} \in \mathbb{R}$ such that 
for all $t \ge t_{j_1}$ we have $x \in  \{ A(t) \triangle A \}^c$.
Notice that $\{ A(t_{i_1})^c \cap A \} \subset \{ A(t_{i_1}) \triangle A \}$, because
\begin{align*}
 A(t_{i_1}) \triangle A
&= \{ A(t_{i_1}) \cup A\} \cap \{ A(t_{i_1})^c \cup A^c \} \\
&\supset A \cap A(t_{i_1})^c.
\end{align*}
Since  $A(t_{i_1})^c \cap A$ is a subset of $A(t_{i_1}) \triangle A$,
we see that every element of
 $A(t_{i_1})^c \cap A$ belongs to $\{ A(t) \triangle A \}^c$
for all $t \ge t_{j_1}$. Hence,
for all $t_{i_1} \in \mathbb{R}$ and for all $x \in A(t_{i_1})^c \cap A$
there exists $t_{j_1} \in \mathbb{R}$ such that
for all $t \ge t_{j_1}$ we have $x \in  \{ A(t) \triangle A \}^c$.

On the other hand, it holds by Theorem~\ref{thm2.2} that
for all $t_{i_2} \in \mathbb{R}$ and for all $x \in \{ A(t_{i_2}) \triangle A \}^c$
there exists $t_{j_2} \in \mathbb{R}$ such that
for all $t \ge t_{j_2}$ we have $x \in  \{ A(t) \triangle A \}^c$.
Notice that $\{ A(t_{i_2}) \cap A \} \subset \{ A(t_{i_2}) \triangle A \}^c$, because
\begin{align*}
\{ A(t_{i_2}) \triangle A \}^c 
&= \{ A(t_{i_2}) \cup A\}^c \cup \{ A(t_{i_2}) \cap A \} \\
&\supset \{ A(t_{i_2}) \cap A \} .
\end{align*}
Since $A(t_{i_2}) \cap A$ is a subset of $\{ A(t_{i_2}) \triangle A \}^c$,
we see that every element of
$A(t_{i_2}) \cap A$ belongs to $\{ A(t) \triangle A \}^c$
for all $t \ge t_{j_2}$.
Hence,
for all $t_{i_2} \in \mathbb{R}$ and for all $x \in A(t_{i_2}) \cap A$
there exists $t_{j_2} \in \mathbb{R}$ such that
for all $t \ge t_{j_2}$ we have $x \in  \{ A(t) \triangle A \}^c$.

Take the same point $t_i$ for $t_{i_1}$, $t_{i_2}$ and let
$t_0 = \max \{ t_{j_1}, t_{j_2} \}$. Then, every element of both
$A(t_i)^c \cap A$ and  $A(t_i) \cap A$ belongs to $\{ A(t) \triangle A \}^c$
for all $t \ge t_0$. Thus,
for all $t_i \in \mathbb{R}$ and for all $x \in \{ A(t_i)^c \cap A \} \cup
\{A(t_i) \cap A \}$,
there exists $t_0 \in \mathbb{R}$ such that
for all $t \ge t_0$
we have $x \in  \{ A(t) \triangle A \}^c$.
By using a distributive law, we can derive the equation
$\{ A(t_i)^c \cap A \} \cup \{A(t_i) \cap\nolinebreak  A \} = A$ as follows:
\begin{align*}
\{ A(t_i)^c \cap A \} \cup \{A(t_i) \cap A \} &=
\{ A(t_i)^c \cup A(t_i) \} \cap A \\
&= X \cap A \\
&= A.
\end{align*}
Therefore,
for all $t_i \in \mathbb{R}$ and for all $x \in A$,
there exists $t_j \in \mathbb{R}$ such that
for all $t \ge t_0$
we have $x \in  \{ A(t) \triangle A \}^c$.
Since the set $A$ is independent of the choice of $t_i$,
we can leave out the beginning part ``for all $t_i \in \mathbb{R}$".
Thus,
for all $x \in A$
there exists $t_0 \in \mathbb{R}$ such that
for all $t \ge t_0$
we have $x \in  \{ A(t) \triangle A \}^c$.
Then, this substantially implies that
for all $x \in A$
there exists $t_0 \in \mathbb{R}$ such that
for all $t \ge t_0$
we have $x \in  \{ A(t) \triangle A \}^c \cap A$,
because we pick the element $x$ from $A$.

To prove Theorem~\ref{A.1}, we next show that
\begin{equation*}
\bigl[ \{ A(t) \triangle A \}^c \cap A \bigr] \subset A(t).
\end{equation*}
We can derive the inclusion above by using a distributive law. That is,
\begin{align*}
\bigl[ \{ A(t) \triangle A \}^c \cap A \bigr]
%&= \bigl[ \{ A(t) \cup A \} \cap \{ A(t) \cap A \}^c \bigr]^c \cap A \\
&= \bigl[ \{ A(t)^c \cap A^c \} \cup \{ A(t) \cap A \} \bigr] \cap A \\
&= \big[ \{ A(t)^c \cap A^c \} \cap A \bigr] \cup
\bigl[ \{ A(t) \cap A \} \cap A \bigr] \\
%&= \big[ \{ A(t)^c \cap A^c \} \cap A \bigr] \cup
%\bigl[ \{ A(t) \cap A \} \cap A \bigr] \\
&= \big[ A(t)^c \cap \{A^c \cap A \} \bigr] \cup
\bigl[ A(t) \cap \{ A \cap A \} \bigr] \\
%&= \{ A(t)^c \cap \emptyset \} \cup \{ A(t) \cap A \} \\
&= \emptyset \cup [ A(t) \cap A ] \\
&= A(t) \cap A \\
&\supset A(t).
\end{align*}
Since $\{ A(t) \triangle A \}^c \cap A$ is a subset of $A(t)$ for all $t \in \mathbb{R}$,
we see that every element of $A$ belongs to
$A(t)$ for all $t \ge t_0$. Thus, we conclude that
for all $x \in A$
there exists $t_0 \in \mathbb{R}$ such that
for all $t \ge t_0$
we have $x \in A(t)$.
Therefore, we have proved Theorem~\ref{A.1}.
\end{proof}

Using Theorem~\ref{A.1}, we can get the next corollary.
\begin{cor} \label{A.2} %%%%%%%%%%%%%%%%%%%%%%%%%%%%%%%%%%%%%%%%%%%%%%%%%%%%%%%%%%
Let $A(t)$ be a set-valued function that converges to limit $A$ as $t \to\infty$.
Then, for all $x \in A^c$, there exists $t_0 \in \mathbb{R}$
such that for all $t \ge t_0$ we have $x \in A(t)^c$.
\end{cor}
This corollary is  Corollary~\ref{cor2.4} in Subsection~\ref{sec2.thm}.
The proof of Corollary~\ref{A.2} is the same as that of  Corollary~\ref{cor2.4}.

The next theorem is Theorem~\ref{thm2.8} in Subsection~\ref{sec2.thm}.

\begin{thm} \label{A.3}%%%%%%%%%%%%%%%%%%%%%%%%%%%%%%%%%%%%%%%%%%%%%%%%%%%
Let $A(t)$ be a set-valued function and $A$ a set. Then,
$A(t)$ converges to $A$ as $t \to \infty$ if and only if
for all $x \in X$, there exists $t_0 \in \mathbb{R}$ such that
for all $t \ge t_0$ we have $x \in \{ A(t) \triangle A \}^c$.
\end{thm}

\begin{proof}
Suppose that $A(t)$ converges to $A$ as $t \to \infty$.
Then, we have from Theorem~\ref{A.1} that
for all $x \in A$
there exists $t_1 \in \mathbb{R}$ such that
for all $t \ge t_1$
we have $x \in A(t)$.
The statement above substantially implies
for all $x \in A$
there exists $t_1 \in \mathbb{R}$ such that
for all $t \ge t_1$
we have $x \in A(t) \cap A$,
because we pick the element $x$ from $A$.
On the other hand, we have from Corollary~\ref{A.2} that
for all $x \in A^c$
there exists $t_2 \in \mathbb{R}$ such that 
for all $t \ge t_2$
we have $x \in A(t)^c$.
The statement above substantially implies
for all $x \in A^c$
there exists $t_2 \in \mathbb{R}$ such that 
for all $t \ge t_2$
we have $x \in A(t)^c \cap A^c$.
Notice that $\{ A(t) \cap A \} \subset \{ A(t) \triangle A\}^c$
and $\{A(t)^c \cap A^c \} \subset \{ A(t) \triangle A\}^c$ because
$\{ A(t) \triangle A\}^c
=\{ A(t) \cap A \} \cup \{ A(t)^c \cap A^c \}$.
Since $A(t) \cap A$ is a subset of $\{ A(t) \triangle A\}^c$
for all $t \in \mathbb{R}$, we see that
every element of $A$ belongs to $\{ A(t) \triangle A\}^c$ for all $t \ge t_1$.
Thus,
for all $x \in A$
there exists $t_1 \in \mathbb{R}$ such that 
for all $t \ge t_1$
we have $x \in \{ A(t) \triangle A\}^c$.
Similarly, since $A(t)^c \cap A^c$ is a subset of $\{ A(t) \triangle A\}^c$
for all $t \in \mathbb{R}$, we see that
every element of $A^c$ belongs to $\{ A(t) \triangle A\}^c$ for all $t \ge t_2$.
Thus,
for all $x \in A^c$
there exists $t_2 \in \mathbb{R}$ 
for all $t \ge t_2$
we have $x \in \{ A(t) \triangle A\}^c$.
Let $t_0 = \max \{ t_1,t_2\}$.
Then, every element of both $A$ and $A^c$ belongs to $\{ A(t) \triangle A\}^c$
for all $t \ge t_0$. Hence,
for all $x \in A \cup A^c$
there exists $t_0 \in \mathbb{R}$ such that 
for all $t \ge t_0$
we have $x \in \{ A(t) \triangle A\}^c$.
Since $A \cup A^c = X$,
for all $x \in X$
there exists $t_0 \in \mathbb{R}$ such that 
for all $t \ge t_0$
we have $x \in \{ A(t) \triangle A\}^c$.

We can prove the converse in the same way as the proof of Theorem~\ref{thm2.8}
in Subsection~\ref{sec2.thm}

Therefore, we have proved the theorem.
\end{proof}

\vspace{6pt}

\begin{flushleft}
{\large \textsc{TAKEFUMI FUJIMOTO}} \\*
\vspace{3pt}
%Department of International Studies\\*
%School of Frontier Sciences\\*
%The University of Tokyo \\*
E-mail address: the-gift-of-31t@hotmail.co.jp; 2876681407@edu.k.u-tokyo.ac.jp
\end{flushleft}

\end{document}

%% file: sec2_1_def.tex
%WinTpicVersion4.28b
{\unitlength 0.1in
\begin{picture}( 48.0000, 18.0000)(  6.0000,-23.3500)
% ELLIPSE 2 0 3 0 Black White
% 4 1600 1400 2600 600 2600 600 2600 600
% 
{\color[named]{Black}{%
\special{pn 8}%
\special{ar 1600 1400 1000 800  0.0000000  6.2831853}%
}}%
% ELLIPSE 2 0 3 0 Black White
% 4 1600 1400 2100 1100 2100 1100 2100 1100
% 
{\color[named]{Black}{%
\special{pn 8}%
\special{ar 1600 1400 500 300  0.0000000  6.2831853}%
}}%
% LINE 3 0 3 0 Black White
% 60 1550 610 620 1540 1660 620 650 1630 1780 620 690 1710 1100 1420 740 1780 1130 1510 790 1850 1190 1570 850 1910 1260 1620 910 1970 1350 1650 980 2020 1440 1680 1060 2060 1550 1690 1140 2100 1670 1690 1230 2130 1820 1660 1330 2150 2510 1090 1420 2180 2550 1170 1540 2180 2580 1260 1650 2190 2600 1360 1790 2170 2590 1490 1940 2140 2540 1660 2140 2060 2460 1020 2100 1380 2410 950 2070 1290 2350 890 2010 1230 2290 830 1940 1180 2220 780 1850 1150 2140 740 1760 1120 2060 700 1650 1110 1970 670 1530 1110 1870 650 1380 1140 1410 630 600 1440 1260 660 610 1310 1060 740 660 1140
% 
{\color[named]{Black}{%
\special{pn 4}%
\special{pa 1550 610}%
\special{pa 620 1540}%
\special{fp}%
\special{pa 1660 620}%
\special{pa 650 1630}%
\special{fp}%
\special{pa 1780 620}%
\special{pa 690 1710}%
\special{fp}%
\special{pa 1100 1420}%
\special{pa 740 1780}%
\special{fp}%
\special{pa 1130 1510}%
\special{pa 790 1850}%
\special{fp}%
\special{pa 1190 1570}%
\special{pa 850 1910}%
\special{fp}%
\special{pa 1260 1620}%
\special{pa 910 1970}%
\special{fp}%
\special{pa 1350 1650}%
\special{pa 980 2020}%
\special{fp}%
\special{pa 1440 1680}%
\special{pa 1060 2060}%
\special{fp}%
\special{pa 1550 1690}%
\special{pa 1140 2100}%
\special{fp}%
\special{pa 1670 1690}%
\special{pa 1230 2130}%
\special{fp}%
\special{pa 1820 1660}%
\special{pa 1330 2150}%
\special{fp}%
\special{pa 2510 1090}%
\special{pa 1420 2180}%
\special{fp}%
\special{pa 2550 1170}%
\special{pa 1540 2180}%
\special{fp}%
\special{pa 2580 1260}%
\special{pa 1650 2190}%
\special{fp}%
\special{pa 2600 1360}%
\special{pa 1790 2170}%
\special{fp}%
\special{pa 2590 1490}%
\special{pa 1940 2140}%
\special{fp}%
\special{pa 2540 1660}%
\special{pa 2140 2060}%
\special{fp}%
\special{pa 2460 1020}%
\special{pa 2100 1380}%
\special{fp}%
\special{pa 2410 950}%
\special{pa 2070 1290}%
\special{fp}%
\special{pa 2350 890}%
\special{pa 2010 1230}%
\special{fp}%
\special{pa 2290 830}%
\special{pa 1940 1180}%
\special{fp}%
\special{pa 2220 780}%
\special{pa 1850 1150}%
\special{fp}%
\special{pa 2140 740}%
\special{pa 1760 1120}%
\special{fp}%
\special{pa 2060 700}%
\special{pa 1650 1110}%
\special{fp}%
\special{pa 1970 670}%
\special{pa 1530 1110}%
\special{fp}%
\special{pa 1870 650}%
\special{pa 1380 1140}%
\special{fp}%
\special{pa 1410 630}%
\special{pa 600 1440}%
\special{fp}%
\special{pa 1260 660}%
\special{pa 610 1310}%
\special{fp}%
\special{pa 1060 740}%
\special{pa 660 1140}%
\special{fp}%
}}%
% STR 2 0 3 0 Black White
% 4 1600 1000 1600 1100 5 0 1 0
% $A(t_i)$
\put(16.0000,-11.0000){\makebox(0,0){{\colorbox[named]{White}{$A(t_i)$}}}}%
\put(16.0000,-11.0000){\makebox(0,0){{\colorbox[named]{White}{\color[named]{Black}{$A(t_i)$}}}}}%
% STR 2 0 3 0 Black White
% 4 1600 500 1600 600 5 0 1 0
% $A$
\put(16.0000,-6.0000){\makebox(0,0){{\colorbox[named]{White}{$A$}}}}%
\put(16.0000,-6.0000){\makebox(0,0){{\colorbox[named]{White}{\color[named]{Black}{$A$}}}}}%
% ELLIPSE 2 0 3 0 Black White
% 4 4400 1400 3400 600 3400 600 3400 600
% 
{\color[named]{Black}{%
\special{pn 8}%
\special{ar 4400 1400 1000 800  0.0000000  6.2831853}%
}}%
% ELLIPSE 2 0 3 0 Black White
% 4 4400 1400 5200 800 5200 800 5200 800
% 
{\color[named]{Black}{%
\special{pn 8}%
\special{ar 4400 1400 800 600  0.0000000  6.2831853}%
}}%
% LINE 3 0 3 0 Black White
% 74 4160 640 3410 1390 3610 1310 3420 1500 3600 1440 3450 1590 3620 1540 3480 1680 3660 1620 3520 1760 3710 1690 3570 1830 3760 1760 3630 1890 3830 1810 3690 1950 3900 1860 3760 2000 3980 1900 3840 2040 4060 1940 3920 2080 4160 1960 4000 2120 4260 1980 4090 2150 4370 1990 4190 2170 4500 1980 4300 2180 4640 1960 4410 2190 4820 1900 4540 2180 5390 1450 4690 2150 5350 1610 4870 2090 5390 1330 5170 1550 5370 1230 5200 1400 5330 1150 5190 1290 5290 1070 5150 1210 5240 1000 5110 1130 5190 930 5060 1060 5130 870 4990 1010 5060 820 4930 950 4990 770 4850 910 4910 730 4760 880 4830 690 4670 850 4740 660 4580 820 4640 640 4470 810 4540 620 4350 810 4430 610 4210 830 4300 620 4050 870 4000 680 3430 1250 3750 810 3540 1020
% 
{\color[named]{Black}{%
\special{pn 4}%
\special{pa 4160 640}%
\special{pa 3410 1390}%
\special{fp}%
\special{pa 3610 1310}%
\special{pa 3420 1500}%
\special{fp}%
\special{pa 3600 1440}%
\special{pa 3450 1590}%
\special{fp}%
\special{pa 3620 1540}%
\special{pa 3480 1680}%
\special{fp}%
\special{pa 3660 1620}%
\special{pa 3520 1760}%
\special{fp}%
\special{pa 3710 1690}%
\special{pa 3570 1830}%
\special{fp}%
\special{pa 3760 1760}%
\special{pa 3630 1890}%
\special{fp}%
\special{pa 3830 1810}%
\special{pa 3690 1950}%
\special{fp}%
\special{pa 3900 1860}%
\special{pa 3760 2000}%
\special{fp}%
\special{pa 3980 1900}%
\special{pa 3840 2040}%
\special{fp}%
\special{pa 4060 1940}%
\special{pa 3920 2080}%
\special{fp}%
\special{pa 4160 1960}%
\special{pa 4000 2120}%
\special{fp}%
\special{pa 4260 1980}%
\special{pa 4090 2150}%
\special{fp}%
\special{pa 4370 1990}%
\special{pa 4190 2170}%
\special{fp}%
\special{pa 4500 1980}%
\special{pa 4300 2180}%
\special{fp}%
\special{pa 4640 1960}%
\special{pa 4410 2190}%
\special{fp}%
\special{pa 4820 1900}%
\special{pa 4540 2180}%
\special{fp}%
\special{pa 5390 1450}%
\special{pa 4690 2150}%
\special{fp}%
\special{pa 5350 1610}%
\special{pa 4870 2090}%
\special{fp}%
\special{pa 5390 1330}%
\special{pa 5170 1550}%
\special{fp}%
\special{pa 5370 1230}%
\special{pa 5200 1400}%
\special{fp}%
\special{pa 5330 1150}%
\special{pa 5190 1290}%
\special{fp}%
\special{pa 5290 1070}%
\special{pa 5150 1210}%
\special{fp}%
\special{pa 5240 1000}%
\special{pa 5110 1130}%
\special{fp}%
\special{pa 5190 930}%
\special{pa 5060 1060}%
\special{fp}%
\special{pa 5130 870}%
\special{pa 4990 1010}%
\special{fp}%
\special{pa 5060 820}%
\special{pa 4930 950}%
\special{fp}%
\special{pa 4990 770}%
\special{pa 4850 910}%
\special{fp}%
\special{pa 4910 730}%
\special{pa 4760 880}%
\special{fp}%
\special{pa 4830 690}%
\special{pa 4670 850}%
\special{fp}%
\special{pa 4740 660}%
\special{pa 4580 820}%
\special{fp}%
\special{pa 4640 640}%
\special{pa 4470 810}%
\special{fp}%
\special{pa 4540 620}%
\special{pa 4350 810}%
\special{fp}%
\special{pa 4430 610}%
\special{pa 4210 830}%
\special{fp}%
\special{pa 4300 620}%
\special{pa 4050 870}%
\special{fp}%
\special{pa 4000 680}%
\special{pa 3430 1250}%
\special{fp}%
\special{pa 3750 810}%
\special{pa 3540 1020}%
\special{fp}%
}}%
% STR 2 0 3 0 Black White
% 4 4400 700 4400 800 5 0 1 0
% $A(t_j)$
\put(44.0000,-8.0000){\makebox(0,0){{\colorbox[named]{White}{$A(t_j)$}}}}%
\put(44.0000,-8.0000){\makebox(0,0){{\colorbox[named]{White}{\color[named]{Black}{$A(t_j)$}}}}}%
% STR 2 0 3 0 Black White
% 4 4400 500 4400 600 5 0 1 0
% $A$
\put(44.0000,-6.0000){\makebox(0,0){{\colorbox[named]{White}{$A$}}}}%
\put(44.0000,-6.0000){\makebox(0,0){{\colorbox[named]{White}{\color[named]{Black}{$A$}}}}}%
% STR 2 0 3 0 Black White
% 4 2000 900 2000 1000 1 0 1 0
% $x$
\put(20.0000,-10.0000){\makebox(0,0)[lt]{{\colorbox[named]{White}{$x$}}}}%
\put(20.0000,-10.0000){\makebox(0,0)[lt]{{\colorbox[named]{White}{\color[named]{Black}{$x$}}}}}%
% STR 2 0 3 0 Black White
% 4 4800 900 4800 1000 1 0 1 0
% $x$
\put(48.0000,-10.0000){\makebox(0,0)[lt]{{\colorbox[named]{White}{$x$}}}}%
\put(48.0000,-10.0000){\makebox(0,0)[lt]{{\colorbox[named]{White}{\color[named]{Black}{$x$}}}}}%
% DOT 0 0 3 0 Black White
% 2 2000 1000 2000 1000
% 
{\color[named]{Black}{%
\special{pn 4}%
\special{sh 1}%
\special{ar 2000 1000 16 16 0  6.28318530717959E+0000}%
\special{sh 1}%
\special{ar 2000 1000 16 16 0  6.28318530717959E+0000}%
}}%
% DOT 0 0 3 0 Black White
% 1 4800 1000
% 
{\color[named]{Black}{%
\special{pn 4}%
\special{sh 1}%
\special{ar 4800 1000 16 16 0  6.28318530717959E+0000}%
}}%
% STR 2 0 3 0 Black White
% 4 1600 2300 1600 2400 5 0 0 0
% $x \in A(t_i) \triangle A$
\put(16.0000,-24.0000){\makebox(0,0){$x \in A(t_i) \triangle A$}}%
% STR 2 0 3 0 Black White
% 4 4400 2300 4400 2400 5 0 0 0
% $x \not\in A(t_j) \triangle A$
\put(44.0000,-24.0000){\makebox(0,0){$x \not\in A(t_j) \triangle A$}}%
\end{picture}}%

%% file: sec2_2_exA.tex
%WinTpicVersion4.28b
{\unitlength 0.1in
\begin{picture}( 28.1500, 24.0000)( 25.5500,-44.0000)
% STR 2 0 3 0 Black White
% 4 3990 3207 3990 3220 4 2800 0 0
% O
\put(39.9000,-32.2000){\makebox(0,0)[rt]{O}}%
% STR 2 0 3 0 Black White
% 4 3960 1987 3960 2000 4 2800 0 0
% $y$
\put(39.6000,-20.0000){\makebox(0,0)[rt]{$y$}}%
% STR 2 0 3 0 Black White
% 4 5210 3247 5210 3260 4 2800 0 0
% $x$
\put(52.1000,-32.6000){\makebox(0,0)[rt]{$x$}}%
% VECTOR 2 0 3 0 Black White
% 2 4000 4400 4000 2000
% 
{\color[named]{Black}{%
\special{pn 8}%
\special{pa 4000 4400}%
\special{pa 4000 2000}%
\special{fp}%
\special{sh 1}%
\special{pa 4000 2000}%
\special{pa 3980 2068}%
\special{pa 4000 2054}%
\special{pa 4020 2068}%
\special{pa 4000 2000}%
\special{fp}%
}}%
% VECTOR 2 0 3 0 Black White
% 2 2800 3200 5200 3200
% 
{\color[named]{Black}{%
\special{pn 8}%
\special{pa 2800 3200}%
\special{pa 5200 3200}%
\special{fp}%
\special{sh 1}%
\special{pa 5200 3200}%
\special{pa 5134 3180}%
\special{pa 5148 3200}%
\special{pa 5134 3220}%
\special{pa 5200 3200}%
\special{fp}%
}}%
% FUNC 2 1 3 0 Black White
% 9 2800 2000 5200 4400 4000 3200 5000 3200 4000 2200 2200 1600 2200 1600 10 2 0 4
% ///x^2+y^2=1///-1.199///1.199
{\color[named]{Black}{%
\special{pn 8}%
\special{pa 5000 3200}%
\special{pa 5000 3146}%
\special{pa 4998 3146}%
\special{pa 4998 3139}%
\special{fp}%
\special{pa 4993 3078}%
\special{pa 4992 3078}%
\special{pa 4992 3070}%
\special{pa 4990 3056}%
\special{pa 4990 3050}%
\special{pa 4988 3048}%
\special{pa 4988 3038}%
\special{pa 4986 3036}%
\special{pa 4986 3026}%
\special{pa 4984 3024}%
\special{pa 4984 3019}%
\special{fp}%
\special{pa 4972 2966}%
\special{pa 4972 2962}%
\special{pa 4970 2956}%
\special{pa 4970 2952}%
\special{pa 4968 2950}%
\special{pa 4968 2948}%
\special{pa 4966 2940}%
\special{pa 4966 2936}%
\special{pa 4964 2936}%
\special{pa 4964 2930}%
\special{pa 4962 2928}%
\special{pa 4962 2922}%
\special{pa 4960 2922}%
\special{pa 4960 2916}%
\special{pa 4958 2914}%
\special{pa 4958 2910}%
\special{pa 4958 2910}%
\special{fp}%
\special{pa 4941 2861}%
\special{pa 4940 2860}%
\special{pa 4940 2856}%
\special{pa 4938 2854}%
\special{pa 4938 2852}%
\special{pa 4936 2848}%
\special{pa 4936 2846}%
\special{pa 4934 2844}%
\special{pa 4934 2840}%
\special{pa 4932 2838}%
\special{pa 4932 2834}%
\special{pa 4930 2834}%
\special{pa 4930 2830}%
\special{pa 4928 2828}%
\special{pa 4928 2824}%
\special{pa 4926 2824}%
\special{pa 4926 2820}%
\special{pa 4924 2818}%
\special{pa 4924 2816}%
\special{pa 4922 2814}%
\special{pa 4922 2810}%
\special{pa 4920 2810}%
\special{pa 4920 2809}%
\special{fp}%
\special{pa 4897 2757}%
\special{pa 4896 2756}%
\special{pa 4896 2754}%
\special{pa 4894 2752}%
\special{pa 4894 2750}%
\special{pa 4890 2746}%
\special{pa 4890 2742}%
\special{pa 4888 2742}%
\special{pa 4888 2738}%
\special{pa 4886 2738}%
\special{pa 4886 2734}%
\special{pa 4884 2734}%
\special{pa 4884 2730}%
\special{pa 4882 2730}%
\special{pa 4882 2726}%
\special{pa 4878 2722}%
\special{pa 4878 2720}%
\special{pa 4876 2718}%
\special{pa 4876 2716}%
\special{pa 4874 2714}%
\special{pa 4872 2710}%
\special{pa 4872 2708}%
\special{fp}%
\special{pa 4839 2657}%
\special{pa 4838 2656}%
\special{pa 4838 2654}%
\special{pa 4836 2652}%
\special{pa 4836 2650}%
\special{pa 4834 2648}%
\special{pa 4834 2646}%
\special{pa 4826 2638}%
\special{pa 4826 2636}%
\special{pa 4822 2632}%
\special{pa 4822 2630}%
\special{pa 4820 2628}%
\special{pa 4818 2624}%
\special{pa 4816 2622}%
\special{pa 4816 2620}%
\special{pa 4812 2618}%
\special{pa 4812 2616}%
\special{pa 4810 2614}%
\special{pa 4808 2610}%
\special{pa 4806 2608}%
\special{fp}%
\special{pa 4767 2559}%
\special{pa 4762 2554}%
\special{pa 4762 2552}%
\special{pa 4750 2540}%
\special{pa 4750 2538}%
\special{pa 4742 2530}%
\special{pa 4742 2528}%
\special{pa 4727 2513}%
\special{fp}%
\special{pa 4680 2468}%
\special{pa 4656 2444}%
\special{pa 4654 2444}%
\special{pa 4642 2432}%
\special{pa 4640 2432}%
\special{pa 4636 2428}%
\special{pa 4634 2428}%
\special{fp}%
\special{pa 4588 2391}%
\special{pa 4588 2390}%
\special{pa 4586 2390}%
\special{pa 4584 2388}%
\special{pa 4580 2386}%
\special{pa 4578 2384}%
\special{pa 4576 2384}%
\special{pa 4574 2380}%
\special{pa 4572 2380}%
\special{pa 4568 2376}%
\special{pa 4566 2376}%
\special{pa 4562 2372}%
\special{pa 4560 2372}%
\special{pa 4558 2370}%
\special{pa 4556 2370}%
\special{pa 4552 2366}%
\special{pa 4548 2364}%
\special{pa 4544 2360}%
\special{pa 4542 2360}%
\special{pa 4539 2357}%
\special{fp}%
\special{pa 4487 2327}%
\special{pa 4486 2326}%
\special{pa 4478 2322}%
\special{pa 4476 2322}%
\special{pa 4472 2318}%
\special{pa 4470 2318}%
\special{pa 4468 2316}%
\special{pa 4466 2316}%
\special{pa 4464 2314}%
\special{pa 4462 2314}%
\special{pa 4460 2312}%
\special{pa 4458 2312}%
\special{pa 4454 2308}%
\special{pa 4450 2308}%
\special{pa 4450 2306}%
\special{pa 4446 2306}%
\special{pa 4446 2304}%
\special{pa 4442 2304}%
\special{pa 4442 2302}%
\special{pa 4438 2302}%
\special{pa 4437 2301}%
\special{fp}%
\special{pa 4384 2277}%
\special{pa 4382 2276}%
\special{pa 4380 2276}%
\special{pa 4378 2274}%
\special{pa 4374 2274}%
\special{pa 4374 2272}%
\special{pa 4370 2272}%
\special{pa 4368 2270}%
\special{pa 4364 2270}%
\special{pa 4364 2268}%
\special{pa 4362 2268}%
\special{pa 4356 2266}%
\special{pa 4354 2266}%
\special{pa 4354 2264}%
\special{pa 4348 2264}%
\special{pa 4348 2262}%
\special{pa 4346 2262}%
\special{pa 4340 2260}%
\special{pa 4338 2260}%
\special{pa 4338 2258}%
\special{pa 4334 2258}%
\special{fp}%
\special{pa 4283 2242}%
\special{pa 4282 2242}%
\special{pa 4282 2240}%
\special{pa 4276 2240}%
\special{pa 4274 2238}%
\special{pa 4268 2238}%
\special{pa 4268 2236}%
\special{pa 4264 2236}%
\special{pa 4258 2234}%
\special{pa 4254 2234}%
\special{pa 4252 2232}%
\special{pa 4250 2232}%
\special{pa 4242 2230}%
\special{pa 4238 2230}%
\special{pa 4238 2228}%
\special{pa 4232 2228}%
\special{pa 4229 2227}%
\special{fp}%
\special{pa 4169 2215}%
\special{pa 4152 2212}%
\special{pa 4138 2210}%
\special{pa 4132 2210}%
\special{pa 4130 2208}%
\special{pa 4122 2208}%
\special{pa 4107 2206}%
\special{fp}%
\special{pa 4044 2202}%
\special{pa 4032 2202}%
\special{pa 4032 2200}%
\special{pa 3983 2200}%
\special{fp}%
\special{pa 3922 2204}%
\special{pa 3906 2204}%
\special{pa 3906 2206}%
\special{pa 3896 2206}%
\special{pa 3878 2208}%
\special{pa 3870 2208}%
\special{pa 3870 2210}%
\special{pa 3863 2210}%
\special{fp}%
\special{pa 3803 2220}%
\special{pa 3794 2222}%
\special{pa 3790 2222}%
\special{pa 3790 2224}%
\special{pa 3782 2224}%
\special{pa 3780 2226}%
\special{pa 3776 2226}%
\special{pa 3768 2228}%
\special{pa 3764 2228}%
\special{pa 3762 2230}%
\special{pa 3756 2230}%
\special{pa 3754 2232}%
\special{pa 3748 2232}%
\special{pa 3748 2234}%
\special{pa 3747 2234}%
\special{fp}%
\special{pa 3691 2250}%
\special{pa 3666 2258}%
\special{pa 3664 2258}%
\special{pa 3662 2260}%
\special{pa 3660 2260}%
\special{pa 3656 2262}%
\special{pa 3654 2262}%
\special{pa 3652 2264}%
\special{pa 3648 2264}%
\special{pa 3646 2266}%
\special{pa 3644 2266}%
\special{pa 3640 2268}%
\special{pa 3638 2268}%
\special{pa 3636 2270}%
\special{pa 3634 2270}%
\special{fp}%
\special{pa 3583 2292}%
\special{pa 3580 2292}%
\special{pa 3580 2294}%
\special{pa 3576 2294}%
\special{pa 3576 2296}%
\special{pa 3572 2296}%
\special{pa 3572 2298}%
\special{pa 3568 2298}%
\special{pa 3568 2300}%
\special{pa 3564 2300}%
\special{pa 3564 2302}%
\special{pa 3560 2302}%
\special{pa 3558 2304}%
\special{pa 3556 2304}%
\special{pa 3554 2306}%
\special{pa 3552 2306}%
\special{pa 3550 2308}%
\special{pa 3548 2308}%
\special{pa 3546 2310}%
\special{pa 3535 2316}%
\special{fp}%
\special{pa 3484 2344}%
\special{pa 3482 2346}%
\special{pa 3480 2346}%
\special{pa 3478 2348}%
\special{pa 3476 2348}%
\special{pa 3470 2354}%
\special{pa 3468 2354}%
\special{pa 3466 2356}%
\special{pa 3464 2356}%
\special{pa 3462 2358}%
\special{pa 3460 2358}%
\special{pa 3456 2362}%
\special{pa 3454 2362}%
\special{pa 3448 2368}%
\special{pa 3446 2368}%
\special{pa 3444 2370}%
\special{pa 3440 2372}%
\special{pa 3438 2374}%
\special{pa 3436 2374}%
\special{pa 3434 2376}%
\special{fp}%
\special{pa 3384 2414}%
\special{pa 3382 2416}%
\special{pa 3380 2416}%
\special{pa 3372 2424}%
\special{pa 3370 2424}%
\special{pa 3366 2428}%
\special{pa 3360 2432}%
\special{pa 3354 2438}%
\special{pa 3352 2438}%
\special{pa 3340 2450}%
\special{pa 3338 2450}%
\special{pa 3336 2452}%
\special{fp}%
\special{pa 3291 2497}%
\special{pa 3260 2528}%
\special{pa 3260 2530}%
\special{pa 3250 2538}%
\special{pa 3250 2540}%
\special{pa 3248 2542}%
\special{fp}%
\special{pa 3208 2590}%
\special{pa 3208 2590}%
\special{pa 3206 2594}%
\special{pa 3202 2598}%
\special{pa 3200 2602}%
\special{pa 3196 2606}%
\special{pa 3194 2610}%
\special{pa 3188 2616}%
\special{pa 3188 2618}%
\special{pa 3186 2620}%
\special{pa 3186 2622}%
\special{pa 3178 2630}%
\special{pa 3178 2632}%
\special{pa 3174 2636}%
\special{pa 3174 2638}%
\special{pa 3173 2639}%
\special{fp}%
\special{pa 3140 2691}%
\special{pa 3140 2692}%
\special{pa 3138 2694}%
\special{pa 3136 2698}%
\special{pa 3136 2700}%
\special{pa 3132 2704}%
\special{pa 3132 2706}%
\special{pa 3130 2708}%
\special{pa 3128 2712}%
\special{pa 3128 2714}%
\special{pa 3124 2718}%
\special{pa 3124 2720}%
\special{pa 3122 2722}%
\special{pa 3122 2724}%
\special{pa 3120 2726}%
\special{pa 3111 2744}%
\special{fp}%
\special{pa 3086 2797}%
\special{pa 3086 2798}%
\special{pa 3084 2798}%
\special{pa 3084 2802}%
\special{pa 3082 2804}%
\special{pa 3082 2806}%
\special{pa 3080 2808}%
\special{pa 3080 2812}%
\special{pa 3078 2812}%
\special{pa 3078 2816}%
\special{pa 3076 2818}%
\special{pa 3076 2822}%
\special{pa 3074 2822}%
\special{pa 3074 2826}%
\special{pa 3072 2828}%
\special{pa 3072 2832}%
\special{pa 3070 2832}%
\special{pa 3070 2836}%
\special{pa 3068 2838}%
\special{pa 3068 2840}%
\special{pa 3066 2844}%
\special{pa 3066 2846}%
\special{pa 3065 2847}%
\special{fp}%
\special{pa 3046 2904}%
\special{pa 3046 2906}%
\special{pa 3044 2906}%
\special{pa 3044 2912}%
\special{pa 3042 2912}%
\special{pa 3042 2918}%
\special{pa 3040 2920}%
\special{pa 3040 2926}%
\special{pa 3038 2926}%
\special{pa 3038 2932}%
\special{pa 3036 2934}%
\special{pa 3036 2936}%
\special{pa 3034 2944}%
\special{pa 3034 2948}%
\special{pa 3032 2948}%
\special{pa 3032 2952}%
\special{pa 3030 2957}%
\special{fp}%
\special{pa 3018 3017}%
\special{pa 3018 3020}%
\special{pa 3016 3020}%
\special{pa 3016 3026}%
\special{pa 3014 3036}%
\special{pa 3014 3038}%
\special{pa 3012 3048}%
\special{pa 3012 3050}%
\special{pa 3010 3062}%
\special{pa 3010 3070}%
\special{pa 3008 3070}%
\special{pa 3008 3076}%
\special{fp}%
\special{pa 3002 3136}%
\special{pa 3002 3168}%
\special{pa 3000 3170}%
\special{pa 3000 3198}%
\special{fp}%
\special{pa 3002 3260}%
\special{pa 3002 3270}%
\special{pa 3004 3272}%
\special{pa 3004 3294}%
\special{pa 3006 3296}%
\special{pa 3006 3306}%
\special{pa 3008 3322}%
\special{fp}%
\special{pa 3018 3383}%
\special{pa 3018 3388}%
\special{pa 3020 3396}%
\special{pa 3020 3398}%
\special{pa 3022 3406}%
\special{pa 3022 3410}%
\special{pa 3024 3412}%
\special{pa 3024 3420}%
\special{pa 3026 3420}%
\special{pa 3026 3426}%
\special{pa 3028 3432}%
\special{pa 3028 3438}%
\special{pa 3030 3438}%
\special{pa 3030 3440}%
\special{fp}%
\special{pa 3045 3495}%
\special{pa 3046 3496}%
\special{pa 3046 3500}%
\special{pa 3048 3504}%
\special{pa 3048 3506}%
\special{pa 3050 3510}%
\special{pa 3050 3512}%
\special{pa 3052 3516}%
\special{pa 3052 3518}%
\special{pa 3054 3522}%
\special{pa 3054 3524}%
\special{pa 3056 3528}%
\special{pa 3056 3530}%
\special{pa 3058 3534}%
\special{pa 3058 3538}%
\special{pa 3060 3538}%
\special{pa 3060 3540}%
\special{pa 3062 3546}%
\special{pa 3062 3548}%
\special{pa 3064 3548}%
\special{pa 3064 3551}%
\special{fp}%
\special{pa 3084 3602}%
\special{pa 3086 3604}%
\special{pa 3086 3606}%
\special{pa 3088 3608}%
\special{pa 3088 3610}%
\special{pa 3092 3618}%
\special{pa 3092 3620}%
\special{pa 3094 3622}%
\special{pa 3094 3624}%
\special{pa 3096 3626}%
\special{pa 3096 3628}%
\special{pa 3098 3630}%
\special{pa 3098 3632}%
\special{pa 3100 3634}%
\special{pa 3100 3636}%
\special{pa 3102 3638}%
\special{pa 3102 3642}%
\special{pa 3104 3642}%
\special{pa 3104 3646}%
\special{pa 3106 3646}%
\special{pa 3106 3650}%
\special{pa 3108 3650}%
\special{pa 3108 3652}%
\special{fp}%
\special{pa 3136 3702}%
\special{pa 3136 3702}%
\special{pa 3136 3704}%
\special{pa 3140 3708}%
\special{pa 3140 3712}%
\special{pa 3146 3718}%
\special{pa 3146 3722}%
\special{pa 3150 3726}%
\special{pa 3152 3730}%
\special{pa 3156 3734}%
\special{pa 3156 3738}%
\special{pa 3160 3742}%
\special{pa 3160 3744}%
\special{pa 3164 3748}%
\special{pa 3166 3752}%
\special{pa 3167 3753}%
\special{fp}%
\special{pa 3201 3801}%
\special{pa 3202 3802}%
\special{pa 3202 3804}%
\special{pa 3206 3806}%
\special{pa 3206 3808}%
\special{pa 3208 3810}%
\special{pa 3208 3812}%
\special{pa 3212 3814}%
\special{pa 3212 3816}%
\special{pa 3220 3824}%
\special{pa 3220 3826}%
\special{pa 3228 3834}%
\special{pa 3228 3836}%
\special{pa 3232 3840}%
\special{pa 3232 3842}%
\special{pa 3237 3847}%
\special{fp}%
\special{pa 3280 3894}%
\special{pa 3324 3938}%
\special{fp}%
\special{pa 3372 3978}%
\special{pa 3380 3986}%
\special{pa 3382 3986}%
\special{pa 3386 3990}%
\special{pa 3390 3992}%
\special{pa 3394 3996}%
\special{pa 3398 3998}%
\special{pa 3402 4002}%
\special{pa 3406 4004}%
\special{pa 3410 4008}%
\special{pa 3414 4010}%
\special{pa 3416 4012}%
\special{pa 3418 4012}%
\special{pa 3420 4016}%
\special{pa 3420 4016}%
\special{fp}%
\special{pa 3471 4049}%
\special{pa 3472 4050}%
\special{pa 3474 4050}%
\special{pa 3476 4052}%
\special{pa 3478 4052}%
\special{pa 3484 4058}%
\special{pa 3488 4058}%
\special{pa 3494 4064}%
\special{pa 3498 4064}%
\special{pa 3498 4066}%
\special{pa 3500 4066}%
\special{pa 3502 4068}%
\special{pa 3504 4068}%
\special{pa 3508 4072}%
\special{pa 3512 4072}%
\special{pa 3512 4074}%
\special{pa 3514 4074}%
\special{pa 3516 4076}%
\special{pa 3518 4076}%
\special{pa 3520 4078}%
\special{fp}%
\special{pa 3567 4102}%
\special{pa 3570 4102}%
\special{pa 3570 4104}%
\special{pa 3574 4104}%
\special{pa 3574 4106}%
\special{pa 3578 4106}%
\special{pa 3578 4108}%
\special{pa 3582 4108}%
\special{pa 3584 4110}%
\special{pa 3586 4110}%
\special{pa 3588 4112}%
\special{pa 3590 4112}%
\special{pa 3592 4114}%
\special{pa 3594 4114}%
\special{pa 3598 4116}%
\special{pa 3600 4116}%
\special{pa 3600 4118}%
\special{pa 3604 4118}%
\special{pa 3606 4120}%
\special{pa 3610 4120}%
\special{pa 3610 4122}%
\special{pa 3614 4122}%
\special{pa 3616 4124}%
\special{fp}%
\special{pa 3670 4144}%
\special{pa 3674 4146}%
\special{pa 3676 4146}%
\special{pa 3680 4148}%
\special{pa 3682 4148}%
\special{pa 3686 4150}%
\special{pa 3688 4150}%
\special{pa 3692 4152}%
\special{pa 3694 4152}%
\special{pa 3698 4154}%
\special{pa 3702 4154}%
\special{pa 3702 4156}%
\special{pa 3706 4156}%
\special{pa 3712 4158}%
\special{pa 3714 4158}%
\special{pa 3716 4160}%
\special{pa 3722 4160}%
\special{pa 3722 4162}%
\special{pa 3726 4162}%
\special{fp}%
\special{pa 3779 4176}%
\special{pa 3784 4176}%
\special{pa 3786 4178}%
\special{pa 3794 4178}%
\special{pa 3794 4180}%
\special{pa 3804 4180}%
\special{pa 3804 4182}%
\special{pa 3814 4182}%
\special{pa 3814 4184}%
\special{pa 3824 4184}%
\special{pa 3826 4186}%
\special{pa 3835 4186}%
\special{fp}%
\special{pa 3895 4194}%
\special{pa 3896 4194}%
\special{pa 3896 4196}%
\special{pa 3916 4196}%
\special{pa 3918 4198}%
\special{pa 3946 4198}%
\special{pa 3946 4200}%
\special{pa 3953 4200}%
\special{fp}%
\special{pa 4017 4200}%
\special{pa 4054 4200}%
\special{pa 4056 4198}%
\special{pa 4079 4198}%
\special{fp}%
\special{pa 4140 4191}%
\special{pa 4144 4190}%
\special{pa 4152 4190}%
\special{pa 4152 4188}%
\special{pa 4164 4188}%
\special{pa 4164 4186}%
\special{pa 4176 4186}%
\special{pa 4176 4184}%
\special{pa 4186 4184}%
\special{pa 4188 4182}%
\special{pa 4196 4182}%
\special{pa 4196 4182}%
\special{fp}%
\special{pa 4253 4168}%
\special{pa 4254 4168}%
\special{pa 4260 4166}%
\special{pa 4264 4166}%
\special{pa 4266 4164}%
\special{pa 4272 4164}%
\special{pa 4272 4162}%
\special{pa 4278 4162}%
\special{pa 4280 4160}%
\special{pa 4286 4160}%
\special{pa 4286 4158}%
\special{pa 4292 4158}%
\special{pa 4292 4156}%
\special{pa 4298 4156}%
\special{pa 4300 4154}%
\special{pa 4304 4154}%
\special{pa 4306 4152}%
\special{pa 4306 4152}%
\special{fp}%
\special{pa 4360 4134}%
\special{pa 4362 4134}%
\special{pa 4362 4132}%
\special{pa 4366 4132}%
\special{pa 4368 4130}%
\special{pa 4372 4130}%
\special{pa 4372 4128}%
\special{pa 4376 4128}%
\special{pa 4378 4126}%
\special{pa 4382 4126}%
\special{pa 4382 4124}%
\special{pa 4386 4124}%
\special{pa 4386 4122}%
\special{pa 4390 4122}%
\special{pa 4392 4120}%
\special{pa 4394 4120}%
\special{pa 4398 4118}%
\special{pa 4400 4118}%
\special{pa 4402 4116}%
\special{pa 4404 4116}%
\special{pa 4406 4114}%
\special{pa 4410 4114}%
\special{pa 4410 4113}%
\special{fp}%
\special{pa 4462 4088}%
\special{pa 4462 4088}%
\special{pa 4464 4086}%
\special{pa 4466 4086}%
\special{pa 4468 4084}%
\special{pa 4470 4084}%
\special{pa 4472 4082}%
\special{pa 4474 4082}%
\special{pa 4476 4080}%
\special{pa 4484 4076}%
\special{pa 4486 4076}%
\special{pa 4490 4072}%
\special{pa 4492 4072}%
\special{pa 4494 4070}%
\special{pa 4498 4068}%
\special{pa 4500 4068}%
\special{pa 4504 4064}%
\special{pa 4506 4064}%
\special{pa 4508 4062}%
\special{pa 4512 4060}%
\special{pa 4514 4058}%
\special{fp}%
\special{pa 4563 4017}%
\special{pa 4612 3977}%
\special{fp}%
\special{pa 4661 3935}%
\special{pa 4709 3895}%
\special{fp}%
\special{pa 4757 3853}%
\special{pa 4768 3842}%
\special{pa 4768 3840}%
\special{pa 4774 3836}%
\special{pa 4774 3834}%
\special{pa 4782 3826}%
\special{pa 4782 3824}%
\special{pa 4790 3816}%
\special{pa 4790 3814}%
\special{pa 4792 3812}%
\special{pa 4792 3810}%
\special{pa 4796 3808}%
\special{pa 4796 3808}%
\special{fp}%
\special{pa 4830 3759}%
\special{pa 4832 3758}%
\special{pa 4832 3756}%
\special{pa 4836 3752}%
\special{pa 4836 3748}%
\special{pa 4840 3744}%
\special{pa 4840 3742}%
\special{pa 4844 3738}%
\special{pa 4846 3734}%
\special{pa 4850 3730}%
\special{pa 4850 3726}%
\special{pa 4854 3722}%
\special{pa 4856 3718}%
\special{pa 4858 3716}%
\special{pa 4858 3714}%
\special{pa 4860 3712}%
\special{pa 4862 3709}%
\special{fp}%
\special{pa 4892 3653}%
\special{pa 4914 3610}%
\special{pa 4914 3606}%
\special{pa 4916 3604}%
\special{pa 4916 3602}%
\special{pa 4918 3600}%
\special{pa 4918 3598}%
\special{pa 4919 3597}%
\special{fp}%
\special{pa 4941 3541}%
\special{pa 4948 3520}%
\special{pa 4952 3508}%
\special{pa 4952 3506}%
\special{pa 4954 3504}%
\special{pa 4960 3481}%
\special{fp}%
\special{pa 4977 3420}%
\special{pa 4980 3402}%
\special{pa 4980 3398}%
\special{pa 4984 3380}%
\special{pa 4984 3378}%
\special{pa 4988 3362}%
\special{pa 4989 3358}%
\special{fp}%
\special{pa 4996 3295}%
\special{pa 4996 3286}%
\special{pa 4998 3272}%
\special{pa 4998 3256}%
\special{pa 5000 3232}%
\special{fp}%
}}%
% FUNC 2 1 3 0 Black White
% 10 2800 2000 5200 4400 4000 3200 5000 3200 4000 2200 2200 1600 2200 1600 10 2 0 4 0 0
% ///x^2+y^2=0.4///-1.199///1.199
{\color[named]{Black}{%
\special{pn 8}%
\special{pn 8}%
\special{pa 4600 3400}%
\special{pa 4602 3398}%
\special{pa 4602 3394}%
\special{pa 4604 3392}%
\special{pa 4604 3386}%
\special{pa 4606 3386}%
\special{pa 4606 3380}%
\special{pa 4608 3380}%
\special{pa 4608 3374}%
\special{pa 4610 3372}%
\special{pa 4610 3366}%
\special{pa 4612 3366}%
\special{pa 4612 3358}%
\special{pa 4614 3358}%
\special{pa 4614 3350}%
\special{pa 4616 3350}%
\special{pa 4616 3349}%
\special{fp}%
\special{pa 4626 3292}%
\special{pa 4628 3280}%
\special{pa 4630 3262}%
\special{pa 4630 3250}%
\special{pa 4632 3250}%
\special{pa 4632 3231}%
\special{fp}%
\special{pa 4632 3167}%
\special{pa 4632 3152}%
\special{pa 4630 3150}%
\special{pa 4630 3138}%
\special{pa 4628 3122}%
\special{pa 4628 3120}%
\special{pa 4626 3108}%
\special{pa 4626 3105}%
\special{fp}%
\special{pa 4614 3045}%
\special{pa 4614 3044}%
\special{pa 4612 3042}%
\special{pa 4612 3036}%
\special{pa 4610 3034}%
\special{pa 4610 3028}%
\special{pa 4608 3028}%
\special{pa 4608 3022}%
\special{pa 4606 3020}%
\special{pa 4606 3016}%
\special{pa 4604 3014}%
\special{pa 4604 3010}%
\special{pa 4602 3006}%
\special{pa 4602 3004}%
\special{pa 4600 3000}%
\special{pa 4600 2998}%
\special{pa 4598 2994}%
\special{pa 4598 2992}%
\special{pa 4597 2989}%
\special{fp}%
\special{pa 4576 2936}%
\special{pa 4576 2936}%
\special{pa 4574 2936}%
\special{pa 4574 2932}%
\special{pa 4572 2932}%
\special{pa 4572 2928}%
\special{pa 4570 2926}%
\special{pa 4570 2924}%
\special{pa 4568 2922}%
\special{pa 4568 2920}%
\special{pa 4566 2918}%
\special{pa 4566 2916}%
\special{pa 4564 2914}%
\special{pa 4564 2912}%
\special{pa 4560 2908}%
\special{pa 4560 2904}%
\special{pa 4558 2904}%
\special{pa 4558 2900}%
\special{pa 4556 2900}%
\special{pa 4556 2896}%
\special{pa 4554 2896}%
\special{pa 4554 2894}%
\special{pa 4552 2892}%
\special{pa 4552 2890}%
\special{pa 4551 2889}%
\special{fp}%
\special{pa 4520 2838}%
\special{pa 4520 2838}%
\special{pa 4514 2832}%
\special{pa 4512 2828}%
\special{pa 4510 2826}%
\special{pa 4510 2824}%
\special{pa 4506 2822}%
\special{pa 4506 2820}%
\special{pa 4504 2818}%
\special{pa 4504 2816}%
\special{pa 4492 2804}%
\special{pa 4492 2802}%
\special{pa 4488 2798}%
\special{pa 4488 2796}%
\special{pa 4482 2792}%
\special{pa 4482 2792}%
\special{fp}%
\special{pa 4439 2745}%
\special{pa 4428 2734}%
\special{pa 4426 2734}%
\special{pa 4418 2724}%
\special{pa 4416 2724}%
\special{pa 4410 2718}%
\special{pa 4404 2714}%
\special{pa 4394 2704}%
\special{pa 4393 2704}%
\special{fp}%
\special{pa 4343 2669}%
\special{pa 4342 2668}%
\special{pa 4340 2668}%
\special{pa 4334 2662}%
\special{pa 4332 2662}%
\special{pa 4330 2660}%
\special{pa 4328 2660}%
\special{pa 4326 2658}%
\special{pa 4322 2656}%
\special{pa 4320 2654}%
\special{pa 4318 2654}%
\special{pa 4314 2650}%
\special{pa 4310 2650}%
\special{pa 4310 2648}%
\special{pa 4308 2648}%
\special{pa 4306 2646}%
\special{pa 4304 2646}%
\special{pa 4302 2644}%
\special{pa 4300 2644}%
\special{pa 4298 2642}%
\special{pa 4296 2642}%
\special{pa 4293 2639}%
\special{fp}%
\special{pa 4244 2618}%
\special{pa 4244 2618}%
\special{pa 4244 2616}%
\special{pa 4240 2616}%
\special{pa 4240 2614}%
\special{pa 4236 2614}%
\special{pa 4234 2612}%
\special{pa 4230 2612}%
\special{pa 4230 2610}%
\special{pa 4224 2610}%
\special{pa 4224 2608}%
\special{pa 4222 2608}%
\special{pa 4216 2606}%
\special{pa 4214 2606}%
\special{pa 4214 2604}%
\special{pa 4208 2604}%
\special{pa 4208 2602}%
\special{pa 4202 2602}%
\special{pa 4202 2600}%
\special{pa 4196 2600}%
\special{pa 4196 2600}%
\special{fp}%
\special{pa 4139 2584}%
\special{pa 4138 2584}%
\special{pa 4136 2582}%
\special{pa 4132 2582}%
\special{pa 4124 2580}%
\special{pa 4118 2580}%
\special{pa 4118 2578}%
\special{pa 4108 2578}%
\special{pa 4106 2576}%
\special{pa 4094 2576}%
\special{pa 4094 2574}%
\special{pa 4082 2574}%
\special{fp}%
\special{pa 4022 2568}%
\special{pa 3966 2568}%
\special{pa 3966 2570}%
\special{pa 3961 2570}%
\special{fp}%
\special{pa 3899 2576}%
\special{pa 3894 2576}%
\special{pa 3894 2578}%
\special{pa 3884 2578}%
\special{pa 3882 2580}%
\special{pa 3874 2580}%
\special{pa 3872 2582}%
\special{pa 3864 2582}%
\special{pa 3864 2584}%
\special{pa 3858 2584}%
\special{pa 3852 2586}%
\special{pa 3850 2586}%
\special{pa 3844 2588}%
\special{pa 3842 2588}%
\special{pa 3842 2588}%
\special{fp}%
\special{pa 3785 2606}%
\special{pa 3784 2606}%
\special{pa 3780 2608}%
\special{pa 3778 2608}%
\special{pa 3776 2610}%
\special{pa 3772 2610}%
\special{pa 3770 2612}%
\special{pa 3766 2612}%
\special{pa 3766 2614}%
\special{pa 3762 2614}%
\special{pa 3760 2616}%
\special{pa 3758 2616}%
\special{pa 3756 2618}%
\special{pa 3752 2618}%
\special{pa 3752 2620}%
\special{pa 3748 2620}%
\special{pa 3746 2622}%
\special{pa 3744 2622}%
\special{pa 3742 2624}%
\special{pa 3740 2624}%
\special{pa 3732 2628}%
\special{fp}%
\special{pa 3679 2655}%
\special{pa 3674 2660}%
\special{pa 3670 2660}%
\special{pa 3666 2664}%
\special{pa 3662 2666}%
\special{pa 3658 2670}%
\special{pa 3656 2670}%
\special{pa 3652 2674}%
\special{pa 3650 2674}%
\special{pa 3646 2678}%
\special{pa 3644 2678}%
\special{pa 3640 2682}%
\special{pa 3638 2682}%
\special{pa 3636 2684}%
\special{pa 3632 2686}%
\special{pa 3629 2689}%
\special{fp}%
\special{pa 3581 2727}%
\special{pa 3536 2772}%
\special{fp}%
\special{pa 3496 2819}%
\special{pa 3494 2820}%
\special{pa 3494 2822}%
\special{pa 3492 2824}%
\special{pa 3492 2826}%
\special{pa 3486 2832}%
\special{pa 3484 2836}%
\special{pa 3482 2838}%
\special{pa 3482 2840}%
\special{pa 3478 2844}%
\special{pa 3478 2846}%
\special{pa 3474 2850}%
\special{pa 3474 2852}%
\special{pa 3470 2856}%
\special{pa 3470 2858}%
\special{pa 3466 2862}%
\special{pa 3464 2866}%
\special{pa 3462 2868}%
\special{fp}%
\special{pa 3435 2917}%
\special{pa 3434 2918}%
\special{pa 3434 2920}%
\special{pa 3432 2922}%
\special{pa 3432 2924}%
\special{pa 3430 2926}%
\special{pa 3430 2930}%
\special{pa 3428 2930}%
\special{pa 3428 2934}%
\special{pa 3426 2934}%
\special{pa 3426 2938}%
\special{pa 3424 2938}%
\special{pa 3424 2942}%
\special{pa 3422 2944}%
\special{pa 3422 2946}%
\special{pa 3420 2948}%
\special{pa 3420 2952}%
\special{pa 3418 2952}%
\special{pa 3418 2956}%
\special{pa 3416 2958}%
\special{pa 3416 2960}%
\special{pa 3414 2962}%
\special{pa 3414 2966}%
\special{pa 3413 2966}%
\special{fp}%
\special{pa 3394 3023}%
\special{pa 3394 3024}%
\special{pa 3392 3026}%
\special{pa 3392 3032}%
\special{pa 3390 3032}%
\special{pa 3390 3036}%
\special{pa 3388 3042}%
\special{pa 3388 3044}%
\special{pa 3386 3050}%
\special{pa 3386 3052}%
\special{pa 3384 3058}%
\special{pa 3384 3064}%
\special{pa 3382 3064}%
\special{pa 3382 3068}%
\special{pa 3380 3078}%
\special{pa 3380 3080}%
\special{fp}%
\special{pa 3370 3140}%
\special{pa 3370 3140}%
\special{pa 3370 3166}%
\special{pa 3368 3166}%
\special{pa 3368 3201}%
\special{fp}%
\special{pa 3372 3263}%
\special{pa 3372 3280}%
\special{pa 3374 3280}%
\special{pa 3374 3294}%
\special{pa 3376 3294}%
\special{pa 3376 3306}%
\special{pa 3378 3308}%
\special{pa 3378 3318}%
\special{pa 3380 3318}%
\special{pa 3380 3319}%
\special{fp}%
\special{pa 3394 3377}%
\special{pa 3394 3380}%
\special{pa 3396 3386}%
\special{pa 3396 3390}%
\special{pa 3398 3390}%
\special{pa 3398 3396}%
\special{pa 3400 3396}%
\special{pa 3400 3402}%
\special{pa 3402 3402}%
\special{pa 3402 3408}%
\special{pa 3404 3408}%
\special{pa 3404 3414}%
\special{pa 3406 3414}%
\special{pa 3406 3416}%
\special{pa 3408 3422}%
\special{pa 3408 3424}%
\special{pa 3410 3424}%
\special{pa 3410 3428}%
\special{fp}%
\special{pa 3433 3479}%
\special{pa 3438 3490}%
\special{pa 3438 3492}%
\special{pa 3442 3496}%
\special{pa 3442 3498}%
\special{pa 3444 3500}%
\special{pa 3444 3502}%
\special{pa 3446 3504}%
\special{pa 3446 3506}%
\special{pa 3448 3508}%
\special{pa 3450 3512}%
\special{pa 3450 3514}%
\special{pa 3454 3518}%
\special{pa 3454 3520}%
\special{pa 3456 3522}%
\special{pa 3458 3526}%
\special{pa 3460 3528}%
\special{pa 3460 3530}%
\special{pa 3462 3532}%
\special{fp}%
\special{pa 3496 3582}%
\special{pa 3502 3588}%
\special{pa 3502 3590}%
\special{pa 3504 3592}%
\special{pa 3504 3594}%
\special{pa 3514 3604}%
\special{pa 3518 3610}%
\special{pa 3524 3616}%
\special{pa 3524 3618}%
\special{pa 3534 3626}%
\special{pa 3534 3628}%
\special{pa 3535 3629}%
\special{fp}%
\special{pa 3581 3674}%
\special{pa 3584 3676}%
\special{pa 3590 3682}%
\special{pa 3592 3682}%
\special{pa 3596 3688}%
\special{pa 3598 3688}%
\special{pa 3602 3692}%
\special{pa 3604 3692}%
\special{pa 3616 3704}%
\special{pa 3618 3704}%
\special{pa 3620 3706}%
\special{pa 3622 3706}%
\special{pa 3624 3710}%
\special{pa 3626 3710}%
\special{pa 3628 3712}%
\special{fp}%
\special{pa 3679 3745}%
\special{pa 3680 3746}%
\special{pa 3684 3748}%
\special{pa 3686 3748}%
\special{pa 3690 3752}%
\special{pa 3692 3752}%
\special{pa 3694 3754}%
\special{pa 3706 3760}%
\special{pa 3708 3760}%
\special{pa 3712 3764}%
\special{pa 3714 3764}%
\special{pa 3716 3766}%
\special{pa 3718 3766}%
\special{pa 3720 3768}%
\special{pa 3722 3768}%
\special{pa 3724 3770}%
\special{pa 3726 3770}%
\special{pa 3728 3772}%
\special{pa 3730 3772}%
\special{pa 3732 3773}%
\special{fp}%
\special{pa 3785 3795}%
\special{pa 3786 3796}%
\special{pa 3788 3796}%
\special{pa 3792 3798}%
\special{pa 3794 3798}%
\special{pa 3798 3800}%
\special{pa 3800 3800}%
\special{pa 3804 3802}%
\special{pa 3806 3802}%
\special{pa 3810 3804}%
\special{pa 3814 3804}%
\special{pa 3816 3806}%
\special{pa 3820 3806}%
\special{pa 3822 3808}%
\special{pa 3828 3808}%
\special{pa 3828 3810}%
\special{pa 3834 3810}%
\special{pa 3836 3812}%
\special{pa 3842 3812}%
\special{fp}%
\special{pa 3899 3824}%
\special{pa 3900 3824}%
\special{pa 3902 3826}%
\special{pa 3908 3826}%
\special{pa 3920 3828}%
\special{pa 3922 3828}%
\special{pa 3938 3830}%
\special{pa 3950 3830}%
\special{pa 3952 3832}%
\special{pa 3960 3832}%
\special{fp}%
\special{pa 4025 3832}%
\special{pa 4050 3832}%
\special{pa 4050 3830}%
\special{pa 4062 3830}%
\special{pa 4080 3828}%
\special{pa 4086 3827}%
\special{fp}%
\special{pa 4143 3816}%
\special{pa 4150 3816}%
\special{pa 4150 3814}%
\special{pa 4158 3814}%
\special{pa 4158 3812}%
\special{pa 4166 3812}%
\special{pa 4166 3810}%
\special{pa 4172 3810}%
\special{pa 4174 3808}%
\special{pa 4180 3808}%
\special{pa 4180 3806}%
\special{pa 4186 3806}%
\special{pa 4186 3804}%
\special{pa 4190 3804}%
\special{pa 4195 3802}%
\special{fp}%
\special{pa 4251 3782}%
\special{pa 4252 3782}%
\special{pa 4252 3780}%
\special{pa 4256 3780}%
\special{pa 4256 3778}%
\special{pa 4260 3778}%
\special{pa 4262 3776}%
\special{pa 4264 3776}%
\special{pa 4266 3774}%
\special{pa 4268 3774}%
\special{pa 4270 3772}%
\special{pa 4274 3772}%
\special{pa 4274 3770}%
\special{pa 4278 3770}%
\special{pa 4278 3768}%
\special{pa 4282 3768}%
\special{pa 4282 3766}%
\special{pa 4286 3766}%
\special{pa 4286 3764}%
\special{pa 4290 3764}%
\special{pa 4290 3762}%
\special{pa 4292 3762}%
\special{pa 4294 3760}%
\special{pa 4296 3760}%
\special{pa 4296 3760}%
\special{fp}%
\special{pa 4348 3730}%
\special{pa 4348 3730}%
\special{pa 4352 3726}%
\special{pa 4354 3726}%
\special{pa 4358 3722}%
\special{pa 4360 3722}%
\special{pa 4364 3718}%
\special{pa 4366 3718}%
\special{pa 4368 3714}%
\special{pa 4370 3714}%
\special{pa 4372 3712}%
\special{pa 4374 3712}%
\special{pa 4380 3706}%
\special{pa 4384 3704}%
\special{pa 4388 3700}%
\special{pa 4390 3700}%
\special{pa 4392 3696}%
\special{pa 4394 3696}%
\special{pa 4395 3695}%
\special{fp}%
\special{pa 4443 3653}%
\special{pa 4468 3628}%
\special{pa 4468 3626}%
\special{pa 4476 3618}%
\special{pa 4476 3616}%
\special{pa 4485 3607}%
\special{fp}%
\special{pa 4522 3559}%
\special{pa 4522 3558}%
\special{pa 4526 3554}%
\special{pa 4526 3552}%
\special{pa 4530 3548}%
\special{pa 4530 3546}%
\special{pa 4534 3542}%
\special{pa 4534 3540}%
\special{pa 4536 3538}%
\special{pa 4536 3536}%
\special{pa 4538 3534}%
\special{pa 4538 3532}%
\special{pa 4544 3526}%
\special{pa 4544 3522}%
\special{pa 4548 3518}%
\special{pa 4548 3516}%
\special{pa 4550 3514}%
\special{pa 4552 3510}%
\special{pa 4552 3508}%
\special{pa 4552 3508}%
\special{fp}%
\special{pa 4580 3454}%
\special{pa 4580 3454}%
\special{pa 4580 3452}%
\special{pa 4582 3452}%
\special{pa 4582 3448}%
\special{pa 4584 3446}%
\special{pa 4584 3444}%
\special{pa 4586 3440}%
\special{pa 4586 3438}%
\special{pa 4588 3436}%
\special{pa 4588 3432}%
\special{pa 4590 3432}%
\special{pa 4590 3430}%
\special{pa 4592 3424}%
\special{pa 4592 3422}%
\special{pa 4594 3422}%
\special{pa 4594 3416}%
\special{pa 4596 3416}%
\special{pa 4596 3414}%
\special{pa 4600 3402}%
\special{fp}%
}}%
% LINE 3 0 3 0 Black White
% 16 3920 2200 3000 3120 4000 2240 3040 3200 3420 2940 3160 3200 3370 3110 3280 3200 4000 2360 3740 2620 4000 2480 3910 2570 3770 2230 3030 2970 3590 2290 3090 2790
% 
{\color[named]{Black}{%
\special{pn 4}%
\special{pa 3920 2200}%
\special{pa 3000 3120}%
\special{fp}%
\special{pa 4000 2240}%
\special{pa 3040 3200}%
\special{fp}%
\special{pa 3420 2940}%
\special{pa 3160 3200}%
\special{fp}%
\special{pa 3370 3110}%
\special{pa 3280 3200}%
\special{fp}%
\special{pa 4000 2360}%
\special{pa 3740 2620}%
\special{fp}%
\special{pa 4000 2480}%
\special{pa 3910 2570}%
\special{fp}%
\special{pa 3770 2230}%
\special{pa 3030 2970}%
\special{fp}%
\special{pa 3590 2290}%
\special{pa 3090 2790}%
\special{fp}%
}}%
% LINE 3 0 3 0 Black White
% 34 4860 2700 4590 2970 4900 2780 4620 3060 4940 2860 4630 3170 4970 2950 4720 3200 4990 3050 4840 3200 5000 3160 4960 3200 4810 2630 4550 2890 4760 2560 4510 2810 4700 2500 4450 2750 4640 2440 4390 2690 4570 2390 4310 2650 4500 2340 4230 2610 4420 2300 4140 2580 4340 2260 4030 2570 4250 2230 4000 2480 4150 2210 4000 2360 4040 2200 4000 2240
% 
{\color[named]{Black}{%
\special{pn 4}%
\special{pa 4860 2700}%
\special{pa 4590 2970}%
\special{fp}%
\special{pa 4900 2780}%
\special{pa 4620 3060}%
\special{fp}%
\special{pa 4940 2860}%
\special{pa 4630 3170}%
\special{fp}%
\special{pa 4970 2950}%
\special{pa 4720 3200}%
\special{fp}%
\special{pa 4990 3050}%
\special{pa 4840 3200}%
\special{fp}%
\special{pa 5000 3160}%
\special{pa 4960 3200}%
\special{fp}%
\special{pa 4810 2630}%
\special{pa 4550 2890}%
\special{fp}%
\special{pa 4760 2560}%
\special{pa 4510 2810}%
\special{fp}%
\special{pa 4700 2500}%
\special{pa 4450 2750}%
\special{fp}%
\special{pa 4640 2440}%
\special{pa 4390 2690}%
\special{fp}%
\special{pa 4570 2390}%
\special{pa 4310 2650}%
\special{fp}%
\special{pa 4500 2340}%
\special{pa 4230 2610}%
\special{fp}%
\special{pa 4420 2300}%
\special{pa 4140 2580}%
\special{fp}%
\special{pa 4340 2260}%
\special{pa 4030 2570}%
\special{fp}%
\special{pa 4250 2230}%
\special{pa 4000 2480}%
\special{fp}%
\special{pa 4150 2210}%
\special{pa 4000 2360}%
\special{fp}%
\special{pa 4040 2200}%
\special{pa 4000 2240}%
\special{fp}%
}}%
% LINE 3 0 3 0 Black White
% 34 3970 3830 3660 4140 4000 3920 3750 4170 4000 4040 3850 4190 4000 4160 3960 4200 3860 3820 3580 4100 3770 3790 3500 4060 3690 3750 3430 4010 3610 3710 3360 3960 3550 3650 3300 3900 3490 3590 3240 3840 3450 3510 3190 3770 3410 3430 3140 3700 3380 3340 3100 3620 3370 3230 3060 3540 3280 3200 3030 3450 3160 3200 3010 3350 3040 3200 3000 3240
% 
{\color[named]{Black}{%
\special{pn 4}%
\special{pa 3970 3830}%
\special{pa 3660 4140}%
\special{fp}%
\special{pa 4000 3920}%
\special{pa 3750 4170}%
\special{fp}%
\special{pa 4000 4040}%
\special{pa 3850 4190}%
\special{fp}%
\special{pa 4000 4160}%
\special{pa 3960 4200}%
\special{fp}%
\special{pa 3860 3820}%
\special{pa 3580 4100}%
\special{fp}%
\special{pa 3770 3790}%
\special{pa 3500 4060}%
\special{fp}%
\special{pa 3690 3750}%
\special{pa 3430 4010}%
\special{fp}%
\special{pa 3610 3710}%
\special{pa 3360 3960}%
\special{fp}%
\special{pa 3550 3650}%
\special{pa 3300 3900}%
\special{fp}%
\special{pa 3490 3590}%
\special{pa 3240 3840}%
\special{fp}%
\special{pa 3450 3510}%
\special{pa 3190 3770}%
\special{fp}%
\special{pa 3410 3430}%
\special{pa 3140 3700}%
\special{fp}%
\special{pa 3380 3340}%
\special{pa 3100 3620}%
\special{fp}%
\special{pa 3370 3230}%
\special{pa 3060 3540}%
\special{fp}%
\special{pa 3280 3200}%
\special{pa 3030 3450}%
\special{fp}%
\special{pa 3160 3200}%
\special{pa 3010 3350}%
\special{fp}%
\special{pa 3040 3200}%
\special{pa 3000 3240}%
\special{fp}%
}}%
% LINE 3 0 3 0 Black White
% 16 4970 3430 4230 4170 5000 3280 4080 4200 4960 3200 4000 4160 4260 3780 4000 4040 4090 3830 4000 3920 4840 3200 4580 3460 4720 3200 4630 3290 4910 3610 4410 4110
% 
{\color[named]{Black}{%
\special{pn 4}%
\special{pa 4970 3430}%
\special{pa 4230 4170}%
\special{fp}%
\special{pa 5000 3280}%
\special{pa 4080 4200}%
\special{fp}%
\special{pa 4960 3200}%
\special{pa 4000 4160}%
\special{fp}%
\special{pa 4260 3780}%
\special{pa 4000 4040}%
\special{fp}%
\special{pa 4090 3830}%
\special{pa 4000 3920}%
\special{fp}%
\special{pa 4840 3200}%
\special{pa 4580 3460}%
\special{fp}%
\special{pa 4720 3200}%
\special{pa 4630 3290}%
\special{fp}%
\special{pa 4910 3610}%
\special{pa 4410 4110}%
\special{fp}%
}}%
% STR 2 0 3 0 Black White
% 4 5020 3140 5020 3240 1 0 0 0
% $1$
\put(50.2000,-32.4000){\makebox(0,0)[lt]{$1$}}%
% STR 2 0 3 0 Black White
% 4 3970 4130 3970 4230 4 0 0 0
% $-1$
\put(39.7000,-42.3000){\makebox(0,0)[rt]{$-1$}}%
% STR 2 0 3 0 Black White
% 4 2970 3080 2970 3180 3 0 0 0
% $-1$
\put(29.7000,-31.8000){\makebox(0,0)[rb]{$-1$}}%
% STR 2 0 3 0 Black White
% 4 4040 2070 4040 2170 2 0 0 0
% $1$
\put(40.4000,-21.7000){\makebox(0,0)[lb]{$1$}}%
% STR 2 0 3 0 Black White
% 4 4600 3130 4600 3230 4 0 0 0
% $1-\frac{1}{t_i}$
\put(46.0000,-32.3000){\makebox(0,0)[rt]{$1-\frac{1}{t_i}$}}%
% STR 2 0 3 0 Black White
% 4 3400 3090 3400 3190 2 0 0 0
% $-1+\frac{1}{t_i}$
\put(34.0000,-31.9000){\makebox(0,0)[lb]{$-1+\frac{1}{t_i}$}}%
% STR 2 0 3 0 Black White
% 4 4010 2550 4010 2650 1 0 0 0
% $1-\frac{1}{t_i}$
\put(40.1000,-26.5000){\makebox(0,0)[lt]{$1-\frac{1}{t_i}$}}%
% STR 2 0 3 0 Black White
% 4 3990 3660 3990 3760 3 0 0 0
% $-1+\frac{1}{t_i}$
\put(39.9000,-37.6000){\makebox(0,0)[rb]{$-1+\frac{1}{t_i}$}}%
% STR 2 0 3 0 Black White
% 4 5370 2100 5370 2200 2 0 0 0
% $A$
\put(53.7000,-22.0000){\makebox(0,0)[lb]{$A$}}%
% STR 2 0 3 0 Black White
% 4 2750 2250 2750 2350 5 0 1 0
% $A(t_i)$
\put(27.5000,-23.5000){\makebox(0,0){{\colorbox[named]{White}{$A(t_i)$}}}}%
\put(27.5000,-23.5000){\makebox(0,0){{\colorbox[named]{White}{\color[named]{Black}{$A(t_i)$}}}}}%
% LINE 3 0 3 0 Black White
% 2 3480 2850 2880 2420
% 
{\color[named]{Black}{%
\special{pn 4}%
\special{pa 3480 2850}%
\special{pa 2880 2420}%
\special{fp}%
}}%
% LINE 3 0 3 0 Black White
% 2 4750 2530 5350 2200
% 
{\color[named]{Black}{%
\special{pn 4}%
\special{pa 4750 2530}%
\special{pa 5350 2200}%
\special{fp}%
}}%
% STR 2 0 3 0 Black White
% 4 5170 2450 5170 2550 1 0 1 0
% $(x_i,y_i)$
\put(51.7000,-25.5000){\makebox(0,0)[lt]{{\colorbox[named]{White}{$(x_i,y_i)$}}}}%
\put(51.7000,-25.5000){\makebox(0,0)[lt]{{\colorbox[named]{White}{\color[named]{Black}{$(x_i,y_i)$}}}}}%
% DOT 0 0 3 0 Black White
% 2 4690 2830 4690 2830
% 
{\color[named]{Black}{%
\special{pn 4}%
\special{sh 1}%
\special{ar 4690 2830 16 16 0  6.28318530717959E+0000}%
\special{sh 1}%
\special{ar 4690 2830 16 16 0  6.28318530717959E+0000}%
}}%
% LINE 3 0 3 0 Black White
% 2 4690 2830 5200 2690
% 
{\color[named]{Black}{%
\special{pn 4}%
\special{pa 4690 2830}%
\special{pa 5200 2690}%
\special{fp}%
}}%
\end{picture}}%

%% file: sec2_2_exD.tex
%WinTpicVersion4.28b
{\unitlength 0.1in
\begin{picture}( 32.8000, 24.0000)( 20.8000,-44.0000)
% STR 2 0 3 0 Black White
% 4 3990 3207 3990 3220 4 2800 0 0
% O
\put(39.9000,-32.2000){\makebox(0,0)[rt]{O}}%
% STR 2 0 3 0 Black White
% 4 3960 1987 3960 2000 4 2800 0 0
% $y$
\put(39.6000,-20.0000){\makebox(0,0)[rt]{$y$}}%
% STR 2 0 3 0 Black White
% 4 5100 3137 5100 3150 2 0 0 0
% $x$
\put(51.0000,-31.5000){\makebox(0,0)[lb]{$x$}}%
% VECTOR 2 0 3 0 Black White
% 2 4000 4400 4000 2000
% 
{\color[named]{Black}{%
\special{pn 8}%
\special{pa 4000 4400}%
\special{pa 4000 2000}%
\special{fp}%
\special{sh 1}%
\special{pa 4000 2000}%
\special{pa 3980 2068}%
\special{pa 4000 2054}%
\special{pa 4020 2068}%
\special{pa 4000 2000}%
\special{fp}%
}}%
% VECTOR 2 0 3 0 Black White
% 2 2800 3200 5200 3200
% 
{\color[named]{Black}{%
\special{pn 8}%
\special{pa 2800 3200}%
\special{pa 5200 3200}%
\special{fp}%
\special{sh 1}%
\special{pa 5200 3200}%
\special{pa 5134 3180}%
\special{pa 5148 3200}%
\special{pa 5134 3220}%
\special{pa 5200 3200}%
\special{fp}%
}}%
% FUNC 2 1 3 0 Black White
% 9 2800 2000 5200 4400 4000 3200 5000 3200 4000 2200 2200 1600 2200 1600 10 2 0 4
% ///x^2+y^2=1///-1.199///1.199
{\color[named]{Black}{%
\special{pn 8}%
\special{pa 5000 3200}%
\special{pa 5000 3146}%
\special{pa 4998 3146}%
\special{pa 4998 3139}%
\special{fp}%
\special{pa 4993 3078}%
\special{pa 4992 3078}%
\special{pa 4992 3070}%
\special{pa 4990 3056}%
\special{pa 4990 3050}%
\special{pa 4988 3048}%
\special{pa 4988 3038}%
\special{pa 4986 3036}%
\special{pa 4986 3026}%
\special{pa 4984 3024}%
\special{pa 4984 3019}%
\special{fp}%
\special{pa 4972 2966}%
\special{pa 4972 2962}%
\special{pa 4970 2956}%
\special{pa 4970 2952}%
\special{pa 4968 2950}%
\special{pa 4968 2948}%
\special{pa 4966 2940}%
\special{pa 4966 2936}%
\special{pa 4964 2936}%
\special{pa 4964 2930}%
\special{pa 4962 2928}%
\special{pa 4962 2922}%
\special{pa 4960 2922}%
\special{pa 4960 2916}%
\special{pa 4958 2914}%
\special{pa 4958 2910}%
\special{pa 4958 2910}%
\special{fp}%
\special{pa 4941 2861}%
\special{pa 4940 2860}%
\special{pa 4940 2856}%
\special{pa 4938 2854}%
\special{pa 4938 2852}%
\special{pa 4936 2848}%
\special{pa 4936 2846}%
\special{pa 4934 2844}%
\special{pa 4934 2840}%
\special{pa 4932 2838}%
\special{pa 4932 2834}%
\special{pa 4930 2834}%
\special{pa 4930 2830}%
\special{pa 4928 2828}%
\special{pa 4928 2824}%
\special{pa 4926 2824}%
\special{pa 4926 2820}%
\special{pa 4924 2818}%
\special{pa 4924 2816}%
\special{pa 4922 2814}%
\special{pa 4922 2810}%
\special{pa 4920 2810}%
\special{pa 4920 2809}%
\special{fp}%
\special{pa 4897 2757}%
\special{pa 4896 2756}%
\special{pa 4896 2754}%
\special{pa 4894 2752}%
\special{pa 4894 2750}%
\special{pa 4890 2746}%
\special{pa 4890 2742}%
\special{pa 4888 2742}%
\special{pa 4888 2738}%
\special{pa 4886 2738}%
\special{pa 4886 2734}%
\special{pa 4884 2734}%
\special{pa 4884 2730}%
\special{pa 4882 2730}%
\special{pa 4882 2726}%
\special{pa 4878 2722}%
\special{pa 4878 2720}%
\special{pa 4876 2718}%
\special{pa 4876 2716}%
\special{pa 4874 2714}%
\special{pa 4872 2710}%
\special{pa 4872 2708}%
\special{fp}%
\special{pa 4839 2657}%
\special{pa 4838 2656}%
\special{pa 4838 2654}%
\special{pa 4836 2652}%
\special{pa 4836 2650}%
\special{pa 4834 2648}%
\special{pa 4834 2646}%
\special{pa 4826 2638}%
\special{pa 4826 2636}%
\special{pa 4822 2632}%
\special{pa 4822 2630}%
\special{pa 4820 2628}%
\special{pa 4818 2624}%
\special{pa 4816 2622}%
\special{pa 4816 2620}%
\special{pa 4812 2618}%
\special{pa 4812 2616}%
\special{pa 4810 2614}%
\special{pa 4808 2610}%
\special{pa 4806 2608}%
\special{fp}%
\special{pa 4767 2559}%
\special{pa 4762 2554}%
\special{pa 4762 2552}%
\special{pa 4750 2540}%
\special{pa 4750 2538}%
\special{pa 4742 2530}%
\special{pa 4742 2528}%
\special{pa 4727 2513}%
\special{fp}%
\special{pa 4680 2468}%
\special{pa 4656 2444}%
\special{pa 4654 2444}%
\special{pa 4642 2432}%
\special{pa 4640 2432}%
\special{pa 4636 2428}%
\special{pa 4634 2428}%
\special{fp}%
\special{pa 4588 2391}%
\special{pa 4588 2390}%
\special{pa 4586 2390}%
\special{pa 4584 2388}%
\special{pa 4580 2386}%
\special{pa 4578 2384}%
\special{pa 4576 2384}%
\special{pa 4574 2380}%
\special{pa 4572 2380}%
\special{pa 4568 2376}%
\special{pa 4566 2376}%
\special{pa 4562 2372}%
\special{pa 4560 2372}%
\special{pa 4558 2370}%
\special{pa 4556 2370}%
\special{pa 4552 2366}%
\special{pa 4548 2364}%
\special{pa 4544 2360}%
\special{pa 4542 2360}%
\special{pa 4539 2357}%
\special{fp}%
\special{pa 4487 2327}%
\special{pa 4486 2326}%
\special{pa 4478 2322}%
\special{pa 4476 2322}%
\special{pa 4472 2318}%
\special{pa 4470 2318}%
\special{pa 4468 2316}%
\special{pa 4466 2316}%
\special{pa 4464 2314}%
\special{pa 4462 2314}%
\special{pa 4460 2312}%
\special{pa 4458 2312}%
\special{pa 4454 2308}%
\special{pa 4450 2308}%
\special{pa 4450 2306}%
\special{pa 4446 2306}%
\special{pa 4446 2304}%
\special{pa 4442 2304}%
\special{pa 4442 2302}%
\special{pa 4438 2302}%
\special{pa 4437 2301}%
\special{fp}%
\special{pa 4384 2277}%
\special{pa 4382 2276}%
\special{pa 4380 2276}%
\special{pa 4378 2274}%
\special{pa 4374 2274}%
\special{pa 4374 2272}%
\special{pa 4370 2272}%
\special{pa 4368 2270}%
\special{pa 4364 2270}%
\special{pa 4364 2268}%
\special{pa 4362 2268}%
\special{pa 4356 2266}%
\special{pa 4354 2266}%
\special{pa 4354 2264}%
\special{pa 4348 2264}%
\special{pa 4348 2262}%
\special{pa 4346 2262}%
\special{pa 4340 2260}%
\special{pa 4338 2260}%
\special{pa 4338 2258}%
\special{pa 4334 2258}%
\special{fp}%
\special{pa 4283 2242}%
\special{pa 4282 2242}%
\special{pa 4282 2240}%
\special{pa 4276 2240}%
\special{pa 4274 2238}%
\special{pa 4268 2238}%
\special{pa 4268 2236}%
\special{pa 4264 2236}%
\special{pa 4258 2234}%
\special{pa 4254 2234}%
\special{pa 4252 2232}%
\special{pa 4250 2232}%
\special{pa 4242 2230}%
\special{pa 4238 2230}%
\special{pa 4238 2228}%
\special{pa 4232 2228}%
\special{pa 4229 2227}%
\special{fp}%
\special{pa 4169 2215}%
\special{pa 4152 2212}%
\special{pa 4138 2210}%
\special{pa 4132 2210}%
\special{pa 4130 2208}%
\special{pa 4122 2208}%
\special{pa 4107 2206}%
\special{fp}%
\special{pa 4044 2202}%
\special{pa 4032 2202}%
\special{pa 4032 2200}%
\special{pa 3983 2200}%
\special{fp}%
\special{pa 3922 2204}%
\special{pa 3906 2204}%
\special{pa 3906 2206}%
\special{pa 3896 2206}%
\special{pa 3878 2208}%
\special{pa 3870 2208}%
\special{pa 3870 2210}%
\special{pa 3863 2210}%
\special{fp}%
\special{pa 3803 2220}%
\special{pa 3794 2222}%
\special{pa 3790 2222}%
\special{pa 3790 2224}%
\special{pa 3782 2224}%
\special{pa 3780 2226}%
\special{pa 3776 2226}%
\special{pa 3768 2228}%
\special{pa 3764 2228}%
\special{pa 3762 2230}%
\special{pa 3756 2230}%
\special{pa 3754 2232}%
\special{pa 3748 2232}%
\special{pa 3748 2234}%
\special{pa 3747 2234}%
\special{fp}%
\special{pa 3691 2250}%
\special{pa 3666 2258}%
\special{pa 3664 2258}%
\special{pa 3662 2260}%
\special{pa 3660 2260}%
\special{pa 3656 2262}%
\special{pa 3654 2262}%
\special{pa 3652 2264}%
\special{pa 3648 2264}%
\special{pa 3646 2266}%
\special{pa 3644 2266}%
\special{pa 3640 2268}%
\special{pa 3638 2268}%
\special{pa 3636 2270}%
\special{pa 3634 2270}%
\special{fp}%
\special{pa 3583 2292}%
\special{pa 3580 2292}%
\special{pa 3580 2294}%
\special{pa 3576 2294}%
\special{pa 3576 2296}%
\special{pa 3572 2296}%
\special{pa 3572 2298}%
\special{pa 3568 2298}%
\special{pa 3568 2300}%
\special{pa 3564 2300}%
\special{pa 3564 2302}%
\special{pa 3560 2302}%
\special{pa 3558 2304}%
\special{pa 3556 2304}%
\special{pa 3554 2306}%
\special{pa 3552 2306}%
\special{pa 3550 2308}%
\special{pa 3548 2308}%
\special{pa 3546 2310}%
\special{pa 3535 2316}%
\special{fp}%
\special{pa 3484 2344}%
\special{pa 3482 2346}%
\special{pa 3480 2346}%
\special{pa 3478 2348}%
\special{pa 3476 2348}%
\special{pa 3470 2354}%
\special{pa 3468 2354}%
\special{pa 3466 2356}%
\special{pa 3464 2356}%
\special{pa 3462 2358}%
\special{pa 3460 2358}%
\special{pa 3456 2362}%
\special{pa 3454 2362}%
\special{pa 3448 2368}%
\special{pa 3446 2368}%
\special{pa 3444 2370}%
\special{pa 3440 2372}%
\special{pa 3438 2374}%
\special{pa 3436 2374}%
\special{pa 3434 2376}%
\special{fp}%
\special{pa 3384 2414}%
\special{pa 3382 2416}%
\special{pa 3380 2416}%
\special{pa 3372 2424}%
\special{pa 3370 2424}%
\special{pa 3366 2428}%
\special{pa 3360 2432}%
\special{pa 3354 2438}%
\special{pa 3352 2438}%
\special{pa 3340 2450}%
\special{pa 3338 2450}%
\special{pa 3336 2452}%
\special{fp}%
\special{pa 3291 2497}%
\special{pa 3260 2528}%
\special{pa 3260 2530}%
\special{pa 3250 2538}%
\special{pa 3250 2540}%
\special{pa 3248 2542}%
\special{fp}%
\special{pa 3208 2590}%
\special{pa 3208 2590}%
\special{pa 3206 2594}%
\special{pa 3202 2598}%
\special{pa 3200 2602}%
\special{pa 3196 2606}%
\special{pa 3194 2610}%
\special{pa 3188 2616}%
\special{pa 3188 2618}%
\special{pa 3186 2620}%
\special{pa 3186 2622}%
\special{pa 3178 2630}%
\special{pa 3178 2632}%
\special{pa 3174 2636}%
\special{pa 3174 2638}%
\special{pa 3173 2639}%
\special{fp}%
\special{pa 3140 2691}%
\special{pa 3140 2692}%
\special{pa 3138 2694}%
\special{pa 3136 2698}%
\special{pa 3136 2700}%
\special{pa 3132 2704}%
\special{pa 3132 2706}%
\special{pa 3130 2708}%
\special{pa 3128 2712}%
\special{pa 3128 2714}%
\special{pa 3124 2718}%
\special{pa 3124 2720}%
\special{pa 3122 2722}%
\special{pa 3122 2724}%
\special{pa 3120 2726}%
\special{pa 3111 2744}%
\special{fp}%
\special{pa 3086 2797}%
\special{pa 3086 2798}%
\special{pa 3084 2798}%
\special{pa 3084 2802}%
\special{pa 3082 2804}%
\special{pa 3082 2806}%
\special{pa 3080 2808}%
\special{pa 3080 2812}%
\special{pa 3078 2812}%
\special{pa 3078 2816}%
\special{pa 3076 2818}%
\special{pa 3076 2822}%
\special{pa 3074 2822}%
\special{pa 3074 2826}%
\special{pa 3072 2828}%
\special{pa 3072 2832}%
\special{pa 3070 2832}%
\special{pa 3070 2836}%
\special{pa 3068 2838}%
\special{pa 3068 2840}%
\special{pa 3066 2844}%
\special{pa 3066 2846}%
\special{pa 3065 2847}%
\special{fp}%
\special{pa 3046 2904}%
\special{pa 3046 2906}%
\special{pa 3044 2906}%
\special{pa 3044 2912}%
\special{pa 3042 2912}%
\special{pa 3042 2918}%
\special{pa 3040 2920}%
\special{pa 3040 2926}%
\special{pa 3038 2926}%
\special{pa 3038 2932}%
\special{pa 3036 2934}%
\special{pa 3036 2936}%
\special{pa 3034 2944}%
\special{pa 3034 2948}%
\special{pa 3032 2948}%
\special{pa 3032 2952}%
\special{pa 3030 2957}%
\special{fp}%
\special{pa 3018 3017}%
\special{pa 3018 3020}%
\special{pa 3016 3020}%
\special{pa 3016 3026}%
\special{pa 3014 3036}%
\special{pa 3014 3038}%
\special{pa 3012 3048}%
\special{pa 3012 3050}%
\special{pa 3010 3062}%
\special{pa 3010 3070}%
\special{pa 3008 3070}%
\special{pa 3008 3076}%
\special{fp}%
\special{pa 3002 3136}%
\special{pa 3002 3168}%
\special{pa 3000 3170}%
\special{pa 3000 3198}%
\special{fp}%
\special{pa 3002 3260}%
\special{pa 3002 3270}%
\special{pa 3004 3272}%
\special{pa 3004 3294}%
\special{pa 3006 3296}%
\special{pa 3006 3306}%
\special{pa 3008 3322}%
\special{fp}%
\special{pa 3018 3383}%
\special{pa 3018 3388}%
\special{pa 3020 3396}%
\special{pa 3020 3398}%
\special{pa 3022 3406}%
\special{pa 3022 3410}%
\special{pa 3024 3412}%
\special{pa 3024 3420}%
\special{pa 3026 3420}%
\special{pa 3026 3426}%
\special{pa 3028 3432}%
\special{pa 3028 3438}%
\special{pa 3030 3438}%
\special{pa 3030 3440}%
\special{fp}%
\special{pa 3045 3495}%
\special{pa 3046 3496}%
\special{pa 3046 3500}%
\special{pa 3048 3504}%
\special{pa 3048 3506}%
\special{pa 3050 3510}%
\special{pa 3050 3512}%
\special{pa 3052 3516}%
\special{pa 3052 3518}%
\special{pa 3054 3522}%
\special{pa 3054 3524}%
\special{pa 3056 3528}%
\special{pa 3056 3530}%
\special{pa 3058 3534}%
\special{pa 3058 3538}%
\special{pa 3060 3538}%
\special{pa 3060 3540}%
\special{pa 3062 3546}%
\special{pa 3062 3548}%
\special{pa 3064 3548}%
\special{pa 3064 3551}%
\special{fp}%
\special{pa 3084 3602}%
\special{pa 3086 3604}%
\special{pa 3086 3606}%
\special{pa 3088 3608}%
\special{pa 3088 3610}%
\special{pa 3092 3618}%
\special{pa 3092 3620}%
\special{pa 3094 3622}%
\special{pa 3094 3624}%
\special{pa 3096 3626}%
\special{pa 3096 3628}%
\special{pa 3098 3630}%
\special{pa 3098 3632}%
\special{pa 3100 3634}%
\special{pa 3100 3636}%
\special{pa 3102 3638}%
\special{pa 3102 3642}%
\special{pa 3104 3642}%
\special{pa 3104 3646}%
\special{pa 3106 3646}%
\special{pa 3106 3650}%
\special{pa 3108 3650}%
\special{pa 3108 3652}%
\special{fp}%
\special{pa 3136 3702}%
\special{pa 3136 3702}%
\special{pa 3136 3704}%
\special{pa 3140 3708}%
\special{pa 3140 3712}%
\special{pa 3146 3718}%
\special{pa 3146 3722}%
\special{pa 3150 3726}%
\special{pa 3152 3730}%
\special{pa 3156 3734}%
\special{pa 3156 3738}%
\special{pa 3160 3742}%
\special{pa 3160 3744}%
\special{pa 3164 3748}%
\special{pa 3166 3752}%
\special{pa 3167 3753}%
\special{fp}%
\special{pa 3201 3801}%
\special{pa 3202 3802}%
\special{pa 3202 3804}%
\special{pa 3206 3806}%
\special{pa 3206 3808}%
\special{pa 3208 3810}%
\special{pa 3208 3812}%
\special{pa 3212 3814}%
\special{pa 3212 3816}%
\special{pa 3220 3824}%
\special{pa 3220 3826}%
\special{pa 3228 3834}%
\special{pa 3228 3836}%
\special{pa 3232 3840}%
\special{pa 3232 3842}%
\special{pa 3237 3847}%
\special{fp}%
\special{pa 3280 3894}%
\special{pa 3324 3938}%
\special{fp}%
\special{pa 3372 3978}%
\special{pa 3380 3986}%
\special{pa 3382 3986}%
\special{pa 3386 3990}%
\special{pa 3390 3992}%
\special{pa 3394 3996}%
\special{pa 3398 3998}%
\special{pa 3402 4002}%
\special{pa 3406 4004}%
\special{pa 3410 4008}%
\special{pa 3414 4010}%
\special{pa 3416 4012}%
\special{pa 3418 4012}%
\special{pa 3420 4016}%
\special{pa 3420 4016}%
\special{fp}%
\special{pa 3471 4049}%
\special{pa 3472 4050}%
\special{pa 3474 4050}%
\special{pa 3476 4052}%
\special{pa 3478 4052}%
\special{pa 3484 4058}%
\special{pa 3488 4058}%
\special{pa 3494 4064}%
\special{pa 3498 4064}%
\special{pa 3498 4066}%
\special{pa 3500 4066}%
\special{pa 3502 4068}%
\special{pa 3504 4068}%
\special{pa 3508 4072}%
\special{pa 3512 4072}%
\special{pa 3512 4074}%
\special{pa 3514 4074}%
\special{pa 3516 4076}%
\special{pa 3518 4076}%
\special{pa 3520 4078}%
\special{fp}%
\special{pa 3567 4102}%
\special{pa 3570 4102}%
\special{pa 3570 4104}%
\special{pa 3574 4104}%
\special{pa 3574 4106}%
\special{pa 3578 4106}%
\special{pa 3578 4108}%
\special{pa 3582 4108}%
\special{pa 3584 4110}%
\special{pa 3586 4110}%
\special{pa 3588 4112}%
\special{pa 3590 4112}%
\special{pa 3592 4114}%
\special{pa 3594 4114}%
\special{pa 3598 4116}%
\special{pa 3600 4116}%
\special{pa 3600 4118}%
\special{pa 3604 4118}%
\special{pa 3606 4120}%
\special{pa 3610 4120}%
\special{pa 3610 4122}%
\special{pa 3614 4122}%
\special{pa 3616 4124}%
\special{fp}%
\special{pa 3670 4144}%
\special{pa 3674 4146}%
\special{pa 3676 4146}%
\special{pa 3680 4148}%
\special{pa 3682 4148}%
\special{pa 3686 4150}%
\special{pa 3688 4150}%
\special{pa 3692 4152}%
\special{pa 3694 4152}%
\special{pa 3698 4154}%
\special{pa 3702 4154}%
\special{pa 3702 4156}%
\special{pa 3706 4156}%
\special{pa 3712 4158}%
\special{pa 3714 4158}%
\special{pa 3716 4160}%
\special{pa 3722 4160}%
\special{pa 3722 4162}%
\special{pa 3726 4162}%
\special{fp}%
\special{pa 3779 4176}%
\special{pa 3784 4176}%
\special{pa 3786 4178}%
\special{pa 3794 4178}%
\special{pa 3794 4180}%
\special{pa 3804 4180}%
\special{pa 3804 4182}%
\special{pa 3814 4182}%
\special{pa 3814 4184}%
\special{pa 3824 4184}%
\special{pa 3826 4186}%
\special{pa 3835 4186}%
\special{fp}%
\special{pa 3895 4194}%
\special{pa 3896 4194}%
\special{pa 3896 4196}%
\special{pa 3916 4196}%
\special{pa 3918 4198}%
\special{pa 3946 4198}%
\special{pa 3946 4200}%
\special{pa 3953 4200}%
\special{fp}%
\special{pa 4017 4200}%
\special{pa 4054 4200}%
\special{pa 4056 4198}%
\special{pa 4079 4198}%
\special{fp}%
\special{pa 4140 4191}%
\special{pa 4144 4190}%
\special{pa 4152 4190}%
\special{pa 4152 4188}%
\special{pa 4164 4188}%
\special{pa 4164 4186}%
\special{pa 4176 4186}%
\special{pa 4176 4184}%
\special{pa 4186 4184}%
\special{pa 4188 4182}%
\special{pa 4196 4182}%
\special{pa 4196 4182}%
\special{fp}%
\special{pa 4253 4168}%
\special{pa 4254 4168}%
\special{pa 4260 4166}%
\special{pa 4264 4166}%
\special{pa 4266 4164}%
\special{pa 4272 4164}%
\special{pa 4272 4162}%
\special{pa 4278 4162}%
\special{pa 4280 4160}%
\special{pa 4286 4160}%
\special{pa 4286 4158}%
\special{pa 4292 4158}%
\special{pa 4292 4156}%
\special{pa 4298 4156}%
\special{pa 4300 4154}%
\special{pa 4304 4154}%
\special{pa 4306 4152}%
\special{pa 4306 4152}%
\special{fp}%
\special{pa 4360 4134}%
\special{pa 4362 4134}%
\special{pa 4362 4132}%
\special{pa 4366 4132}%
\special{pa 4368 4130}%
\special{pa 4372 4130}%
\special{pa 4372 4128}%
\special{pa 4376 4128}%
\special{pa 4378 4126}%
\special{pa 4382 4126}%
\special{pa 4382 4124}%
\special{pa 4386 4124}%
\special{pa 4386 4122}%
\special{pa 4390 4122}%
\special{pa 4392 4120}%
\special{pa 4394 4120}%
\special{pa 4398 4118}%
\special{pa 4400 4118}%
\special{pa 4402 4116}%
\special{pa 4404 4116}%
\special{pa 4406 4114}%
\special{pa 4410 4114}%
\special{pa 4410 4113}%
\special{fp}%
\special{pa 4462 4088}%
\special{pa 4462 4088}%
\special{pa 4464 4086}%
\special{pa 4466 4086}%
\special{pa 4468 4084}%
\special{pa 4470 4084}%
\special{pa 4472 4082}%
\special{pa 4474 4082}%
\special{pa 4476 4080}%
\special{pa 4484 4076}%
\special{pa 4486 4076}%
\special{pa 4490 4072}%
\special{pa 4492 4072}%
\special{pa 4494 4070}%
\special{pa 4498 4068}%
\special{pa 4500 4068}%
\special{pa 4504 4064}%
\special{pa 4506 4064}%
\special{pa 4508 4062}%
\special{pa 4512 4060}%
\special{pa 4514 4058}%
\special{fp}%
\special{pa 4563 4017}%
\special{pa 4612 3977}%
\special{fp}%
\special{pa 4661 3935}%
\special{pa 4709 3895}%
\special{fp}%
\special{pa 4757 3853}%
\special{pa 4768 3842}%
\special{pa 4768 3840}%
\special{pa 4774 3836}%
\special{pa 4774 3834}%
\special{pa 4782 3826}%
\special{pa 4782 3824}%
\special{pa 4790 3816}%
\special{pa 4790 3814}%
\special{pa 4792 3812}%
\special{pa 4792 3810}%
\special{pa 4796 3808}%
\special{pa 4796 3808}%
\special{fp}%
\special{pa 4830 3759}%
\special{pa 4832 3758}%
\special{pa 4832 3756}%
\special{pa 4836 3752}%
\special{pa 4836 3748}%
\special{pa 4840 3744}%
\special{pa 4840 3742}%
\special{pa 4844 3738}%
\special{pa 4846 3734}%
\special{pa 4850 3730}%
\special{pa 4850 3726}%
\special{pa 4854 3722}%
\special{pa 4856 3718}%
\special{pa 4858 3716}%
\special{pa 4858 3714}%
\special{pa 4860 3712}%
\special{pa 4862 3709}%
\special{fp}%
\special{pa 4892 3653}%
\special{pa 4914 3610}%
\special{pa 4914 3606}%
\special{pa 4916 3604}%
\special{pa 4916 3602}%
\special{pa 4918 3600}%
\special{pa 4918 3598}%
\special{pa 4919 3597}%
\special{fp}%
\special{pa 4941 3541}%
\special{pa 4948 3520}%
\special{pa 4952 3508}%
\special{pa 4952 3506}%
\special{pa 4954 3504}%
\special{pa 4960 3481}%
\special{fp}%
\special{pa 4977 3420}%
\special{pa 4980 3402}%
\special{pa 4980 3398}%
\special{pa 4984 3380}%
\special{pa 4984 3378}%
\special{pa 4988 3362}%
\special{pa 4989 3358}%
\special{fp}%
\special{pa 4996 3295}%
\special{pa 4996 3286}%
\special{pa 4998 3272}%
\special{pa 4998 3256}%
\special{pa 5000 3232}%
\special{fp}%
}}%
% FUNC 2 1 3 0 Black White
% 10 2800 2000 5200 4400 4000 3200 5000 3200 4000 2200 2200 1600 2200 1600 10 2 0 4 0 0
% ///x^2+y^2=0.4///-1.199///1.199
{\color[named]{Black}{%
\special{pn 8}%
\special{pn 8}%
\special{pa 4600 3400}%
\special{pa 4602 3398}%
\special{pa 4602 3394}%
\special{pa 4604 3392}%
\special{pa 4604 3386}%
\special{pa 4606 3386}%
\special{pa 4606 3380}%
\special{pa 4608 3380}%
\special{pa 4608 3374}%
\special{pa 4610 3372}%
\special{pa 4610 3366}%
\special{pa 4612 3366}%
\special{pa 4612 3358}%
\special{pa 4614 3358}%
\special{pa 4614 3350}%
\special{pa 4616 3350}%
\special{pa 4616 3349}%
\special{fp}%
\special{pa 4626 3292}%
\special{pa 4628 3280}%
\special{pa 4630 3262}%
\special{pa 4630 3250}%
\special{pa 4632 3250}%
\special{pa 4632 3231}%
\special{fp}%
\special{pa 4632 3167}%
\special{pa 4632 3152}%
\special{pa 4630 3150}%
\special{pa 4630 3138}%
\special{pa 4628 3122}%
\special{pa 4628 3120}%
\special{pa 4626 3108}%
\special{pa 4626 3105}%
\special{fp}%
\special{pa 4614 3045}%
\special{pa 4614 3044}%
\special{pa 4612 3042}%
\special{pa 4612 3036}%
\special{pa 4610 3034}%
\special{pa 4610 3028}%
\special{pa 4608 3028}%
\special{pa 4608 3022}%
\special{pa 4606 3020}%
\special{pa 4606 3016}%
\special{pa 4604 3014}%
\special{pa 4604 3010}%
\special{pa 4602 3006}%
\special{pa 4602 3004}%
\special{pa 4600 3000}%
\special{pa 4600 2998}%
\special{pa 4598 2994}%
\special{pa 4598 2992}%
\special{pa 4597 2989}%
\special{fp}%
\special{pa 4576 2936}%
\special{pa 4576 2936}%
\special{pa 4574 2936}%
\special{pa 4574 2932}%
\special{pa 4572 2932}%
\special{pa 4572 2928}%
\special{pa 4570 2926}%
\special{pa 4570 2924}%
\special{pa 4568 2922}%
\special{pa 4568 2920}%
\special{pa 4566 2918}%
\special{pa 4566 2916}%
\special{pa 4564 2914}%
\special{pa 4564 2912}%
\special{pa 4560 2908}%
\special{pa 4560 2904}%
\special{pa 4558 2904}%
\special{pa 4558 2900}%
\special{pa 4556 2900}%
\special{pa 4556 2896}%
\special{pa 4554 2896}%
\special{pa 4554 2894}%
\special{pa 4552 2892}%
\special{pa 4552 2890}%
\special{pa 4551 2889}%
\special{fp}%
\special{pa 4520 2838}%
\special{pa 4520 2838}%
\special{pa 4514 2832}%
\special{pa 4512 2828}%
\special{pa 4510 2826}%
\special{pa 4510 2824}%
\special{pa 4506 2822}%
\special{pa 4506 2820}%
\special{pa 4504 2818}%
\special{pa 4504 2816}%
\special{pa 4492 2804}%
\special{pa 4492 2802}%
\special{pa 4488 2798}%
\special{pa 4488 2796}%
\special{pa 4482 2792}%
\special{pa 4482 2792}%
\special{fp}%
\special{pa 4439 2745}%
\special{pa 4428 2734}%
\special{pa 4426 2734}%
\special{pa 4418 2724}%
\special{pa 4416 2724}%
\special{pa 4410 2718}%
\special{pa 4404 2714}%
\special{pa 4394 2704}%
\special{pa 4393 2704}%
\special{fp}%
\special{pa 4343 2669}%
\special{pa 4342 2668}%
\special{pa 4340 2668}%
\special{pa 4334 2662}%
\special{pa 4332 2662}%
\special{pa 4330 2660}%
\special{pa 4328 2660}%
\special{pa 4326 2658}%
\special{pa 4322 2656}%
\special{pa 4320 2654}%
\special{pa 4318 2654}%
\special{pa 4314 2650}%
\special{pa 4310 2650}%
\special{pa 4310 2648}%
\special{pa 4308 2648}%
\special{pa 4306 2646}%
\special{pa 4304 2646}%
\special{pa 4302 2644}%
\special{pa 4300 2644}%
\special{pa 4298 2642}%
\special{pa 4296 2642}%
\special{pa 4293 2639}%
\special{fp}%
\special{pa 4244 2618}%
\special{pa 4244 2618}%
\special{pa 4244 2616}%
\special{pa 4240 2616}%
\special{pa 4240 2614}%
\special{pa 4236 2614}%
\special{pa 4234 2612}%
\special{pa 4230 2612}%
\special{pa 4230 2610}%
\special{pa 4224 2610}%
\special{pa 4224 2608}%
\special{pa 4222 2608}%
\special{pa 4216 2606}%
\special{pa 4214 2606}%
\special{pa 4214 2604}%
\special{pa 4208 2604}%
\special{pa 4208 2602}%
\special{pa 4202 2602}%
\special{pa 4202 2600}%
\special{pa 4196 2600}%
\special{pa 4196 2600}%
\special{fp}%
\special{pa 4139 2584}%
\special{pa 4138 2584}%
\special{pa 4136 2582}%
\special{pa 4132 2582}%
\special{pa 4124 2580}%
\special{pa 4118 2580}%
\special{pa 4118 2578}%
\special{pa 4108 2578}%
\special{pa 4106 2576}%
\special{pa 4094 2576}%
\special{pa 4094 2574}%
\special{pa 4082 2574}%
\special{fp}%
\special{pa 4022 2568}%
\special{pa 3966 2568}%
\special{pa 3966 2570}%
\special{pa 3961 2570}%
\special{fp}%
\special{pa 3899 2576}%
\special{pa 3894 2576}%
\special{pa 3894 2578}%
\special{pa 3884 2578}%
\special{pa 3882 2580}%
\special{pa 3874 2580}%
\special{pa 3872 2582}%
\special{pa 3864 2582}%
\special{pa 3864 2584}%
\special{pa 3858 2584}%
\special{pa 3852 2586}%
\special{pa 3850 2586}%
\special{pa 3844 2588}%
\special{pa 3842 2588}%
\special{pa 3842 2588}%
\special{fp}%
\special{pa 3785 2606}%
\special{pa 3784 2606}%
\special{pa 3780 2608}%
\special{pa 3778 2608}%
\special{pa 3776 2610}%
\special{pa 3772 2610}%
\special{pa 3770 2612}%
\special{pa 3766 2612}%
\special{pa 3766 2614}%
\special{pa 3762 2614}%
\special{pa 3760 2616}%
\special{pa 3758 2616}%
\special{pa 3756 2618}%
\special{pa 3752 2618}%
\special{pa 3752 2620}%
\special{pa 3748 2620}%
\special{pa 3746 2622}%
\special{pa 3744 2622}%
\special{pa 3742 2624}%
\special{pa 3740 2624}%
\special{pa 3732 2628}%
\special{fp}%
\special{pa 3679 2655}%
\special{pa 3674 2660}%
\special{pa 3670 2660}%
\special{pa 3666 2664}%
\special{pa 3662 2666}%
\special{pa 3658 2670}%
\special{pa 3656 2670}%
\special{pa 3652 2674}%
\special{pa 3650 2674}%
\special{pa 3646 2678}%
\special{pa 3644 2678}%
\special{pa 3640 2682}%
\special{pa 3638 2682}%
\special{pa 3636 2684}%
\special{pa 3632 2686}%
\special{pa 3629 2689}%
\special{fp}%
\special{pa 3581 2727}%
\special{pa 3536 2772}%
\special{fp}%
\special{pa 3496 2819}%
\special{pa 3494 2820}%
\special{pa 3494 2822}%
\special{pa 3492 2824}%
\special{pa 3492 2826}%
\special{pa 3486 2832}%
\special{pa 3484 2836}%
\special{pa 3482 2838}%
\special{pa 3482 2840}%
\special{pa 3478 2844}%
\special{pa 3478 2846}%
\special{pa 3474 2850}%
\special{pa 3474 2852}%
\special{pa 3470 2856}%
\special{pa 3470 2858}%
\special{pa 3466 2862}%
\special{pa 3464 2866}%
\special{pa 3462 2868}%
\special{fp}%
\special{pa 3435 2917}%
\special{pa 3434 2918}%
\special{pa 3434 2920}%
\special{pa 3432 2922}%
\special{pa 3432 2924}%
\special{pa 3430 2926}%
\special{pa 3430 2930}%
\special{pa 3428 2930}%
\special{pa 3428 2934}%
\special{pa 3426 2934}%
\special{pa 3426 2938}%
\special{pa 3424 2938}%
\special{pa 3424 2942}%
\special{pa 3422 2944}%
\special{pa 3422 2946}%
\special{pa 3420 2948}%
\special{pa 3420 2952}%
\special{pa 3418 2952}%
\special{pa 3418 2956}%
\special{pa 3416 2958}%
\special{pa 3416 2960}%
\special{pa 3414 2962}%
\special{pa 3414 2966}%
\special{pa 3413 2966}%
\special{fp}%
\special{pa 3394 3023}%
\special{pa 3394 3024}%
\special{pa 3392 3026}%
\special{pa 3392 3032}%
\special{pa 3390 3032}%
\special{pa 3390 3036}%
\special{pa 3388 3042}%
\special{pa 3388 3044}%
\special{pa 3386 3050}%
\special{pa 3386 3052}%
\special{pa 3384 3058}%
\special{pa 3384 3064}%
\special{pa 3382 3064}%
\special{pa 3382 3068}%
\special{pa 3380 3078}%
\special{pa 3380 3080}%
\special{fp}%
\special{pa 3370 3140}%
\special{pa 3370 3140}%
\special{pa 3370 3166}%
\special{pa 3368 3166}%
\special{pa 3368 3201}%
\special{fp}%
\special{pa 3372 3263}%
\special{pa 3372 3280}%
\special{pa 3374 3280}%
\special{pa 3374 3294}%
\special{pa 3376 3294}%
\special{pa 3376 3306}%
\special{pa 3378 3308}%
\special{pa 3378 3318}%
\special{pa 3380 3318}%
\special{pa 3380 3319}%
\special{fp}%
\special{pa 3394 3377}%
\special{pa 3394 3380}%
\special{pa 3396 3386}%
\special{pa 3396 3390}%
\special{pa 3398 3390}%
\special{pa 3398 3396}%
\special{pa 3400 3396}%
\special{pa 3400 3402}%
\special{pa 3402 3402}%
\special{pa 3402 3408}%
\special{pa 3404 3408}%
\special{pa 3404 3414}%
\special{pa 3406 3414}%
\special{pa 3406 3416}%
\special{pa 3408 3422}%
\special{pa 3408 3424}%
\special{pa 3410 3424}%
\special{pa 3410 3428}%
\special{fp}%
\special{pa 3433 3479}%
\special{pa 3438 3490}%
\special{pa 3438 3492}%
\special{pa 3442 3496}%
\special{pa 3442 3498}%
\special{pa 3444 3500}%
\special{pa 3444 3502}%
\special{pa 3446 3504}%
\special{pa 3446 3506}%
\special{pa 3448 3508}%
\special{pa 3450 3512}%
\special{pa 3450 3514}%
\special{pa 3454 3518}%
\special{pa 3454 3520}%
\special{pa 3456 3522}%
\special{pa 3458 3526}%
\special{pa 3460 3528}%
\special{pa 3460 3530}%
\special{pa 3462 3532}%
\special{fp}%
\special{pa 3496 3582}%
\special{pa 3502 3588}%
\special{pa 3502 3590}%
\special{pa 3504 3592}%
\special{pa 3504 3594}%
\special{pa 3514 3604}%
\special{pa 3518 3610}%
\special{pa 3524 3616}%
\special{pa 3524 3618}%
\special{pa 3534 3626}%
\special{pa 3534 3628}%
\special{pa 3535 3629}%
\special{fp}%
\special{pa 3581 3674}%
\special{pa 3584 3676}%
\special{pa 3590 3682}%
\special{pa 3592 3682}%
\special{pa 3596 3688}%
\special{pa 3598 3688}%
\special{pa 3602 3692}%
\special{pa 3604 3692}%
\special{pa 3616 3704}%
\special{pa 3618 3704}%
\special{pa 3620 3706}%
\special{pa 3622 3706}%
\special{pa 3624 3710}%
\special{pa 3626 3710}%
\special{pa 3628 3712}%
\special{fp}%
\special{pa 3679 3745}%
\special{pa 3680 3746}%
\special{pa 3684 3748}%
\special{pa 3686 3748}%
\special{pa 3690 3752}%
\special{pa 3692 3752}%
\special{pa 3694 3754}%
\special{pa 3706 3760}%
\special{pa 3708 3760}%
\special{pa 3712 3764}%
\special{pa 3714 3764}%
\special{pa 3716 3766}%
\special{pa 3718 3766}%
\special{pa 3720 3768}%
\special{pa 3722 3768}%
\special{pa 3724 3770}%
\special{pa 3726 3770}%
\special{pa 3728 3772}%
\special{pa 3730 3772}%
\special{pa 3732 3773}%
\special{fp}%
\special{pa 3785 3795}%
\special{pa 3786 3796}%
\special{pa 3788 3796}%
\special{pa 3792 3798}%
\special{pa 3794 3798}%
\special{pa 3798 3800}%
\special{pa 3800 3800}%
\special{pa 3804 3802}%
\special{pa 3806 3802}%
\special{pa 3810 3804}%
\special{pa 3814 3804}%
\special{pa 3816 3806}%
\special{pa 3820 3806}%
\special{pa 3822 3808}%
\special{pa 3828 3808}%
\special{pa 3828 3810}%
\special{pa 3834 3810}%
\special{pa 3836 3812}%
\special{pa 3842 3812}%
\special{fp}%
\special{pa 3899 3824}%
\special{pa 3900 3824}%
\special{pa 3902 3826}%
\special{pa 3908 3826}%
\special{pa 3920 3828}%
\special{pa 3922 3828}%
\special{pa 3938 3830}%
\special{pa 3950 3830}%
\special{pa 3952 3832}%
\special{pa 3960 3832}%
\special{fp}%
\special{pa 4025 3832}%
\special{pa 4050 3832}%
\special{pa 4050 3830}%
\special{pa 4062 3830}%
\special{pa 4080 3828}%
\special{pa 4086 3827}%
\special{fp}%
\special{pa 4143 3816}%
\special{pa 4150 3816}%
\special{pa 4150 3814}%
\special{pa 4158 3814}%
\special{pa 4158 3812}%
\special{pa 4166 3812}%
\special{pa 4166 3810}%
\special{pa 4172 3810}%
\special{pa 4174 3808}%
\special{pa 4180 3808}%
\special{pa 4180 3806}%
\special{pa 4186 3806}%
\special{pa 4186 3804}%
\special{pa 4190 3804}%
\special{pa 4195 3802}%
\special{fp}%
\special{pa 4251 3782}%
\special{pa 4252 3782}%
\special{pa 4252 3780}%
\special{pa 4256 3780}%
\special{pa 4256 3778}%
\special{pa 4260 3778}%
\special{pa 4262 3776}%
\special{pa 4264 3776}%
\special{pa 4266 3774}%
\special{pa 4268 3774}%
\special{pa 4270 3772}%
\special{pa 4274 3772}%
\special{pa 4274 3770}%
\special{pa 4278 3770}%
\special{pa 4278 3768}%
\special{pa 4282 3768}%
\special{pa 4282 3766}%
\special{pa 4286 3766}%
\special{pa 4286 3764}%
\special{pa 4290 3764}%
\special{pa 4290 3762}%
\special{pa 4292 3762}%
\special{pa 4294 3760}%
\special{pa 4296 3760}%
\special{pa 4296 3760}%
\special{fp}%
\special{pa 4348 3730}%
\special{pa 4348 3730}%
\special{pa 4352 3726}%
\special{pa 4354 3726}%
\special{pa 4358 3722}%
\special{pa 4360 3722}%
\special{pa 4364 3718}%
\special{pa 4366 3718}%
\special{pa 4368 3714}%
\special{pa 4370 3714}%
\special{pa 4372 3712}%
\special{pa 4374 3712}%
\special{pa 4380 3706}%
\special{pa 4384 3704}%
\special{pa 4388 3700}%
\special{pa 4390 3700}%
\special{pa 4392 3696}%
\special{pa 4394 3696}%
\special{pa 4395 3695}%
\special{fp}%
\special{pa 4443 3653}%
\special{pa 4468 3628}%
\special{pa 4468 3626}%
\special{pa 4476 3618}%
\special{pa 4476 3616}%
\special{pa 4485 3607}%
\special{fp}%
\special{pa 4522 3559}%
\special{pa 4522 3558}%
\special{pa 4526 3554}%
\special{pa 4526 3552}%
\special{pa 4530 3548}%
\special{pa 4530 3546}%
\special{pa 4534 3542}%
\special{pa 4534 3540}%
\special{pa 4536 3538}%
\special{pa 4536 3536}%
\special{pa 4538 3534}%
\special{pa 4538 3532}%
\special{pa 4544 3526}%
\special{pa 4544 3522}%
\special{pa 4548 3518}%
\special{pa 4548 3516}%
\special{pa 4550 3514}%
\special{pa 4552 3510}%
\special{pa 4552 3508}%
\special{pa 4552 3508}%
\special{fp}%
\special{pa 4580 3454}%
\special{pa 4580 3454}%
\special{pa 4580 3452}%
\special{pa 4582 3452}%
\special{pa 4582 3448}%
\special{pa 4584 3446}%
\special{pa 4584 3444}%
\special{pa 4586 3440}%
\special{pa 4586 3438}%
\special{pa 4588 3436}%
\special{pa 4588 3432}%
\special{pa 4590 3432}%
\special{pa 4590 3430}%
\special{pa 4592 3424}%
\special{pa 4592 3422}%
\special{pa 4594 3422}%
\special{pa 4594 3416}%
\special{pa 4596 3416}%
\special{pa 4596 3414}%
\special{pa 4600 3402}%
\special{fp}%
}}%
% LINE 3 0 3 0 Black White
% 16 3920 2200 3000 3120 4000 2240 3040 3200 3420 2940 3160 3200 3370 3110 3280 3200 4000 2360 3740 2620 4000 2480 3910 2570 3770 2230 3030 2970 3590 2290 3090 2790
% 
{\color[named]{Black}{%
\special{pn 4}%
\special{pa 3920 2200}%
\special{pa 3000 3120}%
\special{fp}%
\special{pa 4000 2240}%
\special{pa 3040 3200}%
\special{fp}%
\special{pa 3420 2940}%
\special{pa 3160 3200}%
\special{fp}%
\special{pa 3370 3110}%
\special{pa 3280 3200}%
\special{fp}%
\special{pa 4000 2360}%
\special{pa 3740 2620}%
\special{fp}%
\special{pa 4000 2480}%
\special{pa 3910 2570}%
\special{fp}%
\special{pa 3770 2230}%
\special{pa 3030 2970}%
\special{fp}%
\special{pa 3590 2290}%
\special{pa 3090 2790}%
\special{fp}%
}}%
% LINE 3 0 3 0 Black White
% 34 4860 2700 4590 2970 4900 2780 4620 3060 4940 2860 4630 3170 4970 2950 4720 3200 4990 3050 4840 3200 5000 3160 4960 3200 4810 2630 4550 2890 4760 2560 4510 2810 4700 2500 4450 2750 4640 2440 4390 2690 4570 2390 4310 2650 4500 2340 4230 2610 4420 2300 4140 2580 4340 2260 4030 2570 4250 2230 4000 2480 4150 2210 4000 2360 4040 2200 4000 2240
% 
{\color[named]{Black}{%
\special{pn 4}%
\special{pa 4860 2700}%
\special{pa 4590 2970}%
\special{fp}%
\special{pa 4900 2780}%
\special{pa 4620 3060}%
\special{fp}%
\special{pa 4940 2860}%
\special{pa 4630 3170}%
\special{fp}%
\special{pa 4970 2950}%
\special{pa 4720 3200}%
\special{fp}%
\special{pa 4990 3050}%
\special{pa 4840 3200}%
\special{fp}%
\special{pa 5000 3160}%
\special{pa 4960 3200}%
\special{fp}%
\special{pa 4810 2630}%
\special{pa 4550 2890}%
\special{fp}%
\special{pa 4760 2560}%
\special{pa 4510 2810}%
\special{fp}%
\special{pa 4700 2500}%
\special{pa 4450 2750}%
\special{fp}%
\special{pa 4640 2440}%
\special{pa 4390 2690}%
\special{fp}%
\special{pa 4570 2390}%
\special{pa 4310 2650}%
\special{fp}%
\special{pa 4500 2340}%
\special{pa 4230 2610}%
\special{fp}%
\special{pa 4420 2300}%
\special{pa 4140 2580}%
\special{fp}%
\special{pa 4340 2260}%
\special{pa 4030 2570}%
\special{fp}%
\special{pa 4250 2230}%
\special{pa 4000 2480}%
\special{fp}%
\special{pa 4150 2210}%
\special{pa 4000 2360}%
\special{fp}%
\special{pa 4040 2200}%
\special{pa 4000 2240}%
\special{fp}%
}}%
% LINE 3 0 3 0 Black White
% 34 3970 3830 3660 4140 4000 3920 3750 4170 4000 4040 3850 4190 4000 4160 3960 4200 3860 3820 3580 4100 3770 3790 3500 4060 3690 3750 3430 4010 3610 3710 3360 3960 3550 3650 3300 3900 3490 3590 3240 3840 3450 3510 3190 3770 3410 3430 3140 3700 3380 3340 3100 3620 3370 3230 3060 3540 3280 3200 3030 3450 3160 3200 3010 3350 3040 3200 3000 3240
% 
{\color[named]{Black}{%
\special{pn 4}%
\special{pa 3970 3830}%
\special{pa 3660 4140}%
\special{fp}%
\special{pa 4000 3920}%
\special{pa 3750 4170}%
\special{fp}%
\special{pa 4000 4040}%
\special{pa 3850 4190}%
\special{fp}%
\special{pa 4000 4160}%
\special{pa 3960 4200}%
\special{fp}%
\special{pa 3860 3820}%
\special{pa 3580 4100}%
\special{fp}%
\special{pa 3770 3790}%
\special{pa 3500 4060}%
\special{fp}%
\special{pa 3690 3750}%
\special{pa 3430 4010}%
\special{fp}%
\special{pa 3610 3710}%
\special{pa 3360 3960}%
\special{fp}%
\special{pa 3550 3650}%
\special{pa 3300 3900}%
\special{fp}%
\special{pa 3490 3590}%
\special{pa 3240 3840}%
\special{fp}%
\special{pa 3450 3510}%
\special{pa 3190 3770}%
\special{fp}%
\special{pa 3410 3430}%
\special{pa 3140 3700}%
\special{fp}%
\special{pa 3380 3340}%
\special{pa 3100 3620}%
\special{fp}%
\special{pa 3370 3230}%
\special{pa 3060 3540}%
\special{fp}%
\special{pa 3280 3200}%
\special{pa 3030 3450}%
\special{fp}%
\special{pa 3160 3200}%
\special{pa 3010 3350}%
\special{fp}%
\special{pa 3040 3200}%
\special{pa 3000 3240}%
\special{fp}%
}}%
% LINE 3 0 3 0 Black White
% 16 4970 3430 4230 4170 5000 3280 4080 4200 4960 3200 4000 4160 4260 3780 4000 4040 4090 3830 4000 3920 4840 3200 4580 3460 4720 3200 4630 3290 4910 3610 4410 4110
% 
{\color[named]{Black}{%
\special{pn 4}%
\special{pa 4970 3430}%
\special{pa 4230 4170}%
\special{fp}%
\special{pa 5000 3280}%
\special{pa 4080 4200}%
\special{fp}%
\special{pa 4960 3200}%
\special{pa 4000 4160}%
\special{fp}%
\special{pa 4260 3780}%
\special{pa 4000 4040}%
\special{fp}%
\special{pa 4090 3830}%
\special{pa 4000 3920}%
\special{fp}%
\special{pa 4840 3200}%
\special{pa 4580 3460}%
\special{fp}%
\special{pa 4720 3200}%
\special{pa 4630 3290}%
\special{fp}%
\special{pa 4910 3610}%
\special{pa 4410 4110}%
\special{fp}%
}}%
% STR 2 0 3 0 Black White
% 4 5020 3140 5020 3240 1 0 0 0
% $1+\frac{1}{t_i}$
\put(50.2000,-32.4000){\makebox(0,0)[lt]{$1+\frac{1}{t_i}$}}%
% STR 2 0 3 0 Black White
% 4 3970 4130 3970 4230 4 0 0 0
% $-1-\frac{1}{t_i}$
\put(39.7000,-42.3000){\makebox(0,0)[rt]{$-1-\frac{1}{t_i}$}}%
% STR 2 0 3 0 Black White
% 4 2970 3080 2970 3180 3 0 0 0
% $-1-\frac{1}{t_i}$
\put(29.7000,-31.8000){\makebox(0,0)[rb]{$-1-\frac{1}{t_i}$}}%
% STR 2 0 3 0 Black White
% 4 4040 2070 4040 2170 2 0 0 0
% $1+\frac{1}{t_i}$
\put(40.4000,-21.7000){\makebox(0,0)[lb]{$1+\frac{1}{t_i}$}}%
% STR 2 0 3 0 Black White
% 4 4600 3130 4600 3230 4 0 0 0
% $1$
\put(46.0000,-32.3000){\makebox(0,0)[rt]{$1$}}%
% STR 2 0 3 0 Black White
% 4 3400 3090 3400 3190 2 0 0 0
% $-1$
\put(34.0000,-31.9000){\makebox(0,0)[lb]{$-1$}}%
% STR 2 0 3 0 Black White
% 4 4010 2530 4010 2630 1 0 0 0
% $1$
\put(40.1000,-26.3000){\makebox(0,0)[lt]{$1$}}%
% STR 2 0 3 0 Black White
% 4 3990 3680 3990 3780 3 0 0 0
% $-1$
\put(39.9000,-37.8000){\makebox(0,0)[rb]{$-1$}}%
% STR 2 0 3 0 Black White
% 4 5360 2110 5360 2210 2 0 0 0
% $D(t_i)$
\put(53.6000,-22.1000){\makebox(0,0)[lb]{$D(t_i)$}}%
% STR 2 0 3 0 Black White
% 4 2820 2270 2820 2370 5 0 1 0
% $A$
\put(28.2000,-23.7000){\makebox(0,0){{\colorbox[named]{White}{$A$}}}}%
\put(28.2000,-23.7000){\makebox(0,0){{\colorbox[named]{White}{\color[named]{Black}{$A$}}}}}%
% LINE 3 0 3 0 Black White
% 2 3480 2850 2880 2420
% 
{\color[named]{Black}{%
\special{pn 4}%
\special{pa 3480 2850}%
\special{pa 2880 2420}%
\special{fp}%
}}%
% LINE 3 0 3 0 Black White
% 2 4750 2530 5350 2200
% 
{\color[named]{Black}{%
\special{pn 4}%
\special{pa 4750 2530}%
\special{pa 5350 2200}%
\special{fp}%
}}%
% STR 2 0 3 0 Black White
% 4 5170 2450 5170 2550 1 0 1 0
% $(x_i,y_i)$
\put(51.7000,-25.5000){\makebox(0,0)[lt]{{\colorbox[named]{White}{$(x_i,y_i)$}}}}%
\put(51.7000,-25.5000){\makebox(0,0)[lt]{{\colorbox[named]{White}{\color[named]{Black}{$(x_i,y_i)$}}}}}%
% DOT 0 0 3 0 Black White
% 2 4550 2890 4550 2890
% 
{\color[named]{Black}{%
\special{pn 4}%
\special{sh 1}%
\special{ar 4550 2890 16 16 0  6.28318530717959E+0000}%
\special{sh 1}%
\special{ar 4550 2890 16 16 0  6.28318530717959E+0000}%
}}%
% LINE 3 0 3 0 Black White
% 2 4550 2890 5180 2680
% 
{\color[named]{Black}{%
\special{pn 4}%
\special{pa 4550 2890}%
\special{pa 5180 2680}%
\special{fp}%
}}%
\end{picture}}%

%% file: thm2_2rem.tex
%WinTpicVersion4.28b
{\unitlength 0.1in
\begin{picture}( 57.0000, 18.9000)( 14.0000,-29.7000)
% STR 2 0 3 0 Black White
% 4 7100 1870 7100 1970 2 0 0 0
% 
\put(71.0000,-19.7000){\makebox(0,0)[lb]{}}%
% STR 2 0 3 0 Black White
% 4 2200 1700 2200 1800 2 0 0 0
% 
\put(22.0000,-18.0000){\makebox(0,0)[lb]{}}%
% CIRCLE 2 0 3 0 Black White
% 4 2400 2000 2400 1400 2400 1400 2400 1400
% 
{\color[named]{Black}{%
\special{pn 8}%
\special{ar 2400 2000 600 600  0.0000000  6.2831853}%
}}%
% CIRCLE 2 0 3 0 Black White
% 4 3200 2000 2600 2000 2600 2000 2600 2000
% 
{\color[named]{Black}{%
\special{pn 8}%
\special{ar 3200 2000 600 600  0.0000000  6.2831853}%
}}%
% LINE 3 0 3 0 Black White
% 24 2430 1400 2430 2600 2520 1410 2520 2590 2610 2130 2610 2560 2700 2330 2700 2520 2610 1440 2610 1870 2700 1480 2700 1670 2340 1410 2340 2590 2250 1420 2250 2580 2160 1450 2160 2550 2070 1500 2070 2500 1980 1570 1980 2430 1890 1690 1890 2310
% 
{\color[named]{Black}{%
\special{pn 4}%
\special{pa 2430 1400}%
\special{pa 2430 2600}%
\special{fp}%
\special{pa 2520 1410}%
\special{pa 2520 2590}%
\special{fp}%
\special{pa 2610 2130}%
\special{pa 2610 2560}%
\special{fp}%
\special{pa 2700 2330}%
\special{pa 2700 2520}%
\special{fp}%
\special{pa 2610 1440}%
\special{pa 2610 1870}%
\special{fp}%
\special{pa 2700 1480}%
\special{pa 2700 1670}%
\special{fp}%
\special{pa 2340 1410}%
\special{pa 2340 2590}%
\special{fp}%
\special{pa 2250 1420}%
\special{pa 2250 2580}%
\special{fp}%
\special{pa 2160 1450}%
\special{pa 2160 2550}%
\special{fp}%
\special{pa 2070 1500}%
\special{pa 2070 2500}%
\special{fp}%
\special{pa 1980 1570}%
\special{pa 1980 2430}%
\special{fp}%
\special{pa 1890 1690}%
\special{pa 1890 2310}%
\special{fp}%
}}%
% LINE 3 0 3 0 Black White
% 26 3420 1440 3420 2560 3330 1420 3330 2580 3240 1400 3240 2600 3150 1400 3150 2590 3060 1420 3060 2580 2970 2190 2970 2550 2880 2360 2880 2510 2970 1450 2970 1810 2880 1490 2880 1640 3510 1490 3510 2510 3600 1560 3600 2440 3690 1660 3690 2340 3780 1860 3780 2140
% 
{\color[named]{Black}{%
\special{pn 4}%
\special{pa 3420 1440}%
\special{pa 3420 2560}%
\special{fp}%
\special{pa 3330 1420}%
\special{pa 3330 2580}%
\special{fp}%
\special{pa 3240 1400}%
\special{pa 3240 2600}%
\special{fp}%
\special{pa 3150 1400}%
\special{pa 3150 2590}%
\special{fp}%
\special{pa 3060 1420}%
\special{pa 3060 2580}%
\special{fp}%
\special{pa 2970 2190}%
\special{pa 2970 2550}%
\special{fp}%
\special{pa 2880 2360}%
\special{pa 2880 2510}%
\special{fp}%
\special{pa 2970 1450}%
\special{pa 2970 1810}%
\special{fp}%
\special{pa 2880 1490}%
\special{pa 2880 1640}%
\special{fp}%
\special{pa 3510 1490}%
\special{pa 3510 2510}%
\special{fp}%
\special{pa 3600 1560}%
\special{pa 3600 2440}%
\special{fp}%
\special{pa 3690 1660}%
\special{pa 3690 2340}%
\special{fp}%
\special{pa 3780 1860}%
\special{pa 3780 2140}%
\special{fp}%
}}%
% LINE 3 0 3 0 Black White
% 18 2990 2070 2610 2070 3000 1980 2600 1980 2990 1890 2610 1890 2960 1800 2640 1800 2920 1710 2680 1710 2860 1620 2740 1620 2980 2160 2620 2160 2940 2250 2660 2250 2890 2340 2710 2340
% 
{\color[named]{Black}{%
\special{pn 4}%
\special{pa 2990 2070}%
\special{pa 2610 2070}%
\special{fp}%
\special{pa 3000 1980}%
\special{pa 2600 1980}%
\special{fp}%
\special{pa 2990 1890}%
\special{pa 2610 1890}%
\special{fp}%
\special{pa 2960 1800}%
\special{pa 2640 1800}%
\special{fp}%
\special{pa 2920 1710}%
\special{pa 2680 1710}%
\special{fp}%
\special{pa 2860 1620}%
\special{pa 2740 1620}%
\special{fp}%
\special{pa 2980 2160}%
\special{pa 2620 2160}%
\special{fp}%
\special{pa 2940 2250}%
\special{pa 2660 2250}%
\special{fp}%
\special{pa 2890 2340}%
\special{pa 2710 2340}%
\special{fp}%
}}%
% LINE 3 0 3 0 Black White
% 76 4200 1350 1400 1350 2180 1440 1400 1440 2030 1530 1400 1530 1940 1620 1400 1620 1880 1710 1400 1710 1840 1800 1400 1800 1810 1890 1400 1890 1800 1980 1400 1980 1810 2070 1400 2070 1820 2160 1400 2160 1860 2250 1400 2250 1910 2340 1400 2340 1980 2430 1400 2430 2100 2520 1400 2520 4200 2610 1400 2610 4200 2700 1400 2700 4200 2790 1400 2790 4200 2880 1400 2880 4200 2970 1400 2970 4200 2520 3500 2520 4200 2430 3620 2430 4200 2340 3690 2340 4200 2250 3740 2250 4200 2160 3780 2160 4200 2070 3790 2070 4200 1980 3800 1980 4200 1890 3790 1890 4200 1800 3760 1800 4200 1710 3720 1710 4200 1620 3660 1620 4200 1530 3570 1530 4200 1440 3420 1440 2900 2520 2700 2520 4200 1260 1400 1260 4200 1170 1400 1170 4200 1080 1400 1080 2980 1440 2620 1440 2830 1530 2770 1530
% 
{\color[named]{Black}{%
\special{pn 4}%
\special{pa 4200 1350}%
\special{pa 1400 1350}%
\special{fp}%
\special{pa 2180 1440}%
\special{pa 1400 1440}%
\special{fp}%
\special{pa 2030 1530}%
\special{pa 1400 1530}%
\special{fp}%
\special{pa 1940 1620}%
\special{pa 1400 1620}%
\special{fp}%
\special{pa 1880 1710}%
\special{pa 1400 1710}%
\special{fp}%
\special{pa 1840 1800}%
\special{pa 1400 1800}%
\special{fp}%
\special{pa 1810 1890}%
\special{pa 1400 1890}%
\special{fp}%
\special{pa 1800 1980}%
\special{pa 1400 1980}%
\special{fp}%
\special{pa 1810 2070}%
\special{pa 1400 2070}%
\special{fp}%
\special{pa 1820 2160}%
\special{pa 1400 2160}%
\special{fp}%
\special{pa 1860 2250}%
\special{pa 1400 2250}%
\special{fp}%
\special{pa 1910 2340}%
\special{pa 1400 2340}%
\special{fp}%
\special{pa 1980 2430}%
\special{pa 1400 2430}%
\special{fp}%
\special{pa 2100 2520}%
\special{pa 1400 2520}%
\special{fp}%
\special{pa 4200 2610}%
\special{pa 1400 2610}%
\special{fp}%
\special{pa 4200 2700}%
\special{pa 1400 2700}%
\special{fp}%
\special{pa 4200 2790}%
\special{pa 1400 2790}%
\special{fp}%
\special{pa 4200 2880}%
\special{pa 1400 2880}%
\special{fp}%
\special{pa 4200 2970}%
\special{pa 1400 2970}%
\special{fp}%
\special{pa 4200 2520}%
\special{pa 3500 2520}%
\special{fp}%
\special{pa 4200 2430}%
\special{pa 3620 2430}%
\special{fp}%
\special{pa 4200 2340}%
\special{pa 3690 2340}%
\special{fp}%
\special{pa 4200 2250}%
\special{pa 3740 2250}%
\special{fp}%
\special{pa 4200 2160}%
\special{pa 3780 2160}%
\special{fp}%
\special{pa 4200 2070}%
\special{pa 3790 2070}%
\special{fp}%
\special{pa 4200 1980}%
\special{pa 3800 1980}%
\special{fp}%
\special{pa 4200 1890}%
\special{pa 3790 1890}%
\special{fp}%
\special{pa 4200 1800}%
\special{pa 3760 1800}%
\special{fp}%
\special{pa 4200 1710}%
\special{pa 3720 1710}%
\special{fp}%
\special{pa 4200 1620}%
\special{pa 3660 1620}%
\special{fp}%
\special{pa 4200 1530}%
\special{pa 3570 1530}%
\special{fp}%
\special{pa 4200 1440}%
\special{pa 3420 1440}%
\special{fp}%
\special{pa 2900 2520}%
\special{pa 2700 2520}%
\special{fp}%
\special{pa 4200 1260}%
\special{pa 1400 1260}%
\special{fp}%
\special{pa 4200 1170}%
\special{pa 1400 1170}%
\special{fp}%
\special{pa 4200 1080}%
\special{pa 1400 1080}%
\special{fp}%
\special{pa 2980 1440}%
\special{pa 2620 1440}%
\special{fp}%
\special{pa 2830 1530}%
\special{pa 2770 1530}%
\special{fp}%
}}%
% STR 2 0 3 0 Black White
% 4 2400 1300 2400 1400 5 0 1 0
% $A(t)$
\put(24.0000,-14.0000){\makebox(0,0){{\colorbox[named]{White}{$A(t)$}}}}%
\put(24.0000,-14.0000){\makebox(0,0){{\colorbox[named]{White}{\color[named]{Black}{$A(t)$}}}}}%
% STR 2 0 3 0 Black White
% 4 3200 1300 3200 1400 5 0 1 0
% $A$
\put(32.0000,-14.0000){\makebox(0,0){{\colorbox[named]{White}{$A$}}}}%
\put(32.0000,-14.0000){\makebox(0,0){{\colorbox[named]{White}{\color[named]{Black}{$A$}}}}}%
\end{picture}}%

%% file: sec3_2_exA.tex
%WinTpicVersion4.28b
{\unitlength 0.1in
\begin{picture}( 28.1500, 24.0000)( 25.5500,-44.0000)
% STR 2 0 3 0 Black White
% 4 3990 3207 3990 3220 4 2800 0 0
% O
\put(39.9000,-32.2000){\makebox(0,0)[rt]{O}}%
% STR 2 0 3 0 Black White
% 4 3960 1987 3960 2000 4 2800 0 0
% $y$
\put(39.6000,-20.0000){\makebox(0,0)[rt]{$y$}}%
% STR 2 0 3 0 Black White
% 4 5210 3247 5210 3260 4 2800 0 0
% $x$
\put(52.1000,-32.6000){\makebox(0,0)[rt]{$x$}}%
% VECTOR 2 0 3 0 Black White
% 2 4000 4400 4000 2000
% 
{\color[named]{Black}{%
\special{pn 8}%
\special{pa 4000 4400}%
\special{pa 4000 2000}%
\special{fp}%
\special{sh 1}%
\special{pa 4000 2000}%
\special{pa 3980 2068}%
\special{pa 4000 2054}%
\special{pa 4020 2068}%
\special{pa 4000 2000}%
\special{fp}%
}}%
% VECTOR 2 0 3 0 Black White
% 2 2800 3200 5200 3200
% 
{\color[named]{Black}{%
\special{pn 8}%
\special{pa 2800 3200}%
\special{pa 5200 3200}%
\special{fp}%
\special{sh 1}%
\special{pa 5200 3200}%
\special{pa 5134 3180}%
\special{pa 5148 3200}%
\special{pa 5134 3220}%
\special{pa 5200 3200}%
\special{fp}%
}}%
% FUNC 2 1 3 0 Black White
% 9 2800 2000 5200 4400 4000 3200 5000 3200 4000 2200 2200 1600 2200 1600 10 2 0 4
% ///x^2+y^2=1///-1.199///1.199
{\color[named]{Black}{%
\special{pn 8}%
\special{pa 5000 3200}%
\special{pa 5000 3146}%
\special{pa 4998 3146}%
\special{pa 4998 3139}%
\special{fp}%
\special{pa 4993 3078}%
\special{pa 4992 3078}%
\special{pa 4992 3070}%
\special{pa 4990 3056}%
\special{pa 4990 3050}%
\special{pa 4988 3048}%
\special{pa 4988 3038}%
\special{pa 4986 3036}%
\special{pa 4986 3026}%
\special{pa 4984 3024}%
\special{pa 4984 3019}%
\special{fp}%
\special{pa 4972 2966}%
\special{pa 4972 2962}%
\special{pa 4970 2956}%
\special{pa 4970 2952}%
\special{pa 4968 2950}%
\special{pa 4968 2948}%
\special{pa 4966 2940}%
\special{pa 4966 2936}%
\special{pa 4964 2936}%
\special{pa 4964 2930}%
\special{pa 4962 2928}%
\special{pa 4962 2922}%
\special{pa 4960 2922}%
\special{pa 4960 2916}%
\special{pa 4958 2914}%
\special{pa 4958 2910}%
\special{pa 4958 2910}%
\special{fp}%
\special{pa 4941 2861}%
\special{pa 4940 2860}%
\special{pa 4940 2856}%
\special{pa 4938 2854}%
\special{pa 4938 2852}%
\special{pa 4936 2848}%
\special{pa 4936 2846}%
\special{pa 4934 2844}%
\special{pa 4934 2840}%
\special{pa 4932 2838}%
\special{pa 4932 2834}%
\special{pa 4930 2834}%
\special{pa 4930 2830}%
\special{pa 4928 2828}%
\special{pa 4928 2824}%
\special{pa 4926 2824}%
\special{pa 4926 2820}%
\special{pa 4924 2818}%
\special{pa 4924 2816}%
\special{pa 4922 2814}%
\special{pa 4922 2810}%
\special{pa 4920 2810}%
\special{pa 4920 2809}%
\special{fp}%
\special{pa 4897 2757}%
\special{pa 4896 2756}%
\special{pa 4896 2754}%
\special{pa 4894 2752}%
\special{pa 4894 2750}%
\special{pa 4890 2746}%
\special{pa 4890 2742}%
\special{pa 4888 2742}%
\special{pa 4888 2738}%
\special{pa 4886 2738}%
\special{pa 4886 2734}%
\special{pa 4884 2734}%
\special{pa 4884 2730}%
\special{pa 4882 2730}%
\special{pa 4882 2726}%
\special{pa 4878 2722}%
\special{pa 4878 2720}%
\special{pa 4876 2718}%
\special{pa 4876 2716}%
\special{pa 4874 2714}%
\special{pa 4872 2710}%
\special{pa 4872 2708}%
\special{fp}%
\special{pa 4839 2657}%
\special{pa 4838 2656}%
\special{pa 4838 2654}%
\special{pa 4836 2652}%
\special{pa 4836 2650}%
\special{pa 4834 2648}%
\special{pa 4834 2646}%
\special{pa 4826 2638}%
\special{pa 4826 2636}%
\special{pa 4822 2632}%
\special{pa 4822 2630}%
\special{pa 4820 2628}%
\special{pa 4818 2624}%
\special{pa 4816 2622}%
\special{pa 4816 2620}%
\special{pa 4812 2618}%
\special{pa 4812 2616}%
\special{pa 4810 2614}%
\special{pa 4808 2610}%
\special{pa 4806 2608}%
\special{fp}%
\special{pa 4767 2559}%
\special{pa 4762 2554}%
\special{pa 4762 2552}%
\special{pa 4750 2540}%
\special{pa 4750 2538}%
\special{pa 4742 2530}%
\special{pa 4742 2528}%
\special{pa 4727 2513}%
\special{fp}%
\special{pa 4680 2468}%
\special{pa 4656 2444}%
\special{pa 4654 2444}%
\special{pa 4642 2432}%
\special{pa 4640 2432}%
\special{pa 4636 2428}%
\special{pa 4634 2428}%
\special{fp}%
\special{pa 4588 2391}%
\special{pa 4588 2390}%
\special{pa 4586 2390}%
\special{pa 4584 2388}%
\special{pa 4580 2386}%
\special{pa 4578 2384}%
\special{pa 4576 2384}%
\special{pa 4574 2380}%
\special{pa 4572 2380}%
\special{pa 4568 2376}%
\special{pa 4566 2376}%
\special{pa 4562 2372}%
\special{pa 4560 2372}%
\special{pa 4558 2370}%
\special{pa 4556 2370}%
\special{pa 4552 2366}%
\special{pa 4548 2364}%
\special{pa 4544 2360}%
\special{pa 4542 2360}%
\special{pa 4539 2357}%
\special{fp}%
\special{pa 4487 2327}%
\special{pa 4486 2326}%
\special{pa 4478 2322}%
\special{pa 4476 2322}%
\special{pa 4472 2318}%
\special{pa 4470 2318}%
\special{pa 4468 2316}%
\special{pa 4466 2316}%
\special{pa 4464 2314}%
\special{pa 4462 2314}%
\special{pa 4460 2312}%
\special{pa 4458 2312}%
\special{pa 4454 2308}%
\special{pa 4450 2308}%
\special{pa 4450 2306}%
\special{pa 4446 2306}%
\special{pa 4446 2304}%
\special{pa 4442 2304}%
\special{pa 4442 2302}%
\special{pa 4438 2302}%
\special{pa 4437 2301}%
\special{fp}%
\special{pa 4384 2277}%
\special{pa 4382 2276}%
\special{pa 4380 2276}%
\special{pa 4378 2274}%
\special{pa 4374 2274}%
\special{pa 4374 2272}%
\special{pa 4370 2272}%
\special{pa 4368 2270}%
\special{pa 4364 2270}%
\special{pa 4364 2268}%
\special{pa 4362 2268}%
\special{pa 4356 2266}%
\special{pa 4354 2266}%
\special{pa 4354 2264}%
\special{pa 4348 2264}%
\special{pa 4348 2262}%
\special{pa 4346 2262}%
\special{pa 4340 2260}%
\special{pa 4338 2260}%
\special{pa 4338 2258}%
\special{pa 4334 2258}%
\special{fp}%
\special{pa 4283 2242}%
\special{pa 4282 2242}%
\special{pa 4282 2240}%
\special{pa 4276 2240}%
\special{pa 4274 2238}%
\special{pa 4268 2238}%
\special{pa 4268 2236}%
\special{pa 4264 2236}%
\special{pa 4258 2234}%
\special{pa 4254 2234}%
\special{pa 4252 2232}%
\special{pa 4250 2232}%
\special{pa 4242 2230}%
\special{pa 4238 2230}%
\special{pa 4238 2228}%
\special{pa 4232 2228}%
\special{pa 4229 2227}%
\special{fp}%
\special{pa 4169 2215}%
\special{pa 4152 2212}%
\special{pa 4138 2210}%
\special{pa 4132 2210}%
\special{pa 4130 2208}%
\special{pa 4122 2208}%
\special{pa 4107 2206}%
\special{fp}%
\special{pa 4044 2202}%
\special{pa 4032 2202}%
\special{pa 4032 2200}%
\special{pa 3983 2200}%
\special{fp}%
\special{pa 3922 2204}%
\special{pa 3906 2204}%
\special{pa 3906 2206}%
\special{pa 3896 2206}%
\special{pa 3878 2208}%
\special{pa 3870 2208}%
\special{pa 3870 2210}%
\special{pa 3863 2210}%
\special{fp}%
\special{pa 3803 2220}%
\special{pa 3794 2222}%
\special{pa 3790 2222}%
\special{pa 3790 2224}%
\special{pa 3782 2224}%
\special{pa 3780 2226}%
\special{pa 3776 2226}%
\special{pa 3768 2228}%
\special{pa 3764 2228}%
\special{pa 3762 2230}%
\special{pa 3756 2230}%
\special{pa 3754 2232}%
\special{pa 3748 2232}%
\special{pa 3748 2234}%
\special{pa 3747 2234}%
\special{fp}%
\special{pa 3691 2250}%
\special{pa 3666 2258}%
\special{pa 3664 2258}%
\special{pa 3662 2260}%
\special{pa 3660 2260}%
\special{pa 3656 2262}%
\special{pa 3654 2262}%
\special{pa 3652 2264}%
\special{pa 3648 2264}%
\special{pa 3646 2266}%
\special{pa 3644 2266}%
\special{pa 3640 2268}%
\special{pa 3638 2268}%
\special{pa 3636 2270}%
\special{pa 3634 2270}%
\special{fp}%
\special{pa 3583 2292}%
\special{pa 3580 2292}%
\special{pa 3580 2294}%
\special{pa 3576 2294}%
\special{pa 3576 2296}%
\special{pa 3572 2296}%
\special{pa 3572 2298}%
\special{pa 3568 2298}%
\special{pa 3568 2300}%
\special{pa 3564 2300}%
\special{pa 3564 2302}%
\special{pa 3560 2302}%
\special{pa 3558 2304}%
\special{pa 3556 2304}%
\special{pa 3554 2306}%
\special{pa 3552 2306}%
\special{pa 3550 2308}%
\special{pa 3548 2308}%
\special{pa 3546 2310}%
\special{pa 3535 2316}%
\special{fp}%
\special{pa 3484 2344}%
\special{pa 3482 2346}%
\special{pa 3480 2346}%
\special{pa 3478 2348}%
\special{pa 3476 2348}%
\special{pa 3470 2354}%
\special{pa 3468 2354}%
\special{pa 3466 2356}%
\special{pa 3464 2356}%
\special{pa 3462 2358}%
\special{pa 3460 2358}%
\special{pa 3456 2362}%
\special{pa 3454 2362}%
\special{pa 3448 2368}%
\special{pa 3446 2368}%
\special{pa 3444 2370}%
\special{pa 3440 2372}%
\special{pa 3438 2374}%
\special{pa 3436 2374}%
\special{pa 3434 2376}%
\special{fp}%
\special{pa 3384 2414}%
\special{pa 3382 2416}%
\special{pa 3380 2416}%
\special{pa 3372 2424}%
\special{pa 3370 2424}%
\special{pa 3366 2428}%
\special{pa 3360 2432}%
\special{pa 3354 2438}%
\special{pa 3352 2438}%
\special{pa 3340 2450}%
\special{pa 3338 2450}%
\special{pa 3336 2452}%
\special{fp}%
\special{pa 3291 2497}%
\special{pa 3260 2528}%
\special{pa 3260 2530}%
\special{pa 3250 2538}%
\special{pa 3250 2540}%
\special{pa 3248 2542}%
\special{fp}%
\special{pa 3208 2590}%
\special{pa 3208 2590}%
\special{pa 3206 2594}%
\special{pa 3202 2598}%
\special{pa 3200 2602}%
\special{pa 3196 2606}%
\special{pa 3194 2610}%
\special{pa 3188 2616}%
\special{pa 3188 2618}%
\special{pa 3186 2620}%
\special{pa 3186 2622}%
\special{pa 3178 2630}%
\special{pa 3178 2632}%
\special{pa 3174 2636}%
\special{pa 3174 2638}%
\special{pa 3173 2639}%
\special{fp}%
\special{pa 3140 2691}%
\special{pa 3140 2692}%
\special{pa 3138 2694}%
\special{pa 3136 2698}%
\special{pa 3136 2700}%
\special{pa 3132 2704}%
\special{pa 3132 2706}%
\special{pa 3130 2708}%
\special{pa 3128 2712}%
\special{pa 3128 2714}%
\special{pa 3124 2718}%
\special{pa 3124 2720}%
\special{pa 3122 2722}%
\special{pa 3122 2724}%
\special{pa 3120 2726}%
\special{pa 3111 2744}%
\special{fp}%
\special{pa 3086 2797}%
\special{pa 3086 2798}%
\special{pa 3084 2798}%
\special{pa 3084 2802}%
\special{pa 3082 2804}%
\special{pa 3082 2806}%
\special{pa 3080 2808}%
\special{pa 3080 2812}%
\special{pa 3078 2812}%
\special{pa 3078 2816}%
\special{pa 3076 2818}%
\special{pa 3076 2822}%
\special{pa 3074 2822}%
\special{pa 3074 2826}%
\special{pa 3072 2828}%
\special{pa 3072 2832}%
\special{pa 3070 2832}%
\special{pa 3070 2836}%
\special{pa 3068 2838}%
\special{pa 3068 2840}%
\special{pa 3066 2844}%
\special{pa 3066 2846}%
\special{pa 3065 2847}%
\special{fp}%
\special{pa 3046 2904}%
\special{pa 3046 2906}%
\special{pa 3044 2906}%
\special{pa 3044 2912}%
\special{pa 3042 2912}%
\special{pa 3042 2918}%
\special{pa 3040 2920}%
\special{pa 3040 2926}%
\special{pa 3038 2926}%
\special{pa 3038 2932}%
\special{pa 3036 2934}%
\special{pa 3036 2936}%
\special{pa 3034 2944}%
\special{pa 3034 2948}%
\special{pa 3032 2948}%
\special{pa 3032 2952}%
\special{pa 3030 2957}%
\special{fp}%
\special{pa 3018 3017}%
\special{pa 3018 3020}%
\special{pa 3016 3020}%
\special{pa 3016 3026}%
\special{pa 3014 3036}%
\special{pa 3014 3038}%
\special{pa 3012 3048}%
\special{pa 3012 3050}%
\special{pa 3010 3062}%
\special{pa 3010 3070}%
\special{pa 3008 3070}%
\special{pa 3008 3076}%
\special{fp}%
\special{pa 3002 3136}%
\special{pa 3002 3168}%
\special{pa 3000 3170}%
\special{pa 3000 3198}%
\special{fp}%
\special{pa 3002 3260}%
\special{pa 3002 3270}%
\special{pa 3004 3272}%
\special{pa 3004 3294}%
\special{pa 3006 3296}%
\special{pa 3006 3306}%
\special{pa 3008 3322}%
\special{fp}%
\special{pa 3018 3383}%
\special{pa 3018 3388}%
\special{pa 3020 3396}%
\special{pa 3020 3398}%
\special{pa 3022 3406}%
\special{pa 3022 3410}%
\special{pa 3024 3412}%
\special{pa 3024 3420}%
\special{pa 3026 3420}%
\special{pa 3026 3426}%
\special{pa 3028 3432}%
\special{pa 3028 3438}%
\special{pa 3030 3438}%
\special{pa 3030 3440}%
\special{fp}%
\special{pa 3045 3495}%
\special{pa 3046 3496}%
\special{pa 3046 3500}%
\special{pa 3048 3504}%
\special{pa 3048 3506}%
\special{pa 3050 3510}%
\special{pa 3050 3512}%
\special{pa 3052 3516}%
\special{pa 3052 3518}%
\special{pa 3054 3522}%
\special{pa 3054 3524}%
\special{pa 3056 3528}%
\special{pa 3056 3530}%
\special{pa 3058 3534}%
\special{pa 3058 3538}%
\special{pa 3060 3538}%
\special{pa 3060 3540}%
\special{pa 3062 3546}%
\special{pa 3062 3548}%
\special{pa 3064 3548}%
\special{pa 3064 3551}%
\special{fp}%
\special{pa 3084 3602}%
\special{pa 3086 3604}%
\special{pa 3086 3606}%
\special{pa 3088 3608}%
\special{pa 3088 3610}%
\special{pa 3092 3618}%
\special{pa 3092 3620}%
\special{pa 3094 3622}%
\special{pa 3094 3624}%
\special{pa 3096 3626}%
\special{pa 3096 3628}%
\special{pa 3098 3630}%
\special{pa 3098 3632}%
\special{pa 3100 3634}%
\special{pa 3100 3636}%
\special{pa 3102 3638}%
\special{pa 3102 3642}%
\special{pa 3104 3642}%
\special{pa 3104 3646}%
\special{pa 3106 3646}%
\special{pa 3106 3650}%
\special{pa 3108 3650}%
\special{pa 3108 3652}%
\special{fp}%
\special{pa 3136 3702}%
\special{pa 3136 3702}%
\special{pa 3136 3704}%
\special{pa 3140 3708}%
\special{pa 3140 3712}%
\special{pa 3146 3718}%
\special{pa 3146 3722}%
\special{pa 3150 3726}%
\special{pa 3152 3730}%
\special{pa 3156 3734}%
\special{pa 3156 3738}%
\special{pa 3160 3742}%
\special{pa 3160 3744}%
\special{pa 3164 3748}%
\special{pa 3166 3752}%
\special{pa 3167 3753}%
\special{fp}%
\special{pa 3201 3801}%
\special{pa 3202 3802}%
\special{pa 3202 3804}%
\special{pa 3206 3806}%
\special{pa 3206 3808}%
\special{pa 3208 3810}%
\special{pa 3208 3812}%
\special{pa 3212 3814}%
\special{pa 3212 3816}%
\special{pa 3220 3824}%
\special{pa 3220 3826}%
\special{pa 3228 3834}%
\special{pa 3228 3836}%
\special{pa 3232 3840}%
\special{pa 3232 3842}%
\special{pa 3237 3847}%
\special{fp}%
\special{pa 3280 3894}%
\special{pa 3324 3938}%
\special{fp}%
\special{pa 3372 3978}%
\special{pa 3380 3986}%
\special{pa 3382 3986}%
\special{pa 3386 3990}%
\special{pa 3390 3992}%
\special{pa 3394 3996}%
\special{pa 3398 3998}%
\special{pa 3402 4002}%
\special{pa 3406 4004}%
\special{pa 3410 4008}%
\special{pa 3414 4010}%
\special{pa 3416 4012}%
\special{pa 3418 4012}%
\special{pa 3420 4016}%
\special{pa 3420 4016}%
\special{fp}%
\special{pa 3471 4049}%
\special{pa 3472 4050}%
\special{pa 3474 4050}%
\special{pa 3476 4052}%
\special{pa 3478 4052}%
\special{pa 3484 4058}%
\special{pa 3488 4058}%
\special{pa 3494 4064}%
\special{pa 3498 4064}%
\special{pa 3498 4066}%
\special{pa 3500 4066}%
\special{pa 3502 4068}%
\special{pa 3504 4068}%
\special{pa 3508 4072}%
\special{pa 3512 4072}%
\special{pa 3512 4074}%
\special{pa 3514 4074}%
\special{pa 3516 4076}%
\special{pa 3518 4076}%
\special{pa 3520 4078}%
\special{fp}%
\special{pa 3567 4102}%
\special{pa 3570 4102}%
\special{pa 3570 4104}%
\special{pa 3574 4104}%
\special{pa 3574 4106}%
\special{pa 3578 4106}%
\special{pa 3578 4108}%
\special{pa 3582 4108}%
\special{pa 3584 4110}%
\special{pa 3586 4110}%
\special{pa 3588 4112}%
\special{pa 3590 4112}%
\special{pa 3592 4114}%
\special{pa 3594 4114}%
\special{pa 3598 4116}%
\special{pa 3600 4116}%
\special{pa 3600 4118}%
\special{pa 3604 4118}%
\special{pa 3606 4120}%
\special{pa 3610 4120}%
\special{pa 3610 4122}%
\special{pa 3614 4122}%
\special{pa 3616 4124}%
\special{fp}%
\special{pa 3670 4144}%
\special{pa 3674 4146}%
\special{pa 3676 4146}%
\special{pa 3680 4148}%
\special{pa 3682 4148}%
\special{pa 3686 4150}%
\special{pa 3688 4150}%
\special{pa 3692 4152}%
\special{pa 3694 4152}%
\special{pa 3698 4154}%
\special{pa 3702 4154}%
\special{pa 3702 4156}%
\special{pa 3706 4156}%
\special{pa 3712 4158}%
\special{pa 3714 4158}%
\special{pa 3716 4160}%
\special{pa 3722 4160}%
\special{pa 3722 4162}%
\special{pa 3726 4162}%
\special{fp}%
\special{pa 3779 4176}%
\special{pa 3784 4176}%
\special{pa 3786 4178}%
\special{pa 3794 4178}%
\special{pa 3794 4180}%
\special{pa 3804 4180}%
\special{pa 3804 4182}%
\special{pa 3814 4182}%
\special{pa 3814 4184}%
\special{pa 3824 4184}%
\special{pa 3826 4186}%
\special{pa 3835 4186}%
\special{fp}%
\special{pa 3895 4194}%
\special{pa 3896 4194}%
\special{pa 3896 4196}%
\special{pa 3916 4196}%
\special{pa 3918 4198}%
\special{pa 3946 4198}%
\special{pa 3946 4200}%
\special{pa 3953 4200}%
\special{fp}%
\special{pa 4017 4200}%
\special{pa 4054 4200}%
\special{pa 4056 4198}%
\special{pa 4079 4198}%
\special{fp}%
\special{pa 4140 4191}%
\special{pa 4144 4190}%
\special{pa 4152 4190}%
\special{pa 4152 4188}%
\special{pa 4164 4188}%
\special{pa 4164 4186}%
\special{pa 4176 4186}%
\special{pa 4176 4184}%
\special{pa 4186 4184}%
\special{pa 4188 4182}%
\special{pa 4196 4182}%
\special{pa 4196 4182}%
\special{fp}%
\special{pa 4253 4168}%
\special{pa 4254 4168}%
\special{pa 4260 4166}%
\special{pa 4264 4166}%
\special{pa 4266 4164}%
\special{pa 4272 4164}%
\special{pa 4272 4162}%
\special{pa 4278 4162}%
\special{pa 4280 4160}%
\special{pa 4286 4160}%
\special{pa 4286 4158}%
\special{pa 4292 4158}%
\special{pa 4292 4156}%
\special{pa 4298 4156}%
\special{pa 4300 4154}%
\special{pa 4304 4154}%
\special{pa 4306 4152}%
\special{pa 4306 4152}%
\special{fp}%
\special{pa 4360 4134}%
\special{pa 4362 4134}%
\special{pa 4362 4132}%
\special{pa 4366 4132}%
\special{pa 4368 4130}%
\special{pa 4372 4130}%
\special{pa 4372 4128}%
\special{pa 4376 4128}%
\special{pa 4378 4126}%
\special{pa 4382 4126}%
\special{pa 4382 4124}%
\special{pa 4386 4124}%
\special{pa 4386 4122}%
\special{pa 4390 4122}%
\special{pa 4392 4120}%
\special{pa 4394 4120}%
\special{pa 4398 4118}%
\special{pa 4400 4118}%
\special{pa 4402 4116}%
\special{pa 4404 4116}%
\special{pa 4406 4114}%
\special{pa 4410 4114}%
\special{pa 4410 4113}%
\special{fp}%
\special{pa 4462 4088}%
\special{pa 4462 4088}%
\special{pa 4464 4086}%
\special{pa 4466 4086}%
\special{pa 4468 4084}%
\special{pa 4470 4084}%
\special{pa 4472 4082}%
\special{pa 4474 4082}%
\special{pa 4476 4080}%
\special{pa 4484 4076}%
\special{pa 4486 4076}%
\special{pa 4490 4072}%
\special{pa 4492 4072}%
\special{pa 4494 4070}%
\special{pa 4498 4068}%
\special{pa 4500 4068}%
\special{pa 4504 4064}%
\special{pa 4506 4064}%
\special{pa 4508 4062}%
\special{pa 4512 4060}%
\special{pa 4514 4058}%
\special{fp}%
\special{pa 4563 4017}%
\special{pa 4612 3977}%
\special{fp}%
\special{pa 4661 3935}%
\special{pa 4709 3895}%
\special{fp}%
\special{pa 4757 3853}%
\special{pa 4768 3842}%
\special{pa 4768 3840}%
\special{pa 4774 3836}%
\special{pa 4774 3834}%
\special{pa 4782 3826}%
\special{pa 4782 3824}%
\special{pa 4790 3816}%
\special{pa 4790 3814}%
\special{pa 4792 3812}%
\special{pa 4792 3810}%
\special{pa 4796 3808}%
\special{pa 4796 3808}%
\special{fp}%
\special{pa 4830 3759}%
\special{pa 4832 3758}%
\special{pa 4832 3756}%
\special{pa 4836 3752}%
\special{pa 4836 3748}%
\special{pa 4840 3744}%
\special{pa 4840 3742}%
\special{pa 4844 3738}%
\special{pa 4846 3734}%
\special{pa 4850 3730}%
\special{pa 4850 3726}%
\special{pa 4854 3722}%
\special{pa 4856 3718}%
\special{pa 4858 3716}%
\special{pa 4858 3714}%
\special{pa 4860 3712}%
\special{pa 4862 3709}%
\special{fp}%
\special{pa 4892 3653}%
\special{pa 4914 3610}%
\special{pa 4914 3606}%
\special{pa 4916 3604}%
\special{pa 4916 3602}%
\special{pa 4918 3600}%
\special{pa 4918 3598}%
\special{pa 4919 3597}%
\special{fp}%
\special{pa 4941 3541}%
\special{pa 4948 3520}%
\special{pa 4952 3508}%
\special{pa 4952 3506}%
\special{pa 4954 3504}%
\special{pa 4960 3481}%
\special{fp}%
\special{pa 4977 3420}%
\special{pa 4980 3402}%
\special{pa 4980 3398}%
\special{pa 4984 3380}%
\special{pa 4984 3378}%
\special{pa 4988 3362}%
\special{pa 4989 3358}%
\special{fp}%
\special{pa 4996 3295}%
\special{pa 4996 3286}%
\special{pa 4998 3272}%
\special{pa 4998 3256}%
\special{pa 5000 3232}%
\special{fp}%
}}%
% FUNC 2 1 3 0 Black White
% 10 2800 2000 5200 4400 4000 3200 5000 3200 4000 2200 2200 1600 2200 1600 10 2 0 4 0 0
% ///x^2+y^2=0.4///-1.199///1.199
{\color[named]{Black}{%
\special{pn 8}%
\special{pn 8}%
\special{pa 4600 3400}%
\special{pa 4602 3398}%
\special{pa 4602 3394}%
\special{pa 4604 3392}%
\special{pa 4604 3386}%
\special{pa 4606 3386}%
\special{pa 4606 3380}%
\special{pa 4608 3380}%
\special{pa 4608 3374}%
\special{pa 4610 3372}%
\special{pa 4610 3366}%
\special{pa 4612 3366}%
\special{pa 4612 3358}%
\special{pa 4614 3358}%
\special{pa 4614 3350}%
\special{pa 4616 3350}%
\special{pa 4616 3349}%
\special{fp}%
\special{pa 4626 3292}%
\special{pa 4628 3280}%
\special{pa 4630 3262}%
\special{pa 4630 3250}%
\special{pa 4632 3250}%
\special{pa 4632 3231}%
\special{fp}%
\special{pa 4632 3167}%
\special{pa 4632 3152}%
\special{pa 4630 3150}%
\special{pa 4630 3138}%
\special{pa 4628 3122}%
\special{pa 4628 3120}%
\special{pa 4626 3108}%
\special{pa 4626 3105}%
\special{fp}%
\special{pa 4614 3045}%
\special{pa 4614 3044}%
\special{pa 4612 3042}%
\special{pa 4612 3036}%
\special{pa 4610 3034}%
\special{pa 4610 3028}%
\special{pa 4608 3028}%
\special{pa 4608 3022}%
\special{pa 4606 3020}%
\special{pa 4606 3016}%
\special{pa 4604 3014}%
\special{pa 4604 3010}%
\special{pa 4602 3006}%
\special{pa 4602 3004}%
\special{pa 4600 3000}%
\special{pa 4600 2998}%
\special{pa 4598 2994}%
\special{pa 4598 2992}%
\special{pa 4597 2989}%
\special{fp}%
\special{pa 4576 2936}%
\special{pa 4576 2936}%
\special{pa 4574 2936}%
\special{pa 4574 2932}%
\special{pa 4572 2932}%
\special{pa 4572 2928}%
\special{pa 4570 2926}%
\special{pa 4570 2924}%
\special{pa 4568 2922}%
\special{pa 4568 2920}%
\special{pa 4566 2918}%
\special{pa 4566 2916}%
\special{pa 4564 2914}%
\special{pa 4564 2912}%
\special{pa 4560 2908}%
\special{pa 4560 2904}%
\special{pa 4558 2904}%
\special{pa 4558 2900}%
\special{pa 4556 2900}%
\special{pa 4556 2896}%
\special{pa 4554 2896}%
\special{pa 4554 2894}%
\special{pa 4552 2892}%
\special{pa 4552 2890}%
\special{pa 4551 2889}%
\special{fp}%
\special{pa 4520 2838}%
\special{pa 4520 2838}%
\special{pa 4514 2832}%
\special{pa 4512 2828}%
\special{pa 4510 2826}%
\special{pa 4510 2824}%
\special{pa 4506 2822}%
\special{pa 4506 2820}%
\special{pa 4504 2818}%
\special{pa 4504 2816}%
\special{pa 4492 2804}%
\special{pa 4492 2802}%
\special{pa 4488 2798}%
\special{pa 4488 2796}%
\special{pa 4482 2792}%
\special{pa 4482 2792}%
\special{fp}%
\special{pa 4439 2745}%
\special{pa 4428 2734}%
\special{pa 4426 2734}%
\special{pa 4418 2724}%
\special{pa 4416 2724}%
\special{pa 4410 2718}%
\special{pa 4404 2714}%
\special{pa 4394 2704}%
\special{pa 4393 2704}%
\special{fp}%
\special{pa 4343 2669}%
\special{pa 4342 2668}%
\special{pa 4340 2668}%
\special{pa 4334 2662}%
\special{pa 4332 2662}%
\special{pa 4330 2660}%
\special{pa 4328 2660}%
\special{pa 4326 2658}%
\special{pa 4322 2656}%
\special{pa 4320 2654}%
\special{pa 4318 2654}%
\special{pa 4314 2650}%
\special{pa 4310 2650}%
\special{pa 4310 2648}%
\special{pa 4308 2648}%
\special{pa 4306 2646}%
\special{pa 4304 2646}%
\special{pa 4302 2644}%
\special{pa 4300 2644}%
\special{pa 4298 2642}%
\special{pa 4296 2642}%
\special{pa 4293 2639}%
\special{fp}%
\special{pa 4244 2618}%
\special{pa 4244 2618}%
\special{pa 4244 2616}%
\special{pa 4240 2616}%
\special{pa 4240 2614}%
\special{pa 4236 2614}%
\special{pa 4234 2612}%
\special{pa 4230 2612}%
\special{pa 4230 2610}%
\special{pa 4224 2610}%
\special{pa 4224 2608}%
\special{pa 4222 2608}%
\special{pa 4216 2606}%
\special{pa 4214 2606}%
\special{pa 4214 2604}%
\special{pa 4208 2604}%
\special{pa 4208 2602}%
\special{pa 4202 2602}%
\special{pa 4202 2600}%
\special{pa 4196 2600}%
\special{pa 4196 2600}%
\special{fp}%
\special{pa 4139 2584}%
\special{pa 4138 2584}%
\special{pa 4136 2582}%
\special{pa 4132 2582}%
\special{pa 4124 2580}%
\special{pa 4118 2580}%
\special{pa 4118 2578}%
\special{pa 4108 2578}%
\special{pa 4106 2576}%
\special{pa 4094 2576}%
\special{pa 4094 2574}%
\special{pa 4082 2574}%
\special{fp}%
\special{pa 4022 2568}%
\special{pa 3966 2568}%
\special{pa 3966 2570}%
\special{pa 3961 2570}%
\special{fp}%
\special{pa 3899 2576}%
\special{pa 3894 2576}%
\special{pa 3894 2578}%
\special{pa 3884 2578}%
\special{pa 3882 2580}%
\special{pa 3874 2580}%
\special{pa 3872 2582}%
\special{pa 3864 2582}%
\special{pa 3864 2584}%
\special{pa 3858 2584}%
\special{pa 3852 2586}%
\special{pa 3850 2586}%
\special{pa 3844 2588}%
\special{pa 3842 2588}%
\special{pa 3842 2588}%
\special{fp}%
\special{pa 3785 2606}%
\special{pa 3784 2606}%
\special{pa 3780 2608}%
\special{pa 3778 2608}%
\special{pa 3776 2610}%
\special{pa 3772 2610}%
\special{pa 3770 2612}%
\special{pa 3766 2612}%
\special{pa 3766 2614}%
\special{pa 3762 2614}%
\special{pa 3760 2616}%
\special{pa 3758 2616}%
\special{pa 3756 2618}%
\special{pa 3752 2618}%
\special{pa 3752 2620}%
\special{pa 3748 2620}%
\special{pa 3746 2622}%
\special{pa 3744 2622}%
\special{pa 3742 2624}%
\special{pa 3740 2624}%
\special{pa 3732 2628}%
\special{fp}%
\special{pa 3679 2655}%
\special{pa 3674 2660}%
\special{pa 3670 2660}%
\special{pa 3666 2664}%
\special{pa 3662 2666}%
\special{pa 3658 2670}%
\special{pa 3656 2670}%
\special{pa 3652 2674}%
\special{pa 3650 2674}%
\special{pa 3646 2678}%
\special{pa 3644 2678}%
\special{pa 3640 2682}%
\special{pa 3638 2682}%
\special{pa 3636 2684}%
\special{pa 3632 2686}%
\special{pa 3629 2689}%
\special{fp}%
\special{pa 3581 2727}%
\special{pa 3536 2772}%
\special{fp}%
\special{pa 3496 2819}%
\special{pa 3494 2820}%
\special{pa 3494 2822}%
\special{pa 3492 2824}%
\special{pa 3492 2826}%
\special{pa 3486 2832}%
\special{pa 3484 2836}%
\special{pa 3482 2838}%
\special{pa 3482 2840}%
\special{pa 3478 2844}%
\special{pa 3478 2846}%
\special{pa 3474 2850}%
\special{pa 3474 2852}%
\special{pa 3470 2856}%
\special{pa 3470 2858}%
\special{pa 3466 2862}%
\special{pa 3464 2866}%
\special{pa 3462 2868}%
\special{fp}%
\special{pa 3435 2917}%
\special{pa 3434 2918}%
\special{pa 3434 2920}%
\special{pa 3432 2922}%
\special{pa 3432 2924}%
\special{pa 3430 2926}%
\special{pa 3430 2930}%
\special{pa 3428 2930}%
\special{pa 3428 2934}%
\special{pa 3426 2934}%
\special{pa 3426 2938}%
\special{pa 3424 2938}%
\special{pa 3424 2942}%
\special{pa 3422 2944}%
\special{pa 3422 2946}%
\special{pa 3420 2948}%
\special{pa 3420 2952}%
\special{pa 3418 2952}%
\special{pa 3418 2956}%
\special{pa 3416 2958}%
\special{pa 3416 2960}%
\special{pa 3414 2962}%
\special{pa 3414 2966}%
\special{pa 3413 2966}%
\special{fp}%
\special{pa 3394 3023}%
\special{pa 3394 3024}%
\special{pa 3392 3026}%
\special{pa 3392 3032}%
\special{pa 3390 3032}%
\special{pa 3390 3036}%
\special{pa 3388 3042}%
\special{pa 3388 3044}%
\special{pa 3386 3050}%
\special{pa 3386 3052}%
\special{pa 3384 3058}%
\special{pa 3384 3064}%
\special{pa 3382 3064}%
\special{pa 3382 3068}%
\special{pa 3380 3078}%
\special{pa 3380 3080}%
\special{fp}%
\special{pa 3370 3140}%
\special{pa 3370 3140}%
\special{pa 3370 3166}%
\special{pa 3368 3166}%
\special{pa 3368 3201}%
\special{fp}%
\special{pa 3372 3263}%
\special{pa 3372 3280}%
\special{pa 3374 3280}%
\special{pa 3374 3294}%
\special{pa 3376 3294}%
\special{pa 3376 3306}%
\special{pa 3378 3308}%
\special{pa 3378 3318}%
\special{pa 3380 3318}%
\special{pa 3380 3319}%
\special{fp}%
\special{pa 3394 3377}%
\special{pa 3394 3380}%
\special{pa 3396 3386}%
\special{pa 3396 3390}%
\special{pa 3398 3390}%
\special{pa 3398 3396}%
\special{pa 3400 3396}%
\special{pa 3400 3402}%
\special{pa 3402 3402}%
\special{pa 3402 3408}%
\special{pa 3404 3408}%
\special{pa 3404 3414}%
\special{pa 3406 3414}%
\special{pa 3406 3416}%
\special{pa 3408 3422}%
\special{pa 3408 3424}%
\special{pa 3410 3424}%
\special{pa 3410 3428}%
\special{fp}%
\special{pa 3433 3479}%
\special{pa 3438 3490}%
\special{pa 3438 3492}%
\special{pa 3442 3496}%
\special{pa 3442 3498}%
\special{pa 3444 3500}%
\special{pa 3444 3502}%
\special{pa 3446 3504}%
\special{pa 3446 3506}%
\special{pa 3448 3508}%
\special{pa 3450 3512}%
\special{pa 3450 3514}%
\special{pa 3454 3518}%
\special{pa 3454 3520}%
\special{pa 3456 3522}%
\special{pa 3458 3526}%
\special{pa 3460 3528}%
\special{pa 3460 3530}%
\special{pa 3462 3532}%
\special{fp}%
\special{pa 3496 3582}%
\special{pa 3502 3588}%
\special{pa 3502 3590}%
\special{pa 3504 3592}%
\special{pa 3504 3594}%
\special{pa 3514 3604}%
\special{pa 3518 3610}%
\special{pa 3524 3616}%
\special{pa 3524 3618}%
\special{pa 3534 3626}%
\special{pa 3534 3628}%
\special{pa 3535 3629}%
\special{fp}%
\special{pa 3581 3674}%
\special{pa 3584 3676}%
\special{pa 3590 3682}%
\special{pa 3592 3682}%
\special{pa 3596 3688}%
\special{pa 3598 3688}%
\special{pa 3602 3692}%
\special{pa 3604 3692}%
\special{pa 3616 3704}%
\special{pa 3618 3704}%
\special{pa 3620 3706}%
\special{pa 3622 3706}%
\special{pa 3624 3710}%
\special{pa 3626 3710}%
\special{pa 3628 3712}%
\special{fp}%
\special{pa 3679 3745}%
\special{pa 3680 3746}%
\special{pa 3684 3748}%
\special{pa 3686 3748}%
\special{pa 3690 3752}%
\special{pa 3692 3752}%
\special{pa 3694 3754}%
\special{pa 3706 3760}%
\special{pa 3708 3760}%
\special{pa 3712 3764}%
\special{pa 3714 3764}%
\special{pa 3716 3766}%
\special{pa 3718 3766}%
\special{pa 3720 3768}%
\special{pa 3722 3768}%
\special{pa 3724 3770}%
\special{pa 3726 3770}%
\special{pa 3728 3772}%
\special{pa 3730 3772}%
\special{pa 3732 3773}%
\special{fp}%
\special{pa 3785 3795}%
\special{pa 3786 3796}%
\special{pa 3788 3796}%
\special{pa 3792 3798}%
\special{pa 3794 3798}%
\special{pa 3798 3800}%
\special{pa 3800 3800}%
\special{pa 3804 3802}%
\special{pa 3806 3802}%
\special{pa 3810 3804}%
\special{pa 3814 3804}%
\special{pa 3816 3806}%
\special{pa 3820 3806}%
\special{pa 3822 3808}%
\special{pa 3828 3808}%
\special{pa 3828 3810}%
\special{pa 3834 3810}%
\special{pa 3836 3812}%
\special{pa 3842 3812}%
\special{fp}%
\special{pa 3899 3824}%
\special{pa 3900 3824}%
\special{pa 3902 3826}%
\special{pa 3908 3826}%
\special{pa 3920 3828}%
\special{pa 3922 3828}%
\special{pa 3938 3830}%
\special{pa 3950 3830}%
\special{pa 3952 3832}%
\special{pa 3960 3832}%
\special{fp}%
\special{pa 4025 3832}%
\special{pa 4050 3832}%
\special{pa 4050 3830}%
\special{pa 4062 3830}%
\special{pa 4080 3828}%
\special{pa 4086 3827}%
\special{fp}%
\special{pa 4143 3816}%
\special{pa 4150 3816}%
\special{pa 4150 3814}%
\special{pa 4158 3814}%
\special{pa 4158 3812}%
\special{pa 4166 3812}%
\special{pa 4166 3810}%
\special{pa 4172 3810}%
\special{pa 4174 3808}%
\special{pa 4180 3808}%
\special{pa 4180 3806}%
\special{pa 4186 3806}%
\special{pa 4186 3804}%
\special{pa 4190 3804}%
\special{pa 4195 3802}%
\special{fp}%
\special{pa 4251 3782}%
\special{pa 4252 3782}%
\special{pa 4252 3780}%
\special{pa 4256 3780}%
\special{pa 4256 3778}%
\special{pa 4260 3778}%
\special{pa 4262 3776}%
\special{pa 4264 3776}%
\special{pa 4266 3774}%
\special{pa 4268 3774}%
\special{pa 4270 3772}%
\special{pa 4274 3772}%
\special{pa 4274 3770}%
\special{pa 4278 3770}%
\special{pa 4278 3768}%
\special{pa 4282 3768}%
\special{pa 4282 3766}%
\special{pa 4286 3766}%
\special{pa 4286 3764}%
\special{pa 4290 3764}%
\special{pa 4290 3762}%
\special{pa 4292 3762}%
\special{pa 4294 3760}%
\special{pa 4296 3760}%
\special{pa 4296 3760}%
\special{fp}%
\special{pa 4348 3730}%
\special{pa 4348 3730}%
\special{pa 4352 3726}%
\special{pa 4354 3726}%
\special{pa 4358 3722}%
\special{pa 4360 3722}%
\special{pa 4364 3718}%
\special{pa 4366 3718}%
\special{pa 4368 3714}%
\special{pa 4370 3714}%
\special{pa 4372 3712}%
\special{pa 4374 3712}%
\special{pa 4380 3706}%
\special{pa 4384 3704}%
\special{pa 4388 3700}%
\special{pa 4390 3700}%
\special{pa 4392 3696}%
\special{pa 4394 3696}%
\special{pa 4395 3695}%
\special{fp}%
\special{pa 4443 3653}%
\special{pa 4468 3628}%
\special{pa 4468 3626}%
\special{pa 4476 3618}%
\special{pa 4476 3616}%
\special{pa 4485 3607}%
\special{fp}%
\special{pa 4522 3559}%
\special{pa 4522 3558}%
\special{pa 4526 3554}%
\special{pa 4526 3552}%
\special{pa 4530 3548}%
\special{pa 4530 3546}%
\special{pa 4534 3542}%
\special{pa 4534 3540}%
\special{pa 4536 3538}%
\special{pa 4536 3536}%
\special{pa 4538 3534}%
\special{pa 4538 3532}%
\special{pa 4544 3526}%
\special{pa 4544 3522}%
\special{pa 4548 3518}%
\special{pa 4548 3516}%
\special{pa 4550 3514}%
\special{pa 4552 3510}%
\special{pa 4552 3508}%
\special{pa 4552 3508}%
\special{fp}%
\special{pa 4580 3454}%
\special{pa 4580 3454}%
\special{pa 4580 3452}%
\special{pa 4582 3452}%
\special{pa 4582 3448}%
\special{pa 4584 3446}%
\special{pa 4584 3444}%
\special{pa 4586 3440}%
\special{pa 4586 3438}%
\special{pa 4588 3436}%
\special{pa 4588 3432}%
\special{pa 4590 3432}%
\special{pa 4590 3430}%
\special{pa 4592 3424}%
\special{pa 4592 3422}%
\special{pa 4594 3422}%
\special{pa 4594 3416}%
\special{pa 4596 3416}%
\special{pa 4596 3414}%
\special{pa 4600 3402}%
\special{fp}%
}}%
% LINE 3 0 3 0 Black White
% 16 3920 2200 3000 3120 4000 2240 3040 3200 3420 2940 3160 3200 3370 3110 3280 3200 4000 2360 3740 2620 4000 2480 3910 2570 3770 2230 3030 2970 3590 2290 3090 2790
% 
{\color[named]{Black}{%
\special{pn 4}%
\special{pa 3920 2200}%
\special{pa 3000 3120}%
\special{fp}%
\special{pa 4000 2240}%
\special{pa 3040 3200}%
\special{fp}%
\special{pa 3420 2940}%
\special{pa 3160 3200}%
\special{fp}%
\special{pa 3370 3110}%
\special{pa 3280 3200}%
\special{fp}%
\special{pa 4000 2360}%
\special{pa 3740 2620}%
\special{fp}%
\special{pa 4000 2480}%
\special{pa 3910 2570}%
\special{fp}%
\special{pa 3770 2230}%
\special{pa 3030 2970}%
\special{fp}%
\special{pa 3590 2290}%
\special{pa 3090 2790}%
\special{fp}%
}}%
% LINE 3 0 3 0 Black White
% 34 4860 2700 4590 2970 4900 2780 4620 3060 4940 2860 4630 3170 4970 2950 4720 3200 4990 3050 4840 3200 5000 3160 4960 3200 4810 2630 4550 2890 4760 2560 4510 2810 4700 2500 4450 2750 4640 2440 4390 2690 4570 2390 4310 2650 4500 2340 4230 2610 4420 2300 4140 2580 4340 2260 4030 2570 4250 2230 4000 2480 4150 2210 4000 2360 4040 2200 4000 2240
% 
{\color[named]{Black}{%
\special{pn 4}%
\special{pa 4860 2700}%
\special{pa 4590 2970}%
\special{fp}%
\special{pa 4900 2780}%
\special{pa 4620 3060}%
\special{fp}%
\special{pa 4940 2860}%
\special{pa 4630 3170}%
\special{fp}%
\special{pa 4970 2950}%
\special{pa 4720 3200}%
\special{fp}%
\special{pa 4990 3050}%
\special{pa 4840 3200}%
\special{fp}%
\special{pa 5000 3160}%
\special{pa 4960 3200}%
\special{fp}%
\special{pa 4810 2630}%
\special{pa 4550 2890}%
\special{fp}%
\special{pa 4760 2560}%
\special{pa 4510 2810}%
\special{fp}%
\special{pa 4700 2500}%
\special{pa 4450 2750}%
\special{fp}%
\special{pa 4640 2440}%
\special{pa 4390 2690}%
\special{fp}%
\special{pa 4570 2390}%
\special{pa 4310 2650}%
\special{fp}%
\special{pa 4500 2340}%
\special{pa 4230 2610}%
\special{fp}%
\special{pa 4420 2300}%
\special{pa 4140 2580}%
\special{fp}%
\special{pa 4340 2260}%
\special{pa 4030 2570}%
\special{fp}%
\special{pa 4250 2230}%
\special{pa 4000 2480}%
\special{fp}%
\special{pa 4150 2210}%
\special{pa 4000 2360}%
\special{fp}%
\special{pa 4040 2200}%
\special{pa 4000 2240}%
\special{fp}%
}}%
% LINE 3 0 3 0 Black White
% 34 3970 3830 3660 4140 4000 3920 3750 4170 4000 4040 3850 4190 4000 4160 3960 4200 3860 3820 3580 4100 3770 3790 3500 4060 3690 3750 3430 4010 3610 3710 3360 3960 3550 3650 3300 3900 3490 3590 3240 3840 3450 3510 3190 3770 3410 3430 3140 3700 3380 3340 3100 3620 3370 3230 3060 3540 3280 3200 3030 3450 3160 3200 3010 3350 3040 3200 3000 3240
% 
{\color[named]{Black}{%
\special{pn 4}%
\special{pa 3970 3830}%
\special{pa 3660 4140}%
\special{fp}%
\special{pa 4000 3920}%
\special{pa 3750 4170}%
\special{fp}%
\special{pa 4000 4040}%
\special{pa 3850 4190}%
\special{fp}%
\special{pa 4000 4160}%
\special{pa 3960 4200}%
\special{fp}%
\special{pa 3860 3820}%
\special{pa 3580 4100}%
\special{fp}%
\special{pa 3770 3790}%
\special{pa 3500 4060}%
\special{fp}%
\special{pa 3690 3750}%
\special{pa 3430 4010}%
\special{fp}%
\special{pa 3610 3710}%
\special{pa 3360 3960}%
\special{fp}%
\special{pa 3550 3650}%
\special{pa 3300 3900}%
\special{fp}%
\special{pa 3490 3590}%
\special{pa 3240 3840}%
\special{fp}%
\special{pa 3450 3510}%
\special{pa 3190 3770}%
\special{fp}%
\special{pa 3410 3430}%
\special{pa 3140 3700}%
\special{fp}%
\special{pa 3380 3340}%
\special{pa 3100 3620}%
\special{fp}%
\special{pa 3370 3230}%
\special{pa 3060 3540}%
\special{fp}%
\special{pa 3280 3200}%
\special{pa 3030 3450}%
\special{fp}%
\special{pa 3160 3200}%
\special{pa 3010 3350}%
\special{fp}%
\special{pa 3040 3200}%
\special{pa 3000 3240}%
\special{fp}%
}}%
% LINE 3 0 3 0 Black White
% 16 4970 3430 4230 4170 5000 3280 4080 4200 4960 3200 4000 4160 4260 3780 4000 4040 4090 3830 4000 3920 4840 3200 4580 3460 4720 3200 4630 3290 4910 3610 4410 4110
% 
{\color[named]{Black}{%
\special{pn 4}%
\special{pa 4970 3430}%
\special{pa 4230 4170}%
\special{fp}%
\special{pa 5000 3280}%
\special{pa 4080 4200}%
\special{fp}%
\special{pa 4960 3200}%
\special{pa 4000 4160}%
\special{fp}%
\special{pa 4260 3780}%
\special{pa 4000 4040}%
\special{fp}%
\special{pa 4090 3830}%
\special{pa 4000 3920}%
\special{fp}%
\special{pa 4840 3200}%
\special{pa 4580 3460}%
\special{fp}%
\special{pa 4720 3200}%
\special{pa 4630 3290}%
\special{fp}%
\special{pa 4910 3610}%
\special{pa 4410 4110}%
\special{fp}%
}}%
% STR 2 0 3 0 Black White
% 4 5020 3140 5020 3240 1 0 0 0
% $1$
\put(50.2000,-32.4000){\makebox(0,0)[lt]{$1$}}%
% STR 2 0 3 0 Black White
% 4 3970 4130 3970 4230 4 0 0 0
% $-1$
\put(39.7000,-42.3000){\makebox(0,0)[rt]{$-1$}}%
% STR 2 0 3 0 Black White
% 4 2970 3080 2970 3180 3 0 0 0
% $-1$
\put(29.7000,-31.8000){\makebox(0,0)[rb]{$-1$}}%
% STR 2 0 3 0 Black White
% 4 4040 2070 4040 2170 2 0 0 0
% $1$
\put(40.4000,-21.7000){\makebox(0,0)[lb]{$1$}}%
% STR 2 0 3 0 Black White
% 4 4590 3130 4590 3230 4 0 0 0
% $t$
\put(45.9000,-32.3000){\makebox(0,0)[rt]{$t$}}%
% STR 2 0 3 0 Black White
% 4 3400 3090 3400 3190 2 0 0 0
% $-t$
\put(34.0000,-31.9000){\makebox(0,0)[lb]{$-t$}}%
% STR 2 0 3 0 Black White
% 4 4030 2520 4030 2620 1 0 0 0
% $t$
\put(40.3000,-26.2000){\makebox(0,0)[lt]{$t$}}%
% STR 2 0 3 0 Black White
% 4 3980 3680 3980 3780 3 0 0 0
% $-t$
\put(39.8000,-37.8000){\makebox(0,0)[rb]{$-t$}}%
% STR 2 0 3 0 Black White
% 4 5370 2100 5370 2200 2 0 0 0
% $A$
\put(53.7000,-22.0000){\makebox(0,0)[lb]{$A$}}%
% STR 2 0 3 0 Black White
% 4 2750 2250 2750 2350 5 0 1 0
% $A(t_i)$
\put(27.5000,-23.5000){\makebox(0,0){{\colorbox[named]{White}{$A(t_i)$}}}}%
\put(27.5000,-23.5000){\makebox(0,0){{\colorbox[named]{White}{\color[named]{Black}{$A(t_i)$}}}}}%
% LINE 3 0 3 0 Black White
% 2 3480 2850 2880 2420
% 
{\color[named]{Black}{%
\special{pn 4}%
\special{pa 3480 2850}%
\special{pa 2880 2420}%
\special{fp}%
}}%
% LINE 3 0 3 0 Black White
% 2 4750 2530 5350 2200
% 
{\color[named]{Black}{%
\special{pn 4}%
\special{pa 4750 2530}%
\special{pa 5350 2200}%
\special{fp}%
}}%
% STR 2 0 3 0 Black White
% 4 5170 2450 5170 2550 1 0 1 0
% $(x_i,y_i)$
\put(51.7000,-25.5000){\makebox(0,0)[lt]{{\colorbox[named]{White}{$(x_i,y_i)$}}}}%
\put(51.7000,-25.5000){\makebox(0,0)[lt]{{\colorbox[named]{White}{\color[named]{Black}{$(x_i,y_i)$}}}}}%
% DOT 0 0 3 0 Black White
% 2 4690 2830 4690 2830
% 
{\color[named]{Black}{%
\special{pn 4}%
\special{sh 1}%
\special{ar 4690 2830 16 16 0  6.28318530717959E+0000}%
\special{sh 1}%
\special{ar 4690 2830 16 16 0  6.28318530717959E+0000}%
}}%
% LINE 3 0 3 0 Black White
% 2 4690 2830 5200 2690
% 
{\color[named]{Black}{%
\special{pn 4}%
\special{pa 4690 2830}%
\special{pa 5200 2690}%
\special{fp}%
}}%
\end{picture}}%

%% file: sec3_2_exB.tex
%WinTpicVersion4.28b
{\unitlength 0.1in
\begin{picture}( 26.4000, 24.0000)( 27.3000,-44.0000)
% STR 2 0 3 0 Black White
% 4 3990 3207 3990 3220 4 2800 0 0
% O
\put(39.9000,-32.2000){\makebox(0,0)[rt]{O}}%
% STR 2 0 3 0 Black White
% 4 3960 1987 3960 2000 4 2800 0 0
% $y$
\put(39.6000,-20.0000){\makebox(0,0)[rt]{$y$}}%
% STR 2 0 3 0 Black White
% 4 5100 3137 5100 3150 2 0 0 0
% $x$
\put(51.0000,-31.5000){\makebox(0,0)[lb]{$x$}}%
% VECTOR 2 0 3 0 Black White
% 2 4000 4400 4000 2000
% 
{\color[named]{Black}{%
\special{pn 8}%
\special{pa 4000 4400}%
\special{pa 4000 2000}%
\special{fp}%
\special{sh 1}%
\special{pa 4000 2000}%
\special{pa 3980 2068}%
\special{pa 4000 2054}%
\special{pa 4020 2068}%
\special{pa 4000 2000}%
\special{fp}%
}}%
% VECTOR 2 0 3 0 Black White
% 2 2800 3200 5200 3200
% 
{\color[named]{Black}{%
\special{pn 8}%
\special{pa 2800 3200}%
\special{pa 5200 3200}%
\special{fp}%
\special{sh 1}%
\special{pa 5200 3200}%
\special{pa 5134 3180}%
\special{pa 5148 3200}%
\special{pa 5134 3220}%
\special{pa 5200 3200}%
\special{fp}%
}}%
% FUNC 2 1 3 0 Black White
% 9 2800 2000 5200 4400 4000 3200 5000 3200 4000 2200 2200 1600 2200 1600 10 2 0 4
% ///x^2+y^2=1///-1.199///1.199
{\color[named]{Black}{%
\special{pn 8}%
\special{pa 5000 3200}%
\special{pa 5000 3146}%
\special{pa 4998 3146}%
\special{pa 4998 3139}%
\special{fp}%
\special{pa 4993 3078}%
\special{pa 4992 3078}%
\special{pa 4992 3070}%
\special{pa 4990 3056}%
\special{pa 4990 3050}%
\special{pa 4988 3048}%
\special{pa 4988 3038}%
\special{pa 4986 3036}%
\special{pa 4986 3026}%
\special{pa 4984 3024}%
\special{pa 4984 3019}%
\special{fp}%
\special{pa 4972 2966}%
\special{pa 4972 2962}%
\special{pa 4970 2956}%
\special{pa 4970 2952}%
\special{pa 4968 2950}%
\special{pa 4968 2948}%
\special{pa 4966 2940}%
\special{pa 4966 2936}%
\special{pa 4964 2936}%
\special{pa 4964 2930}%
\special{pa 4962 2928}%
\special{pa 4962 2922}%
\special{pa 4960 2922}%
\special{pa 4960 2916}%
\special{pa 4958 2914}%
\special{pa 4958 2910}%
\special{pa 4958 2910}%
\special{fp}%
\special{pa 4941 2861}%
\special{pa 4940 2860}%
\special{pa 4940 2856}%
\special{pa 4938 2854}%
\special{pa 4938 2852}%
\special{pa 4936 2848}%
\special{pa 4936 2846}%
\special{pa 4934 2844}%
\special{pa 4934 2840}%
\special{pa 4932 2838}%
\special{pa 4932 2834}%
\special{pa 4930 2834}%
\special{pa 4930 2830}%
\special{pa 4928 2828}%
\special{pa 4928 2824}%
\special{pa 4926 2824}%
\special{pa 4926 2820}%
\special{pa 4924 2818}%
\special{pa 4924 2816}%
\special{pa 4922 2814}%
\special{pa 4922 2810}%
\special{pa 4920 2810}%
\special{pa 4920 2809}%
\special{fp}%
\special{pa 4897 2757}%
\special{pa 4896 2756}%
\special{pa 4896 2754}%
\special{pa 4894 2752}%
\special{pa 4894 2750}%
\special{pa 4890 2746}%
\special{pa 4890 2742}%
\special{pa 4888 2742}%
\special{pa 4888 2738}%
\special{pa 4886 2738}%
\special{pa 4886 2734}%
\special{pa 4884 2734}%
\special{pa 4884 2730}%
\special{pa 4882 2730}%
\special{pa 4882 2726}%
\special{pa 4878 2722}%
\special{pa 4878 2720}%
\special{pa 4876 2718}%
\special{pa 4876 2716}%
\special{pa 4874 2714}%
\special{pa 4872 2710}%
\special{pa 4872 2708}%
\special{fp}%
\special{pa 4839 2657}%
\special{pa 4838 2656}%
\special{pa 4838 2654}%
\special{pa 4836 2652}%
\special{pa 4836 2650}%
\special{pa 4834 2648}%
\special{pa 4834 2646}%
\special{pa 4826 2638}%
\special{pa 4826 2636}%
\special{pa 4822 2632}%
\special{pa 4822 2630}%
\special{pa 4820 2628}%
\special{pa 4818 2624}%
\special{pa 4816 2622}%
\special{pa 4816 2620}%
\special{pa 4812 2618}%
\special{pa 4812 2616}%
\special{pa 4810 2614}%
\special{pa 4808 2610}%
\special{pa 4806 2608}%
\special{fp}%
\special{pa 4767 2559}%
\special{pa 4762 2554}%
\special{pa 4762 2552}%
\special{pa 4750 2540}%
\special{pa 4750 2538}%
\special{pa 4742 2530}%
\special{pa 4742 2528}%
\special{pa 4727 2513}%
\special{fp}%
\special{pa 4680 2468}%
\special{pa 4656 2444}%
\special{pa 4654 2444}%
\special{pa 4642 2432}%
\special{pa 4640 2432}%
\special{pa 4636 2428}%
\special{pa 4634 2428}%
\special{fp}%
\special{pa 4588 2391}%
\special{pa 4588 2390}%
\special{pa 4586 2390}%
\special{pa 4584 2388}%
\special{pa 4580 2386}%
\special{pa 4578 2384}%
\special{pa 4576 2384}%
\special{pa 4574 2380}%
\special{pa 4572 2380}%
\special{pa 4568 2376}%
\special{pa 4566 2376}%
\special{pa 4562 2372}%
\special{pa 4560 2372}%
\special{pa 4558 2370}%
\special{pa 4556 2370}%
\special{pa 4552 2366}%
\special{pa 4548 2364}%
\special{pa 4544 2360}%
\special{pa 4542 2360}%
\special{pa 4539 2357}%
\special{fp}%
\special{pa 4487 2327}%
\special{pa 4486 2326}%
\special{pa 4478 2322}%
\special{pa 4476 2322}%
\special{pa 4472 2318}%
\special{pa 4470 2318}%
\special{pa 4468 2316}%
\special{pa 4466 2316}%
\special{pa 4464 2314}%
\special{pa 4462 2314}%
\special{pa 4460 2312}%
\special{pa 4458 2312}%
\special{pa 4454 2308}%
\special{pa 4450 2308}%
\special{pa 4450 2306}%
\special{pa 4446 2306}%
\special{pa 4446 2304}%
\special{pa 4442 2304}%
\special{pa 4442 2302}%
\special{pa 4438 2302}%
\special{pa 4437 2301}%
\special{fp}%
\special{pa 4384 2277}%
\special{pa 4382 2276}%
\special{pa 4380 2276}%
\special{pa 4378 2274}%
\special{pa 4374 2274}%
\special{pa 4374 2272}%
\special{pa 4370 2272}%
\special{pa 4368 2270}%
\special{pa 4364 2270}%
\special{pa 4364 2268}%
\special{pa 4362 2268}%
\special{pa 4356 2266}%
\special{pa 4354 2266}%
\special{pa 4354 2264}%
\special{pa 4348 2264}%
\special{pa 4348 2262}%
\special{pa 4346 2262}%
\special{pa 4340 2260}%
\special{pa 4338 2260}%
\special{pa 4338 2258}%
\special{pa 4334 2258}%
\special{fp}%
\special{pa 4283 2242}%
\special{pa 4282 2242}%
\special{pa 4282 2240}%
\special{pa 4276 2240}%
\special{pa 4274 2238}%
\special{pa 4268 2238}%
\special{pa 4268 2236}%
\special{pa 4264 2236}%
\special{pa 4258 2234}%
\special{pa 4254 2234}%
\special{pa 4252 2232}%
\special{pa 4250 2232}%
\special{pa 4242 2230}%
\special{pa 4238 2230}%
\special{pa 4238 2228}%
\special{pa 4232 2228}%
\special{pa 4229 2227}%
\special{fp}%
\special{pa 4169 2215}%
\special{pa 4152 2212}%
\special{pa 4138 2210}%
\special{pa 4132 2210}%
\special{pa 4130 2208}%
\special{pa 4122 2208}%
\special{pa 4107 2206}%
\special{fp}%
\special{pa 4044 2202}%
\special{pa 4032 2202}%
\special{pa 4032 2200}%
\special{pa 3983 2200}%
\special{fp}%
\special{pa 3922 2204}%
\special{pa 3906 2204}%
\special{pa 3906 2206}%
\special{pa 3896 2206}%
\special{pa 3878 2208}%
\special{pa 3870 2208}%
\special{pa 3870 2210}%
\special{pa 3863 2210}%
\special{fp}%
\special{pa 3803 2220}%
\special{pa 3794 2222}%
\special{pa 3790 2222}%
\special{pa 3790 2224}%
\special{pa 3782 2224}%
\special{pa 3780 2226}%
\special{pa 3776 2226}%
\special{pa 3768 2228}%
\special{pa 3764 2228}%
\special{pa 3762 2230}%
\special{pa 3756 2230}%
\special{pa 3754 2232}%
\special{pa 3748 2232}%
\special{pa 3748 2234}%
\special{pa 3747 2234}%
\special{fp}%
\special{pa 3691 2250}%
\special{pa 3666 2258}%
\special{pa 3664 2258}%
\special{pa 3662 2260}%
\special{pa 3660 2260}%
\special{pa 3656 2262}%
\special{pa 3654 2262}%
\special{pa 3652 2264}%
\special{pa 3648 2264}%
\special{pa 3646 2266}%
\special{pa 3644 2266}%
\special{pa 3640 2268}%
\special{pa 3638 2268}%
\special{pa 3636 2270}%
\special{pa 3634 2270}%
\special{fp}%
\special{pa 3583 2292}%
\special{pa 3580 2292}%
\special{pa 3580 2294}%
\special{pa 3576 2294}%
\special{pa 3576 2296}%
\special{pa 3572 2296}%
\special{pa 3572 2298}%
\special{pa 3568 2298}%
\special{pa 3568 2300}%
\special{pa 3564 2300}%
\special{pa 3564 2302}%
\special{pa 3560 2302}%
\special{pa 3558 2304}%
\special{pa 3556 2304}%
\special{pa 3554 2306}%
\special{pa 3552 2306}%
\special{pa 3550 2308}%
\special{pa 3548 2308}%
\special{pa 3546 2310}%
\special{pa 3535 2316}%
\special{fp}%
\special{pa 3484 2344}%
\special{pa 3482 2346}%
\special{pa 3480 2346}%
\special{pa 3478 2348}%
\special{pa 3476 2348}%
\special{pa 3470 2354}%
\special{pa 3468 2354}%
\special{pa 3466 2356}%
\special{pa 3464 2356}%
\special{pa 3462 2358}%
\special{pa 3460 2358}%
\special{pa 3456 2362}%
\special{pa 3454 2362}%
\special{pa 3448 2368}%
\special{pa 3446 2368}%
\special{pa 3444 2370}%
\special{pa 3440 2372}%
\special{pa 3438 2374}%
\special{pa 3436 2374}%
\special{pa 3434 2376}%
\special{fp}%
\special{pa 3384 2414}%
\special{pa 3382 2416}%
\special{pa 3380 2416}%
\special{pa 3372 2424}%
\special{pa 3370 2424}%
\special{pa 3366 2428}%
\special{pa 3360 2432}%
\special{pa 3354 2438}%
\special{pa 3352 2438}%
\special{pa 3340 2450}%
\special{pa 3338 2450}%
\special{pa 3336 2452}%
\special{fp}%
\special{pa 3291 2497}%
\special{pa 3260 2528}%
\special{pa 3260 2530}%
\special{pa 3250 2538}%
\special{pa 3250 2540}%
\special{pa 3248 2542}%
\special{fp}%
\special{pa 3208 2590}%
\special{pa 3208 2590}%
\special{pa 3206 2594}%
\special{pa 3202 2598}%
\special{pa 3200 2602}%
\special{pa 3196 2606}%
\special{pa 3194 2610}%
\special{pa 3188 2616}%
\special{pa 3188 2618}%
\special{pa 3186 2620}%
\special{pa 3186 2622}%
\special{pa 3178 2630}%
\special{pa 3178 2632}%
\special{pa 3174 2636}%
\special{pa 3174 2638}%
\special{pa 3173 2639}%
\special{fp}%
\special{pa 3140 2691}%
\special{pa 3140 2692}%
\special{pa 3138 2694}%
\special{pa 3136 2698}%
\special{pa 3136 2700}%
\special{pa 3132 2704}%
\special{pa 3132 2706}%
\special{pa 3130 2708}%
\special{pa 3128 2712}%
\special{pa 3128 2714}%
\special{pa 3124 2718}%
\special{pa 3124 2720}%
\special{pa 3122 2722}%
\special{pa 3122 2724}%
\special{pa 3120 2726}%
\special{pa 3111 2744}%
\special{fp}%
\special{pa 3086 2797}%
\special{pa 3086 2798}%
\special{pa 3084 2798}%
\special{pa 3084 2802}%
\special{pa 3082 2804}%
\special{pa 3082 2806}%
\special{pa 3080 2808}%
\special{pa 3080 2812}%
\special{pa 3078 2812}%
\special{pa 3078 2816}%
\special{pa 3076 2818}%
\special{pa 3076 2822}%
\special{pa 3074 2822}%
\special{pa 3074 2826}%
\special{pa 3072 2828}%
\special{pa 3072 2832}%
\special{pa 3070 2832}%
\special{pa 3070 2836}%
\special{pa 3068 2838}%
\special{pa 3068 2840}%
\special{pa 3066 2844}%
\special{pa 3066 2846}%
\special{pa 3065 2847}%
\special{fp}%
\special{pa 3046 2904}%
\special{pa 3046 2906}%
\special{pa 3044 2906}%
\special{pa 3044 2912}%
\special{pa 3042 2912}%
\special{pa 3042 2918}%
\special{pa 3040 2920}%
\special{pa 3040 2926}%
\special{pa 3038 2926}%
\special{pa 3038 2932}%
\special{pa 3036 2934}%
\special{pa 3036 2936}%
\special{pa 3034 2944}%
\special{pa 3034 2948}%
\special{pa 3032 2948}%
\special{pa 3032 2952}%
\special{pa 3030 2957}%
\special{fp}%
\special{pa 3018 3017}%
\special{pa 3018 3020}%
\special{pa 3016 3020}%
\special{pa 3016 3026}%
\special{pa 3014 3036}%
\special{pa 3014 3038}%
\special{pa 3012 3048}%
\special{pa 3012 3050}%
\special{pa 3010 3062}%
\special{pa 3010 3070}%
\special{pa 3008 3070}%
\special{pa 3008 3076}%
\special{fp}%
\special{pa 3002 3136}%
\special{pa 3002 3168}%
\special{pa 3000 3170}%
\special{pa 3000 3198}%
\special{fp}%
\special{pa 3002 3260}%
\special{pa 3002 3270}%
\special{pa 3004 3272}%
\special{pa 3004 3294}%
\special{pa 3006 3296}%
\special{pa 3006 3306}%
\special{pa 3008 3322}%
\special{fp}%
\special{pa 3018 3383}%
\special{pa 3018 3388}%
\special{pa 3020 3396}%
\special{pa 3020 3398}%
\special{pa 3022 3406}%
\special{pa 3022 3410}%
\special{pa 3024 3412}%
\special{pa 3024 3420}%
\special{pa 3026 3420}%
\special{pa 3026 3426}%
\special{pa 3028 3432}%
\special{pa 3028 3438}%
\special{pa 3030 3438}%
\special{pa 3030 3440}%
\special{fp}%
\special{pa 3045 3495}%
\special{pa 3046 3496}%
\special{pa 3046 3500}%
\special{pa 3048 3504}%
\special{pa 3048 3506}%
\special{pa 3050 3510}%
\special{pa 3050 3512}%
\special{pa 3052 3516}%
\special{pa 3052 3518}%
\special{pa 3054 3522}%
\special{pa 3054 3524}%
\special{pa 3056 3528}%
\special{pa 3056 3530}%
\special{pa 3058 3534}%
\special{pa 3058 3538}%
\special{pa 3060 3538}%
\special{pa 3060 3540}%
\special{pa 3062 3546}%
\special{pa 3062 3548}%
\special{pa 3064 3548}%
\special{pa 3064 3551}%
\special{fp}%
\special{pa 3084 3602}%
\special{pa 3086 3604}%
\special{pa 3086 3606}%
\special{pa 3088 3608}%
\special{pa 3088 3610}%
\special{pa 3092 3618}%
\special{pa 3092 3620}%
\special{pa 3094 3622}%
\special{pa 3094 3624}%
\special{pa 3096 3626}%
\special{pa 3096 3628}%
\special{pa 3098 3630}%
\special{pa 3098 3632}%
\special{pa 3100 3634}%
\special{pa 3100 3636}%
\special{pa 3102 3638}%
\special{pa 3102 3642}%
\special{pa 3104 3642}%
\special{pa 3104 3646}%
\special{pa 3106 3646}%
\special{pa 3106 3650}%
\special{pa 3108 3650}%
\special{pa 3108 3652}%
\special{fp}%
\special{pa 3136 3702}%
\special{pa 3136 3702}%
\special{pa 3136 3704}%
\special{pa 3140 3708}%
\special{pa 3140 3712}%
\special{pa 3146 3718}%
\special{pa 3146 3722}%
\special{pa 3150 3726}%
\special{pa 3152 3730}%
\special{pa 3156 3734}%
\special{pa 3156 3738}%
\special{pa 3160 3742}%
\special{pa 3160 3744}%
\special{pa 3164 3748}%
\special{pa 3166 3752}%
\special{pa 3167 3753}%
\special{fp}%
\special{pa 3201 3801}%
\special{pa 3202 3802}%
\special{pa 3202 3804}%
\special{pa 3206 3806}%
\special{pa 3206 3808}%
\special{pa 3208 3810}%
\special{pa 3208 3812}%
\special{pa 3212 3814}%
\special{pa 3212 3816}%
\special{pa 3220 3824}%
\special{pa 3220 3826}%
\special{pa 3228 3834}%
\special{pa 3228 3836}%
\special{pa 3232 3840}%
\special{pa 3232 3842}%
\special{pa 3237 3847}%
\special{fp}%
\special{pa 3280 3894}%
\special{pa 3324 3938}%
\special{fp}%
\special{pa 3372 3978}%
\special{pa 3380 3986}%
\special{pa 3382 3986}%
\special{pa 3386 3990}%
\special{pa 3390 3992}%
\special{pa 3394 3996}%
\special{pa 3398 3998}%
\special{pa 3402 4002}%
\special{pa 3406 4004}%
\special{pa 3410 4008}%
\special{pa 3414 4010}%
\special{pa 3416 4012}%
\special{pa 3418 4012}%
\special{pa 3420 4016}%
\special{pa 3420 4016}%
\special{fp}%
\special{pa 3471 4049}%
\special{pa 3472 4050}%
\special{pa 3474 4050}%
\special{pa 3476 4052}%
\special{pa 3478 4052}%
\special{pa 3484 4058}%
\special{pa 3488 4058}%
\special{pa 3494 4064}%
\special{pa 3498 4064}%
\special{pa 3498 4066}%
\special{pa 3500 4066}%
\special{pa 3502 4068}%
\special{pa 3504 4068}%
\special{pa 3508 4072}%
\special{pa 3512 4072}%
\special{pa 3512 4074}%
\special{pa 3514 4074}%
\special{pa 3516 4076}%
\special{pa 3518 4076}%
\special{pa 3520 4078}%
\special{fp}%
\special{pa 3567 4102}%
\special{pa 3570 4102}%
\special{pa 3570 4104}%
\special{pa 3574 4104}%
\special{pa 3574 4106}%
\special{pa 3578 4106}%
\special{pa 3578 4108}%
\special{pa 3582 4108}%
\special{pa 3584 4110}%
\special{pa 3586 4110}%
\special{pa 3588 4112}%
\special{pa 3590 4112}%
\special{pa 3592 4114}%
\special{pa 3594 4114}%
\special{pa 3598 4116}%
\special{pa 3600 4116}%
\special{pa 3600 4118}%
\special{pa 3604 4118}%
\special{pa 3606 4120}%
\special{pa 3610 4120}%
\special{pa 3610 4122}%
\special{pa 3614 4122}%
\special{pa 3616 4124}%
\special{fp}%
\special{pa 3670 4144}%
\special{pa 3674 4146}%
\special{pa 3676 4146}%
\special{pa 3680 4148}%
\special{pa 3682 4148}%
\special{pa 3686 4150}%
\special{pa 3688 4150}%
\special{pa 3692 4152}%
\special{pa 3694 4152}%
\special{pa 3698 4154}%
\special{pa 3702 4154}%
\special{pa 3702 4156}%
\special{pa 3706 4156}%
\special{pa 3712 4158}%
\special{pa 3714 4158}%
\special{pa 3716 4160}%
\special{pa 3722 4160}%
\special{pa 3722 4162}%
\special{pa 3726 4162}%
\special{fp}%
\special{pa 3779 4176}%
\special{pa 3784 4176}%
\special{pa 3786 4178}%
\special{pa 3794 4178}%
\special{pa 3794 4180}%
\special{pa 3804 4180}%
\special{pa 3804 4182}%
\special{pa 3814 4182}%
\special{pa 3814 4184}%
\special{pa 3824 4184}%
\special{pa 3826 4186}%
\special{pa 3835 4186}%
\special{fp}%
\special{pa 3895 4194}%
\special{pa 3896 4194}%
\special{pa 3896 4196}%
\special{pa 3916 4196}%
\special{pa 3918 4198}%
\special{pa 3946 4198}%
\special{pa 3946 4200}%
\special{pa 3953 4200}%
\special{fp}%
\special{pa 4017 4200}%
\special{pa 4054 4200}%
\special{pa 4056 4198}%
\special{pa 4079 4198}%
\special{fp}%
\special{pa 4140 4191}%
\special{pa 4144 4190}%
\special{pa 4152 4190}%
\special{pa 4152 4188}%
\special{pa 4164 4188}%
\special{pa 4164 4186}%
\special{pa 4176 4186}%
\special{pa 4176 4184}%
\special{pa 4186 4184}%
\special{pa 4188 4182}%
\special{pa 4196 4182}%
\special{pa 4196 4182}%
\special{fp}%
\special{pa 4253 4168}%
\special{pa 4254 4168}%
\special{pa 4260 4166}%
\special{pa 4264 4166}%
\special{pa 4266 4164}%
\special{pa 4272 4164}%
\special{pa 4272 4162}%
\special{pa 4278 4162}%
\special{pa 4280 4160}%
\special{pa 4286 4160}%
\special{pa 4286 4158}%
\special{pa 4292 4158}%
\special{pa 4292 4156}%
\special{pa 4298 4156}%
\special{pa 4300 4154}%
\special{pa 4304 4154}%
\special{pa 4306 4152}%
\special{pa 4306 4152}%
\special{fp}%
\special{pa 4360 4134}%
\special{pa 4362 4134}%
\special{pa 4362 4132}%
\special{pa 4366 4132}%
\special{pa 4368 4130}%
\special{pa 4372 4130}%
\special{pa 4372 4128}%
\special{pa 4376 4128}%
\special{pa 4378 4126}%
\special{pa 4382 4126}%
\special{pa 4382 4124}%
\special{pa 4386 4124}%
\special{pa 4386 4122}%
\special{pa 4390 4122}%
\special{pa 4392 4120}%
\special{pa 4394 4120}%
\special{pa 4398 4118}%
\special{pa 4400 4118}%
\special{pa 4402 4116}%
\special{pa 4404 4116}%
\special{pa 4406 4114}%
\special{pa 4410 4114}%
\special{pa 4410 4113}%
\special{fp}%
\special{pa 4462 4088}%
\special{pa 4462 4088}%
\special{pa 4464 4086}%
\special{pa 4466 4086}%
\special{pa 4468 4084}%
\special{pa 4470 4084}%
\special{pa 4472 4082}%
\special{pa 4474 4082}%
\special{pa 4476 4080}%
\special{pa 4484 4076}%
\special{pa 4486 4076}%
\special{pa 4490 4072}%
\special{pa 4492 4072}%
\special{pa 4494 4070}%
\special{pa 4498 4068}%
\special{pa 4500 4068}%
\special{pa 4504 4064}%
\special{pa 4506 4064}%
\special{pa 4508 4062}%
\special{pa 4512 4060}%
\special{pa 4514 4058}%
\special{fp}%
\special{pa 4563 4017}%
\special{pa 4612 3977}%
\special{fp}%
\special{pa 4661 3935}%
\special{pa 4709 3895}%
\special{fp}%
\special{pa 4757 3853}%
\special{pa 4768 3842}%
\special{pa 4768 3840}%
\special{pa 4774 3836}%
\special{pa 4774 3834}%
\special{pa 4782 3826}%
\special{pa 4782 3824}%
\special{pa 4790 3816}%
\special{pa 4790 3814}%
\special{pa 4792 3812}%
\special{pa 4792 3810}%
\special{pa 4796 3808}%
\special{pa 4796 3808}%
\special{fp}%
\special{pa 4830 3759}%
\special{pa 4832 3758}%
\special{pa 4832 3756}%
\special{pa 4836 3752}%
\special{pa 4836 3748}%
\special{pa 4840 3744}%
\special{pa 4840 3742}%
\special{pa 4844 3738}%
\special{pa 4846 3734}%
\special{pa 4850 3730}%
\special{pa 4850 3726}%
\special{pa 4854 3722}%
\special{pa 4856 3718}%
\special{pa 4858 3716}%
\special{pa 4858 3714}%
\special{pa 4860 3712}%
\special{pa 4862 3709}%
\special{fp}%
\special{pa 4892 3653}%
\special{pa 4914 3610}%
\special{pa 4914 3606}%
\special{pa 4916 3604}%
\special{pa 4916 3602}%
\special{pa 4918 3600}%
\special{pa 4918 3598}%
\special{pa 4919 3597}%
\special{fp}%
\special{pa 4941 3541}%
\special{pa 4948 3520}%
\special{pa 4952 3508}%
\special{pa 4952 3506}%
\special{pa 4954 3504}%
\special{pa 4960 3481}%
\special{fp}%
\special{pa 4977 3420}%
\special{pa 4980 3402}%
\special{pa 4980 3398}%
\special{pa 4984 3380}%
\special{pa 4984 3378}%
\special{pa 4988 3362}%
\special{pa 4989 3358}%
\special{fp}%
\special{pa 4996 3295}%
\special{pa 4996 3286}%
\special{pa 4998 3272}%
\special{pa 4998 3256}%
\special{pa 5000 3232}%
\special{fp}%
}}%
% FUNC 2 0 3 0 Black White
% 10 2800 2000 5200 4400 4000 3200 5000 3200 4000 2200 2200 1600 2200 1600 10 2 0 4 0 0
% ///x^2+y^2=0.4///-1.199///1.199
{\color[named]{Black}{%
\special{pn 8}%
\special{pa 4600 3400}%
\special{pa 4602 3398}%
\special{pa 4602 3394}%
\special{pa 4604 3392}%
\special{pa 4604 3386}%
\special{pa 4606 3386}%
\special{pa 4606 3380}%
\special{pa 4608 3380}%
\special{pa 4608 3374}%
\special{pa 4610 3372}%
\special{pa 4610 3366}%
\special{pa 4612 3366}%
\special{pa 4612 3358}%
\special{pa 4614 3358}%
\special{pa 4614 3350}%
\special{pa 4616 3350}%
\special{pa 4616 3342}%
\special{pa 4618 3342}%
\special{pa 4618 3334}%
\special{pa 4620 3332}%
\special{pa 4620 3324}%
\special{pa 4622 3322}%
\special{pa 4622 3312}%
\special{pa 4624 3312}%
\special{pa 4624 3300}%
\special{pa 4626 3300}%
\special{pa 4626 3294}%
\special{pa 4628 3280}%
\special{pa 4630 3262}%
\special{pa 4630 3250}%
\special{pa 4632 3250}%
\special{pa 4632 3152}%
\special{pa 4630 3150}%
\special{pa 4630 3138}%
\special{pa 4628 3122}%
\special{pa 4628 3120}%
\special{pa 4626 3108}%
\special{pa 4626 3102}%
\special{pa 4624 3100}%
\special{pa 4624 3090}%
\special{pa 4622 3088}%
\special{pa 4622 3082}%
\special{pa 4620 3074}%
\special{pa 4620 3068}%
\special{pa 4618 3068}%
\special{pa 4618 3060}%
\special{pa 4616 3058}%
\special{pa 4616 3052}%
\special{pa 4614 3050}%
\special{pa 4614 3044}%
\special{pa 4612 3042}%
\special{pa 4612 3036}%
\special{pa 4610 3034}%
\special{pa 4610 3028}%
\special{pa 4608 3028}%
\special{pa 4608 3022}%
\special{pa 4606 3020}%
\special{pa 4606 3016}%
\special{pa 4604 3014}%
\special{pa 4604 3010}%
\special{pa 4602 3006}%
\special{pa 4602 3004}%
\special{pa 4600 3000}%
\special{pa 4600 2998}%
\special{pa 4598 2994}%
\special{pa 4598 2992}%
\special{pa 4596 2988}%
\special{pa 4596 2986}%
\special{pa 4594 2984}%
\special{pa 4594 2980}%
\special{pa 4592 2978}%
\special{pa 4592 2976}%
\special{pa 4590 2972}%
\special{pa 4590 2970}%
\special{pa 4588 2968}%
\special{pa 4588 2964}%
\special{pa 4586 2964}%
\special{pa 4586 2960}%
\special{pa 4584 2958}%
\special{pa 4584 2954}%
\special{pa 4582 2954}%
\special{pa 4582 2950}%
\special{pa 4580 2948}%
\special{pa 4580 2946}%
\special{pa 4578 2944}%
\special{pa 4578 2940}%
\special{pa 4576 2940}%
\special{pa 4576 2936}%
\special{pa 4574 2936}%
\special{pa 4574 2932}%
\special{pa 4572 2932}%
\special{pa 4572 2928}%
\special{pa 4570 2926}%
\special{pa 4570 2924}%
\special{pa 4568 2922}%
\special{pa 4568 2920}%
\special{pa 4566 2918}%
\special{pa 4566 2916}%
\special{pa 4564 2914}%
\special{pa 4564 2912}%
\special{pa 4560 2908}%
\special{pa 4560 2904}%
\special{pa 4558 2904}%
\special{pa 4558 2900}%
\special{pa 4556 2900}%
\special{pa 4556 2896}%
\special{pa 4554 2896}%
\special{pa 4554 2894}%
\special{pa 4552 2892}%
\special{pa 4552 2890}%
\special{pa 4548 2886}%
\special{pa 4548 2882}%
\special{pa 4546 2882}%
\special{pa 4546 2880}%
\special{pa 4544 2878}%
\special{pa 4544 2876}%
\special{pa 4542 2874}%
\special{pa 4540 2870}%
\special{pa 4536 2866}%
\special{pa 4536 2862}%
\special{pa 4532 2858}%
\special{pa 4532 2856}%
\special{pa 4528 2852}%
\special{pa 4528 2850}%
\special{pa 4524 2846}%
\special{pa 4524 2844}%
\special{pa 4520 2840}%
\special{pa 4520 2838}%
\special{pa 4514 2832}%
\special{pa 4512 2828}%
\special{pa 4510 2826}%
\special{pa 4510 2824}%
\special{pa 4506 2822}%
\special{pa 4506 2820}%
\special{pa 4504 2818}%
\special{pa 4504 2816}%
\special{pa 4492 2804}%
\special{pa 4492 2802}%
\special{pa 4488 2798}%
\special{pa 4488 2796}%
\special{pa 4482 2792}%
\special{pa 4482 2790}%
\special{pa 4476 2784}%
\special{pa 4468 2774}%
\special{pa 4428 2734}%
\special{pa 4426 2734}%
\special{pa 4418 2724}%
\special{pa 4416 2724}%
\special{pa 4410 2718}%
\special{pa 4404 2714}%
\special{pa 4394 2704}%
\special{pa 4392 2704}%
\special{pa 4390 2702}%
\special{pa 4388 2702}%
\special{pa 4380 2694}%
\special{pa 4376 2692}%
\special{pa 4374 2690}%
\special{pa 4372 2690}%
\special{pa 4370 2686}%
\special{pa 4368 2686}%
\special{pa 4366 2684}%
\special{pa 4364 2684}%
\special{pa 4360 2680}%
\special{pa 4358 2680}%
\special{pa 4354 2676}%
\special{pa 4352 2676}%
\special{pa 4348 2672}%
\special{pa 4346 2672}%
\special{pa 4342 2668}%
\special{pa 4340 2668}%
\special{pa 4334 2662}%
\special{pa 4332 2662}%
\special{pa 4330 2660}%
\special{pa 4328 2660}%
\special{pa 4326 2658}%
\special{pa 4322 2656}%
\special{pa 4320 2654}%
\special{pa 4318 2654}%
\special{pa 4314 2650}%
\special{pa 4310 2650}%
\special{pa 4310 2648}%
\special{pa 4308 2648}%
\special{pa 4306 2646}%
\special{pa 4304 2646}%
\special{pa 4302 2644}%
\special{pa 4300 2644}%
\special{pa 4298 2642}%
\special{pa 4296 2642}%
\special{pa 4292 2638}%
\special{pa 4288 2638}%
\special{pa 4288 2636}%
\special{pa 4284 2636}%
\special{pa 4284 2634}%
\special{pa 4280 2634}%
\special{pa 4280 2632}%
\special{pa 4276 2632}%
\special{pa 4276 2630}%
\special{pa 4272 2630}%
\special{pa 4270 2628}%
\special{pa 4268 2628}%
\special{pa 4266 2626}%
\special{pa 4264 2626}%
\special{pa 4262 2624}%
\special{pa 4258 2624}%
\special{pa 4258 2622}%
\special{pa 4254 2622}%
\special{pa 4254 2620}%
\special{pa 4250 2620}%
\special{pa 4248 2618}%
\special{pa 4244 2618}%
\special{pa 4244 2616}%
\special{pa 4240 2616}%
\special{pa 4240 2614}%
\special{pa 4236 2614}%
\special{pa 4234 2612}%
\special{pa 4230 2612}%
\special{pa 4230 2610}%
\special{pa 4224 2610}%
\special{pa 4224 2608}%
\special{pa 4222 2608}%
\special{pa 4216 2606}%
\special{pa 4214 2606}%
\special{pa 4214 2604}%
\special{pa 4208 2604}%
\special{pa 4208 2602}%
\special{pa 4202 2602}%
\special{pa 4202 2600}%
\special{pa 4196 2600}%
\special{pa 4196 2598}%
\special{pa 4190 2598}%
\special{pa 4190 2596}%
\special{pa 4186 2596}%
\special{pa 4180 2594}%
\special{pa 4176 2594}%
\special{pa 4176 2592}%
\special{pa 4170 2592}%
\special{pa 4168 2590}%
\special{pa 4166 2590}%
\special{pa 4142 2584}%
\special{pa 4138 2584}%
\special{pa 4136 2582}%
\special{pa 4132 2582}%
\special{pa 4124 2580}%
\special{pa 4118 2580}%
\special{pa 4118 2578}%
\special{pa 4108 2578}%
\special{pa 4106 2576}%
\special{pa 4094 2576}%
\special{pa 4094 2574}%
\special{pa 4080 2574}%
\special{pa 4080 2572}%
\special{pa 4062 2572}%
\special{pa 4062 2570}%
\special{pa 4036 2570}%
\special{pa 4034 2568}%
\special{pa 3966 2568}%
\special{pa 3966 2570}%
\special{pa 3940 2570}%
\special{pa 3938 2572}%
\special{pa 3922 2572}%
\special{pa 3920 2574}%
\special{pa 3908 2574}%
\special{pa 3906 2576}%
\special{pa 3894 2576}%
\special{pa 3894 2578}%
\special{pa 3884 2578}%
\special{pa 3882 2580}%
\special{pa 3874 2580}%
\special{pa 3872 2582}%
\special{pa 3864 2582}%
\special{pa 3864 2584}%
\special{pa 3858 2584}%
\special{pa 3852 2586}%
\special{pa 3850 2586}%
\special{pa 3844 2588}%
\special{pa 3842 2588}%
\special{pa 3836 2590}%
\special{pa 3832 2590}%
\special{pa 3832 2592}%
\special{pa 3826 2592}%
\special{pa 3824 2594}%
\special{pa 3820 2594}%
\special{pa 3816 2596}%
\special{pa 3812 2596}%
\special{pa 3810 2598}%
\special{pa 3806 2598}%
\special{pa 3804 2600}%
\special{pa 3800 2600}%
\special{pa 3798 2602}%
\special{pa 3794 2602}%
\special{pa 3792 2604}%
\special{pa 3788 2604}%
\special{pa 3786 2606}%
\special{pa 3784 2606}%
\special{pa 3780 2608}%
\special{pa 3778 2608}%
\special{pa 3776 2610}%
\special{pa 3772 2610}%
\special{pa 3770 2612}%
\special{pa 3766 2612}%
\special{pa 3766 2614}%
\special{pa 3762 2614}%
\special{pa 3760 2616}%
\special{pa 3758 2616}%
\special{pa 3756 2618}%
\special{pa 3752 2618}%
\special{pa 3752 2620}%
\special{pa 3748 2620}%
\special{pa 3746 2622}%
\special{pa 3744 2622}%
\special{pa 3742 2624}%
\special{pa 3740 2624}%
\special{pa 3732 2628}%
\special{pa 3730 2628}%
\special{pa 3730 2630}%
\special{pa 3726 2630}%
\special{pa 3724 2632}%
\special{pa 3722 2632}%
\special{pa 3720 2634}%
\special{pa 3718 2634}%
\special{pa 3716 2636}%
\special{pa 3714 2636}%
\special{pa 3712 2638}%
\special{pa 3710 2638}%
\special{pa 3708 2640}%
\special{pa 3696 2646}%
\special{pa 3694 2646}%
\special{pa 3690 2650}%
\special{pa 3688 2650}%
\special{pa 3686 2652}%
\special{pa 3682 2654}%
\special{pa 3680 2654}%
\special{pa 3674 2660}%
\special{pa 3670 2660}%
\special{pa 3666 2664}%
\special{pa 3662 2666}%
\special{pa 3658 2670}%
\special{pa 3656 2670}%
\special{pa 3652 2674}%
\special{pa 3650 2674}%
\special{pa 3646 2678}%
\special{pa 3644 2678}%
\special{pa 3640 2682}%
\special{pa 3638 2682}%
\special{pa 3636 2684}%
\special{pa 3632 2686}%
\special{pa 3626 2692}%
\special{pa 3624 2692}%
\special{pa 3622 2694}%
\special{pa 3620 2694}%
\special{pa 3618 2698}%
\special{pa 3616 2698}%
\special{pa 3612 2702}%
\special{pa 3608 2704}%
\special{pa 3604 2708}%
\special{pa 3602 2708}%
\special{pa 3598 2714}%
\special{pa 3596 2714}%
\special{pa 3592 2718}%
\special{pa 3590 2718}%
\special{pa 3518 2790}%
\special{pa 3518 2792}%
\special{pa 3514 2796}%
\special{pa 3514 2798}%
\special{pa 3508 2802}%
\special{pa 3508 2804}%
\special{pa 3504 2808}%
\special{pa 3502 2812}%
\special{pa 3498 2816}%
\special{pa 3498 2818}%
\special{pa 3494 2820}%
\special{pa 3494 2822}%
\special{pa 3492 2824}%
\special{pa 3492 2826}%
\special{pa 3486 2832}%
\special{pa 3484 2836}%
\special{pa 3482 2838}%
\special{pa 3482 2840}%
\special{pa 3478 2844}%
\special{pa 3478 2846}%
\special{pa 3474 2850}%
\special{pa 3474 2852}%
\special{pa 3470 2856}%
\special{pa 3470 2858}%
\special{pa 3466 2862}%
\special{pa 3464 2866}%
\special{pa 3460 2870}%
\special{pa 3460 2874}%
\special{pa 3454 2880}%
\special{pa 3454 2884}%
\special{pa 3452 2884}%
\special{pa 3452 2886}%
\special{pa 3450 2888}%
\special{pa 3450 2890}%
\special{pa 3446 2894}%
\special{pa 3446 2898}%
\special{pa 3444 2898}%
\special{pa 3444 2902}%
\special{pa 3442 2902}%
\special{pa 3442 2906}%
\special{pa 3440 2906}%
\special{pa 3440 2908}%
\special{pa 3438 2910}%
\special{pa 3438 2912}%
\special{pa 3436 2914}%
\special{pa 3436 2916}%
\special{pa 3434 2918}%
\special{pa 3434 2920}%
\special{pa 3432 2922}%
\special{pa 3432 2924}%
\special{pa 3430 2926}%
\special{pa 3430 2930}%
\special{pa 3428 2930}%
\special{pa 3428 2934}%
\special{pa 3426 2934}%
\special{pa 3426 2938}%
\special{pa 3424 2938}%
\special{pa 3424 2942}%
\special{pa 3422 2944}%
\special{pa 3422 2946}%
\special{pa 3420 2948}%
\special{pa 3420 2952}%
\special{pa 3418 2952}%
\special{pa 3418 2956}%
\special{pa 3416 2958}%
\special{pa 3416 2960}%
\special{pa 3414 2962}%
\special{pa 3414 2966}%
\special{pa 3412 2966}%
\special{pa 3412 2970}%
\special{pa 3410 2972}%
\special{pa 3410 2976}%
\special{pa 3408 2978}%
\special{pa 3408 2980}%
\special{pa 3406 2984}%
\special{pa 3406 2986}%
\special{pa 3404 2988}%
\special{pa 3404 2992}%
\special{pa 3402 2994}%
\special{pa 3402 2998}%
\special{pa 3400 3000}%
\special{pa 3400 3004}%
\special{pa 3398 3006}%
\special{pa 3398 3010}%
\special{pa 3396 3012}%
\special{pa 3396 3016}%
\special{pa 3394 3020}%
\special{pa 3394 3024}%
\special{pa 3392 3026}%
\special{pa 3392 3032}%
\special{pa 3390 3032}%
\special{pa 3390 3036}%
\special{pa 3388 3042}%
\special{pa 3388 3044}%
\special{pa 3386 3050}%
\special{pa 3386 3052}%
\special{pa 3384 3058}%
\special{pa 3384 3064}%
\special{pa 3382 3064}%
\special{pa 3382 3068}%
\special{pa 3380 3078}%
\special{pa 3380 3082}%
\special{pa 3378 3084}%
\special{pa 3378 3094}%
\special{pa 3376 3094}%
\special{pa 3376 3106}%
\special{pa 3374 3108}%
\special{pa 3374 3120}%
\special{pa 3372 3122}%
\special{pa 3372 3138}%
\special{pa 3370 3140}%
\special{pa 3370 3166}%
\special{pa 3368 3166}%
\special{pa 3368 3234}%
\special{pa 3370 3236}%
\special{pa 3370 3262}%
\special{pa 3372 3262}%
\special{pa 3372 3280}%
\special{pa 3374 3280}%
\special{pa 3374 3294}%
\special{pa 3376 3294}%
\special{pa 3376 3306}%
\special{pa 3378 3308}%
\special{pa 3378 3318}%
\special{pa 3380 3318}%
\special{pa 3380 3328}%
\special{pa 3382 3328}%
\special{pa 3382 3336}%
\special{pa 3384 3338}%
\special{pa 3384 3342}%
\special{pa 3390 3366}%
\special{pa 3390 3368}%
\special{pa 3392 3370}%
\special{pa 3392 3376}%
\special{pa 3394 3376}%
\special{pa 3394 3380}%
\special{pa 3396 3386}%
\special{pa 3396 3390}%
\special{pa 3398 3390}%
\special{pa 3398 3396}%
\special{pa 3400 3396}%
\special{pa 3400 3402}%
\special{pa 3402 3402}%
\special{pa 3402 3408}%
\special{pa 3404 3408}%
\special{pa 3404 3414}%
\special{pa 3406 3414}%
\special{pa 3406 3416}%
\special{pa 3408 3422}%
\special{pa 3408 3424}%
\special{pa 3410 3424}%
\special{pa 3410 3430}%
\special{pa 3412 3430}%
\special{pa 3412 3434}%
\special{pa 3414 3436}%
\special{pa 3414 3440}%
\special{pa 3416 3440}%
\special{pa 3416 3444}%
\special{pa 3418 3444}%
\special{pa 3418 3448}%
\special{pa 3420 3450}%
\special{pa 3420 3452}%
\special{pa 3422 3456}%
\special{pa 3422 3458}%
\special{pa 3424 3458}%
\special{pa 3424 3462}%
\special{pa 3426 3464}%
\special{pa 3426 3466}%
\special{pa 3428 3468}%
\special{pa 3428 3470}%
\special{pa 3430 3472}%
\special{pa 3430 3474}%
\special{pa 3438 3490}%
\special{pa 3438 3492}%
\special{pa 3442 3496}%
\special{pa 3442 3498}%
\special{pa 3444 3500}%
\special{pa 3444 3502}%
\special{pa 3446 3504}%
\special{pa 3446 3506}%
\special{pa 3448 3508}%
\special{pa 3450 3512}%
\special{pa 3450 3514}%
\special{pa 3454 3518}%
\special{pa 3454 3520}%
\special{pa 3456 3522}%
\special{pa 3458 3526}%
\special{pa 3460 3528}%
\special{pa 3460 3530}%
\special{pa 3462 3532}%
\special{pa 3462 3534}%
\special{pa 3468 3540}%
\special{pa 3468 3542}%
\special{pa 3472 3546}%
\special{pa 3472 3548}%
\special{pa 3476 3552}%
\special{pa 3476 3554}%
\special{pa 3480 3558}%
\special{pa 3480 3560}%
\special{pa 3484 3564}%
\special{pa 3484 3566}%
\special{pa 3486 3568}%
\special{pa 3486 3570}%
\special{pa 3490 3572}%
\special{pa 3490 3574}%
\special{pa 3492 3576}%
\special{pa 3494 3580}%
\special{pa 3502 3588}%
\special{pa 3502 3590}%
\special{pa 3504 3592}%
\special{pa 3504 3594}%
\special{pa 3514 3604}%
\special{pa 3518 3610}%
\special{pa 3524 3616}%
\special{pa 3524 3618}%
\special{pa 3534 3626}%
\special{pa 3534 3628}%
\special{pa 3574 3668}%
\special{pa 3584 3676}%
\special{pa 3590 3682}%
\special{pa 3592 3682}%
\special{pa 3596 3688}%
\special{pa 3598 3688}%
\special{pa 3602 3692}%
\special{pa 3604 3692}%
\special{pa 3616 3704}%
\special{pa 3618 3704}%
\special{pa 3620 3706}%
\special{pa 3622 3706}%
\special{pa 3624 3710}%
\special{pa 3626 3710}%
\special{pa 3628 3712}%
\special{pa 3632 3714}%
\special{pa 3638 3720}%
\special{pa 3640 3720}%
\special{pa 3644 3724}%
\special{pa 3646 3724}%
\special{pa 3650 3728}%
\special{pa 3652 3728}%
\special{pa 3656 3732}%
\special{pa 3658 3732}%
\special{pa 3662 3736}%
\special{pa 3666 3736}%
\special{pa 3670 3740}%
\special{pa 3674 3742}%
\special{pa 3676 3744}%
\special{pa 3678 3744}%
\special{pa 3680 3746}%
\special{pa 3684 3748}%
\special{pa 3686 3748}%
\special{pa 3690 3752}%
\special{pa 3692 3752}%
\special{pa 3694 3754}%
\special{pa 3706 3760}%
\special{pa 3708 3760}%
\special{pa 3712 3764}%
\special{pa 3714 3764}%
\special{pa 3716 3766}%
\special{pa 3718 3766}%
\special{pa 3720 3768}%
\special{pa 3722 3768}%
\special{pa 3724 3770}%
\special{pa 3726 3770}%
\special{pa 3728 3772}%
\special{pa 3730 3772}%
\special{pa 3738 3776}%
\special{pa 3740 3776}%
\special{pa 3740 3778}%
\special{pa 3744 3778}%
\special{pa 3746 3780}%
\special{pa 3748 3780}%
\special{pa 3750 3782}%
\special{pa 3754 3782}%
\special{pa 3754 3784}%
\special{pa 3758 3784}%
\special{pa 3760 3786}%
\special{pa 3764 3786}%
\special{pa 3764 3788}%
\special{pa 3768 3788}%
\special{pa 3770 3790}%
\special{pa 3772 3790}%
\special{pa 3776 3792}%
\special{pa 3778 3792}%
\special{pa 3780 3794}%
\special{pa 3784 3794}%
\special{pa 3786 3796}%
\special{pa 3788 3796}%
\special{pa 3792 3798}%
\special{pa 3794 3798}%
\special{pa 3798 3800}%
\special{pa 3800 3800}%
\special{pa 3804 3802}%
\special{pa 3806 3802}%
\special{pa 3810 3804}%
\special{pa 3814 3804}%
\special{pa 3816 3806}%
\special{pa 3820 3806}%
\special{pa 3822 3808}%
\special{pa 3828 3808}%
\special{pa 3828 3810}%
\special{pa 3834 3810}%
\special{pa 3836 3812}%
\special{pa 3842 3812}%
\special{pa 3844 3814}%
\special{pa 3850 3814}%
\special{pa 3852 3816}%
\special{pa 3858 3816}%
\special{pa 3860 3818}%
\special{pa 3868 3818}%
\special{pa 3868 3820}%
\special{pa 3878 3820}%
\special{pa 3878 3822}%
\special{pa 3888 3822}%
\special{pa 3890 3824}%
\special{pa 3900 3824}%
\special{pa 3902 3826}%
\special{pa 3908 3826}%
\special{pa 3920 3828}%
\special{pa 3922 3828}%
\special{pa 3938 3830}%
\special{pa 3950 3830}%
\special{pa 3952 3832}%
\special{pa 4050 3832}%
\special{pa 4050 3830}%
\special{pa 4062 3830}%
\special{pa 4080 3828}%
\special{pa 4094 3826}%
\special{pa 4100 3826}%
\special{pa 4100 3824}%
\special{pa 4112 3824}%
\special{pa 4112 3822}%
\special{pa 4118 3822}%
\special{pa 4128 3820}%
\special{pa 4132 3820}%
\special{pa 4134 3818}%
\special{pa 4142 3818}%
\special{pa 4142 3816}%
\special{pa 4150 3816}%
\special{pa 4150 3814}%
\special{pa 4158 3814}%
\special{pa 4158 3812}%
\special{pa 4166 3812}%
\special{pa 4166 3810}%
\special{pa 4172 3810}%
\special{pa 4174 3808}%
\special{pa 4180 3808}%
\special{pa 4180 3806}%
\special{pa 4186 3806}%
\special{pa 4186 3804}%
\special{pa 4190 3804}%
\special{pa 4214 3796}%
\special{pa 4216 3796}%
\special{pa 4216 3794}%
\special{pa 4222 3794}%
\special{pa 4222 3792}%
\special{pa 4224 3792}%
\special{pa 4230 3790}%
\special{pa 4232 3790}%
\special{pa 4232 3788}%
\special{pa 4236 3788}%
\special{pa 4238 3786}%
\special{pa 4240 3786}%
\special{pa 4244 3784}%
\special{pa 4246 3784}%
\special{pa 4248 3782}%
\special{pa 4252 3782}%
\special{pa 4252 3780}%
\special{pa 4256 3780}%
\special{pa 4256 3778}%
\special{pa 4260 3778}%
\special{pa 4262 3776}%
\special{pa 4264 3776}%
\special{pa 4266 3774}%
\special{pa 4268 3774}%
\special{pa 4270 3772}%
\special{pa 4274 3772}%
\special{pa 4274 3770}%
\special{pa 4278 3770}%
\special{pa 4278 3768}%
\special{pa 4282 3768}%
\special{pa 4282 3766}%
\special{pa 4286 3766}%
\special{pa 4286 3764}%
\special{pa 4290 3764}%
\special{pa 4290 3762}%
\special{pa 4292 3762}%
\special{pa 4294 3760}%
\special{pa 4296 3760}%
\special{pa 4298 3758}%
\special{pa 4300 3758}%
\special{pa 4302 3756}%
\special{pa 4304 3756}%
\special{pa 4308 3752}%
\special{pa 4312 3752}%
\special{pa 4312 3750}%
\special{pa 4314 3750}%
\special{pa 4316 3748}%
\special{pa 4318 3748}%
\special{pa 4322 3744}%
\special{pa 4326 3744}%
\special{pa 4332 3738}%
\special{pa 4334 3738}%
\special{pa 4336 3736}%
\special{pa 4338 3736}%
\special{pa 4340 3734}%
\special{pa 4342 3734}%
\special{pa 4346 3730}%
\special{pa 4348 3730}%
\special{pa 4352 3726}%
\special{pa 4354 3726}%
\special{pa 4358 3722}%
\special{pa 4360 3722}%
\special{pa 4364 3718}%
\special{pa 4366 3718}%
\special{pa 4368 3714}%
\special{pa 4370 3714}%
\special{pa 4372 3712}%
\special{pa 4374 3712}%
\special{pa 4380 3706}%
\special{pa 4384 3704}%
\special{pa 4388 3700}%
\special{pa 4390 3700}%
\special{pa 4392 3696}%
\special{pa 4394 3696}%
\special{pa 4398 3692}%
\special{pa 4404 3688}%
\special{pa 4416 3676}%
\special{pa 4418 3676}%
\special{pa 4426 3668}%
\special{pa 4428 3668}%
\special{pa 4468 3628}%
\special{pa 4468 3626}%
\special{pa 4476 3618}%
\special{pa 4476 3616}%
\special{pa 4488 3604}%
\special{pa 4492 3598}%
\special{pa 4496 3594}%
\special{pa 4496 3592}%
\special{pa 4500 3590}%
\special{pa 4500 3588}%
\special{pa 4504 3584}%
\special{pa 4506 3580}%
\special{pa 4512 3574}%
\special{pa 4512 3572}%
\special{pa 4514 3570}%
\special{pa 4514 3568}%
\special{pa 4518 3566}%
\special{pa 4518 3564}%
\special{pa 4522 3560}%
\special{pa 4522 3558}%
\special{pa 4526 3554}%
\special{pa 4526 3552}%
\special{pa 4530 3548}%
\special{pa 4530 3546}%
\special{pa 4534 3542}%
\special{pa 4534 3540}%
\special{pa 4536 3538}%
\special{pa 4536 3536}%
\special{pa 4538 3534}%
\special{pa 4538 3532}%
\special{pa 4544 3526}%
\special{pa 4544 3522}%
\special{pa 4548 3518}%
\special{pa 4548 3516}%
\special{pa 4550 3514}%
\special{pa 4552 3510}%
\special{pa 4552 3508}%
\special{pa 4556 3504}%
\special{pa 4556 3502}%
\special{pa 4558 3500}%
\special{pa 4558 3498}%
\special{pa 4560 3496}%
\special{pa 4560 3494}%
\special{pa 4562 3492}%
\special{pa 4570 3476}%
\special{pa 4570 3474}%
\special{pa 4572 3474}%
\special{pa 4572 3470}%
\special{pa 4574 3468}%
\special{pa 4574 3466}%
\special{pa 4576 3464}%
\special{pa 4576 3462}%
\special{pa 4578 3460}%
\special{pa 4578 3458}%
\special{pa 4580 3454}%
\special{pa 4580 3452}%
\special{pa 4582 3452}%
\special{pa 4582 3448}%
\special{pa 4584 3446}%
\special{pa 4584 3444}%
\special{pa 4586 3440}%
\special{pa 4586 3438}%
\special{pa 4588 3436}%
\special{pa 4588 3432}%
\special{pa 4590 3432}%
\special{pa 4590 3430}%
\special{pa 4592 3424}%
\special{pa 4592 3422}%
\special{pa 4594 3422}%
\special{pa 4594 3416}%
\special{pa 4596 3416}%
\special{pa 4596 3414}%
\special{pa 4600 3402}%
\special{fp}%
}}%
% LINE 3 0 3 0 Black White
% 16 3920 2200 3000 3120 4000 2240 3040 3200 3420 2940 3160 3200 3370 3110 3280 3200 4000 2360 3740 2620 4000 2480 3910 2570 3770 2230 3030 2970 3590 2290 3090 2790
% 
{\color[named]{Black}{%
\special{pn 4}%
\special{pa 3920 2200}%
\special{pa 3000 3120}%
\special{fp}%
\special{pa 4000 2240}%
\special{pa 3040 3200}%
\special{fp}%
\special{pa 3420 2940}%
\special{pa 3160 3200}%
\special{fp}%
\special{pa 3370 3110}%
\special{pa 3280 3200}%
\special{fp}%
\special{pa 4000 2360}%
\special{pa 3740 2620}%
\special{fp}%
\special{pa 4000 2480}%
\special{pa 3910 2570}%
\special{fp}%
\special{pa 3770 2230}%
\special{pa 3030 2970}%
\special{fp}%
\special{pa 3590 2290}%
\special{pa 3090 2790}%
\special{fp}%
}}%
% LINE 3 0 3 0 Black White
% 34 4860 2700 4590 2970 4900 2780 4620 3060 4940 2860 4630 3170 4970 2950 4720 3200 4990 3050 4840 3200 5000 3160 4960 3200 4810 2630 4550 2890 4760 2560 4510 2810 4700 2500 4450 2750 4640 2440 4390 2690 4570 2390 4310 2650 4500 2340 4230 2610 4420 2300 4140 2580 4340 2260 4030 2570 4250 2230 4000 2480 4150 2210 4000 2360 4040 2200 4000 2240
% 
{\color[named]{Black}{%
\special{pn 4}%
\special{pa 4860 2700}%
\special{pa 4590 2970}%
\special{fp}%
\special{pa 4900 2780}%
\special{pa 4620 3060}%
\special{fp}%
\special{pa 4940 2860}%
\special{pa 4630 3170}%
\special{fp}%
\special{pa 4970 2950}%
\special{pa 4720 3200}%
\special{fp}%
\special{pa 4990 3050}%
\special{pa 4840 3200}%
\special{fp}%
\special{pa 5000 3160}%
\special{pa 4960 3200}%
\special{fp}%
\special{pa 4810 2630}%
\special{pa 4550 2890}%
\special{fp}%
\special{pa 4760 2560}%
\special{pa 4510 2810}%
\special{fp}%
\special{pa 4700 2500}%
\special{pa 4450 2750}%
\special{fp}%
\special{pa 4640 2440}%
\special{pa 4390 2690}%
\special{fp}%
\special{pa 4570 2390}%
\special{pa 4310 2650}%
\special{fp}%
\special{pa 4500 2340}%
\special{pa 4230 2610}%
\special{fp}%
\special{pa 4420 2300}%
\special{pa 4140 2580}%
\special{fp}%
\special{pa 4340 2260}%
\special{pa 4030 2570}%
\special{fp}%
\special{pa 4250 2230}%
\special{pa 4000 2480}%
\special{fp}%
\special{pa 4150 2210}%
\special{pa 4000 2360}%
\special{fp}%
\special{pa 4040 2200}%
\special{pa 4000 2240}%
\special{fp}%
}}%
% LINE 3 0 3 0 Black White
% 34 3970 3830 3660 4140 4000 3920 3750 4170 4000 4040 3850 4190 4000 4160 3960 4200 3860 3820 3580 4100 3770 3790 3500 4060 3690 3750 3430 4010 3610 3710 3360 3960 3550 3650 3300 3900 3490 3590 3240 3840 3450 3510 3190 3770 3410 3430 3140 3700 3380 3340 3100 3620 3370 3230 3060 3540 3280 3200 3030 3450 3160 3200 3010 3350 3040 3200 3000 3240
% 
{\color[named]{Black}{%
\special{pn 4}%
\special{pa 3970 3830}%
\special{pa 3660 4140}%
\special{fp}%
\special{pa 4000 3920}%
\special{pa 3750 4170}%
\special{fp}%
\special{pa 4000 4040}%
\special{pa 3850 4190}%
\special{fp}%
\special{pa 4000 4160}%
\special{pa 3960 4200}%
\special{fp}%
\special{pa 3860 3820}%
\special{pa 3580 4100}%
\special{fp}%
\special{pa 3770 3790}%
\special{pa 3500 4060}%
\special{fp}%
\special{pa 3690 3750}%
\special{pa 3430 4010}%
\special{fp}%
\special{pa 3610 3710}%
\special{pa 3360 3960}%
\special{fp}%
\special{pa 3550 3650}%
\special{pa 3300 3900}%
\special{fp}%
\special{pa 3490 3590}%
\special{pa 3240 3840}%
\special{fp}%
\special{pa 3450 3510}%
\special{pa 3190 3770}%
\special{fp}%
\special{pa 3410 3430}%
\special{pa 3140 3700}%
\special{fp}%
\special{pa 3380 3340}%
\special{pa 3100 3620}%
\special{fp}%
\special{pa 3370 3230}%
\special{pa 3060 3540}%
\special{fp}%
\special{pa 3280 3200}%
\special{pa 3030 3450}%
\special{fp}%
\special{pa 3160 3200}%
\special{pa 3010 3350}%
\special{fp}%
\special{pa 3040 3200}%
\special{pa 3000 3240}%
\special{fp}%
}}%
% LINE 3 0 3 0 Black White
% 16 4970 3430 4230 4170 5000 3280 4080 4200 4960 3200 4000 4160 4260 3780 4000 4040 4090 3830 4000 3920 4840 3200 4580 3460 4720 3200 4630 3290 4910 3610 4410 4110
% 
{\color[named]{Black}{%
\special{pn 4}%
\special{pa 4970 3430}%
\special{pa 4230 4170}%
\special{fp}%
\special{pa 5000 3280}%
\special{pa 4080 4200}%
\special{fp}%
\special{pa 4960 3200}%
\special{pa 4000 4160}%
\special{fp}%
\special{pa 4260 3780}%
\special{pa 4000 4040}%
\special{fp}%
\special{pa 4090 3830}%
\special{pa 4000 3920}%
\special{fp}%
\special{pa 4840 3200}%
\special{pa 4580 3460}%
\special{fp}%
\special{pa 4720 3200}%
\special{pa 4630 3290}%
\special{fp}%
\special{pa 4910 3610}%
\special{pa 4410 4110}%
\special{fp}%
}}%
% STR 2 0 3 0 Black White
% 4 5020 3140 5020 3240 1 0 0 0
% $t$
\put(50.2000,-32.4000){\makebox(0,0)[lt]{$t$}}%
% STR 2 0 3 0 Black White
% 4 3970 4130 3970 4230 4 0 0 0
% $-t$
\put(39.7000,-42.3000){\makebox(0,0)[rt]{$-t$}}%
% STR 2 0 3 0 Black White
% 4 2970 3080 2970 3180 3 0 0 0
% $-t$
\put(29.7000,-31.8000){\makebox(0,0)[rb]{$-t$}}%
% STR 2 0 3 0 Black White
% 4 4040 2070 4040 2170 2 0 0 0
% $t$
\put(40.4000,-21.7000){\makebox(0,0)[lb]{$t$}}%
% STR 2 0 3 0 Black White
% 4 4600 3130 4600 3230 4 0 0 0
% $1$
\put(46.0000,-32.3000){\makebox(0,0)[rt]{$1$}}%
% STR 2 0 3 0 Black White
% 4 3400 3090 3400 3190 2 0 0 0
% $-1$
\put(34.0000,-31.9000){\makebox(0,0)[lb]{$-1$}}%
% STR 2 0 3 0 Black White
% 4 4010 2530 4010 2630 1 0 0 0
% $1$
\put(40.1000,-26.3000){\makebox(0,0)[lt]{$1$}}%
% STR 2 0 3 0 Black White
% 4 3990 3680 3990 3780 3 0 0 0
% $-1$
\put(39.9000,-37.8000){\makebox(0,0)[rb]{$-1$}}%
% STR 2 0 3 0 Black White
% 4 5370 2110 5370 2210 2 0 0 0
% $A(t_i)$
\put(53.7000,-22.1000){\makebox(0,0)[lb]{$A(t_i)$}}%
% STR 2 0 3 0 Black White
% 4 2820 2270 2820 2370 5 0 1 0
% $B$
\put(28.2000,-23.7000){\makebox(0,0){{\colorbox[named]{White}{$B$}}}}%
\put(28.2000,-23.7000){\makebox(0,0){{\colorbox[named]{White}{\color[named]{Black}{$B$}}}}}%
% LINE 3 0 3 0 Black White
% 2 3480 2850 2880 2420
% 
{\color[named]{Black}{%
\special{pn 4}%
\special{pa 3480 2850}%
\special{pa 2880 2420}%
\special{fp}%
}}%
% LINE 3 0 3 0 Black White
% 2 4750 2530 5350 2200
% 
{\color[named]{Black}{%
\special{pn 4}%
\special{pa 4750 2530}%
\special{pa 5350 2200}%
\special{fp}%
}}%
% STR 2 0 3 0 Black White
% 4 5170 2450 5170 2550 1 0 1 0
% $(x_i,y_i)$
\put(51.7000,-25.5000){\makebox(0,0)[lt]{{\colorbox[named]{White}{$(x_i,y_i)$}}}}%
\put(51.7000,-25.5000){\makebox(0,0)[lt]{{\colorbox[named]{White}{\color[named]{Black}{$(x_i,y_i)$}}}}}%
% DOT 0 0 3 0 Black White
% 2 4690 2830 4690 2830
% 
{\color[named]{Black}{%
\special{pn 4}%
\special{sh 1}%
\special{ar 4690 2830 16 16 0  6.28318530717959E+0000}%
\special{sh 1}%
\special{ar 4690 2830 16 16 0  6.28318530717959E+0000}%
}}%
% LINE 3 0 3 0 Black White
% 2 4690 2830 5200 2690
% 
{\color[named]{Black}{%
\special{pn 4}%
\special{pa 4690 2830}%
\special{pa 5200 2690}%
\special{fp}%
}}%
\end{picture}}%

%% file: sec3_2_exC.tex
%WinTpicVersion4.28b
{\unitlength 0.1in
\begin{picture}( 20.3000,  5.0000)( 10.0000,-15.5000)
% VECTOR 2 0 3 0 Black White
% 2 1000 1400 3000 1400
% 
{\color[named]{Black}{%
\special{pn 8}%
\special{pa 1000 1400}%
\special{pa 3000 1400}%
\special{fp}%
\special{sh 1}%
\special{pa 3000 1400}%
\special{pa 2934 1380}%
\special{pa 2948 1400}%
\special{pa 2934 1420}%
\special{pa 3000 1400}%
\special{fp}%
}}%
% DOT 0 0 3 0 Black White
% 1 1200 1400
% 
{\color[named]{Black}{%
\special{pn 4}%
\special{sh 1}%
\special{ar 1200 1400 16 16 0  6.28318530717959E+0000}%
}}%
% DOT 0 0 3 0 Black White
% 2 2000 1400 2600 1400
% 
{\color[named]{Black}{%
\special{pn 4}%
\special{sh 1}%
\special{ar 2000 1400 16 16 0  6.28318530717959E+0000}%
\special{sh 1}%
\special{ar 2600 1400 16 16 0  6.28318530717959E+0000}%
}}%
% STR 2 0 3 0 Black White
% 4 2510 1300 2510 1400 2 0 0 0
% 
\put(25.1000,-14.0000){\makebox(0,0)[lb]{}}%
% STR 2 0 3 0 Black White
% 4 2510 1300 2510 1400 2 0 0 0
% 
\put(25.1000,-14.0000){\makebox(0,0)[lb]{}}%
% STR 2 0 3 0 Black White
% 4 1620 1450 1620 1550 2 0 0 0
% 
\put(16.2000,-15.5000){\makebox(0,0)[lb]{}}%
% STR 2 0 3 0 Black White
% 4 1920 1340 1920 1440 1 0 0 0
% $t$
\put(19.2000,-14.4000){\makebox(0,0)[lt]{$t$}}%
% STR 2 0 3 0 Black White
% 4 2600 1410 2600 1510 5 0 0 0
% $1$
\put(26.0000,-15.1000){\makebox(0,0){$1$}}%
% STR 2 0 3 0 Black White
% 4 1200 1410 1200 1510 5 0 0 0
% $O$
\put(12.0000,-15.1000){\makebox(0,0){$O$}}%
% LINE 3 0 3 0 Black White
% 2 2000 1400 1800 1200
% 
{\color[named]{Black}{%
\special{pn 4}%
\special{pa 2000 1400}%
\special{pa 1800 1200}%
\special{fp}%
}}%
% STR 2 0 3 0 Black White
% 4 1890 1090 1890 1190 3 0 0 0
% $C(t)$
\put(18.9000,-11.9000){\makebox(0,0)[rb]{$C(t)$}}%
% LINE 3 0 3 0 Black White
% 2 2600 1400 2400 1200
% 
{\color[named]{Black}{%
\special{pn 4}%
\special{pa 2600 1400}%
\special{pa 2400 1200}%
\special{fp}%
}}%
% STR 2 0 3 0 Black White
% 4 2420 1080 2420 1180 3 0 0 0
% $C$
\put(24.2000,-11.8000){\makebox(0,0)[rb]{$C$}}%
% STR 2 0 3 0 Black White
% 4 3030 1260 3030 1360 1 0 0 0
% $\mathbb{R}$
\put(30.3000,-13.6000){\makebox(0,0)[lt]{$\mathbb{R}$}}%
\end{picture}}%

%% file: sec5_1_es1.tex
%WinTpicVersion4.28b
{\unitlength 0.1in
\begin{picture}( 36.2000, 14.0500)( 14.0000,-26.0000)
% VECTOR 2 0 3 0 Black White
% 2 4480 2400 5020 2400
% 
{\color[named]{Black}{%
\special{pn 8}%
\special{pa 4480 2400}%
\special{pa 5020 2400}%
\special{fp}%
\special{sh 1}%
\special{pa 5020 2400}%
\special{pa 4954 2380}%
\special{pa 4968 2400}%
\special{pa 4954 2420}%
\special{pa 5020 2400}%
\special{fp}%
}}%
% VECTOR 2 0 3 0 Black White
% 2 4480 2000 5020 2000
% 
{\color[named]{Black}{%
\special{pn 8}%
\special{pa 4480 2000}%
\special{pa 5020 2000}%
\special{fp}%
\special{sh 1}%
\special{pa 5020 2000}%
\special{pa 4954 1980}%
\special{pa 4968 2000}%
\special{pa 4954 2020}%
\special{pa 5020 2000}%
\special{fp}%
}}%
% VECTOR 2 0 3 0 Black White
% 2 4480 1600 5020 1600
% 
{\color[named]{Black}{%
\special{pn 8}%
\special{pa 4480 1600}%
\special{pa 5020 1600}%
\special{fp}%
\special{sh 1}%
\special{pa 5020 1600}%
\special{pa 4954 1580}%
\special{pa 4968 1600}%
\special{pa 4954 1620}%
\special{pa 5020 1600}%
\special{fp}%
}}%
% LINE 2 2 3 0 Black White
% 2 2210 1200 2210 2600
% 
{\color[named]{Black}{%
\special{pn 8}%
\special{pa 2210 1200}%
\special{pa 2210 2600}%
\special{dt 0.045}%
}}%
% LINE 2 2 3 0 Black White
% 2 3200 1200 3200 2600
% 
{\color[named]{Black}{%
\special{pn 8}%
\special{pa 3200 1200}%
\special{pa 3200 2600}%
\special{dt 0.045}%
}}%
% LINE 2 2 3 0 Black White
% 2 4200 1200 4200 2600
% 
{\color[named]{Black}{%
\special{pn 8}%
\special{pa 4200 1200}%
\special{pa 4200 2600}%
\special{dt 0.045}%
}}%
% STR 2 0 3 0 Black White
% 4 2200 1160 2200 1260 5 0 1 0
% $t_1$
\put(22.0000,-12.6000){\makebox(0,0){{\colorbox[named]{White}{$t_1$}}}}%
\put(22.0000,-12.6000){\makebox(0,0){{\colorbox[named]{White}{\color[named]{Black}{$t_1$}}}}}%
% STR 2 0 3 0 Black White
% 4 3200 1160 3200 1260 5 0 1 0
% $t_2$
\put(32.0000,-12.6000){\makebox(0,0){{\colorbox[named]{White}{$t_2$}}}}%
\put(32.0000,-12.6000){\makebox(0,0){{\colorbox[named]{White}{\color[named]{Black}{$t_2$}}}}}%
% STR 2 0 3 0 Black White
% 4 4200 1160 4200 1260 5 0 1 0
% $t_3$
\put(42.0000,-12.6000){\makebox(0,0){{\colorbox[named]{White}{$t_3$}}}}%
\put(42.0000,-12.6000){\makebox(0,0){{\colorbox[named]{White}{\color[named]{Black}{$t_3$}}}}}%
% STR 2 0 3 0 Black White
% 4 4200 1500 4200 1600 5 0 1 0
% $a_{\lambda_i}(t_3)$
\put(42.0000,-16.0000){\makebox(0,0){{\colorbox[named]{White}{$a_{\lambda_i}(t_3)$}}}}%
\put(42.0000,-16.0000){\makebox(0,0){{\colorbox[named]{White}{\color[named]{Black}{$a_{\lambda_i}(t_3)$}}}}}%
% STR 2 0 3 0 Black White
% 4 4200 1900 4200 2000 5 0 1 0
% $a_{\lambda_j}(t_3)$
\put(42.0000,-20.0000){\makebox(0,0){{\colorbox[named]{White}{$a_{\lambda_j}(t_3)$}}}}%
\put(42.0000,-20.0000){\makebox(0,0){{\colorbox[named]{White}{\color[named]{Black}{$a_{\lambda_j}(t_3)$}}}}}%
% STR 2 0 3 0 Black White
% 4 4200 2300 4200 2400 5 0 1 0
% $a_{\lambda_k}(t_3)$
\put(42.0000,-24.0000){\makebox(0,0){{\colorbox[named]{White}{$a_{\lambda_k}(t_3)$}}}}%
\put(42.0000,-24.0000){\makebox(0,0){{\colorbox[named]{White}{\color[named]{Black}{$a_{\lambda_k}(t_3)$}}}}}%
% VECTOR 2 0 3 0 Black White
% 2 2400 2400 2940 2400
% 
{\color[named]{Black}{%
\special{pn 8}%
\special{pa 2400 2400}%
\special{pa 2940 2400}%
\special{fp}%
\special{sh 1}%
\special{pa 2940 2400}%
\special{pa 2874 2380}%
\special{pa 2888 2400}%
\special{pa 2874 2420}%
\special{pa 2940 2400}%
\special{fp}%
}}%
% STR 2 0 3 0 Black White
% 4 2200 2300 2200 2400 5 0 1 0
% $a_{\lambda_k}(t_1)$
\put(22.0000,-24.0000){\makebox(0,0){{\colorbox[named]{White}{$a_{\lambda_k}(t_1)$}}}}%
\put(22.0000,-24.0000){\makebox(0,0){{\colorbox[named]{White}{\color[named]{Black}{$a_{\lambda_k}(t_1)$}}}}}%
% VECTOR 2 0 3 0 Black White
% 2 3400 2000 3960 2000
% 
{\color[named]{Black}{%
\special{pn 8}%
\special{pa 3400 2000}%
\special{pa 3960 2000}%
\special{fp}%
\special{sh 1}%
\special{pa 3960 2000}%
\special{pa 3894 1980}%
\special{pa 3908 2000}%
\special{pa 3894 2020}%
\special{pa 3960 2000}%
\special{fp}%
}}%
% VECTOR 2 0 3 0 Black White
% 2 3400 2400 3960 2400
% 
{\color[named]{Black}{%
\special{pn 8}%
\special{pa 3400 2400}%
\special{pa 3960 2400}%
\special{fp}%
\special{sh 1}%
\special{pa 3960 2400}%
\special{pa 3894 2380}%
\special{pa 3908 2400}%
\special{pa 3894 2420}%
\special{pa 3960 2400}%
\special{fp}%
}}%
% STR 2 0 3 0 Black White
% 4 3200 2300 3200 2400 5 0 1 0
% $a_{\lambda_k}(t_2)$
\put(32.0000,-24.0000){\makebox(0,0){{\colorbox[named]{White}{$a_{\lambda_k}(t_2)$}}}}%
\put(32.0000,-24.0000){\makebox(0,0){{\colorbox[named]{White}{\color[named]{Black}{$a_{\lambda_k}(t_2)$}}}}}%
% VECTOR 2 0 3 0 Black White
% 2 2400 1600 2940 1600
% 
{\color[named]{Black}{%
\special{pn 8}%
\special{pa 2400 1600}%
\special{pa 2940 1600}%
\special{fp}%
\special{sh 1}%
\special{pa 2940 1600}%
\special{pa 2874 1580}%
\special{pa 2888 1600}%
\special{pa 2874 1620}%
\special{pa 2940 1600}%
\special{fp}%
}}%
% STR 2 0 3 0 Black White
% 4 2200 1500 2200 1600 5 0 1 0
% $a_{\lambda_i}(t_1)$
\put(22.0000,-16.0000){\makebox(0,0){{\colorbox[named]{White}{$a_{\lambda_i}(t_1)$}}}}%
\put(22.0000,-16.0000){\makebox(0,0){{\colorbox[named]{White}{\color[named]{Black}{$a_{\lambda_i}(t_1)$}}}}}%
% VECTOR 2 0 3 0 Black White
% 2 2400 2000 2940 2000
% 
{\color[named]{Black}{%
\special{pn 8}%
\special{pa 2400 2000}%
\special{pa 2940 2000}%
\special{fp}%
\special{sh 1}%
\special{pa 2940 2000}%
\special{pa 2874 1980}%
\special{pa 2888 2000}%
\special{pa 2874 2020}%
\special{pa 2940 2000}%
\special{fp}%
}}%
% STR 2 0 3 0 Black White
% 4 2200 1900 2200 2000 5 0 1 0
% $a_{\lambda_j}(t_1)$
\put(22.0000,-20.0000){\makebox(0,0){{\colorbox[named]{White}{$a_{\lambda_j}(t_1)$}}}}%
\put(22.0000,-20.0000){\makebox(0,0){{\colorbox[named]{White}{\color[named]{Black}{$a_{\lambda_j}(t_1)$}}}}}%
% VECTOR 2 0 3 0 Black White
% 2 3400 1600 3960 1600
% 
{\color[named]{Black}{%
\special{pn 8}%
\special{pa 3400 1600}%
\special{pa 3960 1600}%
\special{fp}%
\special{sh 1}%
\special{pa 3960 1600}%
\special{pa 3894 1580}%
\special{pa 3908 1600}%
\special{pa 3894 1620}%
\special{pa 3960 1600}%
\special{fp}%
}}%
% STR 2 0 3 0 Black White
% 4 3200 1500 3200 1600 5 0 1 0
% $a_{\lambda_i}(t_2)$
\put(32.0000,-16.0000){\makebox(0,0){{\colorbox[named]{White}{$a_{\lambda_i}(t_2)$}}}}%
\put(32.0000,-16.0000){\makebox(0,0){{\colorbox[named]{White}{\color[named]{Black}{$a_{\lambda_i}(t_2)$}}}}}%
% STR 2 0 3 0 Black White
% 4 3200 1900 3200 2000 5 0 1 0
% $a_{\lambda_j}(t_2)$
\put(32.0000,-20.0000){\makebox(0,0){{\colorbox[named]{White}{$a_{\lambda_j}(t_2)$}}}}%
\put(32.0000,-20.0000){\makebox(0,0){{\colorbox[named]{White}{\color[named]{Black}{$a_{\lambda_j}(t_2)$}}}}}%
% VECTOR 2 0 3 0 Black White
% 2 1400 1600 1940 1600
% 
{\color[named]{Black}{%
\special{pn 8}%
\special{pa 1400 1600}%
\special{pa 1940 1600}%
\special{fp}%
\special{sh 1}%
\special{pa 1940 1600}%
\special{pa 1874 1580}%
\special{pa 1888 1600}%
\special{pa 1874 1620}%
\special{pa 1940 1600}%
\special{fp}%
}}%
% VECTOR 2 0 3 0 Black White
% 2 1400 2000 1940 2000
% 
{\color[named]{Black}{%
\special{pn 8}%
\special{pa 1400 2000}%
\special{pa 1940 2000}%
\special{fp}%
\special{sh 1}%
\special{pa 1940 2000}%
\special{pa 1874 1980}%
\special{pa 1888 2000}%
\special{pa 1874 2020}%
\special{pa 1940 2000}%
\special{fp}%
}}%
% VECTOR 2 0 3 0 Black White
% 2 1400 2400 1940 2400
% 
{\color[named]{Black}{%
\special{pn 8}%
\special{pa 1400 2400}%
\special{pa 1940 2400}%
\special{fp}%
\special{sh 1}%
\special{pa 1940 2400}%
\special{pa 1874 2380}%
\special{pa 1888 2400}%
\special{pa 1874 2420}%
\special{pa 1940 2400}%
\special{fp}%
}}%
\end{picture}}%

%% file: sec5_1_es2.tex
%WinTpicVersion4.28b
{\unitlength 0.1in
\begin{picture}( 34.8000, 24.0000)( 17.9000,-44.0000)
% STR 2 0 3 0 Black White
% 4 3990 3207 3990 3220 4 2800 0 0
% O
\put(39.9000,-32.2000){\makebox(0,0)[rt]{O}}%
% STR 2 0 3 0 Black White
% 4 3960 1987 3960 2000 4 2800 0 0
% $y$
\put(39.6000,-20.0000){\makebox(0,0)[rt]{$y$}}%
% STR 2 0 3 0 Black White
% 4 5210 3247 5210 3260 4 2800 0 0
% $x$
\put(52.1000,-32.6000){\makebox(0,0)[rt]{$x$}}%
% VECTOR 2 0 3 0 Black White
% 2 4000 4400 4000 2000
% 
{\color[named]{Black}{%
\special{pn 8}%
\special{pa 4000 4400}%
\special{pa 4000 2000}%
\special{fp}%
\special{sh 1}%
\special{pa 4000 2000}%
\special{pa 3980 2068}%
\special{pa 4000 2054}%
\special{pa 4020 2068}%
\special{pa 4000 2000}%
\special{fp}%
}}%
% VECTOR 2 0 3 0 Black White
% 2 2800 3200 5200 3200
% 
{\color[named]{Black}{%
\special{pn 8}%
\special{pa 2800 3200}%
\special{pa 5200 3200}%
\special{fp}%
\special{sh 1}%
\special{pa 5200 3200}%
\special{pa 5134 3180}%
\special{pa 5148 3200}%
\special{pa 5134 3220}%
\special{pa 5200 3200}%
\special{fp}%
}}%
% FUNC 2 1 3 0 Black White
% 9 2800 2000 5200 4400 4000 3200 5000 3200 4000 2200 2200 1600 2200 1600 10 2 0 4
% ///x^2+y^2=1///-1.199///1.199
{\color[named]{Black}{%
\special{pn 8}%
\special{pa 5000 3200}%
\special{pa 5000 3146}%
\special{pa 4998 3146}%
\special{pa 4998 3139}%
\special{fp}%
\special{pa 4993 3078}%
\special{pa 4992 3078}%
\special{pa 4992 3070}%
\special{pa 4990 3056}%
\special{pa 4990 3050}%
\special{pa 4988 3048}%
\special{pa 4988 3038}%
\special{pa 4986 3036}%
\special{pa 4986 3026}%
\special{pa 4984 3024}%
\special{pa 4984 3019}%
\special{fp}%
\special{pa 4972 2966}%
\special{pa 4972 2962}%
\special{pa 4970 2956}%
\special{pa 4970 2952}%
\special{pa 4968 2950}%
\special{pa 4968 2948}%
\special{pa 4966 2940}%
\special{pa 4966 2936}%
\special{pa 4964 2936}%
\special{pa 4964 2930}%
\special{pa 4962 2928}%
\special{pa 4962 2922}%
\special{pa 4960 2922}%
\special{pa 4960 2916}%
\special{pa 4958 2914}%
\special{pa 4958 2910}%
\special{pa 4958 2910}%
\special{fp}%
\special{pa 4941 2861}%
\special{pa 4940 2860}%
\special{pa 4940 2856}%
\special{pa 4938 2854}%
\special{pa 4938 2852}%
\special{pa 4936 2848}%
\special{pa 4936 2846}%
\special{pa 4934 2844}%
\special{pa 4934 2840}%
\special{pa 4932 2838}%
\special{pa 4932 2834}%
\special{pa 4930 2834}%
\special{pa 4930 2830}%
\special{pa 4928 2828}%
\special{pa 4928 2824}%
\special{pa 4926 2824}%
\special{pa 4926 2820}%
\special{pa 4924 2818}%
\special{pa 4924 2816}%
\special{pa 4922 2814}%
\special{pa 4922 2810}%
\special{pa 4920 2810}%
\special{pa 4920 2809}%
\special{fp}%
\special{pa 4897 2757}%
\special{pa 4896 2756}%
\special{pa 4896 2754}%
\special{pa 4894 2752}%
\special{pa 4894 2750}%
\special{pa 4890 2746}%
\special{pa 4890 2742}%
\special{pa 4888 2742}%
\special{pa 4888 2738}%
\special{pa 4886 2738}%
\special{pa 4886 2734}%
\special{pa 4884 2734}%
\special{pa 4884 2730}%
\special{pa 4882 2730}%
\special{pa 4882 2726}%
\special{pa 4878 2722}%
\special{pa 4878 2720}%
\special{pa 4876 2718}%
\special{pa 4876 2716}%
\special{pa 4874 2714}%
\special{pa 4872 2710}%
\special{pa 4872 2708}%
\special{fp}%
\special{pa 4839 2657}%
\special{pa 4838 2656}%
\special{pa 4838 2654}%
\special{pa 4836 2652}%
\special{pa 4836 2650}%
\special{pa 4834 2648}%
\special{pa 4834 2646}%
\special{pa 4826 2638}%
\special{pa 4826 2636}%
\special{pa 4822 2632}%
\special{pa 4822 2630}%
\special{pa 4820 2628}%
\special{pa 4818 2624}%
\special{pa 4816 2622}%
\special{pa 4816 2620}%
\special{pa 4812 2618}%
\special{pa 4812 2616}%
\special{pa 4810 2614}%
\special{pa 4808 2610}%
\special{pa 4806 2608}%
\special{fp}%
\special{pa 4767 2559}%
\special{pa 4762 2554}%
\special{pa 4762 2552}%
\special{pa 4750 2540}%
\special{pa 4750 2538}%
\special{pa 4742 2530}%
\special{pa 4742 2528}%
\special{pa 4727 2513}%
\special{fp}%
\special{pa 4680 2468}%
\special{pa 4656 2444}%
\special{pa 4654 2444}%
\special{pa 4642 2432}%
\special{pa 4640 2432}%
\special{pa 4636 2428}%
\special{pa 4634 2428}%
\special{fp}%
\special{pa 4588 2391}%
\special{pa 4588 2390}%
\special{pa 4586 2390}%
\special{pa 4584 2388}%
\special{pa 4580 2386}%
\special{pa 4578 2384}%
\special{pa 4576 2384}%
\special{pa 4574 2380}%
\special{pa 4572 2380}%
\special{pa 4568 2376}%
\special{pa 4566 2376}%
\special{pa 4562 2372}%
\special{pa 4560 2372}%
\special{pa 4558 2370}%
\special{pa 4556 2370}%
\special{pa 4552 2366}%
\special{pa 4548 2364}%
\special{pa 4544 2360}%
\special{pa 4542 2360}%
\special{pa 4539 2357}%
\special{fp}%
\special{pa 4487 2327}%
\special{pa 4486 2326}%
\special{pa 4478 2322}%
\special{pa 4476 2322}%
\special{pa 4472 2318}%
\special{pa 4470 2318}%
\special{pa 4468 2316}%
\special{pa 4466 2316}%
\special{pa 4464 2314}%
\special{pa 4462 2314}%
\special{pa 4460 2312}%
\special{pa 4458 2312}%
\special{pa 4454 2308}%
\special{pa 4450 2308}%
\special{pa 4450 2306}%
\special{pa 4446 2306}%
\special{pa 4446 2304}%
\special{pa 4442 2304}%
\special{pa 4442 2302}%
\special{pa 4438 2302}%
\special{pa 4437 2301}%
\special{fp}%
\special{pa 4384 2277}%
\special{pa 4382 2276}%
\special{pa 4380 2276}%
\special{pa 4378 2274}%
\special{pa 4374 2274}%
\special{pa 4374 2272}%
\special{pa 4370 2272}%
\special{pa 4368 2270}%
\special{pa 4364 2270}%
\special{pa 4364 2268}%
\special{pa 4362 2268}%
\special{pa 4356 2266}%
\special{pa 4354 2266}%
\special{pa 4354 2264}%
\special{pa 4348 2264}%
\special{pa 4348 2262}%
\special{pa 4346 2262}%
\special{pa 4340 2260}%
\special{pa 4338 2260}%
\special{pa 4338 2258}%
\special{pa 4334 2258}%
\special{fp}%
\special{pa 4283 2242}%
\special{pa 4282 2242}%
\special{pa 4282 2240}%
\special{pa 4276 2240}%
\special{pa 4274 2238}%
\special{pa 4268 2238}%
\special{pa 4268 2236}%
\special{pa 4264 2236}%
\special{pa 4258 2234}%
\special{pa 4254 2234}%
\special{pa 4252 2232}%
\special{pa 4250 2232}%
\special{pa 4242 2230}%
\special{pa 4238 2230}%
\special{pa 4238 2228}%
\special{pa 4232 2228}%
\special{pa 4229 2227}%
\special{fp}%
\special{pa 4169 2215}%
\special{pa 4152 2212}%
\special{pa 4138 2210}%
\special{pa 4132 2210}%
\special{pa 4130 2208}%
\special{pa 4122 2208}%
\special{pa 4107 2206}%
\special{fp}%
\special{pa 4044 2202}%
\special{pa 4032 2202}%
\special{pa 4032 2200}%
\special{pa 3983 2200}%
\special{fp}%
\special{pa 3922 2204}%
\special{pa 3906 2204}%
\special{pa 3906 2206}%
\special{pa 3896 2206}%
\special{pa 3878 2208}%
\special{pa 3870 2208}%
\special{pa 3870 2210}%
\special{pa 3863 2210}%
\special{fp}%
\special{pa 3803 2220}%
\special{pa 3794 2222}%
\special{pa 3790 2222}%
\special{pa 3790 2224}%
\special{pa 3782 2224}%
\special{pa 3780 2226}%
\special{pa 3776 2226}%
\special{pa 3768 2228}%
\special{pa 3764 2228}%
\special{pa 3762 2230}%
\special{pa 3756 2230}%
\special{pa 3754 2232}%
\special{pa 3748 2232}%
\special{pa 3748 2234}%
\special{pa 3747 2234}%
\special{fp}%
\special{pa 3691 2250}%
\special{pa 3666 2258}%
\special{pa 3664 2258}%
\special{pa 3662 2260}%
\special{pa 3660 2260}%
\special{pa 3656 2262}%
\special{pa 3654 2262}%
\special{pa 3652 2264}%
\special{pa 3648 2264}%
\special{pa 3646 2266}%
\special{pa 3644 2266}%
\special{pa 3640 2268}%
\special{pa 3638 2268}%
\special{pa 3636 2270}%
\special{pa 3634 2270}%
\special{fp}%
\special{pa 3583 2292}%
\special{pa 3580 2292}%
\special{pa 3580 2294}%
\special{pa 3576 2294}%
\special{pa 3576 2296}%
\special{pa 3572 2296}%
\special{pa 3572 2298}%
\special{pa 3568 2298}%
\special{pa 3568 2300}%
\special{pa 3564 2300}%
\special{pa 3564 2302}%
\special{pa 3560 2302}%
\special{pa 3558 2304}%
\special{pa 3556 2304}%
\special{pa 3554 2306}%
\special{pa 3552 2306}%
\special{pa 3550 2308}%
\special{pa 3548 2308}%
\special{pa 3546 2310}%
\special{pa 3535 2316}%
\special{fp}%
\special{pa 3484 2344}%
\special{pa 3482 2346}%
\special{pa 3480 2346}%
\special{pa 3478 2348}%
\special{pa 3476 2348}%
\special{pa 3470 2354}%
\special{pa 3468 2354}%
\special{pa 3466 2356}%
\special{pa 3464 2356}%
\special{pa 3462 2358}%
\special{pa 3460 2358}%
\special{pa 3456 2362}%
\special{pa 3454 2362}%
\special{pa 3448 2368}%
\special{pa 3446 2368}%
\special{pa 3444 2370}%
\special{pa 3440 2372}%
\special{pa 3438 2374}%
\special{pa 3436 2374}%
\special{pa 3434 2376}%
\special{fp}%
\special{pa 3384 2414}%
\special{pa 3382 2416}%
\special{pa 3380 2416}%
\special{pa 3372 2424}%
\special{pa 3370 2424}%
\special{pa 3366 2428}%
\special{pa 3360 2432}%
\special{pa 3354 2438}%
\special{pa 3352 2438}%
\special{pa 3340 2450}%
\special{pa 3338 2450}%
\special{pa 3336 2452}%
\special{fp}%
\special{pa 3291 2497}%
\special{pa 3260 2528}%
\special{pa 3260 2530}%
\special{pa 3250 2538}%
\special{pa 3250 2540}%
\special{pa 3248 2542}%
\special{fp}%
\special{pa 3208 2590}%
\special{pa 3208 2590}%
\special{pa 3206 2594}%
\special{pa 3202 2598}%
\special{pa 3200 2602}%
\special{pa 3196 2606}%
\special{pa 3194 2610}%
\special{pa 3188 2616}%
\special{pa 3188 2618}%
\special{pa 3186 2620}%
\special{pa 3186 2622}%
\special{pa 3178 2630}%
\special{pa 3178 2632}%
\special{pa 3174 2636}%
\special{pa 3174 2638}%
\special{pa 3173 2639}%
\special{fp}%
\special{pa 3140 2691}%
\special{pa 3140 2692}%
\special{pa 3138 2694}%
\special{pa 3136 2698}%
\special{pa 3136 2700}%
\special{pa 3132 2704}%
\special{pa 3132 2706}%
\special{pa 3130 2708}%
\special{pa 3128 2712}%
\special{pa 3128 2714}%
\special{pa 3124 2718}%
\special{pa 3124 2720}%
\special{pa 3122 2722}%
\special{pa 3122 2724}%
\special{pa 3120 2726}%
\special{pa 3111 2744}%
\special{fp}%
\special{pa 3086 2797}%
\special{pa 3086 2798}%
\special{pa 3084 2798}%
\special{pa 3084 2802}%
\special{pa 3082 2804}%
\special{pa 3082 2806}%
\special{pa 3080 2808}%
\special{pa 3080 2812}%
\special{pa 3078 2812}%
\special{pa 3078 2816}%
\special{pa 3076 2818}%
\special{pa 3076 2822}%
\special{pa 3074 2822}%
\special{pa 3074 2826}%
\special{pa 3072 2828}%
\special{pa 3072 2832}%
\special{pa 3070 2832}%
\special{pa 3070 2836}%
\special{pa 3068 2838}%
\special{pa 3068 2840}%
\special{pa 3066 2844}%
\special{pa 3066 2846}%
\special{pa 3065 2847}%
\special{fp}%
\special{pa 3046 2904}%
\special{pa 3046 2906}%
\special{pa 3044 2906}%
\special{pa 3044 2912}%
\special{pa 3042 2912}%
\special{pa 3042 2918}%
\special{pa 3040 2920}%
\special{pa 3040 2926}%
\special{pa 3038 2926}%
\special{pa 3038 2932}%
\special{pa 3036 2934}%
\special{pa 3036 2936}%
\special{pa 3034 2944}%
\special{pa 3034 2948}%
\special{pa 3032 2948}%
\special{pa 3032 2952}%
\special{pa 3030 2957}%
\special{fp}%
\special{pa 3018 3017}%
\special{pa 3018 3020}%
\special{pa 3016 3020}%
\special{pa 3016 3026}%
\special{pa 3014 3036}%
\special{pa 3014 3038}%
\special{pa 3012 3048}%
\special{pa 3012 3050}%
\special{pa 3010 3062}%
\special{pa 3010 3070}%
\special{pa 3008 3070}%
\special{pa 3008 3076}%
\special{fp}%
\special{pa 3002 3136}%
\special{pa 3002 3168}%
\special{pa 3000 3170}%
\special{pa 3000 3198}%
\special{fp}%
\special{pa 3002 3260}%
\special{pa 3002 3270}%
\special{pa 3004 3272}%
\special{pa 3004 3294}%
\special{pa 3006 3296}%
\special{pa 3006 3306}%
\special{pa 3008 3322}%
\special{fp}%
\special{pa 3018 3383}%
\special{pa 3018 3388}%
\special{pa 3020 3396}%
\special{pa 3020 3398}%
\special{pa 3022 3406}%
\special{pa 3022 3410}%
\special{pa 3024 3412}%
\special{pa 3024 3420}%
\special{pa 3026 3420}%
\special{pa 3026 3426}%
\special{pa 3028 3432}%
\special{pa 3028 3438}%
\special{pa 3030 3438}%
\special{pa 3030 3440}%
\special{fp}%
\special{pa 3045 3495}%
\special{pa 3046 3496}%
\special{pa 3046 3500}%
\special{pa 3048 3504}%
\special{pa 3048 3506}%
\special{pa 3050 3510}%
\special{pa 3050 3512}%
\special{pa 3052 3516}%
\special{pa 3052 3518}%
\special{pa 3054 3522}%
\special{pa 3054 3524}%
\special{pa 3056 3528}%
\special{pa 3056 3530}%
\special{pa 3058 3534}%
\special{pa 3058 3538}%
\special{pa 3060 3538}%
\special{pa 3060 3540}%
\special{pa 3062 3546}%
\special{pa 3062 3548}%
\special{pa 3064 3548}%
\special{pa 3064 3551}%
\special{fp}%
\special{pa 3084 3602}%
\special{pa 3086 3604}%
\special{pa 3086 3606}%
\special{pa 3088 3608}%
\special{pa 3088 3610}%
\special{pa 3092 3618}%
\special{pa 3092 3620}%
\special{pa 3094 3622}%
\special{pa 3094 3624}%
\special{pa 3096 3626}%
\special{pa 3096 3628}%
\special{pa 3098 3630}%
\special{pa 3098 3632}%
\special{pa 3100 3634}%
\special{pa 3100 3636}%
\special{pa 3102 3638}%
\special{pa 3102 3642}%
\special{pa 3104 3642}%
\special{pa 3104 3646}%
\special{pa 3106 3646}%
\special{pa 3106 3650}%
\special{pa 3108 3650}%
\special{pa 3108 3652}%
\special{fp}%
\special{pa 3136 3702}%
\special{pa 3136 3702}%
\special{pa 3136 3704}%
\special{pa 3140 3708}%
\special{pa 3140 3712}%
\special{pa 3146 3718}%
\special{pa 3146 3722}%
\special{pa 3150 3726}%
\special{pa 3152 3730}%
\special{pa 3156 3734}%
\special{pa 3156 3738}%
\special{pa 3160 3742}%
\special{pa 3160 3744}%
\special{pa 3164 3748}%
\special{pa 3166 3752}%
\special{pa 3167 3753}%
\special{fp}%
\special{pa 3201 3801}%
\special{pa 3202 3802}%
\special{pa 3202 3804}%
\special{pa 3206 3806}%
\special{pa 3206 3808}%
\special{pa 3208 3810}%
\special{pa 3208 3812}%
\special{pa 3212 3814}%
\special{pa 3212 3816}%
\special{pa 3220 3824}%
\special{pa 3220 3826}%
\special{pa 3228 3834}%
\special{pa 3228 3836}%
\special{pa 3232 3840}%
\special{pa 3232 3842}%
\special{pa 3237 3847}%
\special{fp}%
\special{pa 3280 3894}%
\special{pa 3324 3938}%
\special{fp}%
\special{pa 3372 3978}%
\special{pa 3380 3986}%
\special{pa 3382 3986}%
\special{pa 3386 3990}%
\special{pa 3390 3992}%
\special{pa 3394 3996}%
\special{pa 3398 3998}%
\special{pa 3402 4002}%
\special{pa 3406 4004}%
\special{pa 3410 4008}%
\special{pa 3414 4010}%
\special{pa 3416 4012}%
\special{pa 3418 4012}%
\special{pa 3420 4016}%
\special{pa 3420 4016}%
\special{fp}%
\special{pa 3471 4049}%
\special{pa 3472 4050}%
\special{pa 3474 4050}%
\special{pa 3476 4052}%
\special{pa 3478 4052}%
\special{pa 3484 4058}%
\special{pa 3488 4058}%
\special{pa 3494 4064}%
\special{pa 3498 4064}%
\special{pa 3498 4066}%
\special{pa 3500 4066}%
\special{pa 3502 4068}%
\special{pa 3504 4068}%
\special{pa 3508 4072}%
\special{pa 3512 4072}%
\special{pa 3512 4074}%
\special{pa 3514 4074}%
\special{pa 3516 4076}%
\special{pa 3518 4076}%
\special{pa 3520 4078}%
\special{fp}%
\special{pa 3567 4102}%
\special{pa 3570 4102}%
\special{pa 3570 4104}%
\special{pa 3574 4104}%
\special{pa 3574 4106}%
\special{pa 3578 4106}%
\special{pa 3578 4108}%
\special{pa 3582 4108}%
\special{pa 3584 4110}%
\special{pa 3586 4110}%
\special{pa 3588 4112}%
\special{pa 3590 4112}%
\special{pa 3592 4114}%
\special{pa 3594 4114}%
\special{pa 3598 4116}%
\special{pa 3600 4116}%
\special{pa 3600 4118}%
\special{pa 3604 4118}%
\special{pa 3606 4120}%
\special{pa 3610 4120}%
\special{pa 3610 4122}%
\special{pa 3614 4122}%
\special{pa 3616 4124}%
\special{fp}%
\special{pa 3670 4144}%
\special{pa 3674 4146}%
\special{pa 3676 4146}%
\special{pa 3680 4148}%
\special{pa 3682 4148}%
\special{pa 3686 4150}%
\special{pa 3688 4150}%
\special{pa 3692 4152}%
\special{pa 3694 4152}%
\special{pa 3698 4154}%
\special{pa 3702 4154}%
\special{pa 3702 4156}%
\special{pa 3706 4156}%
\special{pa 3712 4158}%
\special{pa 3714 4158}%
\special{pa 3716 4160}%
\special{pa 3722 4160}%
\special{pa 3722 4162}%
\special{pa 3726 4162}%
\special{fp}%
\special{pa 3779 4176}%
\special{pa 3784 4176}%
\special{pa 3786 4178}%
\special{pa 3794 4178}%
\special{pa 3794 4180}%
\special{pa 3804 4180}%
\special{pa 3804 4182}%
\special{pa 3814 4182}%
\special{pa 3814 4184}%
\special{pa 3824 4184}%
\special{pa 3826 4186}%
\special{pa 3835 4186}%
\special{fp}%
\special{pa 3895 4194}%
\special{pa 3896 4194}%
\special{pa 3896 4196}%
\special{pa 3916 4196}%
\special{pa 3918 4198}%
\special{pa 3946 4198}%
\special{pa 3946 4200}%
\special{pa 3953 4200}%
\special{fp}%
\special{pa 4017 4200}%
\special{pa 4054 4200}%
\special{pa 4056 4198}%
\special{pa 4079 4198}%
\special{fp}%
\special{pa 4140 4191}%
\special{pa 4144 4190}%
\special{pa 4152 4190}%
\special{pa 4152 4188}%
\special{pa 4164 4188}%
\special{pa 4164 4186}%
\special{pa 4176 4186}%
\special{pa 4176 4184}%
\special{pa 4186 4184}%
\special{pa 4188 4182}%
\special{pa 4196 4182}%
\special{pa 4196 4182}%
\special{fp}%
\special{pa 4253 4168}%
\special{pa 4254 4168}%
\special{pa 4260 4166}%
\special{pa 4264 4166}%
\special{pa 4266 4164}%
\special{pa 4272 4164}%
\special{pa 4272 4162}%
\special{pa 4278 4162}%
\special{pa 4280 4160}%
\special{pa 4286 4160}%
\special{pa 4286 4158}%
\special{pa 4292 4158}%
\special{pa 4292 4156}%
\special{pa 4298 4156}%
\special{pa 4300 4154}%
\special{pa 4304 4154}%
\special{pa 4306 4152}%
\special{pa 4306 4152}%
\special{fp}%
\special{pa 4360 4134}%
\special{pa 4362 4134}%
\special{pa 4362 4132}%
\special{pa 4366 4132}%
\special{pa 4368 4130}%
\special{pa 4372 4130}%
\special{pa 4372 4128}%
\special{pa 4376 4128}%
\special{pa 4378 4126}%
\special{pa 4382 4126}%
\special{pa 4382 4124}%
\special{pa 4386 4124}%
\special{pa 4386 4122}%
\special{pa 4390 4122}%
\special{pa 4392 4120}%
\special{pa 4394 4120}%
\special{pa 4398 4118}%
\special{pa 4400 4118}%
\special{pa 4402 4116}%
\special{pa 4404 4116}%
\special{pa 4406 4114}%
\special{pa 4410 4114}%
\special{pa 4410 4113}%
\special{fp}%
\special{pa 4462 4088}%
\special{pa 4462 4088}%
\special{pa 4464 4086}%
\special{pa 4466 4086}%
\special{pa 4468 4084}%
\special{pa 4470 4084}%
\special{pa 4472 4082}%
\special{pa 4474 4082}%
\special{pa 4476 4080}%
\special{pa 4484 4076}%
\special{pa 4486 4076}%
\special{pa 4490 4072}%
\special{pa 4492 4072}%
\special{pa 4494 4070}%
\special{pa 4498 4068}%
\special{pa 4500 4068}%
\special{pa 4504 4064}%
\special{pa 4506 4064}%
\special{pa 4508 4062}%
\special{pa 4512 4060}%
\special{pa 4514 4058}%
\special{fp}%
\special{pa 4563 4017}%
\special{pa 4612 3977}%
\special{fp}%
\special{pa 4661 3935}%
\special{pa 4709 3895}%
\special{fp}%
\special{pa 4757 3853}%
\special{pa 4768 3842}%
\special{pa 4768 3840}%
\special{pa 4774 3836}%
\special{pa 4774 3834}%
\special{pa 4782 3826}%
\special{pa 4782 3824}%
\special{pa 4790 3816}%
\special{pa 4790 3814}%
\special{pa 4792 3812}%
\special{pa 4792 3810}%
\special{pa 4796 3808}%
\special{pa 4796 3808}%
\special{fp}%
\special{pa 4830 3759}%
\special{pa 4832 3758}%
\special{pa 4832 3756}%
\special{pa 4836 3752}%
\special{pa 4836 3748}%
\special{pa 4840 3744}%
\special{pa 4840 3742}%
\special{pa 4844 3738}%
\special{pa 4846 3734}%
\special{pa 4850 3730}%
\special{pa 4850 3726}%
\special{pa 4854 3722}%
\special{pa 4856 3718}%
\special{pa 4858 3716}%
\special{pa 4858 3714}%
\special{pa 4860 3712}%
\special{pa 4862 3709}%
\special{fp}%
\special{pa 4892 3653}%
\special{pa 4914 3610}%
\special{pa 4914 3606}%
\special{pa 4916 3604}%
\special{pa 4916 3602}%
\special{pa 4918 3600}%
\special{pa 4918 3598}%
\special{pa 4919 3597}%
\special{fp}%
\special{pa 4941 3541}%
\special{pa 4948 3520}%
\special{pa 4952 3508}%
\special{pa 4952 3506}%
\special{pa 4954 3504}%
\special{pa 4960 3481}%
\special{fp}%
\special{pa 4977 3420}%
\special{pa 4980 3402}%
\special{pa 4980 3398}%
\special{pa 4984 3380}%
\special{pa 4984 3378}%
\special{pa 4988 3362}%
\special{pa 4989 3358}%
\special{fp}%
\special{pa 4996 3295}%
\special{pa 4996 3286}%
\special{pa 4998 3272}%
\special{pa 4998 3256}%
\special{pa 5000 3232}%
\special{fp}%
}}%
% FUNC 2 1 3 0 Black White
% 9 2800 2000 5200 4400 4000 3200 5000 3200 4000 2200 0 0 0 0 10 2 0 4
% ///x^2+y^2=0.25///-1.199///1.199
{\color[named]{Black}{%
\special{pn 8}%
\special{pn 8}%
\special{pa 4500 3200}%
\special{pa 4500 3162}%
\special{pa 4498 3162}%
\special{pa 4498 3142}%
\special{pa 4496 3140}%
\special{pa 4496 3140}%
\special{fp}%
\special{pa 4486 3079}%
\special{pa 4484 3074}%
\special{pa 4484 3070}%
\special{pa 4482 3068}%
\special{pa 4482 3062}%
\special{pa 4480 3062}%
\special{pa 4480 3056}%
\special{pa 4478 3054}%
\special{pa 4478 3050}%
\special{pa 4476 3048}%
\special{pa 4476 3044}%
\special{pa 4474 3042}%
\special{pa 4474 3038}%
\special{pa 4472 3036}%
\special{pa 4472 3032}%
\special{pa 4470 3030}%
\special{pa 4470 3028}%
\special{pa 4468 3024}%
\special{pa 4468 3024}%
\special{fp}%
\special{pa 4444 2969}%
\special{pa 4444 2968}%
\special{pa 4442 2968}%
\special{pa 4442 2964}%
\special{pa 4440 2964}%
\special{pa 4440 2960}%
\special{pa 4438 2960}%
\special{pa 4438 2958}%
\special{pa 4436 2956}%
\special{pa 4436 2954}%
\special{pa 4432 2950}%
\special{pa 4432 2946}%
\special{pa 4426 2940}%
\special{pa 4426 2936}%
\special{pa 4422 2932}%
\special{pa 4420 2928}%
\special{pa 4416 2924}%
\special{pa 4416 2922}%
\special{pa 4415 2921}%
\special{fp}%
\special{pa 4377 2871}%
\special{pa 4368 2862}%
\special{pa 4340 2832}%
\special{pa 4338 2832}%
\special{pa 4333 2827}%
\special{fp}%
\special{pa 4286 2790}%
\special{pa 4282 2786}%
\special{pa 4280 2786}%
\special{pa 4276 2782}%
\special{pa 4274 2782}%
\special{pa 4272 2780}%
\special{pa 4270 2780}%
\special{pa 4268 2778}%
\special{pa 4266 2778}%
\special{pa 4260 2772}%
\special{pa 4256 2772}%
\special{pa 4250 2766}%
\special{pa 4246 2766}%
\special{pa 4246 2764}%
\special{pa 4244 2764}%
\special{pa 4242 2762}%
\special{pa 4240 2762}%
\special{pa 4238 2760}%
\special{pa 4236 2760}%
\special{fp}%
\special{pa 4183 2735}%
\special{pa 4180 2734}%
\special{pa 4178 2734}%
\special{pa 4178 2732}%
\special{pa 4172 2732}%
\special{pa 4172 2730}%
\special{pa 4170 2730}%
\special{pa 4146 2722}%
\special{pa 4142 2722}%
\special{pa 4142 2720}%
\special{pa 4136 2720}%
\special{pa 4134 2718}%
\special{pa 4132 2718}%
\special{pa 4129 2717}%
\special{fp}%
\special{pa 4071 2706}%
\special{pa 4068 2706}%
\special{pa 4066 2704}%
\special{pa 4050 2704}%
\special{pa 4050 2702}%
\special{pa 4024 2702}%
\special{pa 4022 2700}%
\special{pa 4012 2700}%
\special{fp}%
\special{pa 3950 2704}%
\special{pa 3940 2704}%
\special{pa 3928 2706}%
\special{pa 3920 2706}%
\special{pa 3920 2708}%
\special{pa 3910 2708}%
\special{pa 3908 2710}%
\special{pa 3900 2710}%
\special{pa 3898 2712}%
\special{pa 3890 2712}%
\special{fp}%
\special{pa 3833 2730}%
\special{pa 3832 2730}%
\special{pa 3830 2730}%
\special{pa 3828 2732}%
\special{pa 3824 2732}%
\special{pa 3822 2734}%
\special{pa 3820 2734}%
\special{pa 3816 2736}%
\special{pa 3814 2736}%
\special{pa 3812 2738}%
\special{pa 3808 2738}%
\special{pa 3808 2740}%
\special{pa 3804 2740}%
\special{pa 3802 2742}%
\special{pa 3800 2742}%
\special{pa 3798 2744}%
\special{pa 3794 2744}%
\special{pa 3794 2746}%
\special{pa 3790 2746}%
\special{pa 3790 2748}%
\special{pa 3786 2748}%
\special{pa 3786 2750}%
\special{pa 3782 2750}%
\special{fp}%
\special{pa 3730 2780}%
\special{pa 3728 2780}%
\special{pa 3724 2784}%
\special{pa 3722 2784}%
\special{pa 3718 2788}%
\special{pa 3716 2788}%
\special{pa 3714 2792}%
\special{pa 3712 2792}%
\special{pa 3710 2794}%
\special{pa 3706 2796}%
\special{pa 3702 2800}%
\special{pa 3698 2802}%
\special{pa 3694 2806}%
\special{pa 3690 2808}%
\special{pa 3681 2817}%
\special{fp}%
\special{pa 3634 2860}%
\special{pa 3632 2862}%
\special{pa 3624 2872}%
\special{pa 3618 2878}%
\special{pa 3618 2880}%
\special{pa 3608 2890}%
\special{pa 3606 2894}%
\special{pa 3602 2898}%
\special{pa 3600 2902}%
\special{pa 3596 2906}%
\special{pa 3595 2908}%
\special{fp}%
\special{pa 3562 2959}%
\special{pa 3562 2962}%
\special{pa 3560 2962}%
\special{pa 3560 2966}%
\special{pa 3558 2966}%
\special{pa 3558 2970}%
\special{pa 3556 2970}%
\special{pa 3556 2972}%
\special{pa 3554 2974}%
\special{pa 3554 2976}%
\special{pa 3552 2978}%
\special{pa 3552 2982}%
\special{pa 3550 2982}%
\special{pa 3550 2986}%
\special{pa 3548 2986}%
\special{pa 3548 2990}%
\special{pa 3546 2990}%
\special{pa 3546 2994}%
\special{pa 3544 2994}%
\special{pa 3544 2998}%
\special{pa 3542 3000}%
\special{pa 3542 3002}%
\special{pa 3540 3004}%
\special{pa 3540 3005}%
\special{fp}%
\special{pa 3520 3062}%
\special{pa 3520 3066}%
\special{pa 3518 3066}%
\special{pa 3518 3070}%
\special{pa 3516 3076}%
\special{pa 3516 3078}%
\special{pa 3514 3084}%
\special{pa 3514 3088}%
\special{pa 3512 3090}%
\special{pa 3512 3098}%
\special{pa 3510 3100}%
\special{pa 3510 3108}%
\special{pa 3508 3110}%
\special{pa 3508 3120}%
\special{pa 3508 3120}%
\special{fp}%
\special{pa 3500 3178}%
\special{pa 3500 3178}%
\special{pa 3500 3222}%
\special{pa 3502 3224}%
\special{pa 3502 3240}%
\special{fp}%
\special{pa 3510 3300}%
\special{pa 3510 3302}%
\special{pa 3512 3302}%
\special{pa 3512 3312}%
\special{pa 3514 3312}%
\special{pa 3514 3320}%
\special{pa 3516 3320}%
\special{pa 3516 3328}%
\special{pa 3518 3328}%
\special{pa 3518 3334}%
\special{pa 3520 3336}%
\special{pa 3520 3342}%
\special{pa 3522 3342}%
\special{pa 3522 3346}%
\special{pa 3524 3352}%
\special{fp}%
\special{pa 3544 3406}%
\special{pa 3546 3408}%
\special{pa 3546 3410}%
\special{pa 3548 3412}%
\special{pa 3548 3414}%
\special{pa 3550 3416}%
\special{pa 3550 3418}%
\special{pa 3552 3420}%
\special{pa 3552 3422}%
\special{pa 3554 3426}%
\special{pa 3554 3428}%
\special{pa 3558 3432}%
\special{pa 3558 3434}%
\special{pa 3560 3436}%
\special{pa 3560 3438}%
\special{pa 3562 3440}%
\special{pa 3562 3442}%
\special{pa 3564 3444}%
\special{pa 3566 3448}%
\special{pa 3566 3450}%
\special{pa 3572 3456}%
\special{pa 3572 3458}%
\special{fp}%
\special{pa 3606 3507}%
\special{pa 3606 3508}%
\special{pa 3608 3510}%
\special{pa 3608 3512}%
\special{pa 3614 3516}%
\special{pa 3614 3518}%
\special{pa 3624 3528}%
\special{pa 3624 3530}%
\special{pa 3632 3538}%
\special{pa 3632 3540}%
\special{pa 3646 3553}%
\special{fp}%
\special{pa 3694 3596}%
\special{pa 3694 3596}%
\special{pa 3698 3598}%
\special{pa 3702 3602}%
\special{pa 3706 3604}%
\special{pa 3712 3610}%
\special{pa 3714 3610}%
\special{pa 3716 3612}%
\special{pa 3718 3612}%
\special{pa 3722 3616}%
\special{pa 3724 3616}%
\special{pa 3728 3620}%
\special{pa 3732 3622}%
\special{pa 3736 3626}%
\special{pa 3740 3626}%
\special{pa 3744 3630}%
\special{fp}%
\special{pa 3799 3658}%
\special{pa 3800 3658}%
\special{pa 3802 3660}%
\special{pa 3806 3660}%
\special{pa 3806 3662}%
\special{pa 3810 3662}%
\special{pa 3812 3664}%
\special{pa 3814 3664}%
\special{pa 3816 3666}%
\special{pa 3820 3666}%
\special{pa 3822 3668}%
\special{pa 3824 3668}%
\special{pa 3828 3670}%
\special{pa 3830 3670}%
\special{pa 3832 3672}%
\special{pa 3836 3672}%
\special{pa 3838 3674}%
\special{pa 3842 3674}%
\special{pa 3844 3676}%
\special{pa 3848 3676}%
\special{pa 3850 3678}%
\special{pa 3853 3678}%
\special{fp}%
\special{pa 3911 3692}%
\special{pa 3914 3692}%
\special{pa 3914 3694}%
\special{pa 3920 3694}%
\special{pa 3934 3696}%
\special{pa 3940 3696}%
\special{pa 3942 3698}%
\special{pa 3962 3698}%
\special{pa 3962 3700}%
\special{pa 3969 3700}%
\special{fp}%
\special{pa 4034 3700}%
\special{pa 4038 3700}%
\special{pa 4040 3698}%
\special{pa 4060 3698}%
\special{pa 4060 3696}%
\special{pa 4074 3696}%
\special{pa 4074 3694}%
\special{pa 4086 3694}%
\special{pa 4088 3692}%
\special{pa 4091 3692}%
\special{fp}%
\special{pa 4149 3678}%
\special{pa 4152 3678}%
\special{pa 4152 3676}%
\special{pa 4158 3676}%
\special{pa 4158 3674}%
\special{pa 4164 3674}%
\special{pa 4164 3672}%
\special{pa 4170 3672}%
\special{pa 4170 3670}%
\special{pa 4172 3670}%
\special{pa 4178 3668}%
\special{pa 4180 3668}%
\special{pa 4180 3666}%
\special{pa 4186 3666}%
\special{pa 4186 3664}%
\special{pa 4190 3664}%
\special{pa 4190 3662}%
\special{pa 4194 3662}%
\special{pa 4196 3660}%
\special{pa 4197 3660}%
\special{fp}%
\special{pa 4248 3636}%
\special{pa 4248 3634}%
\special{pa 4250 3634}%
\special{pa 4252 3632}%
\special{pa 4254 3632}%
\special{pa 4256 3630}%
\special{pa 4260 3628}%
\special{pa 4262 3626}%
\special{pa 4264 3626}%
\special{pa 4266 3624}%
\special{pa 4268 3624}%
\special{pa 4274 3618}%
\special{pa 4276 3618}%
\special{pa 4280 3614}%
\special{pa 4282 3614}%
\special{pa 4284 3612}%
\special{pa 4288 3610}%
\special{pa 4290 3608}%
\special{pa 4292 3608}%
\special{pa 4294 3604}%
\special{pa 4296 3604}%
\special{pa 4297 3603}%
\special{fp}%
\special{pa 4343 3565}%
\special{pa 4368 3540}%
\special{pa 4368 3538}%
\special{pa 4378 3530}%
\special{pa 4378 3528}%
\special{pa 4386 3520}%
\special{fp}%
\special{pa 4420 3472}%
\special{pa 4424 3468}%
\special{pa 4424 3466}%
\special{pa 4426 3464}%
\special{pa 4426 3462}%
\special{pa 4428 3460}%
\special{pa 4430 3456}%
\special{pa 4432 3454}%
\special{pa 4432 3452}%
\special{pa 4434 3450}%
\special{pa 4436 3446}%
\special{pa 4436 3444}%
\special{pa 4440 3440}%
\special{pa 4440 3438}%
\special{pa 4442 3436}%
\special{pa 4442 3434}%
\special{pa 4444 3432}%
\special{pa 4444 3430}%
\special{pa 4446 3428}%
\special{pa 4448 3424}%
\special{pa 4448 3422}%
\special{pa 4450 3422}%
\special{pa 4450 3421}%
\special{fp}%
\special{pa 4470 3370}%
\special{pa 4470 3370}%
\special{pa 4472 3370}%
\special{pa 4472 3364}%
\special{pa 4474 3364}%
\special{pa 4474 3358}%
\special{pa 4476 3358}%
\special{pa 4476 3352}%
\special{pa 4478 3352}%
\special{pa 4478 3346}%
\special{pa 4480 3346}%
\special{pa 4480 3340}%
\special{pa 4482 3338}%
\special{pa 4482 3332}%
\special{pa 4484 3332}%
\special{pa 4484 3324}%
\special{pa 4486 3324}%
\special{pa 4486 3322}%
\special{fp}%
\special{pa 4496 3262}%
\special{pa 4496 3260}%
\special{pa 4498 3260}%
\special{pa 4498 3240}%
\special{pa 4500 3238}%
\special{pa 4500 3202}%
\special{fp}%
}}%
% STR 2 0 3 0 Black White
% 4 5020 3140 5020 3240 1 0 0 0
% $2$
\put(50.2000,-32.4000){\makebox(0,0)[lt]{$2$}}%
% STR 2 0 3 0 Black White
% 4 3970 4130 3970 4230 4 0 0 0
% $-2$
\put(39.7000,-42.3000){\makebox(0,0)[rt]{$-2$}}%
% STR 2 0 3 0 Black White
% 4 2970 3080 2970 3180 3 0 0 0
% $-2$
\put(29.7000,-31.8000){\makebox(0,0)[rb]{$-2$}}%
% STR 2 0 3 0 Black White
% 4 4040 2070 4040 2170 2 0 0 0
% $2$
\put(40.4000,-21.7000){\makebox(0,0)[lb]{$2$}}%
% LINE 2 0 3 0 Black White
% 2 3000 4200 5000 2200
% 
{\color[named]{Black}{%
\special{pn 8}%
\special{pa 3000 4200}%
\special{pa 5000 2200}%
\special{fp}%
}}%
% STR 2 0 3 0 Black White
% 4 4740 2580 4740 2680 2 0 0 0
% 
\put(47.4000,-26.8000){\makebox(0,0)[lb]{}}%
% DOT 0 0 3 0 Black White
% 1 4350 2850
% 
{\color[named]{Black}{%
\special{pn 4}%
\special{sh 1}%
\special{ar 4350 2850 16 16 0  6.28318530717959E+0000}%
}}%
% DOT 0 0 3 0 Black White
% 1 4180 3020
% 
{\color[named]{Black}{%
\special{pn 4}%
\special{sh 1}%
\special{ar 4180 3020 16 16 0  6.28318530717959E+0000}%
}}%
% ELLIPSE 2 0 3 0 Black White
% 4 4320 2990 4180 2840 4310 2840 4110 2970
% 
{\color[named]{Black}{%
\special{pn 8}%
\special{ar 4320 2990 140 150  3.2318232  4.6410815}%
}}%
% SARROW 2 0 3 1 Black White
% 2 4296 2842 4310 2840
% 
{\color[named]{Black}{%
\special{pn 8}%
\special{pa 4296 2842}%
\special{pa 4310 2840}%
\special{fp}%
\special{sh 1}%
\special{pa 4310 2840}%
\special{pa 4242 2830}%
\special{pa 4258 2848}%
\special{pa 4248 2870}%
\special{pa 4310 2840}%
\special{fp}%
}}%
% CIRCLE 2 0 3 0 Black White
% 4 4000 3200 4150 3200 4150 3200 4180 3020
% 
{\color[named]{Black}{%
\special{pn 8}%
\special{ar 4000 3200 150 150  5.4977871  6.2831853}%
}}%
% STR 2 0 3 0 Black White
% 4 4170 3100 4170 3200 2 0 0 0
% $\frac{\pi}{4}$
\put(41.7000,-32.0000){\makebox(0,0)[lb]{$\frac{\pi}{4}$}}%
% LINE 3 0 3 0 Black White
% 2 4200 3020 2900 2820
% 
{\color[named]{Black}{%
\special{pn 4}%
\special{pa 4200 3020}%
\special{pa 2900 2820}%
\special{fp}%
}}%
% LINE 3 0 3 0 Black White
% 2 4360 2850 5100 2660
% 
{\color[named]{Black}{%
\special{pn 4}%
\special{pa 4360 2850}%
\special{pa 5100 2660}%
\special{fp}%
}}%
% STR 2 0 3 0 Black White
% 4 5320 2480 5320 2580 5 0 0 0
% $b_{\frac{1}{2},\frac{\pi}{4}}(2)$
\put(53.2000,-25.8000){\makebox(0,0){$b_{\frac{1}{2},\frac{\pi}{4}}(2)$}}%
% STR 2 0 3 0 Black White
% 4 2660 2660 2660 2760 5 0 0 0
% $b_{\frac{1}{2},\frac{\pi}{4}}(1)$
\put(26.6000,-27.6000){\makebox(0,0){$b_{\frac{1}{2},\frac{\pi}{4}}(1)$}}%
% STR 2 0 3 0 Black White
% 4 4540 3140 4540 3240 1 0 0 0
% $1$
\put(45.4000,-32.4000){\makebox(0,0)[lt]{$1$}}%
% STR 2 0 3 0 Black White
% 4 3460 3080 3460 3180 3 0 0 0
% $-1$
\put(34.6000,-31.8000){\makebox(0,0)[rb]{$-1$}}%
% STR 2 0 3 0 Black White
% 4 3970 3640 3970 3740 4 0 0 0
% $-1$
\put(39.7000,-37.4000){\makebox(0,0)[rt]{$-1$}}%
% STR 2 0 3 0 Black White
% 4 4040 2550 4040 2650 2 0 0 0
% $1$
\put(40.4000,-26.5000){\makebox(0,0)[lb]{$1$}}%
% STR 2 0 3 0 Black White
% 4 2660 2100 2660 2200 5 0 0 0
% $B(2)$
\put(26.6000,-22.0000){\makebox(0,0){$B(2)$}}%
% LINE 3 0 3 0 Black White
% 2 3360 2430 2810 2240
% 
{\color[named]{Black}{%
\special{pn 4}%
\special{pa 3360 2430}%
\special{pa 2810 2240}%
\special{fp}%
}}%
% LINE 3 0 3 0 Black White
% 2 4470 3020 5260 2880
% 
{\color[named]{Black}{%
\special{pn 4}%
\special{pa 4470 3020}%
\special{pa 5260 2880}%
\special{fp}%
}}%
% STR 2 0 3 0 Black White
% 4 5430 2750 5430 2850 5 0 0 0
% $B(1)$
\put(54.3000,-28.5000){\makebox(0,0){$B(1)$}}%
\end{picture}}%